\documentclass [reqno, oneside, , 12pt]{amsart}
\usepackage{amsmath,amssymb,amsfonts,amsthm,enumerate,color,datetime}
\usepackage{graphicx}

\input epsf

\newtheorem{theorem}{Theorem}[section]
\newtheorem*{theoremA*}{Theorem A}
\newtheorem*{theoremB}{Theorem B}
\newtheorem{proposition}[theorem]{Proposition}
\newtheorem{lemma}[theorem]{Lemma}
\newtheorem{corollary}[theorem]{Corollary}
\newtheorem{definition}[theorem]{Definition}

\newtheorem{remark}[theorem]{Remark}
\newtheorem{remarks}[theorem]{Remarks}

\numberwithin{equation}{section}

\theoremstyle{definition}

\begin{document}

\newcommand{\namelistlabel}[1] {\mbox{#1}\hfil}
\newenvironment{namelist}[1]{%
\begin{list}{}
{\let\makelabel\namelistlabel
\settowidth{\labelwidth}{#1}
\setlength{\leftmargin}{1.1\labelwidth}}
}{%
\end{list}}

\def\C{{\mathbb C}}
\def\Cs{\mathcal C}
\def\D{{\mathbb D}}
\def\E{{\mathcal E}}
\def\F{{\mathcal F}}
\def\K{{\mathcal K}}
\def\L{{\mathcal L}}
\def\M{{\mathcal M}}
\def\N{{\mathbb N}}
\def\T{{\mathcal T}}
\def\Sc{{\mathcal S}}
\def\R{{\mathbb R}}
\def\Z{\mathbb Z}
\def \RE {\Re\text{\rm e}}
\def \IM {\Im\text{\rm m}}
\def\i{{\mathbf i}}
\def\j{{\mathbf j}}
\def\x{{\mathbf x}}
\def\y{{\mathbf y}}
\def\u{{\mathbf u}}
\def\v{{\mathbf v}}
\def\s{{\mathbf s}}
\def\t{{\mathbf t}}
\def\0{{\mathbf 0}}
\def\balpha{{\bar\alpha}}
\def\L{\ell}

\def \lan {\langle}
\def \ran {\rangle}
\def \al {\alpha}
\def \la {\lambda}
\def \ph {\varphi}
\def \del {\delta}
\def \de {\partial}
\def \inv{^{-1}}
\def \supp {\text{\rm supp\,}}
\def \<  {{[\![}}
\def \> {{]\!]}}
\newcommand{\lsup}[2]{\ensuremath{{}^{#1}\!{#2}}}   
\def\begeq{\begin{equation}\begin{aligned}}
\def\endeq{\end{aligned}\end{equation}}
\def \eps {\varepsilon}
\def \cJ {\mathcal J}

\title[Flag Kernels on Homogeneous Groups]{Singular Integrals with Flag Kernels\\ on Homogeneous Groups:  I}

\author{Alexander Nagel, Fulvio Ricci, Elias M.~Stein, Stephen Wainger}

\begin{abstract}Let $\mathcal K$ be a flag kernel on a homogeneous nilpotent Lie group $G$. We prove that operators $T$ of the form $T(f)= f*\mathcal K$ form an algebra under composition, and that such operators are bounded on $L^{p}(G)$ for $1<p<\infty$.
\end{abstract}



\address{
Department of Mathematics\\ 
University of Wisconsin\\ 
Madison, WI
53706}
\email{nagel@math.wisc.edu}

\address{
Scuola Normale Superiore\\ Piazza dei Cavalieri 7\\ 56126 Pisa\\
Italy}
\email{fricci@sns.it}

\address{
Department of Mathematics\\ Princeton University\\ Princeton, NJ
08544, USA}
\email{stein@math.princeton.edu}

\address{
Department of Mathematics\\ 
University of Wisconsin\\ 
Madison, WI
53706}
\email{wainger@math.wisc.edu}

\maketitle
{\small
\setcounter{tocdepth}{1} \tableofcontents}

\thispagestyle{empty}

\section{Introduction}\label{Intro}

\allowdisplaybreaks{
This is the first of two papers dealing with singular integral operators with flag kernels on homogeneous nilpotent groups.  Our goal is to show that these operators, along with  appropriate sub-collections, form algebras under composition, and that the operators in question are bounded on $L^p$.  

Operators of this kind arose initially when studying compositions of sub-elliptic operators on the Heisenberg group (such as the sub-Laplacian $\mathcal{L}$ and $\square_b$) with elliptic-type operators.  In particular in \cite{MuRiSt95} one saw that operators of the form $m ( \mathcal{L}, iT)$, (where $m$ is a ``Marcinkiewicz multiplier'') are singular integrals with flag kernels and satisfy $L^p$ estimates.  The theory was extended in \cite{NaRiSt00} to encompass general flag kernels in the Euclidean space $\mathbb{R}^N$, and the resulting operators arising via abelian convolution.  In addition, aspects of the CR theory for quadratic manifolds could be studied via such operators on various step-2 groups.  More recently, flag kernels have been studied in \cite{MR2499336} and \cite{MR2591640}. In view of this, and because of their potential further application, it is desirable to extend the above results in \cite{NaRiSt00} to the setting of homogeneous groups of higher step. To achieve this goal requires however that we substantially recast the approach and techniques used previously, since these were essentially limited to the step 2 case.

Our main results are two-fold.  Suppose $G$ is a homogeneous nilpotent group and $\mathcal{K}$ denotes a distribution on $G$ which is a flag kernel (the requisite definitions are given below in Definition \ref{Def2.1}).


\begin{theoremA*}
The operators $T$ of the form $T ( f
) = f \, \ast \, \mathcal{K}$ form an algebra under composition.
\end{theoremA*}


\begin{theoremB}
The above operators are bounded on
$L^p ( G )$ for $1 < p < \infty$.
\end{theoremB}


Given the complexity of the material, in this introductory section we provide the reader with an outline of the main ideas that enter in the proofs of the above theorems.  Moreover, in order to simplify the presentation we will often not state matters in the most general setting and sometimes describe the situation at hand a little imprecisely.


\subsection{Flag kernels}\label{FlagKernelsA}\quad

\smallskip

We start with a direct sum decomposition $\mathbb{R}^N = \mathbb{R}^{a_1} \oplus  \,\cdots \, \oplus \,\mathbb{R}^{a_n}$, with  $\sum_{j=1}^{n}a_{j} = N$, and we write $\x = ( \x_1 , \x_2, \cdots ,\x_n )$, with $\x_m \in \mathbb{R}^{a_m}$.  We also fix a one-parameter family of dilations $\delta_r$ on $\R^{N}$, given by $\delta_r ( \x ) = ( r^{d_1} \x_1 , \ldots , r^{d_n} \x_n )$, with positive exponents $d_1 < d_2 \cdots < d_k$.\footnote{One can also allow non-isotropic dilations on each subspace $\R^{a_{l}}$. See Section \ref{Dilations} below.} We denote by $Q_k = d_k a_k$ the homogeneous dimension of $\mathbb{R}^{a_k}$.  We also define the partial ``norms'' $N_k ( \x ) = | \x_k |_{e}^{1/d_k}$ where $|\x_{k}|_{e}$ is the standard Euclidean norm on $\R^{a_{k}}$.

In this setting, a flag kernel $\mathcal{K}$ is a distribution on $\mathbb{R}^N$ which is given by integration against a $C^\infty$ function $K ( \x )$ away from $\x_1 = 0$ and which satisfies two types of conditions.  The first are the differential inequalities for $\x_1 \ne 0$: 
\begin{equation}\label{Eqn1.1}
| \partial^\alpha_\x \, K ( \x ) | \, \leq \, C_\alpha \, \prod_{k = 1}^{n} \, ( N_1 ( \x ) \, + \, N_2 ( \x ) \cdots \, + \, N_k ( \x ))^{-Q_k - d_k \alpha_k}
\end{equation}
with $\alpha = ( \alpha_1 , \ldots , \alpha_n )$.  The second are the cancellation conditions.  These are most easily expressed recursively. Let $\langle \K,\varphi\rangle$ denote the action of the distribution $\K$ on a test function $\varphi$. At the beginning of the recursion there is the following condition, in many ways typical of the others:
\begin{equation}\label{Eqn1.2}
\sup\limits_{R} | \langle \mathcal{K} , \varphi_R \rangle | < \infty
\end{equation}
where $R = ( R_1 , R_2 , \cdots R_n )$,  $\varphi_R ( x ) = \varphi( R^{d_1}_1 x_1$,  $R^{d_2}_2 x_2 ,  \cdots R^{d_n}_n x_n )$, and $\varphi$ is an arbitrary $C^\infty$ function which is supported in the unit ball.  More generally, one requires that the action of $\K$ on a test function in some subset of variables $\{\x_{m_{1}}, \ldots, \x_{m_{\beta}}\}$ produces a flag kernel in the remaining variables $\{\x_{l_{1}}, \ldots, \x_{l_{\alpha}}\}$. The precise formulation of these conditions is given in Section \ref{Dilations} and Definition \ref{Def2.1} below.

\smallskip

\subsection{Dyadic decomposition}\label{Sect1.2pjv}\quad

\smallskip

A main tool used in studying flag kernels is their dyadic decomposition into sums of ``bump functions''.  This proceeds as follows. Let $I = ( i_1 , i_2 \cdots , i_n)$ denote any indexing set of integers that satisfies.
\begin{equation}\label{Eqn1.3}
i_1 \, \leq \, i_2 \,\leq  \cdots \, \leq i_{n-1} \, \leq \, i_n \, .
\end{equation}
Also let $\{ \varphi^I \}$ be a family of $C^\infty$ functions supported in the unit ball  that are uniformly bounded in the $C^{(m)}$ norm for each $m$.  Set 
\begin{equation*}
[\varphi^{I}]_I ( \x ) = 2^{- i_1 Q_1 \, - i_2 Q_2 \, \cdots \, - i_n Q_n} \:
\varphi^I ( 2^{-d_1 i_1} \x_1 , \cdots 2^{-d_k i_n} \x_n ) \, ,
\end{equation*}
so that the $[\varphi^{I}]_I$ are $L^1$-normalized.  We say that the $\varphi^{I}$ satisfy the ``strong cancellation'' condition if for each $k$ with $1 \leq k \leq n$, 
\begin{equation}\label{Eqn1.4}
\int\limits_{\mathbb{R}^{a_k}} \, \varphi^{I}(\x_{1}, \ldots, \x_{k},\ldots, \x_{n}) \, d \x_k \equiv
0
\end{equation}
when all the inequalities (\ref{Eqn1.3}) for $I$ are strict.  In the case that there are some equalities in (\ref{Eqn1.3}), say $i_{\ell-1}<i_\ell = i_{\ell + 1} = \cdots = i_k<i_{k+1}$, then only cancellation in the collection of corresponding variables is required: 
\begin{equation*}\tag{\ref{Eqn1.4}$^{\prime}$}
\int\limits_{\mathbb{R}^{a_\ell} \oplus \cdots \oplus \,\mathbb{R}^{a_k}}  \varphi^{I} (\x_{1}, \ldots, \x_{\ell}, \ldots, \x_{k},\ldots, \x_{n})\, d\x_\ell \,  \cdots \, d\x_k\, \equiv \, 0.
\end{equation*}
The first result needed is that any sum
\begin{equation}\label{Eqn1.5}
\sum\limits_{I} \, [\varphi^{I}]_{I}
\end{equation}
made up of such bump functions, with cancellation condition (\ref{Eqn1.4}) and (\ref{Eqn1.4}$^{\prime}$), converges in the sense of distributions to a flag kernel, and conversely, any flag kernel $\mathcal{K}$ can be written in this way (of course, not uniquely).

There are two parts to this result (which in effect is stated but not proved completely in \cite{NaRiSt00}).  The first is that the sum in (\ref{Eqn1.5}) is indeed a flag kernel.  To see this, one can use the estimate in Proposition \ref{Prop2.11} given in Appendix II below; one also notes from this that even without the cancellation conditions (\ref{Eqn1.4}) and (\ref{Eqn1.4}$^{\prime}$),  the sum (\ref{Eqn1.5}) satisfies the differential inequalities (\ref{Eqn1.1}). The converse part requires Theorem \ref{Lemma2.3} below, and the observation that the parts of the sum (\ref{Eqn1.5}) contributed by $I$'s where there may be equality in (\ref{Eqn1.3}) give flag distributions corresponding to various ``coarser'' flags.

However, what will be key in what follows is that the strong cancellation conditions (\ref{Eqn1.4}) or (\ref{Eqn1.4}$^{\prime}$) can be weakened, and still lead to the same conclusion.  While these ``weak'' cancellation conditions are somewhat complicated to state (see Definition \ref{Def1.14a2} below), they are easily illustrated in the special step 2 case.  Here we have the decomposition  $\mathbb {R}^n = \mathbb{R}^{a_1} \oplus \mathbb{R}^{a_2}$, $\x = ( \x_1 , \x_2)$. The cancellation condition for the second variable is as before:  $\int \varphi^{I}( \x_1 , \x_2 ) d\x_2 \equiv 0$.  For $\x_1$ the weak cancellation condition takes the form 
\begin{equation}\label{Eqn1.6}
\displaystyle{\int\limits_{\mathbb{R}^{a_1}}} \, \varphi^{I} ( \x_1 , \x_2 ) \, d\x_1 = 2^{-\epsilon ( i_2 - i_1)} \eta^{I} ( \x_2),
\end{equation}
for some $\epsilon >0$, with $I = ( i_1 , i_2)$ and $\eta^{I}$ an $L^1$ normalized bump in the $\x_2$ variable.

In this context the main conclusion (Theorem \ref{Lemma5.4zw}) is that the sum (\ref{Eqn1.5}) is still a flag kernel if the weak-cancellation conditions are assumed instead of (\ref{Eqn1.4})  and (\ref{Eqn1.4}$^{\prime}$), and the functions $\{\varphi^{I}\}$ are allowed to belong to the Schwartz class instead of being compactly supported.  In understanding Definition \ref{Def1.14a2}, one should keep in mind that conditions like (\ref{Eqn1.4}) which involve vanishing of integrals are equivalent with expressions of the $\varphi^{I}$ as the sums of appropriate derivatives. (This is established in  Lemma \ref{Lemma1.12}). 

\smallskip

\subsection{Other properties of flag kernels}\quad

\smallskip

Along with the results about decompositions of flag kernels, there are a number of other properties of these distributions that are worth mentioning and are discussed in Section \ref{FlagKernels}. First, the class of flag kernels is invariant under the change of variables compatible with the structure of the flags.  We have in mind transformations $\x \mapsto \y = F ( \x )$, with  $\y_k = \x_k + P_k ( \x )$, and $P_k$ a homogeneous  polynomial of $\x_1 , \ldots \x_{k - 1}$, of the same degree as $\x_k$.  The fact that $\mathcal{K} \circ F$ satisfies the same differential inequalities (\ref{Eqn1.1}) as $\mathcal{K}$ is nearly obvious, but the requisite cancellation conditions (such as (\ref{Eqn1.2})) are more subtle and involve the weak cancellation of the bump functions.  (See Theorem \ref{Thm3.12}.)

A second fact is that the cancellations required in the definition of a flag kernel can be relaxed.  For example, assuming that the differential inequalities (\ref{Eqn1.1}) hold, then the less restrictive version of (\ref{Eqn1.2}) requires that the supremum is taken only over those $R$ for which $R_1 \geq R_2 \cdots \geq R_n > 0$.   The formulation and proof of the sufficiency of these restricted conditions is in Theorem \ref{restricted}.

Finally we should point out that at the basis of many of our arguments is an earlier characterization in \cite{NaRiSt00} of flag kernels in terms of their Fourier transforms: these are bounded multipliers that satisfy the dual differential inequalities given in Definition \ref{Def2.4}.

\smallskip

\subsection{Graded groups and compositions of flag kernels}\quad

\smallskip

Up to this point our discussion of flag kernels has focused on their definition as distributions on the Euclidean space $\mathbb{R}^N$.  We now consider convolutions with flag kernels on graded nilpotent Lie groups $G$ whose underlying space is $\mathbb{R}^N$. The choice of an appropriate coordinate system on the group $G$, and its multiplication structure, induces a decomposition $\mathbb{R}^N = \mathbb{R}^{a_1} \oplus \cdots \oplus \mathbb{R}^{a_n}$ and allows us to find exponents $d_1 < d_2 \cdots < d_n$ as above so that the dilation $\delta_r ( \x ) = ( \delta^{r_1} \x_1 , \cdots , \delta^{r_n} \x_{n} )$, $\delta > 0$, are automorphisms of $G$.

The proof of Theorem A reduces to the statement that if $\mathcal{K}_1$ and $\mathcal{K}_2$ are a pair of flag kernels, then  $\mathcal{K}_1 \ast \mathcal{K}_2$ is a sum of flag kernels, where the convolution is taken with respect to $G$.  Note that when $G$ is the abelian group $\mathbb{R}^N$, the result follows immediately from the characterization of flag kernels in terms of their Fourier transforms, cited earlier, and in fact the convolution of two flag kernels is a single flag kernel.  In the non-commutative case the proof is not as simple and proceeds as follows.  First write $\mathcal{K}_1 \, = \,  {\sum_{I}} \,[ \varphi^{I}]_{I}$,\,\, $\mathcal{K}_2 \, = \, {\sum_{J}} \, [\psi^{J}]_{J}$ in terms of decompositions with bump functions with strong cancellation.  Now formally
\begin{equation}
\mathcal{K}_1 \ast \mathcal{K}_2 \, = \, \displaystyle{\sum\limits_{I}} \,\displaystyle{\sum\limits_{J}}  [\varphi^{I}]_{I} \, \ast \, [\psi^{J}]_{J} \, .
\end{equation}
We look first at an individual term $[\varphi^{I}]_{I} \ast [\psi^{J}]_{J}$ in the above sum.  It has three properties:
\begin{enumerate}[(a)]

\smallskip

\item $[\varphi^{I}]_{I} \ast [\psi^{J}]_{J} = [\theta^{I,J}]_{K}$ with $[\theta^{I,J}]_{K}$ a ``bump'' scaled according to $K$, where $K = I \vee J$; that is $K = ( k_1 , k_2 , \cdots k_n)$, and $k_m = \max ( i_m , j_m )$ $1 \leq m \leq n$.  This conclusion holds even if we do not assume the cancellation conditions on $[\varphi^{I}]_{I}$ and $[\psi^{J}]_{J}$.

\smallskip

\item Next, because we do have the cancellation conditions (\ref{Eqn1.4}) and (\ref{Eqn1.4}$^\prime$), we have a gain:  There exists $\epsilon > 0$ so that  $[\theta^{I,J}]_{K}$ can be written as a finite sum of terms of the form 
\begin{equation*}
\prod_{l\in A}2^{-\epsilon|i_{l}-j_{l}|}\,\prod_{m\in B}2^{-\epsilon[(i_{m+1}-i_{m})+(j_{m+1}-j_{m})}\,[\widetilde \theta^{I,J}]_{K}
\end{equation*}
where $[\widetilde\theta^{I,J}]_{K}$ is another bump function scaled according to $K$, and $A$ and $B$ are disjoint sets with $A \cup B =\{1, \ldots, n\}$.

\smallskip

\item[(c)] The strong cancellation fails in general for $[\theta^{I,J}]_{K}$, but weak cancellation holds.

\end{enumerate}

The statements (a), (b), and (c) above are contained in Lemmas \ref{Lemma5.1} and \ref{Lemma4.3qw}.  With these assertions proved, one can proceed roughly\footnote{There are actually additional complications. We must first make a preliminary partition of the set of all pairs $(I,J)$, and the result is that $\K_{1}*\K_{2}$ is actually a finite sum of flag kernels.} as follows. We define $\tilde{\theta}_K = \displaystyle{\sum\limits_{I \vee J = K}} \, [\varphi^{I}]_{I} \ast [\varphi^{J}]_J,$ where the sum is taken over all pairs $(I , J)$ for which $I \vee J = K$.  Because of the exponential gain given in (b) this sum converges to a $K$-scaled bump function. Moreover, because of (c), $\tilde{\theta}_K $ satisfies the weak cancellation property.  As a result, the sum ${\sum_{K}} \tilde{\psi}_K$ converges to a flag kernel, and hence $\mathcal{K}_1 \ast \mathcal{K}_2$ is a flag kernel as was to be shown.

We comment briefly on the arguments needed to establish (b) and (c). Here we use the strong cancellation properties of $[\varphi^{I}]_{I}$ (or $[\psi^{J}]_{J}$).  For (b) we express $[\varphi^{I}]_{I}$ as a sum of derivatives with respect to appropriate coordinates, then re-express these in terms of left-invariant vector fields, and finally pass these differentiation to $[\psi^{J}]_{J}$.  The reverse may be done starting with cancellation of $[\psi^{J}]_{J}$.  To obtain (c), the weak cancellation of $[\varphi^{I}]_{I} \ast [\psi^{J}]_{J}$, we begin the same way, but express $[\varphi^{I}]_{I}$ in terms of right-invariant vector fields and then pass these differentiations on the resulting convolution products.  The mechanism underlying this technique is set out in the various lemmas of Section \ref{VectorFields}.

The argument is a little more complex when we are in the case of equality for some of the indices's that arise in $I$ or $J$.  This  in effect involves convolutions with kernels belonging to coarser flags.  The guiding principle for convolutions of such bump functions (or kernels) is that if $\mathcal{K}_j$ are flag kernels corresponding to the flag $\mathcal{F}_j$, $j = 1,2$, then $\mathcal{K}_1 \ast \mathcal{K}_2$ is a flag kernel for the flag $\mathcal{F}$ which is the coarsest flag that is finer than $\mathcal{F}_1$ and $\mathcal{F}_2$. The combinatorics involved are illustrated by several examples given in Sections \ref{Example} and  \ref{FlagAlgebra}.

\smallskip

\subsection{$L^p$ estimates via square functions}\quad

\smallskip

The proof of the $L^p$ estimates (Theorem (\ref{Theorem6.14})) starts with the descending chain of sub-groups $G = G_1 \supset G_2 \supset \cdots \supset G_n$  where  $$G_m = \big\{ \x = ( \x_1 , \x_2 , \cdots \x_n ) : \x_1 = 0,\,\x_2 = 0 \,, \ldots, \,\x_{m-1} = 0\},$$ when $m \geq 2$.  We observe that the dilation's $\delta_r$ restrict to automorphisms of the $G_m$.  We then proceed as follows: 
\begin{enumerate}[(i)]

\item The standard (one-parameter) maximal functions and square functions on each group $G_m$, as given in [FS], are then ``lifted'' (or ``transferred'') to the group $G$.

\smallskip

\item Compositions of those lifted objects lead to ($n$-parameter) maximal functions and square functions on $G$.  Among these is the ``strong'' maximal function
\begin{equation*}
M ( f ) ( \x ) \, = \, \sup \, \frac{1}{m ( R_s)} \, \displaystyle{\int_{R_s}} \, | f ( \x \y^{-1} ) | d\y
\end{equation*}
for which one can prove vector-valued $L^p$ inequalities.  Here $R_s = \{ \x: | \x_k | \leq s_k^{d_k}\}$, with $(s_1 , \cdots , s_n)$ restricted to $s_1 \leq s_2 \cdots \leq s_k$.  There are also a pair of square functions, $S$ and $\mathfrak{S}$, with the property that
\begin{equation}\label{Eqn1.8}
\parallel f \parallel_{L^p} \leq A_p \parallel S ( f ) \parallel_{L^p} \ \mbox{and} \ \parallel \mathfrak{S} ( f ) \parallel_{L^p} \, \leq \, A^\prime_p \parallel f \parallel_{L^p} \, , \ \mbox{for} \ 1 < p < \infty \, .
\end{equation}

\smallskip

\item The connection of these square functions with our operators $T$, given by $T f = f \ast \mathcal{K}$ with $\mathcal{K}$ a flag kernel, comes about because of the point-wise estimate:
\begin{equation}\label{Eqn1.9}
S ( T f ) ( x ) \leq c \,\mathfrak{S} ( f ) ( x ) \, ,
\end{equation}
which is Lemma \ref{Lemma6.13w} below. 
\end{enumerate} 
Now (\ref{Eqn1.8}) together with (\ref{Eqn1.9}) prove the $L^p$ boundedness of our operators.

Among the ideas used to prove (\ref{Eqn1.9}) is the notion of a ``truncated'' flag kernel: such a kernel is truncated at ``width a'', $a \geq 0$, if it  satisfies the conditions such as (\ref{Eqn1.1}), but with $N_1 + \cdots + N_k$ replaced by $a + N_1 + \cdots + N_k$ throughout, (see Definition \ref{DefTruncated}). A key fact that is exploited is that a convolution of a bump of width $b$ with a truncated kernel of width $a$ yields a truncated kernel of width $a + b$.  For this, see Theorem \ref{Theorem6.7}, and its consequence, Theorem \ref{Theorem6.9}.

\subsection{Final remarks}\quad

\smallskip
  
The collection of operators with flag kernels contains both the automorphic (non-isotropic) Calder\'{o}n-Zygmund operators as well as the usual isotropic Calder\'{o}n-Zygmund operators with kernels of compact support, (broadly speaking, the standard pseudo-differential operators of order $0$).  But flag kernels, by their definition, may have singularities away from the origin.  Thus the algebra  we are considering consists of operators that are not necessarily pseudo-local. The study of a narrower algebra that arises naturally, which consists of pseudo-local operators and yet contains both types of Calder\'{o}n-Zygmund operators, will be the subject of the second paper \cite{NRSW2} in this series.

The authors are grateful to Brian Street for conversations and suggestions about the decomposition of flag kernels into sums of dilates of compactly supported functions. We would also like to thank the referee for a very careful reading of the paper. We note that the topic of this paper was the subject of several lectures given by one of us (EMS), in particular at a conference in honor of F.~Treves at Rutgers, April, 2005, and at Washington University and U.C.L.A in April and October 2008. During the preparation of this paper we learned of the work of G\l owacki  \cite{MR2602167}, \cite{MR2679042}, and \cite{Glow2010} where overlapping results are obtained by different methods. We should also mention a forthcoming paper of Brian Street \cite{Street2010} that deals with the $L^2$-theory in a more general context than is done in the present paper.

\section{Dilations and flag kernels on $\R^{N}$}\label{Dilations}

Throughout this paper we shall use standard multi-index notation.  $\Z$ denotes the set of integers and $\N$ denotes the set of non-negative integers. If $\alpha = (\alpha_{1}, \ldots, \alpha_{N})\in \mathbb N^{N}$, then   $|\alpha| = \alpha_{1}+ \cdots + \alpha_{N}$ and $\alpha!= \alpha_{1}!\cdots \alpha_{N}!$. If $\x=(x_{1}, \ldots, x_{N}) \in \R^{N}$, then $\x^{\alpha}= x_{1}^{\alpha_{1}}\cdots x_{N}^{\alpha_{N}}$. For $1\leq j \leq N$,  $\partial_{x_{j}}$ (or more simply $\partial_{j}$) denotes the differential operator $\frac{\partial}{\partial x_{j}}$. If $\alpha\in \N^{N}$ then $\partial^{\alpha}$ denotes the partial differential operator $\partial^{\alpha_{1}}_{1}\cdots \partial^{\alpha_{N}}_{N}$. 

The space of infinitely differentiable real-valued functions on $\R^{N}$ with compact support is denoted by $\mathcal C^{\infty}_{0}(\R^{N})$  and the space of Schwartz functions is denoted by $\mathcal S(\R^{N})$. The basic semi-norms on these spaces are defined as follows: 
\begin{equation*}\label{E11}
\begin{aligned}
&&&\text{if $\varphi\in \Cs^{\infty}_{0}(\R^{N})$,}& ||\varphi||_{(m)} &= \sup\left\{\big\vert\partial^{\alpha}_{\x}\varphi(\x)\big\vert\,\Big\vert\,|\alpha|\leq m,\,\,\x\in\R^{N}\right\};&&\\
&&&\text{if $\varphi\in \mathcal S(\R^{N})$,}&||\varphi||_{[M]}&= \sup\left\{ \big\vert (1+|\x|_{e})^{\alpha}\partial^{\beta}_{\x}\varphi(\x)\big\vert\,\Big\vert\, |\alpha|+|\beta|\leq M,\, \x\in \R^{N}\right\}.&&
\end{aligned}
\end{equation*}
Here $|\x|_{e}$ denotes the usual Euclidean length of $\x\in\R^{N}$.

\subsection{The basic family of dilations}\quad

\smallskip

Fix positive real numbers $0<d_{1}\leq d_{2}\leq \cdots \leq d_{N}$, and define a one-parameter family of dilations on $\R^{N}$ by setting
\begin{equation}\label{E2}
\delta_{r}[\x]= r\cdot \x = \big(r^{d_{1}}x_{1}, \ldots, r^{d_{N}}x_{N}\big).
\end{equation}
Also fix a  smooth homogeneous norm $|\x|$ on $\R^{N}$ so that $|r\cdot \x|=r\,|\x|$. The {homogeneous ball} of radius $r$ is $B(r)= \{\x\in\R^{N}\,\big\vert\,|\x|<r\}$, and the {homogeneous dimension} of $\R^{N}$ (relative to this family of dilations) is $Q = d_{1}+ \cdots + d_{N}$. Recall that $|\x|_{e}=\sqrt{x_{1}^{2}+\cdots +x_{N}^{2}}$ denotes the ordinary Euclidean length of a vector $\x\in \R^{N}$.  If $m(\x) = c\,\x^{\alpha}=cx_{1}^{\alpha_{1}}\cdots x_{N}^{\alpha_{N}}$ is a monomial, then $m(r\cdot \x) = r^{\alpha_{1}d_{1}+ \cdots \alpha_{N}d_{N}}m(\x)$, and the {homogeneous degree} of $m$ is $\Delta(m) = \alpha_{1}d_{1}+ \cdots \alpha_{N}d_{N}$. In particular, the homogeneous degree of a constant is zero. We shall agree that if the homogeneous degree of a monomial is negative, the monomial itself must be identically zero. With this convention, if $m$ is any monomial, we have
\begin{equation}\label{E6}
\Delta(\partial_{{j}}m) = \Delta(m)-d_{j}.
\end{equation}
We denote by $\mathcal H_{d}$ the space of real-valued polynomials which are sums of monomials of homogeneous degree $d$.  We have the following easy result.

\begin{proposition}\label{Prop2.1mn}
If $P$ is a polynomial, then $P\in \mathcal H_{d}$ if and only if $P(\x)=\sum_{\alpha\in \mathfrak H_{d}}c_{\alpha}\x^{\alpha}$ where $\mathfrak H_{d}= \Big\{\alpha=(\alpha_{1}, \ldots, \alpha_{N}) \in \mathbb N^{n}\,\Big\vert\, \sum_{j=1}^{N}\alpha_{j}d_{j}=d\Big\}$. Moreover:
\begin{enumerate}[{\rm(1)}]

\smallskip

\item \label{E8} if $P\in \mathcal H_{d}$ then $P(r\cdot \x)= r^{d}\,P(\x)$;

\smallskip

\item \label{E9} if $P\in \mathcal H_{d_{1}}$ and $Q\in \mathcal H_{d_{2}}$, then $PQ\in \mathcal H_{d_{1}+d_{2}}$;

\smallskip

\item \label{E.10} if $P\in \mathcal H_{d}$ then $\partial_{k}(P)(\x)\equiv 0$ if $d_{k}>d$.
\end{enumerate}
\end{proposition}

\subsection{Standard flags and flag kernels in $\R^{N}$}\quad

\smallskip

If $X$ is an $N$-dimensional vector space, an \emph{$n$-step} flag in $X$ is a collection of subspaces $X_{j}\subseteq X$,  $1\leq j \leq n$, such that $(0) \subsetneq X_{1}\subsetneq X_{2}\subsetneq \cdots \subsetneq X_{n-1}\subsetneq X_{n}= X$. When $X=\R^{N}$ we single out a special class of \emph{standard} flags parameterized by partitions $N= a_{1}+ \cdots + a_{n}$ where each $a_{j}$ is a positive integer. We write
\begin{equation}\label{Eqn2.5tyu}
\R^{N}= \R^{a_{1}}\oplus \cdots \oplus \R^{a_{n}},
\end{equation}
and we write $\x\in \R^{N}$ as $\x=(\x_{1}, \ldots, \x_{n})$ with $\x_{j}\in \R^{a_{j}}$. With an abuse of notation, we identify $\R^{a_{k}}$ with vectors in $\R^{N}$ of the form $(\mathbf 0, \ldots, \mathbf 0, \x_{k},\mathbf 0, \ldots, \mathbf 0)$. Then the \emph{standard flag} $\F$ associated to the partition $N=a_{1}+ \cdots +a_{n}$ and to the decomposition (\ref{Eqn2.5tyu}) is given by
\begin{equation} \label{Eqn2.6tyu}
(0)\subset \R^{a_{n}}\subset \R^{a_{n-1}}\oplus\R^{a_{n}}\subset \cdots \subset \R^{a_{2}}\oplus \cdots \oplus \R^{a_{n}}\subset \R^{a_{1}}\oplus \cdots \oplus \R^{a_{n}}=\R^{N}.
\end{equation}

In dealing with such decompositions and flags, it is important to make clear which variables in $\R^{N}$ appear in which factor $\R^{a_{l}}$. We can write $\x\in \R^{N}$ either as $\x=(x_{1}, \ldots, x_{N})$ with each $x_{j}\in \R$, or as $\x=(\x_{1}, \ldots, \x_{n})$ with $\x_{l}= (x_{p_{l}}, \ldots, x_{q_{l}}) \in \R^{a_{l}}$ so that $q_{l}= p_{l}+a_{l}-1$. Denote by $J_{l}= \{p_{l}, p_{l}+1, \ldots, q_{l}\}$ the set of subscripts corresponding to the factor $\R^{a_{l}}$ so that $\{1, \ldots, N\}$ is the disjoint union $J_{1}\, \cup \,\cdots \,\cup \,J_{n}$. There is a mapping \label{DefOfPi} $\pi:\{1, \ldots, N\}\mapsto \{1,\ldots, n\}$ so that $j\in J_{\pi(j)}$ for $1 \leq j \leq N$. Thus for example $\pi(10)=3$ means that the variable $x_{10}$ belongs to the factor $\R^{a_3}$. 

With the family of dilations defined in (\ref{E2}), the action on the subspace $\R^{a_{l}}$ is given by
\begin{equation}\label{E2.1}
r \cdot\x_{l} = \big(r^{d_{p_{l}}}x_{p_{l}},  \ldots, r^{d_{q_{l}}}x_{q_{l}}\big).
 \end{equation}
The homogeneous dimension of $\R^{a_{l}}$ is 
\begin{equation}
Q_{l}= d_{p_{l}}+ \cdots + d_{q_{l}} = \sum_{j\in J_{l}}d_{j}.
\end{equation} 
The function 
\begin{equation}
N_{l}(\x_{l}) = \sup_{p_{l}\leq s \leq q_{l}}|x_{s}|^{1/d_{s}}
\end{equation}
is a homogeneous norm on $\R^{a_{l}}$ so that $N_{l}(r\cdot\x_{l}) = r\,N_{l}(\x_{l})$. If $\alpha=(\alpha_{1}, \ldots, \alpha_{N}) \in \N^{N}$, let $\balpha_{l}= (\alpha_{p_{l}}, \ldots, \alpha_{q_{l}})$, and set
\begin{equation}
\<  \bar\alpha_{l}\> = \alpha_{p_{l}}d_{p_{l}}+ \cdots + \alpha_{q_{l}}d_{q_{l}}= \sum_{j\in J_{l}}\alpha_{j}d_{j}.
\end{equation}

We can introduce a partial order on the set of all standard flags on $\R^{N}$.

\begin{definition} \label{Def} Let $\mathcal A=(a_{1}, \ldots, a_{r})$ and $\mathcal B=(b_{1}, \ldots, b_{s})$ be two partitions of $N$ so that $N=a_{1}+ \cdots +a_{r}=b_{1}+\cdots +b_{s}$.
\begin{enumerate}[{\rm(1)}]
\smallskip
\item  The partition $\mathcal A$  is \emph{finer} than the  partition $\mathcal B$, (or  $\mathcal B$ is \emph{coarser} than $\mathcal A$),  if there are integers $\{1=\alpha_{1}<\alpha_{2}< \cdots < \alpha_{s+1}=r+1)$ so that $b_{k}= \sum_{j=\alpha_{k}}^{\alpha_{k+1}-1}a_{j}$. We write  $\mathcal A \preceq \mathcal B$ or $\mathcal B\succeq \mathcal A$. If $\mathcal A\preceq\mathcal B$ but $\mathcal A\neq \mathcal B$ we write $\mathcal A \prec \mathcal B$ or $\mathcal B\succ \mathcal A$.

\smallskip

\item If $\F_{\mathcal A}$ and $\F_{\mathcal B}$ are the flags corresponding to the two partitions and if $\mathcal A\preceq \mathcal B$ (or $\mathcal A\prec\mathcal B$), we say that the flag $\F_{\mathcal A}$ is finer than $\F_{\mathcal B}$ (or $\F_{\mathcal B}$ is coarser than $\F_{\mathcal A}$) and we also write $\F_{\mathcal A} \preceq \F_{\mathcal B}$ and $\F_{\mathcal B}\succeq \F_{\mathcal A}$ (or $\F_{\mathcal A} \prec \F_{\mathcal B}$ and $\F_{\mathcal B}\succ \F_{\mathcal A}$).
\end{enumerate}
\end{definition}

\smallskip

We recall from \cite{NaRiSt00} the concept of a flag kernel on the vector space $\R^{N}$ associated to the decomposition $ \R^{a_{1}}\oplus \cdots \oplus \R^{a_{n}}$, equipped with the family of dilations given in equation (\ref{E2}). Let $\F$ be the standard flag given in (\ref{Eqn2.6tyu}). In order to formulate the cancellation conditions on the flag kernel, we need notation which allows us to split the variables $\{\x_{1}, \ldots, \x_{n}\}$ into two disjoint sets. Thus if $L=\{l_{1}, \ldots, l_{\alpha}\}$ and $M=\{m_{1}, \ldots, m_{\beta}\}$ are complementary subsets of $\{1, \ldots, n\}$ (so that $\alpha+\beta=n$), let $N_{a}=a_{l_{1}}+ \cdots +a_{l_{\alpha}}$ and $N_{b}= a_{m_{1}}+ \cdots +a_{m_{\beta}}$. Write $\x\in \R^{N}$ as $\x=(\x',\x'')$ where $\x'=(\x_{l_{1}}, \ldots, \x_{l_{\alpha}})$ and $\x''=(\x_{m_{1}}, \ldots, \x_{m_{\beta}})$. If $f$ is a function on $\R^{N_{a}}$ and $g$ is a function on $\R^{N_{b}}$, define a function $f\otimes g$ on $\R^{N}$ by setting
\begin{equation*}
f\otimes g(\x_{1}, \ldots, \x_{n}) = f(\x_{l_{1}}, \ldots, \x_{l_{\alpha}})\,g(\x_{m_{1}}, \ldots, \x_{m_{\beta}}).
\end{equation*}

\begin{definition}\label{Def2.1}
A \emph{flag kernel} adapted to the flag $\mathcal F$ is a distribution $\mathcal K\in \mathcal S'(\R^{N})$ which satisfies the following differential inequalities (part {\rm(\ref{Def2.1a})}) and cancellation conditions (part {\rm(\ref{Def2.1b})}).

\begin{enumerate}[{\rm(a)}]

\smallskip

\item  \label{Def2.1a} 
 For test functions supported away from the subspace $\x_{1}=0$, the distribution $\K$ is given by integration against a  $\mathcal C^{\infty}$-function $K$. Moreover for every $\alpha = (\alpha_{1}, \ldots, \alpha_{N}) \in \mathbb Z^{N}$ there is a constant $C_{\alpha}$ so that if $\balpha_{k}=(\alpha_{p_{k}}, \ldots, \alpha_{q_{k}})$, then for $\x_{1}\neq 0$,
\begin{equation*}
\big\vert\partial^{\alpha}K(\x)\big\vert \leq C_{\alpha}\,\prod_{k=1}^{n} \left[N_{1}(\x_{1}) + \cdots N_{k}(\x_{k})\right]^{-Q_{k}-[\![\balpha_{k}]\!]}.
\end{equation*}

\item \label{Def2.1b}
Let $\{1, \dots, n\}= L \cup M$ with $L=\{l_{1}, \ldots, l_{\alpha}\}$, $M=\{m_{1}, \ldots, m_{\beta}\}$ and $L\cap M = \emptyset$ be any pair of complementary subsets.  For any $\psi \in \Cs^{\infty}_{0}(\R^{N_{b}})$ and any positive real numbers $R_{1}, \ldots, R_{\beta}$, put $\psi_{R}(\x_{m_{1}}, \ldots, \x_{m_{\beta}}) = \psi(R_{1}\cdot\x_{m_{1}}, \ldots, R_{\beta}\cdot\x_{m_{\beta}})$. Define a distribution $\mathcal K^{\#}_{\psi,R}\in \mathcal S'(\R^{a_{l_{1}}+\cdots +a_{l_{r}}})$ by setting $$\big\langle \mathcal K_{\psi,R}^{\#}, \varphi\big\rangle = \big\langle \mathcal K, \psi_{R}\otimes \varphi\big\rangle$$ for any test function $\varphi \in \mathcal S(\R^{a_{l_{1}}+\cdots +a_{l_{r}}})$. Then the distribution $\mathcal K^{\#}_{\psi,R}$ satisfies the differential inequalities of part {\rm(\ref{Def2.1a})} for the decomposition $\R^{a_{l_{1}}}\oplus \cdots \oplus\,\R^{a_{l_{r}}}$. Moreover, the corresponding constants that appear in these differential inequalities are \emph{independent} of the parameters $\{R_{1}, \ldots, R_{s}\}$, and depend only on the constants $\{C_{\alpha}\}$ from part {\rm(\ref{Def2.1a})} and the semi-norms of $\psi$.

\smallskip

\end{enumerate}

\noindent The constants $\{C_{\alpha}\}$ in part {\rm(\ref{Def2.1a})} and the implicit constants in part {\rm(\ref{Def2.1b})} are called the \emph{flag kernel constants} for the flag kernel $\K$.
\end{definition}

\begin{remarks} \label{Remark2.2}\quad

\begin{enumerate}[{\rm (a)}]
\item {\rm This definition proceeds by induction on the length $n$ of the flag. The case $n=1$ corresponds to Calder\'on-Zygmund kernels, and the inductive definition is invoked in part (\ref{Def2.1b}). }

\smallskip

\item {\rm With an abuse of notation, the distribution $\K^{\#}_{\psi,R}$ is often written
\begin{equation*}
K^{\#}_{\psi,R}(\x_{l_{1}}, \ldots, \x_{l_{r}}) = \idotsint\limits_{\R^{a_{m_{1}}}\oplus\cdots\oplus \,\R^{a_{m_{s}}}} K(\x)\,\psi_{R}(\x_{m_{1}}, \ldots, \x_{m_{s}})\,d\x_{m_{1}}\cdots d\x_{m_{s}}.
\end{equation*}}
\end{enumerate}
\end{remarks}

\section{Homogeneous vector fields}\label{VectorFields}\

In Section \ref{Nilpotent} below, where we consider a nilpotent Lie group $G$ whose underlying space is $\R^{N}$, we will need to consider the families of left- and right-invariant vector fields on $G$. At this stage, before we introduce the group structure, we consider instead two spanning sets of vector fields $\{X_{1}, \ldots, X_{N}\}$ and $\{Y_{1}, \ldots Y_{N}\}$ on $\R^{N}$ which are homogeneous with respect to the basic family of dilations given in (\ref{E2}); this means that if $Z_{j}$ is either $X_{j}$ or $Y_{j}$ for $1\leq j \leq N$, then $Z_{j}$ can be written\footnote{Despite some risk of confusion, we do not introduce different notation for the coefficients of $X_{j}$ and $Y_{j}$.}
\begin{equation}\label{E13}
Z_{j}[\psi](\x)= \partial_{{j}}[\psi](\x)+\sum_{d_{l}>d_{j}}P_{j}^{l}(\x) \partial_{{l}}[\psi](\x), 
\end{equation}
{with $P_{j}^{l}\in \mathcal H_{d_{l}-d_{j}}$}. It follows from part (\ref{E.10}) of Proposition \ref{Prop2.1mn} that $\partial_{{k}}(P^{l}_{j})\equiv 0$ if $d_{k} > d_{l}$. Thus we can commute the operators given by multiplication by the polynomial $P^{l}_{j}$ and differentiation with respect to $x_{l}$ and also write
\begin{equation}\label{E14}
Z_{j}[\psi](\x)= \partial_{{j}}[\psi](\x)+\sum_{d_{l}>d_{j}}^{N}\partial_{{l}}[P_{j}^{l}\psi](\x). 
\end{equation}
It follows from (\ref{E13}) or (\ref{E14}) that $Z_{N}= \partial_{{N}}$.

\begin{proposition}\label{Prop2.1}
If $P\in \mathcal H_{d}$ and $Z_{j}$ is either $X_{j}$ or $Y_{j}$, then $Z_{j}[P]\in \mathcal H_{d-d_{j}}$, and if $d_{j}>d$, $Z_{j}[P]\equiv 0$.
\end{proposition}

\begin{proof}
It follows from (\ref{E6}) that if $P\in \mathcal H_{d}$, then $\partial_{{l}}[P] \in \mathcal H_{d-d_{l}}$, and since $P^{l}_{j}\in \mathcal H_{d_{l}-d_{j}}$, it follows from part (\ref{E9}) of Proposition \ref{Prop2.1mn} that $P^{l}_{j}\,\partial_{x_{l}}[P]\in \mathcal H_{d-d_{j}}$. Thus $Z_{j}[P]\in \mathcal H_{d-d_{j}}$. The last conclusion then follows from part (\ref{E.10}) of Proposition \ref{Prop2.1mn} .
\end{proof}

In equations (\ref{E13}) or (\ref{E14}), the vector fields $\{Z_{j}\}$ are written in terms of the Euclidean derivatives. Because these equations are in upper-triangular form, it is easy to solve for the Euclidean derivatives in terms of the vector fields. 

\begin{proposition}\label{Prop1.2}
For each $1\leq j \leq N$ let $Z_{j}$ denote either $X_{j}$ or $Y_{j}$. Then there are polynomials $Q^{l}_{k}\in \mathcal H_{d_{l}-d_{k}}$ such that for $\psi\in \mathcal S(\R^{N})$, 
\begin{equation*}
\partial_{{k}}[\psi](\x)= Z_{k}[\psi]+ \sum_{d_{l}>d_{k}}^{N}Q^{l}_{k}(\x)\,Z_{l}[\psi](\x)=Z_{k}[\psi](\x)+\sum_{d_{l}>d_{k}}^{N}Z_{l}\big[Q^{l}_{k}\psi\big](\x).
\end{equation*}
\end{proposition}

\begin{proof}
We argue by reverse induction on the index $k$. When $k=N$ it follows from equation (\ref{E13}) that $\partial_{{N}}=Z_{N}=X_{N}=Y_{N}$.  To establish the induction step, suppose that the conclusion of the Proposition is true for all indices greater than $k$. From equation (\ref{E13}) and the induction hypothesis, for either choice of $Z_{k}$ we have
\begin{equation*}
\begin{aligned}
\partial_{{k}}[\psi]
&= 
Z_{k}[\psi] -\sum_{d_{m}>d_{k}}^{N}P^{m}_{k}\partial_{{m}}[\psi]
=
Z_{k}[\psi] -\sum_{d_{m}>d_{k}}^{N}P^{m}_{k}\Big[Z_{m}[\psi]+\sum_{d_{l}>d_{m}}^{N}Q^{l}_{m}Z_{l}[\psi]\Big]\\
&=
Z_{k}[\psi] -\sum_{d_{m}>d_{k}}^{N}P^{m}_{k}Z_{m}[\psi]-\sum_{d_{l}>d_{k}}\Big[\sum_{d_{k}<d_{m}<d_{l}}P^{m}_{k}Q^{l}_{m}\Big]\,Z_{l}[\psi].
\end{aligned}
\end{equation*}
But according to part (\ref{E9}) of Proposition \ref{Prop2.1mn},  $P^{m}_{k}\,Q^{l}_{m}\in \mathcal H_{d_{l}-d_{k}}$, and this completes the proof.
\end{proof}

For $\psi\in \mathcal S(\R^{N})$ and $t>0$ set $\psi_{t}(\x) = \psi(t^{-1}\cdot \x)$. Then multiplication by a polynomial $P\in \mathcal H_{d}$ is an operator homogeneous of degree $d$ in the sense that $P(\x)\psi_{t}(\x) = t^{d}(P\varphi)_{t}(\x)$, and the vector fields $X_{j}$ and $Y_{j}$ are operators homogeneous of degree $-d_{j}$ in the sense that $X_{j}[\varphi_{t}](\x) = t^{-d_{j}}(X_{j}\varphi)_{t}(\x)$, $Y_{j}[\varphi_{t}](\x) = t^{-d_{j}}(Y_{j}\varphi)_{t}(\x)$. In particular, the commutators  $[X_{j},X_{k}]$ and $[Y_{j},Y_{k}]$ are vector fields which are homogeneous of degree $-(d_{j}+d_{k})$. It follows that we can write
\begin{equation*}
\begin{aligned}\label{Equation2.14mn}
{[}X_{j},X_{k}{]} &= \sum_{d_{m}\geq d_{j}+d_{k}}Q^{m}_{j,k}(\x)\,\partial_{{m}}= \sum_{d_{m}\geq d_{j}+d_{k}}R^{m}_{j,k}(\x)\,Z_{m}, \,\, \text{with $Q^{m}_{j,k}, R^{m}_{j,k}\in \mathcal H_{d_{m}-d_{j}-d_{k}}$},\\
{[}Y_{j},Y_{k}{]} &= \sum_{d_{m}\geq d_{j}+d_{k}}\widetilde Q^{m}_{j,k}(\x)\,\partial_{{m}}= \sum_{d_{m}\geq d_{j}+d_{k}}\widetilde R^{m}_{j,k}(\x)\,Z_{m}, \,\, \text{with  $\widetilde Q^{m}_{j,k}, \widetilde R^{m}_{j,k}\in \mathcal H_{d_{m}-d_{j}-d_{k}}$}.
\end{aligned}
\end{equation*}
If the operators $\{X_{j}\}$ and $\{Y_{j}\}$ are bases for a Lie algebra (as in the case of left- or right-invariant vector fields),  the coefficients $\{R^{m}_{j,k}\}$ and $\{\widetilde R^{m}_{j,k}\}$ are constants.

\smallskip

Equations (\ref{E13}) or (\ref{E14}) express the vector fields $\{Z_{k}\}$ in terms of the standard derivatives $\{\partial_{j}\}$, and Proposition \ref{Prop1.2} expresses the standard derivatives in terms of the vector fields. We shall need analogous identities for  products of $r$ vector fields $Z_{k_{1}}\cdots Z_{k_{r}}$ or products of $r$ Euclidean derivatives $\partial_{k_{1}}\cdots \partial_{k_{r}}$. The formulas are somewhat complicated, since they involve products of operators of various lengths. To help with the formulation of the results, it will be  convenient to introduce the following notation. 

\begin{definition}\label{disjoint}
Let $k_{1}, \ldots, k_{r}\in \{1, \ldots , N\}$ be a set of $r$ integers, possibly with repetitions.

\begin{enumerate}[(1)]

\smallskip

\item  For any non-empty set $U\subset \{1, \ldots , r\}$, put 
\begin{align*}
d_{U} &= \sum_{\ell\in U}d_{k_{\ell}}, \quad \text{and}\\
\mathfrak I(U) &= \left\{m\in\{1, \ldots, N\}\,\Big\vert\, d_{m}\geq d_{U}\right\}.
\end{align*}
Note that if $U$ consists of two or more elements and $m\in \mathfrak I(U)$, then $d_{m}>\sup_{\ell\in U}d_{k_{\ell}}$.
\smallskip

\item For each integer $1\leq s \leq r$ let $\mathcal U_{s}^{r}$ denote the set of partitions of the set $\{1, \ldots, r\}$ into $s$ non-empty disjoint subsets $\mathcal U = \{U_{1}, \ldots, U_{s}\}$. 
\end{enumerate}
\end{definition}

\noindent The following Proposition then shows how to write products of vector fields in terms of products of Euclidean derivatives.

\begin{proposition}\label{Prop1.3}
Let $k_{1}, \ldots, k_{r}\in \{1, \ldots , N\}$ be a  set of $r$ integers, possibly with repetitions. For $1\leq \ell \leq r$, let $Z_{k_{\ell}}$ denote either $X_{k_{\ell}}$ or $Y_{k_{\ell}}$. Then there are polynomials $P^{m_{\ell}}_{U_{\ell}}\in \mathcal H_{d_{m_{\ell}}-d_{U_{\ell}}}$ such that 
\begin{align*}
Z_{k_{1}}\cdots Z_{k_{r}}[\psi] &= 
\sum_{s=1}^{r}\,\,
\sum_{(U_{1}, \ldots, U_{s})\in \mathcal U^{r}_{s}}\,\,
\sum_{m_{1}\in \mathfrak I(U_{1})}
\cdots 
\sum_{m_{s}\in \mathfrak I(U_{s})}
\partial_{{m_{1}}}\cdots \partial_{{m_{s}}}
\big[P^{m_{1}}_{U_{1}}\cdots P^{m_{s}}_{U_{s}}\psi\big].
\end{align*}
If $s=r$, so that $U_{j}=\{k_{j}\}$, the polynomial $P^{m_{j}}_{U_{j}}(\x)\equiv 1$.
\end{proposition}

\begin{proof}
We argue by induction on $r$. The case $r=1$ is contained in equation (\ref{E14}), so suppose we are given vector fields $\{Z_{k_{1}}, \ldots, Z_{k_{r+1}}\}$ where each $Z_{k_{\ell}}$ is either $X_{k_{\ell}}$ or $Y_{k_{\ell}}$.
 Then by induction
\begin{equation}\label{Eqn2.8io}
\begin{aligned}
Z_{k_{1}}\cdots &Z_{k_{r+1}}[\psi]
=
Z_{k_{1}}\cdots Z_{k_{r}}\big[Z_{k_{r+1}}[\psi]\big]\\
&=
\sum_{s=1}^{r}
\sum_{(U_{1}, \ldots, U_{s})\in \mathcal U^{r}_{s}}
\sum_{m_{1}\in \mathfrak I(U_{1})}\cdots \sum_{m_{s}\in \mathfrak I(U_{s})}\partial_{{m_{1}}}\cdots \partial_{{m_{s}}}\big[P^{m_{1}}_{U_{1}}\cdots P^{m_{s}}_{U_{s}}\big[Z_{k_{r+1}}[\psi]\big]\big].
\end{aligned}
\end{equation}
Since we can write
$
Z_{k_{r+1}}[\psi]= \partial_{{k_{r+1}}}[\psi]+\sum_{\{m\vert d_{m}>d_{k_{r+1}}\}}\partial_{{m}}[P^{m}_{k_{r+1}}\psi]
$
where $P^{m}_{k_{r+1}}\in \mathcal H_{d_{m}-d_{k_{r+1}}}$, the derivative $\partial_{{m_{1}}}\cdots \partial_{{m_{s}}}\big[P^{m_{1}}_{U_{1}}\cdots P^{m_{s}}_{U_{s}}\big[Z_{k_{r+1}}[\psi]\big]\big]$ in the last line of (\ref{Eqn2.8io}) can be written
\begin{align*}
\partial_{{m_{1}}}\cdots \partial_{{m_{s}}}&\big[P^{m_{1}}_{U_{1}}\cdots P^{m_{s}}_{U_{s}}\big[Z_{k_{r+1}}[\psi]\big]\big]\\
&=
\partial_{{m_{1}}}\cdots \partial_{{m_{s}}}\big[P^{m_{1}}_{U_{1}}\cdots P^{m_{s}}_{U_{s}}\big[\partial_{k_{r+1}}[\psi]\big]\big]\\
&\qquad\qquad\qquad
+\sum_{d_{m}>d_{k_{r+1}}}
\partial_{{m_{1}}}\cdots \partial_{{m_{s}}}\big[P^{m_{1}}_{U_{1}}\cdots P^{m_{s}}_{U_{s}}\big[\partial_{m}[P^{m}_{k_{r+1}}\psi]\big]\big]\\
&=
\partial_{{m_{1}}}\cdots \partial_{{m_{s}}}\partial_{k_{r+1}}\big[P^{m_{1}}_{U_{1}}\cdots P^{m_{s}}_{U_{s}}\psi\big] \\
&\qquad\qquad -
\partial_{{m_{1}}}\cdots \partial_{{m_{s}}}\big[\partial_{k_{r+1}}\big[P^{m_{1}}_{U_{1}}\cdots P^{m_{s}}_{U_{s}}\big]\psi\big]\\
&\qquad\qquad\qquad
+
\sum_{d_{m}>d_{k_{r+1}}}
\partial_{{m_{1}}}\cdots \partial_{{m_{s}}}\partial_{m}\big[P^{m_{1}}_{U_{1}}\cdots P^{m_{s}}_{U_{s}}P^{m}_{k_{r+1}}\psi\big]\\
&\qquad\qquad\qquad\qquad\qquad
-
\sum_{d_{m}>d_{k_{r+1}}}
\partial_{{m_{1}}}\cdots \partial_{{m_{s}}}\big[\partial_{m}\big[P^{m_{1}}_{U_{1}}\cdots P^{m_{s}}_{U_{s}}]P^{m}_{k_{r+1}}\psi\big].
\end{align*}
The terms in the first and third lines of the last expression have the right form for the case $r+1$.  Thus in the term $\partial_{{m_{1}}}\cdots \partial_{{m_{s}}}\partial_{k_{r+1}}\big[P^{m_{1}}_{U_{1}}\cdots P^{m_{s}}_{U_{s}}\psi\big]$, $r$ has been replaced by $r+1$, the set $\{1, \ldots, r, r+1\}$ has been decomposed into $s+1$ subsets $\{U_{1}, \ldots, U_{s}, U_{s+1}\}$ where $U_{s+1}= \{k_{r+1}\}$, and $P^{k_{r+1}}_{U_{s+1}}(\x) \equiv 1$. The same is true for each term $\partial_{{m_{1}}}\cdots \partial_{{m_{s}}}\partial_{m}\big[P^{m_{1}}_{U_{1}}\cdots P^{m_{s}}_{U_{s}}P^{m}_{k_{r+1}}\psi\big]$, except that $P^{m}_{U_{s+1}}= P^{m}_{k_{r+1}}$.

For terms in the second and fourth lines, we use the product rule; we write $\partial_{k_{r+1}}\big[P^{m_{1}}_{U_{1}}\cdots P^{m_{s}}_{U_{s}}\big]$ and $\partial_{m}\big[P^{m_{1}}_{U_{1}}\cdots P^{m_{s}}_{U_{s}}]$  as a sum of $s$ terms. Consider, for example, 
\begin{align*}
\sum_{d_{m}>d_{k_{r+1}}}
\partial_{{m_{1}}}\cdots \partial_{{m_{s}}}&\big[\partial_{m}\big[P^{m_{1}}_{U_{1}}]P^{m_{2}}_{U_{2}}\cdots P^{m_{s}}_{U_{s}}P^{m}_{k_{r+1}}\psi\big]\\
&=
\partial_{{m_{1}}}\cdots \partial_{{m_{s}}}
\Big[\big(\sum_{d_{m}>d_{k_{r+1}}}\partial_{m}\big[P^{m_{1}}_{U_{1}}]P^{m}_{k_{r+1}}\big)P^{m_{2}}_{U_{2}}\cdots P^{m_{s}}_{U_{s}}\Big]\\
&=
\partial_{{m_{1}}}\cdots \partial_{{m_{s}}}\big[\widetilde P^{m_{1}}_{\widetilde U_{1}}P^{m_{2}}_{U_{2}}\cdots P^{m_{s}}_{U_{s}}\big]
\end{align*}
where  $\widetilde P^{m_{1}}_{\widetilde U_{1}}=\sum_{d_{m}>d_{k_{r+1}}}\partial_{m}\big[P^{m_{1}}_{U_{1}}]P^{m}_{k_{r+1}} \in \mathcal H_{d_{m_{1}-d_{U_{1}}-d_{k_{r+1}}}}$. These terms also have the right form for the case $r+1$, since we now let $\widetilde U_{1}= U_{1}\cup \{k_{r+1}\}$ so that $\{k_{1}, \ldots, k_{r+1}\} = \widetilde U_{1}\cup U_{2}\cup \cdots \cup U_{s}$.  This establishes the Proposition. \end{proof}

%
%

The next result shows how to write products of Euclidean derivatives in terms of products of vector fields. Since the proof is similar to that of Proposition \ref{Prop1.3}, we omit it.
\begin{proposition}\label{Prop1.3w}
Let $k_{1}, \ldots, k_{r}\in \{1, \ldots , N\}$ be a  set of $r$ integers, possibly with repetitions, chosen from the set $\{1, \ldots, N\}$. 
There are polynomials $Q^{m_{\ell}}_{U_{\ell}}\in \mathcal H_{d_{m_{\ell}}-d_{U_{\ell}}}$ such that 
\begin{align*}
\partial_{{k_{1}}}\cdots\partial_{{k_{r}}}[\psi] 
&=
\sum_{s=1}^{r}\,\,
\sum_{(U_{1}, \ldots, U_{s})\in \mathcal U^{r}_{s}}\,\,
\sum_{m_{1}\in \mathfrak I(U_{1})}
\cdots 
\sum_{m_{s}\in \mathfrak I(U_{s})}
Z_{m_{1}}\cdots Z_{m_{s}}
[Q^{m_{1}}_{U_{1}}\cdots Q^{m_{s}}_{U_{s}}\psi].
\end{align*}
Here each $Z_{j}$ is either $X_{j}$ or $Y_{j}$. If $s=r$, so that $U_{\ell}=\{k_{\ell}\}$, the polynomial $Q^{m_{\ell}}_{U_{\ell}}(\x)\equiv 1$.
\end{proposition}

\section{Normalized bump functions and their dilations}\label{NBFDilations}

\subsection{Families of dilations}\label{Section4.1}\quad

\smallskip

Fix the family of dilations on $\R^{N}$ given in equation (\ref{E2}). We introduce an \emph{$N$-parameter} family of dyadic\footnote{In Section  \ref{LpEstimates}, we shall use a continuous version of this dyadic family. If $\t=(t_{1}, \ldots, t_{N})$ with each $t_{j }>0$, we will set $f_{\t}(\x) = f(t_{1}^{d_{1}}x_{1}, \ldots, t_{N}^{d_{N}}x_{N})$.}  dilations.  For $f\in L^{1}(\R^{N})$ and $I\in \Z^{N}$ set
\begin{equation}\label{E21}
\begin{aligned}
2^{I}\cdot \x &= (2^{-d_{1}i_{1}}x_{1}, \ldots, 2^{-d_{N}i_{N}}x_{N}), \\
\big[f\big]_{I}(\x) &= 2^{-\sum\limits_{\ell=1}^{N}d_{\ell}i_{\ell}} f(2^{I}\cdot \x).
\end{aligned}
\end{equation}
Then $\big|\big|[f]_{I}\big|\big|_{L^{1}(\R^{N})} = ||f||_{L^{1}(\R^{N})}$. The set of  monotone increasing indices is denoted by 
\begin{equation}\label{E22}
E_{N}= \left\{I=(i_{1}, \ldots, i_{N}) \in \Z^{N}\,\big\vert\, i_{1}\leq i_{2}\leq \cdots \leq i_{N}\right\}.
\end{equation}
When we consider flag kernels corresponding to the decomposition $\mathfrak A$ given by $N=a_{1}+\cdots+a_{n}$ and $\R^{N}=\R^{a_{1}}\oplus \cdots \oplus \R^{a_{n}}$, we consider the $n$-parameter family of dilations parameterized by $n$-tuples $I=(i_{1}, \ldots, i_{n}) \in \Z^{n}$:
\begin{equation}\label{E2.2}
\begin{aligned}
2^{I}\cdot\x &=(2^{i_{1}}\cdot\x_{1}, \ldots, 2^{i_{n}}\cdot\x_{n}), \qquad\text{where}\\
2^{i_{\ell}}\cdot\x_{\ell} &= (2^{-d_{p_{\ell}}i_{\ell}}x_{p_{\ell}}, \ldots, 2^{-d_{q_{\ell}}i_{\ell}}x_{q_{\ell}}), \qquad\text{and}\\
[f]_{I}(\x) &= 2^{-\sum\limits_{\ell=1}^{n} Q_{\ell}i_{\ell}}f(2^{I}\cdot \x).
\end{aligned}
\end{equation}
The set of monotone increasing indices in this case is denoted by
\begin{equation}\label{Eqn4.4iop}
\mathcal E_{n}= \left\{I=(i_{1}, \ldots, i_{n}) \in \Z^{n}\,\big\vert\, i_{1}\leq i_{2}\leq \cdots \leq i_{n}\right\}.
\end{equation}
Given the decomposition $\mathfrak A$, there is a mapping $p_{\mathfrak A}:\mathcal E_{n}\mapsto E_{N}$ given by
\begin{equation}\label{E2.18asd}
p_{\mathfrak A}(i_{1}, \ldots, i_{n}) = \big(\,\overset{a_{1}}{\overbrace{i_{1}, \ldots,i_{1}}}\,,\,\overset{a_{2}}{\overbrace{i_{2}, \ldots,i_{2}}}\,, \ldots , \,\overset{a_{n}}{\overbrace{i_{n}, \ldots,i_{n}}}\,\big).
\end{equation}

\smallskip

We shall want to write flag kernels as sums of dilates $[\varphi]_{I}$ of  normalized bump functions $\varphi$. Roughly speaking, a family of functions $\{\varphi_{\alpha}\}$ in $\mathcal C^{\infty}_{0}(\R^{N})$ or $\mathcal S(\R^{N})$ is \textit{normalized} if one has uniform control of the supports (in the case of $\mathcal C^{\infty}_{0}(\R^{N})$) and of the semi-norms $||\varphi_{\alpha}||_{(m)}$ or $||\varphi_{\alpha}||_{[M]}$. The following definition will simplify the precise statements of our results.

\begin{definition}\label{Def1.5}\quad

\smallskip

\begin{enumerate}[{\rm(1)}]

\item If $\varphi,\,\widetilde\varphi\in \Cs^{\infty}_{0}(\R^{n})$,  then $\widetilde\varphi$ is \emph{normalized in terms of $\varphi$} if \footnote{We shall sometimes use the expressions `normalized with respect to' or `normalized relative to' as a variant of `normalized in terms of'.} there are constants $C, C_{m}>0$ and integers $p_{m}\geq 0$ so that:
\begin{enumerate}[{\rm(a)}]

\smallskip

\item If the support of $\varphi$ is contained in the ball $B(\rho)$, then the support of $\widetilde\varphi$ is contained in the ball $B(C\rho)$.

\smallskip

\item For every non-negative integer $m$, $||\widetilde\varphi||_{(m)}\leq C_{m}\,||\varphi||_{(m+p_{m})}$.

\smallskip
\end{enumerate}

\smallskip

\item If $\psi,\,\widetilde \psi\in \mathcal  S(\R^{n})$, then $\widetilde \psi$ is \emph{normalized in terms of $\psi$} if there are constants $C_{N}>0$ and integers $p_{N}\geq 0$ so that $||\widetilde\psi||_{[N]}\leq C_{N}\,||\psi||_{[N+p_{N}]}$ for every non-negative integer $N$.

\smallskip

\item If $P,\,\widetilde P$ are polynomials, then $\widetilde P$ is normalized in terms of $P$ if $\widetilde P$ is obtained from $P$ by multiplying each coefficient by a constant of modulus less than or equal to $1$.
\end{enumerate}

\end{definition}

\smallskip

If $\varphi \in \mathcal C^{\infty}_{0}(\R^{N})$, it is sometimes convenient to write $\varphi$ as a sum of  products of functions of a single variable. That this is possible follows from the following fact.

\begin{proposition}\label{Prop2.5t}
Let $\varphi\in \mathcal C^{\infty}_{0}(\R^{N})$. Then for each $\alpha\in \mathbb N^{N}$ and $1 \leq k \leq N$ there are functions $\varphi_{\alpha,k}\in \mathcal C^{\infty}_{0}(\R^{N})$ so that $$\varphi(x_{1},\ldots, x_{N}) = \sum_{\alpha\in \N^{N}}c_{\alpha}\varphi_{\alpha,{1}}(x_{1}) \cdots \varphi_{\alpha,n}(x_{n}),$$where for any $M>0$, there is a constant $C_{M}$ such that   $|c_{\alpha}|\leq C_{M}\,(1+|\alpha|)^{-M}$.
\end{proposition}

\noindent Proposition \ref{Prop2.5t} follows easily by regarding $\varphi$ as a periodic function in each variable, and then expanding $\varphi$ in a rapidly converging Fourier series.

\smallskip

\subsection{Differentiation and multiplication of dilates of bump functions}\quad

\smallskip

In this section we study the action of differentiation or multiplication by a homogeneous polynomial on dilates of bump functions. The key results are Proposition \ref{Prop2.10jk} and Corollary \ref{Prop2.15jkl} below. We begin with the following result which  follows easily from the definitions and the chain rule:

\begin{proposition}\label{Prop1.6}
Let $\psi\in \mathcal S(\R^{N})$. Then
\begin{align}\label{E23}
2^{+d_{k}i_{k}}\,\partial_{{k}}[\psi]_{I}(\x) &= [\partial_{{k}}\psi]_{I}(\x)&&\text{and}&
2^{-d_{k}i_{k}}\,x_{k}\,[\psi]_{I}(\x) &= [x_{k}\,\psi]_{I}.
\end{align}
More generally, if $I=(i_{1}, \ldots, i_{N})\in \mathbb Z^{N}$, and $\alpha\in \mathbb N^{N}$, then
\begin{align}\label{E24}
2^{+[\![\alpha\cdot I]\!]}\,\partial^{\alpha}[\psi]_{I} &= \big[\partial^{\alpha}\psi\big]_{I}&&\text{and}&
2^{-[\![\alpha\cdot I]\!]}\x^{\alpha}[\psi]_{I}&= \big[\x^{\alpha}\psi\big]_{I},
\end{align}
where $[\![\alpha\cdot I[\!] = \sum_{k=1}^{N} \alpha_{k}d_{k}i_{k}$.
\end{proposition}

\smallskip

We will frequently use the following generalization of the second identity in equation (\ref{E24}).

\begin{proposition}\label{Prop1.7}
Let  $P\in \mathcal H_{d}$ where $d <d_{l}$.  If $I\in E_{N}$, there is a polynomial $P_{I}\in \mathcal H_{d}$, normalized in terms of $P$, so that $P(\x)[\psi]_{I}(\x) =2^{d\,i_{l}} \big[P_{I} \psi\big]_{I}(\x)$ for $\psi\in\mathcal S(\R^{N})$. 
\end{proposition}

\begin{proof} 
Write $P(\x) = \sum_{\alpha\in \mathfrak H_{d}}c_{\alpha}\x^{\alpha}$ with  $\mathfrak H_{d}=\{\alpha=(\alpha_{1}, \ldots, \alpha_{N})\in \N^{N}\,\vert\, \alpha_{1}d_{1}+\cdots +\alpha_{N}d_{N}=d\}$. Since $d<d_{l}$, if $\alpha\in \mathfrak H_{d}$ we have $\alpha_{j}=0$ for $j\geq l$. According to Proposition \ref{Prop1.6}, it follows that $P(\x)[\psi]_{I}(\x) = \sum_{\alpha\in \mathfrak H_{d}}c_{\alpha}2^{[\![\alpha\cdot I]\!]}\big[\x^{\alpha}\psi\big]_{I}(\x)$, and 
\begin{align*}
\<  \alpha\cdot I\> &=\sum_{m=1}^{l-1}\alpha_{m}d_{m}i_{m}=i_{{l}}\sum_{m=1}^{l-1}\alpha_{m}d_{m}-\sum_{m=1}^{l-1}\alpha_{m}d_{m}(i_{{l}}-i_{m})\\
&= 
d\,i_{{l}}-\sum_{m=1}^{l-1}\alpha_{m}d_{m}(i_{l}-i_{m}).
\end{align*}
Thus
\begin{equation*}
P(\x)\big[\psi]_{I}(\x)= 2^{d\,i_{l} }\Big[\sum_{\alpha\in \mathfrak H_{d}}c_{\alpha}2^{-\sum_{m=1}^{l-1}\alpha_{m}d_{m}(i_{l}-i_{m})}\x^{\alpha}\psi\Big]_{I}(\x).
\end{equation*}
But since $I\in E_{N}$, each exponent $-\sum_{m=1}^{l-1}\alpha_{m}d_{m}(i_{l}-i_{m})\leq 0$. The proof is complete if we set $P_{I}(\x) = \sum_{\alpha\in \mathfrak H_{d}}c_{\alpha}2^{-\sum_{m=1}^{l-1}\alpha_{m}d_{m}(i_{l}-i_{m})}\x^{\alpha}$.
\end{proof}

\smallskip

\noindent \textbf{Remark:} Equation (\ref{E23}) shows \label{Normalized} that the operator \,$2^{d_{k}i_{k}}\partial_{x_{k}}$ \,applied to the $I$-dilate  of a normalized function $\varphi$ is the $I$-dilate of a function $\widetilde \varphi$ normalized in terms of $\varphi$, and multiplying the $I$-dilate of normalized function $\varphi$ by \,$2^{-d_{k}i_{k}}x_{k}$ \,is  the  $I$-dilate of a function $\widetilde \varphi$ normalized in terms of $\varphi$. Thus at `scale $I$', the operators $2^{d_{k}i_{k}}\partial_{x_{k}}$ and $2^{-d_{k}i_{k}}x_{k}$ are `invariant'; they map the collection $I$-dilates of normalized functions to itself.  

\smallskip

A key observation, which is used when we consider convolution on homogeneous nilpotent groups, is that we can replace the operator \,$2^{d_{k}i_{k}}\partial_{x_{k}}$  with the operator \,$2^{d_{k}i_{k}}Z_{k}$, or conversely the operator  $2^{d_{k}i_{k}}Z_{k}$ \, by the operator \,$2^{d_{k}i_{k}}\partial_{x_{k}}$, at the cost of introducing an error involving terms $2^{d_{l}i_{l}}Z_{l}$ or $2^{d_{l}i_{l}}\partial_{l}$ where $l>k$ multiplied by a `gain'  $2^{-d_{k}(i_{l}-i_{k})}$.  The precise statement is given in Proposition \ref{Prop1.8} below. 

Let  $P^{l}_{k} \in \mathcal H_{d_{l}-d_{k}}$ be the homogeneous polynomials that are coefficients of a vector field $Z_{k}$ as in equations (\ref{E13}) or (\ref{E14}), and let $Q^{l}_{k}\in \mathcal H_{d_{l}-d_{k}}$ be the polynomials in Proposition \ref{Prop1.2}. Since $d_{l}-d_{k}<d_{l}$, we can use Proposition \ref{Prop1.7} to write
\begin{equation}\label{E25}
\begin{aligned}
P^{l}_{k}(\x)[\psi]_{I}(\x) &= 2^{i_{l}(d_{l}-d_{k})}\big[P^{l}_{k,I}\psi\big]_{I}(\x),\\
Q^{l}_{k}(\x)[\psi]_{I}(\x) &= 2^{i_{l}(d_{l}-d_{k})}\big[Q^{l}_{k,I}\psi\big]_{I}(\x),
\end{aligned}
\end{equation}
where $P^{l}_{k,I},\,Q^{l}_{k,I}\in \mathcal H_{d_{l}-d_{k}}$ are normalized relative to $P^{l}_{k}$ and $Q^{l}_{k}$

\begin{proposition}\label{Prop1.8}
Let $I\in E_{N}$, and let  $\{P^{l}_{k,I}\}$ and $\{Q^{l}_{k,I}\}$ be the homogeneous polynomials defined in equation (\ref{E25}). Let $\psi \in \mathcal C^{\infty}_{0}(\R^{N})$ (respectively $\psi \in \mathcal S(\R^{N})$). Then
\begin{align*}
(2^{d_{k}i_{k}}Z_{k})[\psi]_{I}
&=
(2^{d_{k}i_{k}}\partial_{{k}})[\psi]_{I}
+
\sum_{d_{l}>d_{k}}^{N}2^{-d_{k}(i_{l}-i_{k})}\,(2^{d_{l}i_{l}}\partial_{{l}})[P^{l}_{k,I}\psi]_{I},
\\
(2^{d_{k}i_{k}}\partial_{{k}})[\psi]_{I} 
&=
(2^{d_{k}i_{k}} Z_{k})[\psi]_{I}+
\sum_{d_{l}>d_{k}}^{N} 2^{-d_{k}(i_{l}-i_{k})}(2^{d_{l}i_{l}}Z_{l})[Q^{l}_{k,I}\psi]_{I}.
\end{align*}
The functions $\{P^{l}_{k,I}\psi\}$ and $\{Q^{l}_{k,I}\psi\}$ are normalized with respect to $\psi$ in $\Cs^{\infty}_{0}(\R^{N})$ (respectively in $\Sc(\R^{N})$).
\end{proposition}

\begin{proof}
The first identity follows immediately from equation (\ref{E14}), Proposition \ref{Prop1.7}, and equation (\ref{E25}). To obtain the second, use Propositions \ref{Prop1.2} and equation (\ref{E25}); we have
\begin{align*}
2^{d_{k}i_{k}}\partial_{x_{k}}[\psi]_{I} &=2^{d_{k}i_{k}}Z_{k}[\psi]_{I}+2^{d_{k}i_{k}}\sum_{l=k+1}^{N}Z_{l}Q^{l}_{k}(\x)[\psi]_{I}\\
&=
2^{d_{k}i_{k}}Z_{k}[\psi]_{I}+ \sum_{l=k+1}^{N} 2^{-d_{k}(i_{l}-i_{k})}(2^{d_{l}i_{l}}Z_{l})[Q^{l}_{k,I}\psi]_{I},
\end{align*}
which is the desired formula. 
\end{proof}

\smallskip

\begin{corollary}\label{Cor1.9}
If $\varphi\in \mathcal C^{\infty}_{0}(\R^{N})$ (respectively  $\varphi\in \mathcal S(\R^{N})$) and if $I \in E_{N}$, then  there is a  function $\widetilde\varphi\in \mathcal C^{\infty}_{0}(\R^{N})$, (respectively  $\widetilde\varphi\in \mathcal S(\R^{N})$), normalized with respect to $\varphi$, such that   $ 2^{d_{k}i_{k}}Z_{k}[\varphi]_{I}=[\widetilde\varphi]_{I}$.
\end{corollary}

\smallskip

We shall need an analogue of Proposition \ref{Prop1.8} for $r$-fold products of vector fields or Euclidean derivatives.  If $\{P^{m_{l}}_{U_{l}}\}$ and $\{Q^{m_{l}}_{U_{l}}\}$ are the polynomials appearing in Propositions  \ref{Prop1.3} and \ref{Prop1.3w}, we use Proposition \ref{Prop1.7} to define polynomials $P^{m_{l}}_{U_{l},I}$ and $Q^{m_{l}}_{U_{l},I}$ by the formulas
\begin{equation}\label{Eqn2.20poi}
\begin{aligned}
P^{m_{l}}_{U_{l}}(\x)[\psi]_{I}(\x) &= 2^{i_{l}(d_{m_{l}}-d_{U_{l}})}[P^{m_{l}}_{U_{l},I}\psi]_{I}(\x),\\
Q^{m_{l}}_{U_{l}}(\x)[\psi]_{I}(\x) &= 2^{i_{l}(d_{m_{l}}-d_{U_{l}})}[Q^{m_{l}}_{U_{l},I}\psi]_{I}(\x).
\end{aligned}
\end{equation}

\begin{proposition}\label{Prop2.10jk}
Let $k_{1}, \ldots, k_{r}\in \{1, \ldots, N\}$ be a set of $r$ integers, possibly with repetitions, and let $I \in E_{N}$. Let $Z_{\ell}$ denote either $X_{\ell}$ or $Y_{\ell}$. Then
\begin{align*}
(2^{d_{k_{1}}i_{k_{1}}}Z_{k_{1}}) \cdots &(2^{d_{k_{r}}i_{k_{r}}}Z_{k_{r}})[\psi]_{I} 
\\&=
\sum_{s=1}^{r}
\sum_{(U_{1}, \ldots, U_{s})\in\mathcal U^{r}_{s}}
\sum_{m_{1}\in \mathfrak I(U_{1})}\cdots \sum_{m_{s}\in \mathfrak I(U_{s})}
2^{-\sum_{l=1}^{j}\sum_{t\in U_{l}}d_{k_{l}}(i_{m_{l}}-i_{k_{t}})}\\
&\qquad\qquad\qquad
(2^{d_{m_{1}}i_{m_{1}}}\partial_{m_{1}}) \cdots (2^{d_{m_{s}}i_{m_{s}}}\partial_{m_{s}})[P^{m_{1}}_{U_{1}, I}\cdots P^{m_{s}}_{U_{s}, I}\psi]_{I},
\end{align*}
and
\begin{align*}
(2^{d_{k_{1}}i_{k_{1}}}\partial_{k_{1}}) \cdots &(2^{d_{k_{r}}i_{k_{r}}}\partial_{k_{r}})[\psi]_{I} 
\\&=
\sum_{s=1}^{r}
\sum_{(U_{1}, \ldots, U_{s})\in\mathcal U^{r}_{s}}
\sum_{m_{1}\in \mathfrak I(U_{1})}\!\!\!\cdots \!\!\!\sum_{m_{s}\in \mathfrak I(U_{s})}\!\!\!
2^{-\sum_{l=1}^{j}\sum_{t\in U_{l}}d_{k_{l}}(i_{m_{l}}-i_{k_{t}})}\\
&\qquad\qquad\qquad
(2^{d_{m_{1}}i_{m_{1}}}Z_{m_{1}}) \cdots (2^{d_{m_{s}}i_{m_{s}}}Z_{m_{s}})[Q^{m_{1}}_{U_{1}, I}\cdots Q^{m_{s}}_{U_{s}, I}\psi]_{I}.
\end{align*}
In either identity, if $s=r$ so that $U_{\ell}= \{k_{\ell}\}$, $1 \leq \ell \leq r$, the polynomials $P^{m_{\ell}}_{U_{\ell}}(\x) = Q^{m_{\ell}}_{U_{\ell}}(\x) \equiv 1$.
\end{proposition}

\begin{proof}
Using Proposition \ref{Prop1.3}, we have
\begin{align*}
&(2^{d_{k_{1}}i_{k_{1}}}\partial_{k_{1}}) \cdots (2^{d_{k_{r}}i_{k_{r}}}\partial_{k_{r}})[\psi]_{I}\\
&=
2^{\sum\limits_{\ell=1}^{r}d_{k_{\ell}}i_{k_{\ell}}}
\sum_{s=1}^{r}\,\,\sum_{\mathcal U\in \mathcal U^{r}_{s}}
\sum_{m_{1}\in \mathfrak I(U_{1})}\cdots \sum_{m_{s}\in \mathfrak I(U_{s})}Z_{m_{1}}\cdots Z_{m_{s}}\big[Q^{m_{1}}_{U_{1}}\cdots Q^{m_{s}}_{U_{s}}[\psi]_{I}\big]\\
&=
2^{\sum\limits_{\ell=1}^{r}d_{k_{\ell}}i_{k_{\ell}}}
\sum_{s=1}^{r}\,\,\sum_{\mathcal U\in \mathcal U^{r}_{s}}
\sum_{m_{1}\in \mathfrak I(U_{1})}\cdots \sum_{m_{s}\in \mathfrak I(U_{s})}
2^{\sum\limits_{\ell=1}^{s}i_{m_{\ell}}(d_{m_{\ell}}-d_{U_{\ell}})}\,Z_{m_{1}}\cdots Z_{m_{s}}\big[Q^{m_{1}}_{U_{1},I}\cdots Q^{m_{s}}_{U_{s},I}\,\psi\big]_{I}\\
&=
\sum_{s=1}^{r}\,\,\sum_{\mathcal U\in \mathcal U^{r}_{s}}
\sum_{m_{1}\in \mathfrak I(U_{1})}\cdots \sum_{m_{s}\in \mathfrak I(U_{s})}
2^{\sum\limits_{\ell=1}^{r}d_{k_{\ell}}i_{k_{\ell}}-\sum\limits_{\ell=1}^{s}i_{m_{\ell}}d_{U_{\ell}}}\\
&\qquad\qquad\qquad\qquad\qquad\qquad\qquad\qquad(2^{d_{m_{1}}i_{m_{1}}}Z_{m_{1}})\cdots (2^{d_{m_{s}}i_{m_{s}}}Z_{m_{s}})\big[Q^{m_{1}}_{U_{1},I}\cdots Q^{m_{s}}_{U_{s},I}\,\psi\big]_{I}.
\end{align*}
This completes the proof of the first identity. The second identity is established in the same way.
\end{proof}

In Proposition \ref{Prop2.10jk}, we can rewrite the exponent  in the power of $2$ as follows. Having chosen $\mathcal U = (U_{1}, \ldots, U_{s}) \in \mathcal U^{r}_{s}$ , there is a unique mapping $\sigma=\sigma_{\mathcal U}:\{1, \ldots, r\}\to \{1, \ldots ,s\}$ so that $\ell \in U_{\sigma(\ell)}$ for $1 \leq \ell \leq r$. Then
\begin{align*}
\sum_{\ell=1}^{r}d_{k_{\ell}}i_{k_{\ell}}-\sum_{\ell=1}^{s}i_{m_{\ell}}d_{U_{\ell}}
=
\sum_{\ell=1}^{r}d_{k_{\ell}}i_{k_{\ell}} - \sum_{\ell=1}^{r}i_{m_{\sigma(\ell)}}d_{k_{\ell}}
= 
-\sum_{\ell=1}^{r}d_{k_{\ell}}(i_{m_{\sigma(\ell)}}-i_{k_{\ell}}).
\end{align*}
Since $\ell \in U_{\sigma(\ell)}$ and $m_{\sigma(\ell)}\in \mathfrak I(U_{\sigma(\ell)})$, it follows that $d_{m_{\sigma(\ell)}}\geq d_{U_{\sigma(\ell)}}\geq d_{k_{\ell}}$, with equality only possible if $U_{\sigma(\ell)}= \{\sigma(\ell)\}$. Thus if $I\in E_{N}$, it follows that $i_{m_{\sigma(\ell)}}\geq i_{k_{\ell}}$, in which case
\begin{align*}
2^{\sum_{\ell=1}^{r}d_{k_{\ell}}i_{k_{\ell}}-\sum_{\ell=1}^{s}i_{m_{\ell}}d_{U_{\ell}}}\leq 2^{-\epsilon\sum_{\ell=1}^{r}(i_{m_{\sigma(\ell)}}-i_{k_{\ell}})}
\end{align*}
where we can take $\epsilon = d_{1}>0$.

\smallskip

We can now recast the identities in Proposition \ref{Prop2.10jk} in a way which, although losing  some information, makes them more useful and easier to work with when dealing with flag kernels

\begin{corollary}\label{Prop2.15jkl}
Fix a decomposition $\R^{N}= \R^{a_{1}}\oplus \cdots \oplus \R^{a_{n}}$.
Let $k_{1}, \ldots, k_{r}\in \{1, \ldots, N\}$ be a set of $r$ integers, possibly with repetitions, with $k_{\ell}\in J_{\pi(\ell)}$.\footnote{Recall from page \pageref{DefOfPi} that $\pi:\{1, \ldots, N\}\to \{1, \ldots, n\}$; for any coordinate $x_{j}$, $1 \leq j \leq N$, then $x_{j}$ is a coordinate in the factor $\R^{a_{\pi(j)}}$, and $J_{\pi(j)}$ is the set of indices of all the coordinates in $\R^{a_{\pi(j)}}$.} For $1 \leq \ell \leq r$, let $Z_{k_{\ell}}$ denote either $X_{k_{\ell}}$ or $Y_{k_{\ell}}$. Let $I\in \mathcal E_{n}$, and let $\psi \in \mathcal S(\R^{N})$ or $\mathcal C^{\infty}_{0}(\R^{N})$. 

\begin{enumerate}[{\rm(1)}]
\smallskip
\item  \label{Prop2.15jkl1} The function $(2^{d_{k_{1}}i_{\pi(1)}}Z_{k_{1}}) \cdots (2^{d_{k_{r}}i_{\pi(r)}}Z_{k_{r}})[\psi]_{I}$ can be written as a finite sum of terms of the following form. Decompose the set $\{1, \ldots, r\}$ into two disjoint complementary subsets $A$ and $B$ with $\pi(j) \neq n$ for any $j \in A$ and $B \neq \emptyset$. Let $B= \{\ell_{1}, \ldots, \ell_{s}\}$, and choose integers $\bar m=\{m_{1}, \ldots, m_{s}\}$ so that  each  $m_{t}\in J_{\ell_{t}}$. Then a typical term in the expansion of $(2^{d_{k_{1}}i_{\pi(1)}}Z_{k_{1}}) \cdots (2^{d_{k_{r}}i_{\pi(r)}}Z_{k_{r}})[\psi]_{I}$ is
\begin{align*}
2^{-\epsilon\sum_{j\in A}(i_{{\pi(j)}+1}-i_{{\pi(j)}})}\,(2^{d_{m_{1}}i_{\pi(1)}}\partial_{m_{1}})\cdots (2^{d_{m_{s}}i_{\pi(s)}}\partial_{m_{s}})[\psi_{A,B,\bar m}]_{I}
\end{align*}
where $\psi_{A,B,\bar m}$ is normalized relative to $\psi$. 

\smallskip
\item  \label{Prop2.15jkl2} The function $(2^{d_{k_{1}}i_{\pi(1)}}\partial_{k_{1}}) \cdots (2^{d_{k_{r}}i_{\pi(r)}}\partial_{k_{r}})[\psi]_{I}$ can be written as a finite sum of terms of the following form. Decompose the set $\{1, \ldots, r\}$ into two disjoint complementary subsets $A$ and $B$ with $\pi(j) \neq n$ for any $j \in A$ and $B \neq \emptyset$. Let $B= \{\ell_{1}, \ldots, \ell_{s}\}$, and choose integers $\bar m=\{m_{1}, \ldots, m_{s}\}$ so that  each  $m_{t}\in J_{\ell_{t}}$. Then a typical term in the expansion of $(2^{d_{k_{1}}i_{\pi(1)}}\partial_{k_{1}}) \cdots (2^{d_{k_{r}}i_{\pi(r)}}\partial_{k_{r}})[\psi]_{I}$ is
\begin{align*}
2^{-\epsilon\sum_{j\in A}(i_{{\pi(j)}+1}-i_{{\pi(j)}})}\,(2^{d_{m_{1}}i_{\pi(1)}}Z_{m_{1}})\cdots (2^{d_{m_{s}}i_{\pi(s)}}Z_{m_{s}})[\psi_{A,B,\bar m}]_{I}
\end{align*}
where $\psi_{A,B,\bar m}$ is normalized relative to $\psi$. 
\end{enumerate}
\end{corollary}

\begin{remarks}\label{Remarks4.9}\quad

\smallskip

\begin{enumerate}[{\rm(a)}]

\smallskip

\item {\rm The essential point of part (\ref{Prop2.15jkl1}) in the Corollary is that, when replacing  the operator 
\begin{equation*}
(2^{d_{k_{1}}i_{\pi(1)}}Z_{k_{1}}) \cdots (2^{d_{k_{r}}i_{\pi(r)}}Z_{k_{r}})[\psi]_{I}
\end{equation*}
with a sum of terms of the form
\begin{equation*}
2^{-\epsilon\sum_{j\in A}(i_{{\pi(j)}+1}-i_{{\pi(j)}})}\,(2^{d_{m_{1}}i_{\pi(1)}}\partial_{m_{1}})\cdots (2^{d_{m_{s}}i_{\pi(s)}}\partial_{m_{s}})[\psi_{A,B,\bar m}]_{I},
\end{equation*}
either the factor $(2^{d_{k_{\ell}}i_{\pi(\ell)}}Z_{k_{\ell}})$ is replaced by a term $(2^{d_{m_{\ell}}i_{\pi(\ell)}}\partial_{m_{\ell}})$ where the coordinate $x_{m_{\ell}}$ belongs to the \emph{same} subspace as $x_{k_{\ell}}$ and hence has the same dilation $i_{\pi(\ell)}$, or it is replaced by the gain $2^{-\epsilon(i_{\pi(\ell)+1}-i_{\pi(\ell)})}$. Part (\ref{Prop2.15jkl2}) is the same assertion with the roles of the vector fields and Euclidean differentiation interchanged.}

\smallskip

\item {\rm One term in the expansion of $(2^{d_{k_{1}}i_{\pi(1)}}Z_{k_{1}}) \cdots (2^{d_{k_{r}}i_{\pi(r)}}Z_{k_{r}})[\psi]_{I}$ in part (\ref{Prop2.15jkl1}) arises by letting $A = \emptyset$ so that $B=\{1, \ldots, r\}$, and then choosing $m_{r}=k_{r}$. We have seen that in this case the function $\psi_{A,B,\bar m}= \psi$, so we get the term $$(2^{d_{k_{1}}i_{\pi(1)}}\partial_{k_{1}}) \cdots (2^{d_{k_{r}}i_{\pi(r)}}\partial_{k_{r}})[\psi]_{I}.$$ \emph{Every other term} in the expansion then involves either a gain $2^{-\epsilon(i_{\pi(\ell)+1}-i_{\pi(\ell)})}$ or the replacement of a variable $x_{k_{\ell}}$ by a different variable $x_{m_{\ell}}$ with $d_{m_{\ell}}> d_{k_{\ell}}$. We shall say that any term of the form
\begin{equation*}
2^{-\epsilon\sum_{j\in A}(i_{{\pi(j)}+1}-i_{{\pi(j)}})}\,(2^{d_{m_{1}}i_{\pi(1)}}\partial_{m_{1}})\cdots (2^{d_{m_{s}}i_{\pi(s)}}\partial_{m_{s}})[\psi_{A,B,\bar m}]_{I},
\end{equation*}
where either $A\neq \emptyset$ or some $d_{m_{\ell}}> d_{k_{\ell}}$ is an \emph{allowable error}. Thus the difference
\begin{align*}
(2^{d_{k_{1}}i_{\pi(1)}}Z_{k_{1}}) \cdots (2^{d_{k_{r}}i_{\pi(r)}}Z_{k_{r}})[\psi]_{I}-(2^{d_{k_{1}}i_{\pi(1)}}\partial_{k_{1}}) \cdots (2^{d_{k_{r}}i_{\pi(r)}}\partial_{k_{r}})[\psi]_{I}
\end{align*}
is a sum of allowable errors.}

\smallskip

\item {\rm Since Euclidean derivatives commute, it follows that if $\sigma$ is any permutation of the set $\{1, \ldots, r\}$, then the difference
\begin{align*}
(2^{d_{k_{1}}i_{\pi(1)}}Z_{k_{1}}) \cdots (2^{d_{k_{r}}i_{\pi(r)}}Z_{k_{r}})[\psi]_{I}-(2^{d_{k_{1}}i_{\pi(1)}}Z_{k_{\sigma(1)}}) \cdots (2^{d_{k_{r}}i_{\pi(r)}}Z_{k_{\sigma(r)}})[\psi]_{I}
\end{align*}
is a sum of allowable errors.}

\end{enumerate}
\end{remarks}

\section{Cancellation}

A function is often said to have \emph{cancellation} if its average or integral is zero. For our purposes, we shall need a more refined notion involving integrals in some subset of variables. Let $J=\{j_{1}, \ldots, j_{s}\}\subseteq \{1, \ldots, N\}$. If $\psi\in\Sc(\R^{N})$, write
\begin{equation}
\int \psi(\x)\,d\x_{J} = \int_{\R^{s}}\psi(x_{1}, \ldots, x_{N})\,dx_{j_{1}}\cdots dx_{j_{s}}.
\end{equation}
Note that $\int \psi(\x)\,d\x_{J}$ is then a function of the variables $x_{k}$ for which $k \notin J$. We say that a function $\psi$ has \emph{cancellation in the variables $\{x_{j_{1}}, \ldots, x_{j_{s}}\}$} if $\int_{\R^{s}}\psi(\x)\,d\x_{J}\equiv 0$ where $J=\{j_{1}, \ldots, j_{s}\}$.  Later in Section \ref{StrongWeak} we shall give additional definitions of `strong cancellation' and `weak cancellation' for a function $\psi$.

\subsection{Cancellation and the existence of primitives}\quad

\smallskip

We begin by showing that cancellation in certain collections of variables is equivalent to the existence of appropriate primitives.

\begin{lemma}\label{Lemma1.12}
Let $\psi\in\Sc(\R^{N})$, and let $J_{k}\subseteq \{1, \ldots, n\}$ be non-empty subsets for $1 \leq k \leq r$. If the sets $\{J_{k}\}$ are mutually disjoint, then the following two statements are equivalent:
\begin{enumerate}[{\rm(a)}]
\smallskip
\item \label{1.12a} For $1\leq k \leq r$, $\displaystyle \int \psi(\x)\,d\x_{J_{k}}=0$.

\smallskip

\item \label{1.12b} There are functions $\psi_{j_{1}, \ldots, j_{r}}\in \Sc(\R^{N})$, normalized with respect to $\psi$, such that
\begin{equation*}
\psi(\x) = \sum_{j_{1}\in J_{1}}\cdots \sum_{j_{r}\in J_{r}}\partial_{{j_{1}}}\cdots \partial_{{j_{r}}}\psi_{j_{1}, \ldots, j_{r}}(\x).
\end{equation*}
\end{enumerate}
Moreover, if the function $\psi\in \Cs^{\infty}_{0}(\R^{N})$, then we can choose the functions $\psi_{j_{1}, \ldots, j_{r}}\in \Cs^{\infty}_{0}(\R^{N})$ normalized with respect to $\psi$.
\end{lemma}

It follows easily from the fundamental theorem of calculus that (\ref{1.12b}) implies (\ref{1.12a}). The main content of the Lemma is thus the opposite implication. This will follow by induction on $r$ from the following assertion.

\begin{proposition}\label{Prop1.13}
Let $\psi\in \Cs^{\infty}_{0}(\R^{N})$. Suppose that $J_{1}, J_{2}\subset \{1, \ldots, N\}$ are non-empty and disjoint. If \,\, $\int\psi(\x)\,d\x_{J_{1}}=\int\psi(\x)\,d\x_{J_{2}}=0$, then for each $k\in J_{1}$ there is a function $\psi_{k}\in \Cs^{\infty}_{0}(\R^{N})$, normalized relative to $\psi$, so that: 
\begin{enumerate}[{\rm(i)}]
\item we can write $\psi = \sum_{k\in J_{1}}\partial_{{k}}\psi_{k}$;
\item for each $k\in J_{1}$ we still have $\int \psi_{k}(\x)\,d\x_{J_{2}}=0$.
\end{enumerate}
If $\psi\in \Sc(\R^{N})$, the same conclusions hold except that  the functions $\psi_{k}\in \Sc(\R^{N})$ and are normalized relative to $\psi$.
\end{proposition}

\begin{proof}
By relabeling the coordinates, we can assume that $J_{1}=\{1, \ldots, k\}$, and that $J_{2}\subseteq\{k+1,\ldots, N\}$.  Suppose that $\psi$ has compact support in the set $B=\{\x\in\R^{N}\,\big\vert\,|x_{j}|<a_{j}\}$. Choose $\chi\in\Cs^{\infty}_{0}(\R)$ with support in $[-1,+1]$ such that $\int_{\R}\chi(s)\,ds = 1$, and put $\chi_{j}(t) = a_{j}^{-1}\chi(a_{j}^{-1}t)$ so that $\chi_{j}$ is support in $[-a_{j},+a_{j}]$ and still has integral equal to $1$. Put
\begin{equation*}
\varphi_{1}(\x) = \psi(\x) - \chi_{1}(x_{1})\,\int_{\R}\psi(s, x_{2}, \ldots, x_{N})\,ds
\end{equation*}
and for $2\leq j \leq k$, put
\begin{align*}
\varphi_{j}(\x) &= \Big[\prod_{l=1}^{j-1}\chi_{l}(x_{l}) \Big]\int_{\R^{j-1}}\psi(s_{1}, \ldots, s_{j-1},x_{j}, \ldots, x_{N})\,ds_{1}\cdots ds_{j-1}
\\
&\qquad \qquad\qquad
-
\Big[\prod_{l=1}^{j}\chi_{l}(x_{l})\Big] \int_{\R^{j}}\psi(s_{1}, \ldots, s_{j},x_{j+1}, \ldots, x_{N})\,ds_{1}\cdots ds_{j}.
\end{align*}
Then the functions $\{\varphi_{j}\}$ have the following properties. First, since  $\int\psi(\x)\,d\x_{J_{1}}=0$,
the second integral in the definition of the last function $\varphi_{k}$ is zero, and hence $\psi(\x) = \sum_{j=1}^{k}\varphi_{j}(\x)$. Next, it is clear that each  $\varphi_{j}$ is supported in the set $B$. Finally, for  $1 \leq j \leq k$, 
$$
\int_{\R}\varphi_{j}(x_{1}, \ldots,x_{j-1},s,x_{j+1}, \ldots, x_{N})\,ds= 0,
$$ 
so if we put 
\begin{equation}\label{E29}
\begin{aligned}
\psi_{j}(\x) &= \int_{-\infty}^{x_{j}}\varphi_{j}(x_{1}, \ldots,x_{j-1},s,x_{j+1}, \ldots, x_{N})\,ds \\
&= -\int_{x_{j}}^{\infty}\varphi_{j}(x_{1}, \ldots,x_{j-1},s,x_{j+1}, \ldots, x_{N})\,ds ,
\end{aligned}
\end{equation}
then $\psi_{j}$ is supported on the set $B$, and $ \varphi_{j}(\x) = \partial_{j}\psi_{j}(\x)$. It is clear that one can estimate the size of the derivatives of the functions $\{\psi_{j}\}$ in terms of the derivatives of $\psi$, so $\psi_{j}$ is normalized in terms of $\psi$. Moreover, since $\int\psi(\x)\,d\x_{J_{2}}=0$, it follows from their definitions that $\int\varphi_{j}(\x)\,d\x_{J_{2}}=0$, and hence  $\int\psi_{j}(\x)\,d\x_{J_{2}}=0$. This completes the proof if $\psi\in \Cs^{\infty}_{0}(\R^{N})$. If $\psi\in \Sc(\R^{N})$, the proof goes the same way. One only has to observe from equation (\ref{E29}) that the functions $\psi_{j}\in\Sc(\R^{N})$. This completes the proof of Proposition \ref{Prop1.13}, and hence Lemma \ref{Lemma1.12} is also established.
\end{proof}

\subsection{Strong and weak cancellation}\label{StrongWeak}\quad

\smallskip
In this section we introduce two kinds of cancellation conditions that can be imposed on functions in $\mathcal S(\R^{N})$ or $\mathcal C^{\infty}_{0}(\R^{N})$. As discussed in  Section \ref{FlagKernelsA}, these concepts are used in the context of the  decomposition $\R^{N}= \R^{a_{1}}\oplus \cdots \oplus\R^{a_{n}}$ given in (\ref{Eqn2.5tyu}) where $a_{1}+ \cdots +a_{n}= N$ and each $a_{j}\geq 1$. Recall that if $\x\in \R^{N}$, we write  $\x= (\x_{1}, \ldots, \x_{n})$ where $\x_{l}= (x_{p_{l}}, x_{p_{l}+1}, \ldots, x_{q_{l}})\in \R^{a_{l}}$. We then let $J_{l}$ denote the set of integers $\{p_{l}, p_{l}+1, \ldots, q_{l}\}$. If $\varphi\in \mathcal S(\R^{N})$, set
\begin{equation}\label{Equation2.20zx}
\int_{\R^{a_{l}}}\varphi(\x_{1}, \ldots, \x_{n})\,d\x_{l}= \int_{\R^{a_{l}}}\varphi(\x_{1}, \ldots, \x_{n})\,dx_{p_{l}}\cdots dx_{q_{l}}.
\end{equation}

\smallskip

\begin{definition}\label{Def1.14}
Fix the decomposition (\ref{Eqn2.5tyu}), and let $\varphi\in \Sc(\R^{N})$. The function $\varphi$ has \emph{strong cancellation} if and only if 
\begin{equation*}
\int_{\R^{a_{l}}}\varphi(\x_{1}, \ldots, \x_{n})\,d\x_{\ell}\equiv 0, \qquad 1 \leq \ell \leq n.
\end{equation*}
That is, $\varphi$ has strong cancellation if and only if it has cancellation in each collection of variables $\{x_{p_{l}}, \ldots, x_{q_{l}}\}$ for $1 \leq l \leq n$.
\end{definition}

\begin{remark} \label{Remark2.18}
{\rm It follows from Proposition \ref{Prop1.13} that  $\varphi$ has strong cancellation if and only if there are functions $\varphi_{j_{1}, \ldots, j_{n}}$ normalized with respect to $\varphi$ so that
\begin{equation*}
\varphi = \sum_{j_{1}\in J_{1}}\cdots \sum_{j_{n}\in J_{n}}\partial_{{j_{1}}}\cdots \partial_{{j_{n}}}\varphi_{j_{1}, \ldots, j_{n}}.
\end{equation*}}
\end{remark}

\smallskip

We now introduce a weaker cancellation condition.

\begin{definition}\label{Def1.14a2}
 Fix the decomposition (\ref{Eqn2.5tyu}) of $\R^{N}$.
Let $\varphi\in \Sc(\R^{N})$ or $\varphi\in \mathcal C^{\infty}_{0}(\R^{N})$, and let  $I=(i_{1}, \ldots i_{n}) \in \mathcal E_{n}$. The function $\varphi$ has \emph{weak cancellation} with parameter $\epsilon >0$ relative to the multi-index $I$  if and only if 
\begin{align*}
\varphi &= \sum_{\substack{A\cup B = \{1, \ldots, n\}\\A=\{\alpha_{1}, \ldots, \alpha_{r}\}\\n\notin B}}\,\Big(\prod_{s\in B}2^{-\epsilon(i_{s+1}-i_{s})} \Big)\,
\sum_{j_{1}\in J_{\alpha_{1}}}\cdots \sum_{j_{r}\in J_{\alpha_{r}}} \big(\partial_{j_{1}}\cdots \partial_{j_{r}}\big)[\varphi_{A,B,j_{1}, \ldots, j_{r}}]
\end{align*}
where each $\varphi_{A,B,j_{1}, \ldots, j_{r}}$ is normalized relative to $\varphi$. Here the outer sum is taken over all decompositions of the set $\{1, \ldots, n\}$ into two disjoint subsets $A$ and $B$ such that $n\in A$.
\end{definition}

According to Remark \ref{Remark2.18},  $\varphi$ has strong cancellation if it can be written as a sum of functions of the form $\partial_{{j_{1}}}\cdots \partial_{{j_{n}}}\varphi$; \textit{i.e.} as $n^{th}$-derivatives of functions where there is one derivative in a variable from each of the $n$ subspaces  $\R^{a_{l}}$. Definition \ref{Def1.14a2} imposes a weaker condition; a function $\varphi$ has weak cancellation if again it is a sum of terms, but the term $\partial_{{j_{1}}}\cdots \partial_{{j_{n}}}\varphi_{j_{1}, \ldots, j_{n}}$ is itself replaced by a new sum of terms.  If a derivative $\partial_{j_{l}}$ does not appear, so that the term does not have cancellation in the subspace $\R^{a_{l}}$, there is instead a gain given by $2^{-\epsilon (i_{l+1}-i_{l})}$.

\smallskip

\begin{remarks}\label{Remark2.20}
{\rm There will be occasions when we will use the fact that a function $\varphi$ has weak cancellation with respect to $I\in \mathcal E_{n}$ to draw inferences about existence of primitives and smallness in only some of the subspaces $\{\R^{a_{l}}\}$. In particular, the following assertions follow easily from Definition \ref{Def1.14a2}.}
\begin{enumerate}[{\rm(1)}]
\item {\rm Suppose that  $I\in \mathcal E_{n}$ and that $\varphi$ has weak cancellation relative to $I$. Let $M\subset\{1, \ldots, n\}$. 
Then 
\begin{align*}
\varphi &= \sum_{\substack{A\cup B = M\\A=\{\alpha_{1}, \ldots, \alpha_{r}\}\\n\notin B}}\,\Big(\prod_{s\in B}2^{-\epsilon(i_{s+1}-i_{s})} \Big)\,
\sum_{j_{1}\in J_{\alpha_{1}}}\cdots \sum_{j_{r}\in J_{\alpha_{r}}} \big(\partial_{j_{1}}\cdots \partial_{j_{s}}\big)[\widetilde\varphi_{A,B,j_{1}, \ldots, j_{r}}]
\end{align*}
where each $\widetilde\varphi_{A,B,j_{1}, \ldots, j_{r}}$ is normalized relative to $\varphi$, and where the outer sum is over all subsets $B\subset M$, with the understanding that if $n\in M$ then $n\notin B$. }

\smallskip

\item {\rm In particular, if we take $M= \{1\}$, it follows that if $\varphi$ has weak cancellation relative to $I\in \mathcal E_{n}$, we can write
\begin{equation*}
\varphi= \sum_{r={1}}^{a_{1}}\partial_{r}[\varphi_{r}] +2^{-\epsilon(i_{2}-i_{1})}\varphi_{0}
\end{equation*}
where $\{\varphi_{0}, \varphi_{1}, \ldots, \varphi_{a_{1}}\}$ are normalized relative to $\varphi$. Here, the derivatives are with respect to the variables in the first subspace $\R^{a_{1}}$.}
\end{enumerate}
\end{remarks}

We can also characterize weak cancellation in terms of the ``smallness'' of integrals, to be compared with the characterization of strong cancellation in terms of the vanishing of integrals given in Lemma \ref{Lemma1.12}. For any partition $\{1, \ldots, n\} = A\cup B$ with $A= \{j_{1}, \ldots, j_{a}\}$ and $B= \{k_{1}, \ldots, k_{b}\}$,  write $\x\in \R^{N}$ as $\x=(\x_{A}, \x_{B})$ where $\x_{A}= (\x_{j_{1}}, \ldots, \x_{j_{l}})$, $\x_{B}=(\x_{k_{1}}, \ldots, \x_{k_{b}})$. Let $d\x_{A}= d\x_{j_{1}}\cdots d\x_{j_{l}}$, and let $d\x_{B}=d\x_{k_{1}}\cdots d\x_{k_{b}}$. 

\begin{proposition}\label{Prop5.7ijn}
Let $\epsilon >0$ and $I \in \E_{n}$. A function $\psi\in \mathcal C^{\infty}_{0}(\R^{N})$ has weak cancellation with parameter $\epsilon$ relative to $I$ if and only if for every partition $\{1, \ldots, n\}= A\cup B$ into disjoint subsets with $n \notin B$, we have
\begin{equation*}
\int_{\bigoplus_{k\in B}\R^{a_{k}}}\psi(\x_{A}, \x_{B})\,d\x_{B} = \Big[\prod_{k\in B}2^{-\epsilon(i_{k+1}-i_{k})}\Big]\psi_{A}(\x_{A})
\end{equation*}
where $\psi_{A}\in \mathcal S(\bigoplus_{j \in A}\R^{a_{j}})$ is normalized relative to $\psi$. If $\psi\in C^{\infty}_{0}(\R^{N})$, the functions $\psi_{A}\in C^{\infty}_{0}(\bigoplus_{j \in A}\R^{a_{j}})$.
\end{proposition}

\begin{proof}
It is clear that if $\psi$ has weak cancellation, then it satisfies the condition of Proposition \ref{Prop5.7ijn}, so the main content is the opposite implication.

Let $\chi_{l}\in \mathcal C^{\infty}_{0}(\R^{a_{l}})$ have support in the set where $|\x_{l}|\leq 1$, with $\int_{\R^{a_{l}}}\chi(\x_{l})\,d\x_{l} = 1$. For $\psi\in \mathcal S(\R^{N})$ and $1 \leq l \leq N$, define
\begin{align*}
L_{l}[\psi](\x) &= \psi(\x) - \chi_{l}(\x_{l})\,\int_{\R^{a_{l}}}\psi(\x)\,d\x_{l},\\
M_{l}[\psi](\x) &= \chi_{l}(\x_{l})\,\int_{\R^{a_{l}}}\psi(\x)\,d\x_{l}.
\end{align*}
It is easy to check that the operators $\{L_{1}, \ldots, L_{n}, M_{1}, \ldots, M_{n}\}$ all commute, and that
\begin{align*}
\int_{\R^{a_{l}}}M_{l}[\psi](\x)\,d\x_{l} &= \int_{\R^{a_{l}}}\psi(\x)\,d\x_{l}, \\
\int_{\R^{a_{l}}}L_{l}[\psi](\x)\,d\x_{l} &= 0.
\end{align*}
Since $L_{l}+M_{l}$ is the identity operator, if $\psi\in \mathcal S(\R^{N})$ 
\begin{align*}
\psi(\x) &= \prod_{l=1}^{n}(L_{l}+M_{l})[\psi](\x)\\
&=
\sum_{A\cup B = \{1, \ldots, n\}}\Big(\prod_{j\in A}L_{l}\Big)\Big(\prod_{k\in B}M_{k}\Big)[\psi](\x) = \sum_{A\cup B = \{1, \ldots, n\}}\psi_{A,B}(\x).
\end{align*}
It is clear that the functions $\{\psi_{A,B}\}$ are normalized relative to $\psi$. Note that
\begin{equation*}
\Big(\prod_{k\in B}M_{k}\Big)[\psi](\x) = \Big(\prod_{k\in B}\chi_{k}(\x_{k})\Big)\,\int_{\bigoplus_{k\in B}\R^{a_{k}}}\psi(\x_{A}, \x_{B})\,d\x_{B} 
\end{equation*}
and for every $j \in A$ we have
\begin{equation*}
\int_{\R^{a_{j}}}\psi_{A,B}(\x)\,d\x_{j}=0.
\end{equation*}

Now suppose that $\psi$ satisfies the hypotheses of Proposition \ref{Prop5.7ijn}. For every decomposition $\{1, \ldots, n\} = A\cup B$, it follows that 
\begin{align*}
\Big(\prod_{k\in B}M_{k}\Big)[\psi](\x) =\Big[\prod_{k\in B}2^{-\epsilon(i_{k+1}-i_{k})}\Big]\psi_{A}(\x_{A})\Big(\prod_{k\in B}\chi_{k}(\x_{k})\Big)
\end{align*}
where $\psi_{A}$ is normalized relative to $\psi$. But then
\begin{align*}
\psi_{A,B}(\x) &= \Big(\prod_{j\in A}L_{l}\Big)\Big(\prod_{k\in B}M_{k}\Big)[\psi](\x)\\
&=
\Big[\prod_{k\in B}2^{-\epsilon(i_{k+1}-i_{k})}\Big]\Big(\prod_{j\in A}L_{l}\Big)[\psi_{A}](\x_{A})\,\Big(\prod_{k\in B}\chi_{k}(\x_{k})\Big).
\end{align*}
Now the function $\Big(\prod_{j\in A}L_{l}\Big)[\psi_{A}](\x_{A})\,\Big(\prod_{k\in B}\chi_{k}(\x_{k})\Big)$ has cancellation in each of the subspaces $\R^{a_{j}}$ for $j\in A$. If we write $A= \{\alpha_{1}, \ldots \alpha_{r}\}$,  it follows from Lemma \ref{Lemma1.12} that we can write this function as a sum of terms of the form $\partial_{j_{1}}\cdots \partial_{j_{r}}[\psi_{A,B,\sigma}]$ where the variable $x_{j_{k}}$ belongs to the subspace $\R^{a_{\alpha_{k}}}$. This then makes it clear that $\psi$ has weak cancellation, and completes the proof.
\end{proof}

\section{The structure of flag kernels}\label{FlagKernels}

In this section we establish four important properties of flag kernels. The first (Theorem \ref{Lemma2.3}) shows how to decompose a given flag kernel into a sum of dilates of compactly supported functions with strong cancellation. The second (Theorem \ref{Lemma5.4zw}) shows conversely that a sum of dilates of Schwartz functions with weak cancellation converges to a flag kernel. The third (Theorem \ref{restricted}) shows that the cancellation conditions imposed in part (\ref{Def2.1b}) of Definition \ref{Def2.1} can be relaxed. The fourth (Theorem \ref{Thm3.12}) shows that the family of flag kernels is invariant under appropriately chosen changes of variables.

\subsection{Dyadic decomposition of flag kernels}\label{SecDyadDecomp}\quad

\smallskip

Every flag kernel is a product kernel in the sense of Definition 2.1.1 of \cite{NaRiSt00}, and Corollary 2.2.2 in that reference implies that if $\K$ is a flag kernel, then there are functions $\varphi^{I}\in \mathcal C^{\infty}_{0}(\R^{N})$, $I\in \mathbb Z^{n}$ with strong cancellation and uniformly bounded seminorms $||\varphi^{I}||_{(m)}$, supported in the set where $\frac{1}{2}\leq N_{j}(\x_{j})\leq 4$ for every $j$, such that $\K = \sum_{I\in \mathbb Z^{n}}[\varphi^{I}]_{I}$ in the sense of distributions. Recall that $\E_{n}=\{I=(i_{1}, \ldots, i_{n}) \in \Z^{n}\,\big\vert\,i_{1}\leq i_{2}\leq \cdots \leq i_{n}\}$. Since a flag kernel satisfies better differential inequalities than a general product kernel, one expects that it should be possible to write $\K$ as a sum of dilates of functions $\varphi^{I}\in \mathcal C^{\infty}_{0}(\R^{N})$ where the dilations range \emph{only} over the set $\E_{n}$ instead of over all of $\mathbb Z^{n}$. Such a result is stated in Corollary 2.2.4 of \cite{NaRiSt00}, but the precise statement there is not correct because one must allow additional terms involving flag kernels adapted to coarser flags. This section provides a correct statement and proof.

\begin{theorem}\label{Lemma2.3}
Let $\mathcal K$ be a flag kernel adapted to the standard flag $\mathcal F$ 
\begin{equation*}\tag{\ref{Eqn2.6tyu}}
(0)\subseteq \R^{a_{n}}\subseteq \R^{a_{n-1}}\oplus \R^{a_{n}}\subseteq \cdots \subseteq \R^{a_{3}}\oplus\cdots\oplus \R^{a_{n}}\subseteq \R^{a_{2}}\oplus \cdots \oplus \R^{a_{n}}\subseteq \R^{N}
\end{equation*}
Then there is a decomposition $\K = \K_{0} +\K_{1}+ \cdots +\K_{n}$ with the following properties.
\begin{enumerate}[{\rm(1)}]

\smallskip

\item For each $I \in \mathcal E_{n}$ there is a function $\varphi^{I}\in \Cs^{\infty}_{0}(\R^{N})$ so that $\K_{0}= \sum_{I\in \E_{n}}[\varphi^{I}]_{I}$ with convergence in the sense of distributions. Moreover:

\begin{enumerate}[{\rm (a)}]

\smallskip

\item the support of each function $\varphi^{I}$ is contained in the unit ball $B= \{\x\in\R^{N}\,\big\vert\,|\x|\leq 1\}$;

\smallskip

\item there are constants  $C_{m}>0$ (depending on the constants for the flag kernel $\K$) so that for each $I \in \mathcal E_{n}$ and all $m\geq 0$,   $||\varphi^{I}||_{(m)}\leq C_{m}$;

\smallskip

\item each function $\varphi^{I}$ has strong cancellation in the sense of Definition \ref{Def1.14}: for $1 \leq j \leq n$
$$\int_{R^{a_{j}}}\varphi(\x_{1}, \ldots, \x_{j}, \ldots, \x_{n})\,d\x_{j}=0;$$

\smallskip

\item  for each $I\in \mathcal E_{n}$,  $\varphi^{I}(\x_{1}, \ldots, \x_{n})=0$ for $|\x_{1}|<\frac{1}{8}$.

\smallskip

\end{enumerate}

\smallskip

\item For $1 \leq j \leq n$,  each $\K_{j}$ is a flag kernel adapted to a flag which is \emph{strictly coarser} than $\F$.

\end{enumerate}
\end{theorem}

\smallskip
Recall that in Section \ref{Sect1.2pjv} we distinguished between elements $I=(i_{1}, \ldots, i_{n}) \in \mathcal E_{n}$ where we have strict inequality $i_{1}<i_{2}< \cdots < i_{n-1}<i_{n}$, and elements where some of the entries are equal. Underlying this dichotomy is the fact that $\mathcal E_{n}$ is a polyhedral cone in $\Z^{n}$, the set of elements with strict inequality form the `open' $n$-dimensional interior, and elements with various sets of equal elements correspond to lower dimensional faces. In Theorem \ref{Lemma2.3}, we may as well assume that the dyadic sum representing $\K_{0}$ extends over the interior n-tuples with strictly increasing components, leaving the "boundary n-tuples" with repeated indices to parametrize the dyadic terms of the sums representing the other $\K_{j}$'s.

\goodbreak

Using induction on the number of steps in a flag, Theorem \ref{Lemma2.3} immediately gives us the following corollary.

\begin{corollary}\label{Cor3.2xz}
Let $\K$ be a flag kernel adapted to a standard flag $\mathcal F$  of length $n$. There is a finite collection of flags $\{\mathcal F_{k,s}\}$, $1\leq k \leq n$ and $1 \leq s \leq b_{k}$, with the following properties.
\begin{enumerate}[{\rm(1)}]

\smallskip

\item For $k=n$, $b_{n}=1$ and $\mathcal F_{n,1}= \mathcal F$.

\smallskip

\item  For each $k<n$, the flag $\mathcal F_{k,s}$ has length $k$ and is strictly coarser than $\mathcal F$.

\smallskip

\item For each $(k,s)$ there is a uniformly bounded family of functions $\{\varphi^{J}_{k,s}\}\subset\mathcal C^{\infty}_{0}(\R^{N})$, $J\in \E_{k}$, all supported in the unit ball and having strong cancellation relative to the decomposition of $\R^{N}$ corresponding to the flag $\mathcal F_{k,s}$ so that in the sense of distributions,
\begin{equation}\label{Eqn3.1ps}
\K = \sum_{k=1}^{n}\sum_{s=1}^{b_{k}}\sum_{J\in E_{k}}[\varphi^{J}_{k,s}]_{J}.
\end{equation}
\end{enumerate}
\end{corollary}

The proof of Theorem \ref{Lemma2.3} relies on three preliminary results.
The first, which gives the characterization of flag kernels in terms of their Fourier transforms, was established in  \cite{NaRiSt00}. We briefly recall the relevant definitions. Let $(\R^{N})^{*}$ denote the space of linear functionals on $\R^{N}$, and for a subspace $W\subseteq \R^{N}$, let $W^{\perp}\subset(\R^{N})^{*}$ be the subspace of linear functions which are zero on $W$. If $\mathcal F$ is the flag in $\R^{N}$ given in (\ref{Eqn2.6tyu}), the dual flag $\xi\in (\R^{N})^{*}$ is $\mathcal F^{*}$ given by
\begin{equation*}
(0)^{\perp} \supseteq (\R^{a_{n}})^{\perp}\supseteq (\R^{a_{n-1}}\oplus \R^{a_{n}})^{\perp}\supseteq \cdots \supseteq (\R^{a_{2}}\oplus \cdots \oplus \R^{a_{n}})^{\perp} =(\R^{N})^{\perp}.
\end{equation*}
If we identify $(\R^{a_{k}}\oplus\cdots \oplus \R^{a_{n}})^{\perp}$ with $(\R^{a_{1}})^{*}\oplus \cdots \oplus (\R^{a_{k-1}})^{*}$, the dual flag becomes
\begin{equation}\label{E2.5}
(0) \subseteq (\R^{a_{1}})^{*}\subseteq (\R^{a_{1}})^{*}\oplus (\R^{a_{2}})^{*}\subseteq \cdots \subseteq (\R^{a_{1}})^{*}\oplus\cdots \oplus (\R^{a_{n-1}})^{*}\subseteq (\R^{N})^{*}.
\end{equation}
If $\xi \in (\R^{N})^{*}$, we write $\xi = (\xi_{1}, \ldots, \xi_{n})$ where $\xi_{j}=(\xi_{j,1}, \ldots, \xi_{j,a_{j}})\in (\R^{a_{j}})^{*}$. The family of dilations on $\R^{N}$ defined in equation (\ref{E2}) induce a family of dilations on $(\R^{N})^{*}$ so that $\langle r\cdot \x,\xi\rangle= \langle\x,r\cdot \xi\rangle$. We let $|\xi|$ be a smooth homogeneous norm on $(\R^{N})^{*}$, and if $\xi_{j}\in (\R^{a_{j}})^{*}$, we let $|\xi_{j}|$ be the restriction of the norm to this subspace.

\begin{definition} \label{Def2.4}
A \emph{flag multiplier} relative to the flag $\mathcal F^{*}$ given in (\ref{E2.5}) is a function $m(\xi)$ which is infinitely differentiable away from the subspace $\xi_{n}=0$, and which satisfies the differential inequalities
\begin{equation*}
\big\vert\partial^{\bar\alpha_{1}}_{\xi_{1}}\cdots \partial^{\bar\alpha_{n}}_{\xi_{n}}m(\xi)\big\vert \leq C_{\alpha}\,\prod_{j=1}^{n}(|\xi_{j}|+\cdots +|\xi_{n}|)^{-\<  \bar\alpha_{j}\> }.
\end{equation*}
\end{definition}
 
We can now state Theorem 2.3.9 of \cite{NaRiSt00} as follows:

\begin{lemma}\label{Lemma2.5}
Let $\K$ be a flag kernel adapted to the flag $\mathcal F$. Then the Fourier transform of $\K$ is a flag multiplier relative to the dual flag $\mathcal F^{*}$. Conversely, every flag multiplier relative to the flag $\mathcal F^{*}$ is the Fourier transform of a flag kernel adapted to the flag $\mathcal F$.
\end{lemma}

\medskip

The next preliminary result provides a decomposition of test functions in $\mathcal S(\R^{N})$.

\begin{lemma}\label{Lemma2.8}
Let $\psi\in \mathcal S(\R^{N})$. Then there are functions $\{\psi^{k}\}\subset \mathcal C^{\infty}_{0}(\R^{N})$, $k = 0,\,1,\,2,\ldots$ \,such that 
\begin{equation*}
\psi(\x) = \sum_{k=0}^{\infty}2^{-kQ}\psi^{k}(2^{-k}\cdot\x),
\end{equation*}
where $Q$ is the homogeneous dimension of $\R^{N}$. Moreover these functions have the following properties.
\begin{enumerate}[{\rm(a)}]

\smallskip

\item \label{Lemma2.8a} Each $\psi^{k}$ is supported in the unit ball;

\smallskip

\item \label{Lemma2.8b} For any $\delta>0$ and any $\alpha\in \N^{N}$, there exists $M\in\N$ depending on $\delta$ and $\alpha$ so that 
\begin{equation*}
\sup_{\x\in\R^{N}}\big\vert\partial^{\alpha}\psi^{k}(\x)\big\vert \leq ||\psi||_{[M]}\,2^{-k\delta};
\end{equation*}

\smallskip

\item \label{Lemma2.8c} If $\psi$ has strong cancellation, then each $\psi^{k}$ has strong cancellation.
\end{enumerate}
\end{lemma}

\begin{proof}
Choose $\eta\in \mathcal C^{\infty}_{0}(\R)$ supported in $[-1,-\frac{1}{8}]\cup[\frac{1}{8},1]$ with $\eta(t)=\eta(-t)$ such that $$\sum_{k=-\infty}^{+\infty}\eta(2^{-k}t) = 1$$ for all $t\neq 0$. For any $t\in \R$, including $t=0$, set $\eta_{0}(t) = 1-\sum_{k=1}^{\infty}\eta(2^{-k}t)$. Then $\eta_{0}\in \mathcal C^{\infty}_{0}(\R)$ is supported in $[-1,+1]$, and $\eta_{0}(t) + \sum_{k=1}^{\infty}\eta(2^{-k}t) \equiv 1$ for all $t\in \R$. Recall that $\x \to |\x|$ is a smooth homogeneous norm on $\R^{N}$. Set
\begin{equation*}
\psi^{k}(\x) =
\begin{cases}
\psi(\x)\eta_{0}(|\x|) & \text{if $k=0$,}\\\\
2^{kQ}\psi(2^{k}\cdot \x)\eta(|\x|) &\text{if $k\geq 1$}.
\end{cases}
\end{equation*}
Then $2^{-kQ}\psi^{k}(2^{-k}\cdot\x) = \psi(\x)\eta(2^{-k}|\x|)$ for $k\geq 1$, and so $\psi(\x) = \sum_{k=0}^{\infty}2^{-kQ}\psi^{k}(2^{-k}\cdot\x)$. From the choice of $\eta$ and the definition of $\eta_{0}$ it follows that each $\psi^{k}$ is supported on the set where $|\x|\leq 1$, and this gives the assertion (\ref{Lemma2.8a}). 

Since $\psi\in \mathcal S(\R^{N})$, it follows that for every  $M\in \N$ and every $\alpha\in \N^{N}$ with $|\alpha|\leq M$ we have
\begin{equation}\label{E2.9a}
|\partial^{\alpha}\psi(\x)|\leq ||\psi||_{[M]}(1+|\x|_{e})^{-M},
\end{equation}
where $|\x|_{e}$ is the Euclidean length of the vector $\x\in \R^{N}$.
If $\x$ belongs to the support of $\psi^{k}$ when $k\geq 1$, we have $|\x|\geq 2^{-3}$, and so $|2^{k}\cdot\x|\geq 2^{k-3}$. Assertion (\ref{Lemma2.8b}) then follows easily from equation (\ref{E2.9a}) and the fact that $|\x|^{d_{1}}\lesssim |\x|_{e}$ if $|\x|_{e}\geq 1$. 

Finally, suppose that $\psi$ has strong cancellation, so that $\int_{\R^{a_{k}}}\psi(\x_{1}, \ldots, \x_{n})\,d\x_{k}=0$ for $1 \leq k \leq n$. It follows from Lemma \ref{Lemma1.12} that we can write $\psi$ as a finite sum of terms of the form $\partial_{j_{1}}\cdots \partial_{j_{n}}\psi_{j_{1}, \ldots, j_{n}}$ where $x_{j_{k}}$ is a coordinate in $\x_{k}$ for $1\leq k \leq n$, and each function $\psi_{j_{1}, \ldots, j_{n}}\in \mathcal S(\R^{N})$ is normalized relative to $\psi$. Using assertions (\ref{Lemma2.8a}) and (\ref{Lemma2.8b}), we can write
\begin{equation*}
\psi_{j_{1}, \ldots, j_{n}}(\x)=\sum_{k=0}^{\infty}2^{-kQ}\psi_{j_{1}, \ldots, j_{n}}^{k}(2^{-k}\cdot\x),
\end{equation*}
and so
\begin{equation*}
\partial_{j_{1}}\cdots \partial_{j_{n}}\psi_{j_{1}, \ldots, j_{n}}(\x)=\sum_{k=0}^{\infty}2^{-k(d_{j_{1}}+ \cdots+d_{j_{n}}+Q)}\partial_{j_{1}}\cdots \partial_{j_{n}}\psi_{j_{1}, \ldots, j_{n}}^{k}(2^{-k}\cdot\x).
\end{equation*}
Since each term $2^{-k(d_{j_{1}}+ \cdots+d_{j_{n}})}\partial_{j_{1}}\cdots \partial_{j_{n}}\psi_{j_{1}, \ldots, j_{n}}^{k}$ has strong cancellation, summing over a finite number of such terms establishes assertion (\ref{Lemma2.8c}), and completes the proof.
\end{proof}

Finally, we will need the following result which provides a decomposition of test functions in $\mathcal C^{\infty}_{0}(\R^{N})$,

\begin{lemma}\label{Lemma2.7w}
Let $\varphi\in \mathcal C^{\infty}_{0}(\R^{N})$ have compact support in the unit ball, and suppose that $\varphi$ has cancellation in $\x_{1}$; \textit{i.e.}  $\int_{\R^{a_{1}}}\varphi(\x_{1}, \x_{2}, \ldots, \x_{n})\,d\x_{1}=0$. Then there are functions $\{\varphi^{j}\}\subseteq\mathcal C^{\infty}_{0}(\R^{N})$ such that
\begin{equation*}
\varphi(\x_{1}, \x_{2}, \ldots, \x_{n})= \sum_{j=-\infty}^{0}2^{-kQ_{1}}\varphi^{j}(2^{-j}\cdot\x_{1}, \x_{2}, \ldots, \x_{n}),
\end{equation*}
where $Q_{1}$ is the homogeneous dimension of $\R^{a_{1}}$. Moreover, these functions have the following properties.

\begin{enumerate}[{\rm(a)}]

\smallskip

\item Each $\varphi^{j}$ is supported in the unit ball;

\smallskip

\item Each $\varphi^{j}$ is normalized relative to the function $\varphi$;

\smallskip

\item For $-\infty<j\leq 0$ we have $\varphi^{j}(\x_{1}, \x_{2}, \ldots, \x_{n}) = 0$ if $|\x_{1}|\leq \frac{1}{8}$;

\smallskip

\item If $\varphi$ has strong cancellation, then each function $\varphi^{j}$ has strong cancellation.
\end{enumerate}
\end{lemma}

\begin{proof}
Choose $\eta \in \mathcal C^\infty_{0}(\R)$ such that $\eta( t )$ vanishes if $t \geq 2$ or $t\leq \frac{1}{4}$, and such that $\sum_{l \leq 0} \eta( 2^{-l}t) = 1$ for $0<t\leq 1$.  Put $A_{1}(\x_{2}, \ldots, \x_{n}) = 0$, and for $j\leq 0$, put
\begin{align*}
	\chi_j(\x_1) &=\eta\left(2^{-j} |\x_1| \right)\\
	\widetilde{\chi}_j(\x_1) &= \chi_j(\x_1)\Big[ \int_{\R^{a_{1}}} \chi_j (\x_{1})\,d\x_{1}\Big]^{-1}\\
	a_{j}(\x_2,  \dots, \x_n ) &= \int_{\R^{a_{1}}} \varphi(\x_{1},\x_{2},\ldots, \x_{n}) \chi_j(\x_1)\ d\x_1,\\
	A_{j}(\x_{2}, \ldots, \x_{n}) &= \sum_{s=j}^{0}a_{s}(\x_2,  \dots, \x_n ).
\end{align*}
We have $\sum_{j\leq 0}\chi_{j}(\x_{1}) \equiv 1$ for $0<|\x_{1}|\leq 1$, and  since $\varphi$ is supported in the unit ball, we can sum by parts  for $\x_1 \neq0$ to get
\begin{align*}
\varphi(\x) 
&=
\sum_{j\leq 0}\varphi(\x_{1}, \ldots,\x_{n})\chi_{j}(\x_{1})\\
&=
\sum_{j\leq 0}\big[\varphi(\x_{1}, \ldots,\x_{n})\chi_{j}(\x_{1})-\widetilde\chi_{j}(\x_{1})a_{j}( \x_2, \dots, \x_n )\big]\\
&\qquad \qquad \qquad 
+\sum_{j\leq 0}\big(\widetilde \chi_{j}(\x_{1})-\widetilde\chi_{j-1}(\x_{1})\big)A_{j}(\x_{2}, \ldots, \x_{n})
=
\sum_{j\leq 0} \widetilde \varphi_{j}(\x).
\end{align*}
Let 
\begin{equation*}
\varphi^{j}(\x_{1},\x_{2},  \ldots, \x_{n}) = 2^{jQ_{1}}\widetilde \varphi_{j}(2^{j}\cdot \x_{1}, \x_{2}, \ldots, \x_{n}).
\end{equation*}
Then
\begin{equation}\label{E2.10a}
\begin{aligned}
\varphi^{j}(\x) &= 
2^{jQ_{1}}\Big[\varphi(2^{j}\cdot\x_{1}, \ldots,\x_{n})\chi_{0}(\x_{1})-2^{-jQ_{1}}\widetilde\chi_{0}(\x_{1})a_{j}( \x_2, \ \dots, \x_n )
\\
&\qquad\qquad \qquad
+2^{jQ_{1}}\big(\widetilde \chi_{0}(\x_{1})-2^{-Q_{1}}\widetilde\chi_{0}(2\cdot\x_{1})\big)A_{j}(\x_{2}, \ldots, \x_{n})\Big],
\end{aligned}
\end{equation}
and
\begin{equation}\label{E2.11a}
\varphi(\x_{1}, \ldots, \x_{n}) = \sum_{j=-\infty}^{0}2^{-jQ_{1}}\,\varphi^{j}(2^{-j}\cdot \x_{1}, \x_{2}, \ldots, \x_{n}).
\end{equation}

The functions $\{a_{j}\}$ and hence also the functions $\{A_{j}\}$ are infinitely differentiable functions supported in the unit ball of $\R^{a_{2}}\oplus \cdots \oplus \R^{a_{n}}$. Moreover, it follows from the fact that $\chi_{j}(\x_{1})$ is supported on the set $|\x_{1}|\leq 2^{j+1}$ that there is a constant $C$ so that for each integer $m$, we have 
\begin{equation}\label{E2.12a}
||a_{j}||_{(m)}\leq C\,2^{jQ_{1}}||\varphi||_{(m)}.
\end{equation}
The function $\sum_{s=-\infty}^{j}\chi_{s}(\x_{1})$ is also supported on the set $|\x_{1}|\leq 2^{j+1}$, and is bounded independently of $j$. We have
\begin{equation*}
A_{j}(\x_{2}, \ldots, \x_{n}) + \int_{\R^{a_{1}}}\varphi(\x_{1}) \sum_{s=-\infty}^{j+1}\chi_{s}(\x_{1})\,d\x_{1}= \int_{\R^{a_{1}}}\varphi(\x_{1}, \x_{2}, \ldots, \x_{n})\,d\x_{1} = 0,
\end{equation*}
and it thus follows that we also have
\begin{equation}\label{E2.13a}
||A_{j}||_{(m)}\leq C\,2^{jQ_{1}}||\varphi||_{(m)}.
\end{equation}
It is clear from our construction that  each function $\varphi^{j}$ has compact support in the unit ball. It follows from equation (\ref{E2.10a}) that each $\varphi^{j}$ vanishes when $|\x_{1}|\leq \frac{1}{8}$, and also that if $\varphi$ has strong cancellation, then $ \int_{\R^{a_{k}}}\varphi^{j}(\x_{1},\x_{2}, \ldots, \x_{n})\,d\x_{k}=0$ for $1 \leq k \leq n$.  Finally, equations (\ref{E2.10a}), (\ref{E2.12a}), and (\ref{E2.13a}) show that $||\varphi^{j}||_{(m)}\leq C\,2^{-jQ_{1}}||\varphi||_{(m)}$. This completes the proof.
\end{proof}

\medskip

Let $\K$ be a flag kernel adapted to the flag $\mathcal F$, and let $m= \widehat{\mathcal K}$ be the flag multiplier on the flag $\mathcal F^{*}$ which is the Fourier transform of $\K$. 
Choose a function $\eta\in \mathcal C^{\infty}_{0}(\R)$ supported in $[\frac{1}{2},4]$ such that $\sum_{j\in \mathbb Z}\eta(2^{j}t) \equiv 1$ for all $t>0$. For each $I=(i_{1}, \ldots, i_{n})\in \mathcal E_{n}$, set 
\begin{equation}\label{E2.6}
\eta_{I}(\xi) = \eta(2^{i_{1}}|\xi_{1}|)\cdots \eta(2^{i_{n}}|\xi_{n}|).
\end{equation}
Note that $\eta_{I}$ is supported where $|\xi_{j}|\approx 2^{-i_{j}}$ for $1 \leq j \leq n$. We establish

\begin{lemma}\label{Lemma2.6}
Let $m$ be a flag multiplier relative to the flag $\mathcal F^{*}$ given in (\ref{E2.5}). Then
\begin{equation}\label{E2.7}
m(\xi) = \sum_{I\in \mathcal E_{n}} m_{0}(\xi)\eta_{I}(\xi) + \sum_{k=1}^{n}m_{k}(\xi)
\end{equation}
where $m_{0}$ is the Fourier transform of a flag kernel relative to the flag $\mathcal F$, and for $1 \leq k \leq n$, the function $m_{k}$ is the Fourier transform of a flag kernel adapted to a flag \emph{strictly coarser} than $\mathcal F$.
\end{lemma}

\begin{proof}
Let $\theta$ be a smooth function on $\R$ supported where $t \geq 10$ such that $\theta(t) = 1$ for $t\geq 20$.  Write
\begin{equation*}
	m(\xi) = m(\xi) \theta\left( | \xi_{n-1} |\, | \xi_n|^{-1} \right) 
		+ m(\xi)\big[1 - \theta\left( | \xi_{n-1} |\, | \xi_n|^{-1} \right) \big] = n_1(\xi) + m_1(\xi).
\end{equation*}
On the support of $\theta'\left( | \xi_{n-1} |\, | \xi_n|^{-1} \right)$ we have $|\xi_{n-1}| \sim |\xi_n|$. Also, by  homogeneity we have
\begin{equation*}
	\big\vert\partial^{\alpha_{n-1}}_{\xi_{n-1}}(|\xi_{n-1}|)\big\vert
	\leq C |\xi_{n-1}|^{1 - |\alpha_{n-1}| } 
	\qquad\text{and} \qquad
	\big\vert\partial^{\alpha_{n}}_{\xi_{n}}(|\xi_{n}|)\big\vert
	\leq C |\xi_n|^{1-|\alpha_n|}.
\end{equation*}
Thus Lemma \ref{Lemma2.5} implies that $n_1(\xi)$ and $m_1(\xi)$ are flag multipliers relative to the flag (\ref{E2.5}). On the support of $m_{1}$ we have $|\xi_{n-1}|\, |\xi_n|^{-1} \leq 20$.  Thus we can group together the variables $\xi_n$ and $\xi_{n-1}$, and it follows that $m_{1}(\xi)$ is a flag kernel relative to a flag \emph{coarser} than $\mathcal F^{*}$.  Also $|\xi_{n-1}| \geq 10 |\xi_n|$ on the support of $n_{1}$.

Next write
\begin{equation*}
	n_{1}(\xi) = n_{1}(\xi) \theta\left( | \xi_{n-2} |\, | \xi_{n-1}|^{-1} \right) 
		+n_{1}(\xi)\big[1 - \theta\left( | \xi_{n-2} |\, | \xi_{n-1}|^{-1} \right) \big] = n_2(\xi) + m_2(\xi).
\end{equation*}
Since $\theta'\left( | \xi_{n-2} |\, | \xi_{n-1}|^{-1} \right)$ is supported where $|\xi_{n-2}| \sim |\xi_{n-1}|$ and $|\xi_{n-1}| \geq 10 |\xi_n|$, it again follows that $n_2$ and $m_2$ are Fourier transforms of flag kernels relative to the flag $\mathcal F^{*}$. On the support of $m_2$ we have $|\xi_{n-2}|\, |\xi_{n-1}|^{-1} \leq 20$ and $|\xi_{n-1}|\, |\xi_n|^{-1} \leq 20$. Thus we can group together the variables $\xi_n, \xi_{n-1}$ and $\xi_{n-2}$, and it follows that $m_{2}(\xi)$ is a flag kernel relative to a flag \emph{coarser} than $\mathcal F^{*}$. Also $|\xi_{n-1}|\geq 10|\xi_{n-1}|$, and $|\xi_{n-1}|\geq 10|\xi_{n}|$ on the support of $n_{2}$.

We proceed inductively to see that
$$
	m(\xi) = m_{0}(\xi) + \sum_{k=1}^{n} m_k(\xi)
$$
where each $m_s(\xi)$ is the Fourier transform of a coarser flag kernel and ${m}_{0}(\xi)=n_{n}(\xi)$ is supported where $|\xi_j| \geq 10 |\xi_{j+1}|$ for $1 \leq j \leq n-1$.

From our choice of $\eta$, it follows that 
\begin{align*}
1 &= \sum_{J=(j_{1}, \ldots, j_{n})\in \Z^{n}}\eta\left(2^{j_1} | \xi_1 | \right)\cdots \eta\left(2^{j_n}|\xi_n| \right)\\
&=
\sum_{J\in \mathcal E_{n}}\eta\left(2^{j_1} | \xi_1 | \right)\cdots \eta\left(2^{j_n}|\xi_n| \right) +\sum_{J\in \mathbb Z^{n}-E_{n}}\eta\left(2^{j_1} | \xi_1 | \right)\cdots \eta\left(2^{j_n}|\xi_n| \right)
\end{align*}
Thus if $\sum_{J\in \mathcal E_{n}}\eta\left(2^{j_1} | \xi_1 | \right)\cdots \eta\left(2^{j_n}|\xi_n| \right)\neq 1$, there is an $n$-tuple $J=(j_{1}, \ldots, j_{n})$ such and integers $1 \leq r<s\leq n$ such that $j_{r}>j_{s}$ and $\eta\left(2^{j_r} | \xi_r | \right)\, \eta\left(2^{j_s}|\xi_s| \right)\neq 0$. Since $\eta$ is supported on $[\frac{1}{2},4]$, it follows that $|\xi_{r}|\leq 4\,2^{-j_{r}}\leq 4\,2^{-j_{s}}\leq 8|\xi_{s}|$. However, on the support of $m_{0}$ we have $|\xi_{r}|>10|\xi_{s}|$. Thus on the support of $m_{0}$ we have $1 = \sum_{J\in \mathcal E_{n}}\eta\left(2^{j_1} | \xi_1 | \right)\cdots \eta\left(2^{j_n}|\xi_n| \right)$, and so
\begin{equation*}
m(\xi) = \sum_{J\in \mathcal E_{n}}m_{0}(\xi)\eta\left(2^{j_1} | \xi_1 | \right)\cdots \eta\left(2^{j_n}|\xi_n| \right) + \sum_{k=1}^{n} m_k(\xi),
\end{equation*}
which completes the proof.
\end{proof}

\smallskip

We now turn to the proof of Theorem \ref{Lemma2.3}.
If we write $\K_{j}$ as the inverse Fourier transform of the flag multiplier $m_{j}$ of Lemma \ref{Lemma2.6}, we have shown that $\K = \K_{0}+ \sum_{j=1}^{n}\K_{j}$, and for $1 \leq j \leq n$, $\K_{j}$ is a flag kernel adapted to a flag which is coarser than $\mathcal F$. Also since $m_{0}(\xi) = \sum_{J\in \mathcal E_{n}}m_{0}(\xi)\eta_{I}(\xi)$, we can write 
\begin{equation}\label{sdf}
\K_{0}= \sum_{J\in \mathcal E_{n}} \big[\Psi^{I}\big]_{I}
\end{equation}
where
\begin{equation*}
\widehat{\Psi^{I}}(\xi) = m_{0}(2^{-i_{1}}\cdot \xi_{1}, \ldots, 2^{-i_{n}}\cdot\xi_{n}) \eta(|\xi_{1}|)\cdots \eta(|\xi_{n}|),
\end{equation*}
and the sum converges in the sense of distributions. The differential inequalities for $m_{0}$ imply that each function $\Psi^{I}\in \mathcal S(\R^{N})$,  with Schwartz norms uniformly bounded in $I$.  Also since $\Psi^{I}$ vanishes on the coordinate axes, for $1 \leq k \leq n$ we have
\begin{equation}\label{eq:eq10}
	\int \Psi^{I}(\x_1,\ldots,\x_n)\ d\x_k = 0.
\end{equation}



\smallskip

In order to complete the proof of Theorem \ref{Lemma2.3}, it remains to show that we can replace the Schwartz functions $\Psi^{I}$ in equation (\ref{sdf}) by functions in $\mathcal C^{\infty}_{0}(\R^{N})$, all supported in the unit ball with strong cancellation, and which vanish when $|\x_{1}|<\epsilon$. This is done in two steps, using Lemmas \ref{Lemma2.8} and \ref{Lemma2.7w}.

\smallskip

First, according to Lemma \ref{Lemma2.8}, for each $I\in \mathcal E_{n}$ there are functions $\psi^{k,I}\in \mathcal C^{\infty}_{0}(\R^{N})$ each supported in the unit ball and having strong cancellation, so that $\psi^{I}(\x) = \sum_{k=0}^{\infty}2^{-kQ}\,\psi^{k,I}(2^{-k}\cdot \x)$, and for every $\Delta >0$ and every positive integer $m$ there is an integer $p_{m}$ so that 
\begin{equation*}
||\psi^{k,I}||_{(m)}\leq 2^{-k\Delta}\,||\psi^{I}||_{[m+p_{m}]}.
\end{equation*}
Thus $[\psi^{I}]_{I}(\x) = \sum_{k=0}^{\infty}2^{-Q_{1}(k+i_{1}) - \cdots - Q_{n}(k+i_{n})}\psi^{k,I}(2^{-(i_{1}+k)}\cdot\x_{1}, \ldots, 2^{-(i_{n}+k)}\cdot\x_{n})$, and so formally
\begin{align*}
\sum_{I\in \mathcal E_{n}}[\psi^{I}]_{I}(\x) 
&= 
\sum_{k=0}^{\infty}\sum_{I\in \mathcal E_{n}} 2^{-Q_{1}(k+i_{1}) - \cdots - Q_{n}(k+i_{n})}\psi^{k,I}(2^{-(i_{1}+k)}\cdot\x_{1}, \ldots, 2^{-(i_{n}+k)}\cdot\x_{n})\\
&=
\sum_{k=0}^{\infty}\sum_{J\in \mathcal E_{n}} 2^{-Q_{1}(j_{1}) - \cdots - Q_{n}(j_{n})}\psi^{k,I}(2^{-j_{1}}\cdot\x_{1}, \ldots, 2^{-j_{n}}\cdot\x_{n})\\
&=
\sum_{J\in \mathcal E_{n}}2^{-Q_{1}(j_{1}) - \cdots - Q_{n}(j_{n})}\Big[\sum_{k=0}^{\infty}\psi^{k,I}\Big](2^{-j_{1}}\cdot\x_{1}, \ldots, 2^{-j_{n}}\cdot\x_{n})\\
&=
\sum_{J\in \mathcal E_{n}}\Big[\sum_{k=0}^{\infty}\psi^{k,I}\Big]_{J}(\x),
\end{align*}
The estimates we have on the functions $\{\psi^{k,I}\}$ show that the series $\sum_{k=0}^{\infty}\psi^{k,I}= \widetilde\varphi^{I}$ converges in $\mathcal C^{\infty}_{0}(\R^{N})$ to a function in  supported in the unit ball which has strong cancellation. This formal calculation is easily justified by applying it to finite sums of dilates. Thus we have shown $\K_{0}= \sum_{I\in \mathcal E_{n}}[\widetilde\varphi^{I}]_{I}$ with convergence in the sense of distributions, where the functions $\widetilde \varphi^{I}\in \mathcal C^{\infty}_{0}(\R^{N})$ all have support in the unit ball, have strong cancellation, and are uniformly bounded in each semi-norm $||\,\cdot\,||_{(m)}$.

\smallskip


Finally, Lemma \ref{Lemma2.7w} shows that for each $I\in \mathcal E_{n}$, there exist functions $\varphi^{j,I}\in \mathcal C^{\infty}_{0}(\R^{N})$ with strong cancellation, each supported in the unit ball, normalized relative to $\widetilde \varphi^{I}$, and vanishing when $|\x_{1}|\leq \frac{1}{8}$ so that $\widetilde \varphi^{I}(\x_{1}, \x_{2}, \ldots, \x_{n})= \sum_{j=-\infty}^{0}2^{-jQ_{1}}\varphi^{j,I}(2^{-j}\cdot\x_{1}, \x_{2}, \ldots, \x_{n})$. We then have
\begin{equation*}
\begin{aligned}
\big[\widetilde{\varphi}^{I}\big]_{I}&(\x_{1}, \x_{2}, \ldots, \x_{n})\\
&= \sum_{j=-\infty}^{0}2^{-(j+i_{1})Q_{1}-i_{2}Q_{2}-\cdots -i_{n}Q_{n}}\varphi^{j,I}(2^{-(j+i_{1})}\cdot\x_{1}, 2^{-i_{2}}\cdot\x_{2}, \ldots, 2^{-i_{n}}\cdot\x_{n}),
\end{aligned}
\end{equation*}
and so since we are summing only over non-positive indices $j$, we have
\begin{align*}
\sum_{I\in \mathcal E_{n}}[\widetilde \varphi^{I}]_{I}(\x)
&=
\sum_{j=-\infty}^{0}\sum_{I\in \mathcal E_{n}}
2^{-(j+i_{1})Q_{1}-i_{2}Q_{2}-\cdots -i_{n}Q_{n}}\varphi^{j,I}(2^{-(j+i_{1})}\cdot\x_{1}, \ldots, 2^{-i_{n}}\cdot\x_{n})\\
&=
\sum_{j=-\infty}^{0}\sum_{I\in \mathcal E_{n}}
2^{-i_{1}Q_{1}-i_{2}Q_{2}-\cdots -i_{n}Q_{n}}\varphi^{j,I}(2^{-i_{1}}\cdot\x_{1},  \ldots, 2^{-i_{n}}\cdot\x_{n})\\
&=
\sum_{I\in \mathcal E_{n}}
2^{-i_{1}Q_{1}-i_{2}Q_{2}-\cdots -i_{n}Q_{n}}\Big[\sum_{j=-\infty}^{0}\varphi^{j,I}\Big](2^{-i_{1}}\cdot\x_{1},  \ldots, 2^{-i_{n}}\cdot\x_{n})\\
&=
\sum_{I\in \mathcal E_{n}}\Big[\sum_{j=-\infty}^{0}\varphi^{j,I}\Big]_{I}(\x).
\end{align*}
The estimates we have on the functions $\{\varphi^{j,I}\}$ show that the series $\sum_{j=-\infty}^{0}\varphi^{j,I}= \varphi^{I}$ converges in $\mathcal C^{\infty}_{0}(\R^{N})$ to a function in  supported in the unit ball which has strong cancellation, and which vanishes when $|\x_{1}|\leq \frac{1}{8}$. This at last completes the proof of Theorem \ref{Lemma2.3}.

\subsection{Dyadic sums with weak cancellation}\quad

\smallskip

The main result of this section, Theorem \ref{Lemma5.4zw}, is a strengthening of the converse to Theorem \ref{Lemma2.3} which would assert that sums of dilates of appropriate compactly supported bump functions with strong cancellation are flag kernels; we consider sums of Schwartz functions instead of compactly supported functions, and more critically, we assume only weak cancellation instead of strong cancellation relative to the decomposition $\R^{N}= \R^{a_{1}}\oplus \cdots \oplus \R^{a_{n}}$. 


\smallskip

\begin{theorem}\label{Lemma5.4zw}
For each $I\in \mathcal E_{n}$, let $\varphi^{I}\in \mathcal S(\R^{N})$, and suppose
\begin{enumerate}[{\rm (a)}]

\smallskip

\item there are constants $C_{N}>0$ so that $||\varphi^{I}||_{[N]}\leq C_{N}$ for each $I\in \mathcal E_{n}$ and each $N\geq 0$;

\smallskip

\item there is a constant $\epsilon >0$ so that each $\varphi^{I}$ has weak cancellation with respect to $I$ with parameter $\epsilon$ relative to the decomposition $\R^{N}= \R^{a_{1}}\oplus\cdots\oplus\R^{a_{n}}$.
\end{enumerate}
Then we have the following conclusions.
\begin{enumerate}[{\rm(1)}]
\smallskip

\item \label{Thm3.7(1)}For any finite set $F\subset E_{n}$, the function $K_{F} = \sum_{I\in F}[\varphi^{I}]_{I}\in \mathcal S(\R^{N})$ defines a flag kernel $\K_{F}$ for the  flag  $\mathcal F$ with bounds which are independent of the set $F$.

\smallskip

\item \label{Thm3.7(2)}Let $F_{1}\subset F_{2}\subset \cdots \subset F_{m}\subset \cdots $ be any increasing sequence of finite subsets of $E_{N}$ with $E_{N}= \bigcup_{m=1}^{\infty}F_{m}$. Then for any test function $\psi\in \mathcal S(\R^{N})$, 
\begin{equation*}
\lim_{m\to\infty}\langle K_{F_{m}},\psi\rangle = \lim_{m\to\infty}\sum_{I\in F_{m}}\int_{\R^{N}} [\varphi^{I}]_{I}(\x)\psi(\x)\,dx
\end{equation*}
exists and defines a flag kernel $\K\in \mathcal S'(\R^{N})$ which is independent of the choice of the finite subsets. We write this limit as $\displaystyle \mathcal K=\lim_{F\nearrow E_{N}}\sum_{I\in F}[\varphi^{I}]_{I}$.

\end{enumerate}

\end{theorem}

\smallskip

Since the proof of this result is somewhat involved, \label{remarks} let us indicate the main steps. If $F\subset E_{n}$ is any finite set, let $\K_{F}= \sum_{I\in F}[\varphi^{I}]_{I}$. We show (in Proposition \ref{Proposition5.3z}) that $K_{F} = \sum_{I\in \mathcal E_{n}}\big[\varphi^{I}\big]_{I}$ satisfies the correct size estimates,  and also (in Proposition \ref{Lemma5.4y}) that $K_{F}$ satisfies the cancellation conditions  with constants independent of $F$ of Definition \ref{Def2.1}. This will establish part (\ref{Thm3.7(1)}) of Theorem \ref{Lemma5.4zw}. 
To establish the existence of the limit in part (\ref{Thm3.7(2)}), we use the weak cancellation of the functions $\{\varphi^{I}\}$ to show that for any test function $\psi$, the bracket $\big\langle\K_{F},\psi\big\rangle$ can be rewritten as the integral of $\psi$ and its derivatives against locally integrable functions. The existence of the limit then follows from the Lebesgue dominated convergence theorem. 

\smallskip

We now turn to the proofs. Note that in the next Proposition  we impose no cancellation conditions on the functions $\{\varphi^{I}\}$.

\smallskip

\begin{proposition} \label{Proposition5.3z}
For each $I\in \mathcal E_{n}$ let $\varphi^{I}\in \mathcal S(\R^{N})$, and suppose there are constants $C_{M}$ so that for all $I\in \mathcal E_{n}$ and all $M$, $||\varphi^{I}||_{[M]}\leq C_{M}$. Let $F\subset E_{n}$ be a finite subset, and let $K_{F}(\x) = \sum_{I\in F}[\varphi^{I}]_{I}(\x)$. For any $\alpha =(\balpha_{1}; \ldots ; \balpha_{n}) \in \mathbb N^{a_{1}}\times \cdots \times \mathbb N^{a_{n}}$ with $|\alpha|\leq M$, there is a constant $A_{M}$ independent of the finite set $F$ so that
\begin{equation*}
\big\vert\partial^{\alpha}_{\x}K_{F}(\x)\big\vert \leq A_{M}\prod_{j=1}^{n}\left[N_{1}(\x_{1}) + N_{2}(\x_{2}) + \cdots +N_{j}(\x_{j})\right]^{-(Q_{j}+[\![\balpha_{j}]\!])}.
\end{equation*}
\end{proposition}

\begin{proof}

It follows from Proposition \ref{Prop1.6} that
\begin{align*} 
\left\vert\partial^{\alpha}_{\x}K_{F}(\x)\right\vert 
&= 
\Big\vert\sum_{I\in F}2^{-\sum_{j=1}^{n}i_{j}(Q_{j}+[\![\balpha_{j}]\!])}\partial^{\balpha}_{\x}[\varphi^{I}](2^{-i_{1}} \cdot\x_{1}, \ldots, 2^{-i_{n}}\cdot\x_{n})\Big\vert.
\end{align*}
Since $\varphi^{I}\in \mathcal S(\R^{n})$, for any $M$ we have the estimate
\begin{equation*}
\big|\partial^{\balpha}_{\x}[\varphi^{I}](2^{-i_{1}} \x_{1}, \ldots, 2^{-i_{n}}\x_{n})\big|
\leq C_{M}\, \big(1+\sum_{k=1}^{N}2^{-i_{k}}N_{k}(\x_{k})\big)^{-M}
\end{equation*}
Thus 
\begin{equation*}
\left\vert\partial^{\balpha}_{\x}K_{F}(\x)\right\vert
\leq C_{M}\,
\sum_{I\in F}\,2^{-\sum_{j=1}^{n}i_{k}(Q_{k}+[\![\balpha_{k}]\!])}\big(1+\sum_{k=1}^{N}2^{-i_{k}}N_{k}(\x_{k})\big)^{-M}.
\end{equation*}
Proposition \ref{Proposition5.3z} thus follows from estimate (\ref{E2.10}) in Proposition \ref{Prop2.11}.
\end{proof}

\smallskip

The next result provides estimates that will be used in establishing the cancellation conditions (part (\ref{Def2.1b}) of Definition \ref{Def2.1}) for finite sums $\K_{F}= \sum_{I\in F}[\varphi^{I}]_{I}$.  Recall that these cancellation conditions involve integrals of the form
\begin{equation*}
\int K(\x_{1}, \ldots,\x_{n}) \psi(R_{1}\cdot \x_{m_{1}}, \ldots, R_{\beta}\cdot \x_{m_{\beta}})\,d\x_{m_{1}}\cdots d\x_{m_{\beta}}
\end{equation*}
where $\{m_{1}, \ldots, m_{\beta}\}$ is a non-empty subset of $\{1, \ldots, n\}$.  Let  $M= \{m_{1}, \ldots, m_{\beta}\}$, and let $L=\{l_{1}, \ldots, l_{\alpha}\} = \{1,\ldots, n\}\setminus M$. Let $N_{\alpha}= a_{l_{1}}+ \cdots + a_{l_{\alpha}}$ and $N_{\beta}= a_{m_{1}}+ \cdots + a_{m_{\beta}}$ so that $N_{\alpha}+N_{\beta}=N$.
For $\x=(\x_{1}, \ldots, \x_{n})\in \R^{N}$, write 
\begin{align*}
\x=(\x', \x'')
\qquad
\text{with}
\qquad
\begin{cases}\x' =(\x_{l_{1}}, \ldots, \x_{l_{\alpha}})\in \R^{N_{\alpha}}\\
\x''= (\x_{m_{1}}, \ldots, \x_{m_{\beta}})\in \R^{N_{\beta}}
\end{cases}.
\end{align*} 
For $I=(i_{1}, \ldots, i_{n}) \in \mathcal E_{n}$, write 
\begin{align*}
I=(I', I'')\qquad\text{where}\qquad 
\begin{cases}
I'=(i_{l_{1}}, \ldots, i_{l_{\alpha}}) \in E_{\alpha}\\
I''= (i_{m_{1}}, \ldots, i_{m_{\beta}})\in E_{\beta}
\end{cases}.
\end{align*} 
If $R=(R_{1}, \ldots, R_{\beta})$ is a $\beta$-tuple of positive real numbers, write 
\begin{align*}
R\cdot \x''&= (R_{1}\cdot\x_{m_{1}}, \ldots, R_{\beta}\cdot\x_{m_{\beta}}),\\
2^{I''}R\cdot \x''&=(2^{i_{m_{1}}}R_{1}\cdot\x_{m_{1}}, \ldots, 2^{i_{m_{\beta}}}R_{\beta}\cdot\x_{m_{\beta}}).
\end{align*}
Finally, let $P(M)$ denote the set of all partitions of the set $M=\{m_{1}, \ldots, m_{\beta}\}$ into two disjoint (possibly empty) subsets $A$ and $B$. 

\begin{proposition}\label{Prop2.15}
With the above notation, let $I=(i_{1}, \ldots, i_{n})\in \mathcal E_{n}$, and let ${\varphi}\in \mathcal S(\R^{N})$. Let $\psi \in \mathcal C^{\infty}_{0}\left(\R^{N_{\beta}}\right)$  so  that $\psi$ can be regarded as a function of the variables $\x''=(\x_{m_{1}}, \ldots, \x_{m_{\beta}})$.  Let $R= (R_{1}, \ldots , R_{\beta})$ be a $\beta$-tuple of positive real numbers. Define $\Phi$ on $\R^{N_{\alpha}}$ by setting
\begin{align*}
\Phi(\x')
&= 
\int_{\R^{N_{\beta}}}
[\varphi]_{I}(\x',\x'') \,\psi(R_{1}\cdot\x_{m_{1}}, \ldots, R_{\beta}\cdot\x_{m_{\beta}})\,\,d\x_{m_{1}}\cdots d\x_{m_{\beta}}.
\end{align*}
\begin{enumerate}[{\rm(1)}]

\item\label{Prop2.15a} 
There exists $\Theta \in \mathcal S(\R^{N_{\alpha}})$, normalized relative to $\varphi$ and $\psi$ (with constants independent of $(R_{1}, \ldots, R_{\beta})$) such that $\Phi(\x') = [\Theta]_{I'}(\x')$. 

\smallskip

\item\label{Prop2.15b} If  ${\varphi}$ has weak cancellation with respect to $I$, there are constants $C$ and $\epsilon$ independent of $R=(R_{1}, \ldots , R_{\beta})$ so that  $\Theta = \sum_{(A,B)\in P(M)}\Theta_{A,B}$ where for each partition $M=A\cup B$,
\begin{align*}
|\Theta_{A,B}(\x')| \leq C\,
\prod_{m_{r}\in A}2^{-\epsilon(i_{m_{r}+1}-i_{m_{r}})}\,\prod_{m_{s}\in B}\min\big\{(R_{s}2^{i_{m_{s}}})^{+\epsilon},(R_{s}2^{i_{m_{s}}})^{-\epsilon}\big\}.
\end{align*}
Thus for each partition $M=A\cup B$ of the set of variables $\{\x_{m_{1}}, \ldots, \x_{m_{\beta}}\}$ , $|\Theta_{A,B}(\x')|$ is small due to two kinds of gains:  there is a gain $2^{-\epsilon(i_{m_{r}+1}-i_{m_{r}})}$ for every index $m_{r}\in A$, and there is a gain $\min\big\{(R_{r}2^{i_{m_{r}}})^{+\epsilon},(R_{r}2^{i_{m_{r}}})^{-\epsilon}\big\}$ for every index $m_{r}\in B$. 
\end{enumerate}
\end{proposition}

\begin{proof}
Make the change of variables $$\x''=(\x_{m_{1}}, \ldots, \x_{m_{\beta}}) \mapsto (2^{i_{m_{1}}}\cdot\x_{m_{1}}, \ldots, 2^{i_{m_{\beta}}}\cdot\x_{m_{\beta}}).$$ Then
 \begin{align*}
&\Phi(\x')\\
 &= 2^{-\sum_{k=1}^{n}i_{k}Q_{k}}
\int_{\R^{N_{\beta}}}\!\!
\varphi(2^{-i_{1}}\x_{1}, \ldots, 2^{-i_{n}}\cdot \x_{n})
\psi(R_{1}\cdot\x_{m_{1}}, \ldots, R_{\beta}\cdot\x_{m_{\beta}})
\,d\x_{m_{1}}\cdots d\x_{m_{\beta}}\\
&=
\int_{\R^{N_{\beta}}}
\big[\varphi\big]_{I'}(\x',\x'')
\psi(R_{1}2^{i_{m_{1}}}\cdot \x_{m_{1}}, \ldots, R_{\beta}2^{i_{m_{\beta}}}\cdot \x_{m_{\beta}})
\,d\x_{m_{1}}\cdots d\x_{m_{\beta}}\\
&=
\big[\Theta\big]_{I'}(\x')
\end{align*} 
where 
\begin{equation} \label{Eqn3.13po}
\Theta(\x')
= \int_{\R^{N_{\beta}}}
\varphi(\x',\x'')
\psi(R_{1}2^{i_{m_{1}}}\cdot \x_{m_{1}}, \ldots, R_{\beta}2^{i_{m_{\beta}}}\cdot \x_{m_{\beta}})
\,d\x_{m_{1}}\cdots d\x_{m_{\beta}}.
\end{equation}
Now
\begin{align*}
(1+|\x'|)^{M}\partial^{\balpha}_{\x'}\Theta(\x') 
=
\int_{\R^{N_{\beta}}} (1+|\x'|)^{M}\partial^{\balpha}_{\x'}\varphi(\x',\x'')\,\psi(2^{I''}R\cdot \x'')\,d\x'',
\end{align*}
and we can estimate the $\mathcal S(\R^{N_{\alpha}})$-seminorms of $\Theta$ in terms of those of $\varphi$ and the supremum of $|\psi|$, independent of the choice of $R$. This establishes assertion (\ref{Prop2.15a}).

\smallskip

To prove the  decomposition and additional size estimates for $\Theta$ asserted in part (\ref{Prop2.15b}), assume that ${\varphi}$ has weak cancellation relative to $I$.  Since the definiton (\ref{Eqn3.13po}) of $\Theta$ involves integration with respect to the variables $\{\x_{m_{1}}, \ldots, \x_{m_{\beta}}\}$, we will only use the weak cancellation in these variables. We can use  the first of the Remarks \ref{Remark2.20}  to write $\varphi$ as a sum of terms of the form 
\begin{equation*}
\big(\prod_{s\in A}2^{-\epsilon (i_{s+1}-i_{s})}\big)\,\big(\prod_{t\in B}\partial_{\sigma(t)}\big)[\widetilde \varphi _{A,B,\sigma}]
\end{equation*}
where $\{m_{1}, \ldots, m_{\beta}\} = A\cup B$ with $A\cap B= \emptyset$ and $n \in B$ if $n\in \{m_{1}, \ldots, m_{\beta}\}$, $\sigma:B\to \{1, \ldots, N\}$ so that $\sigma(m_{\ell}) \in J_{m_{\ell}}$, and each $\widetilde \varphi _{A,B,\sigma}$ normalized relative to $\varphi$.


If we write $B=\{m_{\ell_{1}}, \ldots, m_{\ell_{s}}\}$, it follows that $\Theta(\x')$ is a finite sum of terms of the form
\begin{equation*}\label{Equation3.13yu}
\begin{aligned}
I(A;B; \sigma) &=
\prod_{s\in A}2^{-\epsilon (i_{s+1}-i_{s})}
\int\limits_{\R^{N_{\beta}}}
\partial_{\sigma(m_{\ell_{1}})}\cdots \partial_{\sigma(m_{\ell_{s}})}[\widetilde \varphi_{A;B;\sigma}]\\
&\qquad\qquad\qquad\qquad
\psi(R_{1}2^{i_{m_{1}}}\cdot \x_{m_{1}}, \ldots, R_{\beta}2^{i_{m_{\beta}}}\cdot \x_{m_{\beta}})\,d\x_{m_{1}}\cdots d\x_{m_{\beta}}.
\end{aligned}
\end{equation*}
We can use integration by parts to move the differentiations $\partial_{\sigma(m_{\ell_{1}})}\cdots \partial_{\sigma(m_{\ell_{s}})}$ from $\widetilde \varphi_{A;B;\sigma}$ to $\psi$. Differentiating the function $\psi(R_{1}2^{i_{m_{1}}}\cdots \x_{m_{1}}, \ldots, R_{\beta}2^{i_{m_{\beta}}}\cdot \x_{m_{\beta}})$ with respect to the variable $x_{\sigma(m_{\ell_{k}})}$  with $\sigma(m_{\ell_{k}}) \in J_{m_{\ell_{k}}}$ brings out a factor $(R_{m_{\ell_{k}}}2^{i_{m_{\ell_{k}}}})^{d_{\sigma(m_{\ell_{k}})}}$, and so we have the estimate
\begin{align*}
\big\vert I(j_{1}, \ldots, j_{\beta}; A; \sigma) \big\vert
 \lesssim
\prod_{s\in A}2^{-\epsilon (i_{s+1}-i_{s})}\,
\prod_{r\in B}(R_{r}2^{i_{r}})^{d_{\sigma(r)}}.
\end{align*}
On the other hand, without integrating by parts, since $\psi$ has compact support, the integral in the variables $\{\x_{m_{r}}\,\big\vert\, r\in B\}$ is taken over the set where $|\x_{m_{r}}|\lesssim (R_{r}2^{i_{m_{r}}})^{Q_{m_{r}}}$ for $r\in B$, and this set has volume bounded by a constant times $\prod_{r\in B}(R_{r}2^{i_{m_{r}}})^{-Q_{m_{r}}}$. It follows that there exists $\epsilon > 0$ so that 
\begin{align}
\big\vert I(j_{1}, \ldots, j_{\beta}; A; \sigma)\big\vert 
\lesssim 
\prod_{s\in A}2^{-\epsilon (i_{s+1}-i_{s})}\,
\prod_{r\in B}\min\big\{(R_{r}2^{i_{m_{r}}})^{+\epsilon},(R_{r}2^{i_{m_{r}}})^{-\epsilon}\big\}.
\end{align}
This completes the proof.
\end{proof}

We now show that if the functions $\{\varphi^{I}\}$ have weak cancellation, then the sum $K_{F}(\x)$ satisfies the cancellation condition (condition (\ref{Def2.1b})) of Definition \ref{Def2.1}. We use the same notation as in Proposition \ref{Prop2.15}. Thus  $L=\{l_{1}, \ldots, l_{\alpha}\}$ and $M= \{m_{1}, \ldots, m_{\beta}\}$ are complementary subsets of $\{1,\,2,\ldots,\,n\}$, and we set  $N_{\alpha}= a_{l_{1}}+ \cdots + a_{l_{\alpha}}$ and $N_{\beta}= a_{m_{1}}+ \cdots + a_{m_{\beta}}$. If $\x=(\x_{1}, \ldots, \x_{n})\in \R^{N}$, we write $\x=(\x', \x'')$,  with $\x' =(\x_{l_{1}}, \ldots, \x_{l_{\alpha}})$ and $\x''= (\x_{m_{1}}, \ldots, \x_{m_{\beta}})$. If $I=(i_{1}, \ldots, i_{n}) \in \mathcal E_{n}$, write $I'=(i_{l_{1}}, \ldots, i_{l_{\alpha}}) \in E_{\alpha}$ and $I''= (i_{m_{1}}, \ldots, i_{m_{\beta}})\in E_{\beta}$







\begin{proposition}\label{Lemma5.4y} 
For each $I\in \mathcal E_{n}$ let $\varphi^{I}\in \mathcal S(\R^{N})$, and suppose there are constants $C_{M}$ so that that for all $I\in \mathcal E_{n}$ and all $M$, $||\varphi^{I}||_{[M]}\leq C_{M}$. Let $F\subset E_{n}$ be a finite subset, and let $K_{F}(\x) = \sum_{I\in F}[\varphi^{I}]_{I}(\x)$.
Let $\psi \in \mathcal C^{\infty}_{0}\left(\R^{N_{\beta}}\right)$ be a bump function in the variables $\x''=(\x_{m_{1}}, \ldots, \x_{m_{\beta}})$.  Let $R= (R_{1}, \ldots , R_{\beta})$ be a $\beta$-tuple of positive real numbers, and let $\bar\gamma = (\gamma_{l_{1}}, \ldots, \gamma_{l_{\alpha}})\in \N^{a_{l_{1}}}\oplus \cdots \oplus \N^{a_{l_{\alpha}}}$. There exists a constant $C_{\bar\gamma}$, independent of $R$ so that
\begin{align*}
\Big\vert 
\partial^{\gamma_{l_{1}}}_{\x_{l_{1}}}\cdots  \partial^{\gamma_{l_{\alpha}}}_{\x_{l_{\alpha}}} \int\limits_{\R^{N_{\beta}}} K_{F}(\x',\x'') \,&\psi(R\cdot \x'')\,d\x''\Big\vert
\leq 
C_{\bar\gamma}
\prod_{p=1}^{\alpha} \left [N_{l_{1}}(\x_{l_{1}}) + \cdots N_{l_{p}}(\x_{l_{p}})\right]^{-Q_{l_{p}}-[\![\gamma_{l_{p}}]\!]}.
\end{align*}
\end{proposition}

\begin{proof}
Recall the definition of $I'$ and $I''$ from just before the statement of the Proposition. 
Using Proposition \ref{Prop2.15}  to write $\int [\varphi^{I}]_{I'}(\x',\x'')\psi(R\cdot \x'')\,d\x''= [\Theta^{I}]_{I'}(\x')$, we have
\begin{align*}
\int K_{F}(\x',\x'') \psi(R\cdot \x'')\,d\x''
&=
\sum_{I\in F\subset E_{n}} \int [\varphi^{I}]_{I'}(\x',\x'')\psi(R\cdot \x'')\,d\x''\\
&
=
\sum_{I\in F\subset E_{n}}[\Theta^{I}]_{I'}(\x').
\end{align*}
Each $\Theta^{I}$ is normalized relative to $\varphi$. 
We write the sum over $I\in F$ as an iterated sum as follows. Let
\begin{align*}
E_{1}&= \left\{I'=(i_{l_{1}}, \ldots, i_{l_{\alpha}})\in \mathbb Z^{\alpha}\,\big\vert\,(i_{1}, \ldots, i_{n})\in F\right\},\\
\intertext{and for $I'\in E_{1}$, let}
E_{2}(I')&= \left\{I''=(i_{m_{1}}, \ldots, i_{m_{\beta}})\}\in \mathbb Z^{\beta}\,\big\vert\, (i_{1}, \ldots i_{n})\in F\subset E_{n}\right\}.
\end{align*}
If $I\in F$, we write $I=(I', I'')$ with $I'\in E_{1}$ and $I''\in E_{2}(I')$.
Then
\begin{align*}
\int K_{F}(\x',\x'') \psi(R\cdot \x'')\,d\x
=
\sum_{I'\in E_{1}}\Big[\sum_{I''\in E_{2}(I')} \Theta^{(I', I'')}\Big]_{I'}(\x').
\end{align*}
We must show that this sum satisfies  the differential inequalities for flags on the space $\R^{a_{l_{1}}}\oplus \cdots \oplus \R^{a_{l_{\alpha}}}$, with constants independent of the finite set $F$. This will follow from Proposition \ref{Proposition5.3z} provided we can show that for each $I'\in E_{1}$, the  sum $\sum_{I''\in E_{2}(I')}\Theta^{(I',I'')}$ converges to a normalized Schwartz function. However,  this follows from the estimates in part (\ref{Prop2.15b}) of Proposition \ref{Prop2.15}.  \end{proof}

%



\smallskip

We now turn to the proof of Theorem \ref{Lemma5.4zw}. As already indicated on page \pageref{remarks} just before the statement of Proposition \ref{Proposition5.3z}, part (1) is an immediate consequence of  Proposition \ref{Proposition5.3z} and Proposition \ref{Lemma5.4y},  so we only need to establish part (2).  Let $\psi \in \mathcal S(\R^{n})$, and for each $I\in \mathcal E_{n}$ let $\varphi^{I}\in \mathcal S(\R^{N})$ have weak cancellation relative to $I$. According to the second of the two Remarks \ref{Remark2.20}, we can write
\begin{equation*}
\varphi^{I}= \sum_{\ell=1}^{a_{1}}\partial_{\ell}[\varphi_{\ell}^{I}] +2^{-\epsilon(i_{2}-i_{1})}\varphi_{0}^{I},
\end{equation*}
and so
\begin{equation*}
[\varphi^{I}]_{I}= \sum_{\ell=1}^{a_{1}}(2^{i_{1}d_{\ell}}\partial_{\ell})[\varphi_{\ell}^{I}]_{I} +2^{-\epsilon(i_{2}-i_{1})}[\varphi_{0}^{I}]_{I}.
\end{equation*}
Then integrating by parts, we have
\begin{align*}
\int_{\R^{N}}[\varphi^{I}(\x)]_{I}\psi(\x)\,d\x = -\sum_{\ell=1}^{a_{1}}\int_{\R^{N}}2^{i_{1}d_{\ell}}[\varphi^{I}_{\ell}]_{I}(\x)\frac{\partial\psi\,\,}{\partial x_{\ell}}(\x)\,d\x
+
\int_{\R^{N}}2^{-\epsilon(i_{2}-i_{1})}[\varphi_{0}^{I}]_{I}\psi(\x)\,d\x.
\end{align*}
Thus if $F\subset E_{n}$ is a finite subset, and $K_{F}(\x) = \sum_{I\in F}[\varphi^{I}]_{I}(\x)$, we have
\begin{align*}
\int_{\R^{N}}&K_{F}(\x)\psi(\x)\,d\x \\
&= 
-\sum_{\ell=1}^{a_{1}}
\int_{\R^{N}}
\Big[\sum_{I\in F}2^{i_{1}d_{\ell}}[\varphi^{I}_{\ell}]_{I}(\x)\Big]\frac{\partial\psi\,\,}{\partial x_{\ell}}(\x)\,d\x
+
\int_{\R^{N}}\Big[\sum_{I\in F}2^{-\epsilon(i_{2}-i_{1})}[\varphi_{0}^{I}]_{I}\Big]\psi(\x)\,d\x
\\
&=
-\sum_{\ell=1}^{a_{1}}\int_{\R^{N}} K_{F}^{\ell}(\x)\,\frac{\partial\psi}{\partial x_{\ell}}(\x)\,d\x + \int_{\R^{N}} K_{F}^{0}(\x)\,\psi(\x)\,d\x 
\end{align*}
where
\begin{align*}
K_{F}^{\ell}(\x) &= \sum_{I\in F}2^{i_{1}d_{\ell}}[\varphi^{I}_{\ell}]_{I}(\x), \quad 1 \leq \ell \leq a_{1},\\
K_{F}^{0}(\x) &= \sum_{I\in F}2^{-\epsilon(i_{2}-i_{1})}[\varphi_{0}^{I}]_{I}.
\end{align*}
If $\alpha=(\alpha_{1}, \ldots, \alpha_{N}) \in \N^{N}$, we have
\begin{align*}
\partial^{\alpha}K_{F}^{\ell}(\x) &=\sum_{I\in F}2^{i_{1}d_{\ell}}2^{-\sum_{j=1}^{n}i_{j}(Q_{j}+[\![\alpha_{j}]\!])}\partial^{\alpha}[\varphi^{I}_{\ell}](2^{-I}\cdot \x),
\\
\partial^{\alpha}K_{F}^{0}(\x) &=\sum_{I\in F}2^{-\epsilon(i_{2}-i_{1})}2^{-\sum_{j=1}^{n}i_{j}(Q_{j}+[\![\alpha_{j}]\!])}\partial^{\alpha}[\varphi^{I}_{0}](2^{-I}\cdot \x).
\end{align*}
It follows from Proposition \ref{Prop2.11} in Appendix II that (at least if $d_{\ell}<Q_{1}+[\![\alpha_{1}]\!]$)
\begin{align*}
\big|\partial^{\alpha} K_{F}^{\ell}(\x)\big|
&\leq
C\,N_{1}(\x_{1})^{d_{\ell}}
\prod_{\substack{j=1}}^{n}[N_{1}(\x_{1})+ \cdots + N_{j}(\x_{j})]^{-(Q_{j}+[\![\alpha_{j}]\!])}, \,\,\text{and}\\
\big|\partial^{\alpha} K_{F}^{0}(\x)\big|
&\leq
C\,N_{1}(\x_{1})^{\epsilon}(N_{1}(\x_{1})+N_{2}(\x_{2}))^{-\epsilon}
\prod_{\substack{j=1}}^{n}[N_{1}(\x_{1})+ \cdots + N_{j}(\x_{j})]^{-(Q_{j}+[\![\alpha_{j}]\!])}.
\end{align*}
The functions on the right hand side of the last two inequalities are integrable on $\R^{N}$. The proof of (2) then follows
from the dominated convergence theorem:
\begin{align*}
\lim_{F\nearrow \mathcal E_{n}} \big\langle K_{F}, \psi\big\rangle = -\sum_{\ell=1}^{a_{1}}\int_{\R^{N}} K^{\ell}(\x)\,\frac{\partial\psi}{\partial x_{\ell}}(\x)\,d\x + \int_{\R^{N}} K^{0}(\x)\,\psi(\x)\,d\x
\end{align*}
where
\begin{align*}
K^{\ell}(\x) &= \sum_{I\in \mathcal E_{n}}2^{i_{1}d_{\ell}}[\varphi^{I}_{\ell}]_{I}(\x), \quad 1 \leq \ell \leq a_{1},\\
K^{0}(\x) &= \sum_{I\in \mathcal E_{n}}2^{-\epsilon(i_{2}-i_{1})}[\varphi_{0}^{I}]_{I}.
\end{align*}

\subsection{Rewriting sums of bump functions with weak cancellation}\quad

\medskip

It follows from Theorem \ref{Lemma5.4zw} that a sum of dilates of normalized bump functions with weak cancellation converges to a flag kernel, and it follows from Theorem \ref{Lemma2.3} that a flag kernel can be written as a sum of dilates of normalized bump functions with strong cancellation plus a sum of flag kernels adapted to strictly coarser flags. It follows that a sum of dilates of functions with \emph{weak} cancellation can be rewritten as sums of dilates of functions with \emph{strong} cancellation relative to coarser flags. In this section we give a direct proof of this fact. The basic idea is to use telescoping series to replace a function with weak cancellation by a sum of functions with strong cancellation plus an error term which belongs to a flag which is coarser than the original flag. Aside from its intrinsic interest, we shall need this observation in the forthcoming paper \cite{NRSW2}. 

Thus consider the standard flag $\mathcal F_{\mathcal A}$ on $\R^{N}$ of step $n$ associated to the decomposition 
\begin{equation*}
(\mathcal A)\qquad \R^{N}= \R^{a_{1}}\oplus\cdots \oplus\R^{a_{n}}.
\end{equation*}
Any strictly coarser flag $\mathcal F_{\mathcal B}\succ\mathcal F_{\mathcal A}$ then arises from a decomposition
\begin{equation*}
(\mathcal B)\qquad \R^{N}= \R^{b_{1}}\oplus\cdots \oplus\R^{b_{m}}
\end{equation*}
where $m<n$ and each $\R^{b_{j}}= \R^{a_{r_{j}}}\oplus \cdots \oplus \R^{a_{s_{j}}}$ where $1=r_{1}$, $n=s_{m}$, $r_{j}\leq s_{j}$ for $1 \leq j \leq m$, and $r_{j+1}=s_{j}+1$ for $1\leq j \leq m-1$. As usual we let $\mathcal E_{n}$ denote the set of $n$-tuples of integers $I=(i_{1},\ldots, i_{n})$ with $i_{1}\leq \cdots \leq i_{n}$. For any strictly coarser flag $\mathcal F_{\mathcal B}$ of step $m<n$ as above, we let $\mathcal E_{\mathcal B}$ denote the set of $m$-tuples of integers $J=(j_{1}, \ldots, j_{m})$ with $j_{1}\leq \cdots \leq j_{m}$.  Given the argument for part (\ref{Thm3.7(2)}) of Theorem \ref{Lemma5.4zw}, we shall only concern ourselves with finite sums, and thus will not need to worry about convergence questions.

\begin{proposition} Let $\mathcal F_{\mathcal A}$ denote the standard flag associated to the decomposition $\R^{N}= \R^{a_{1}}\oplus\cdots \oplus\R^{a_{n}}$. Suppose that
\begin{equation}\label{oiu}
K(\x)= \sum_{I\in \mathcal E_{n}}[\varphi^{I}]_{I}(\x)
\end{equation}
is a finite sum, where each $\varphi^{I}$ is a normalized bump function for the flag $\mathcal F_{\mathcal A}$, with weak cancellation relative to $I\in \mathcal E_{n}$ with parameter $\epsilon$. Then we can write
\begin{equation}\label{strong+coarser}
K(\x) = \sum_{\mathcal B\succeq \mathcal A}\sum_{J\in \mathcal E_{\mathcal B}}[\eta^{J}_{\mathcal B}]_{J}
\end{equation}
where the outer sum is taken over decompositions equal to or coarser than $\mathcal A$, and each $\eta^{J}_{\mathcal B}$ is a normalized bump function which has strong cancellation relative to the flag $\mathcal F_{\mathcal B}$ associated to the decompsition $\mathcal B$.
\end{proposition}

\begin{proof}
We argue by induction on the number of steps $n$ in the original flag $\mathcal A$. For $n=1$ there is nothing to prove, because then there is no distinction between weak and strong cancellation. 

Thus suppose the proposition has been established for all flags of step less than or equal to  $n-1$, and consider the flag of step $n$ corresponding to the decomposition $\mathcal A$. The inductive step itself requires an induction. Let the function $K$ be given by (\ref{oiu}). Since each $\varphi^{I}$ has weak cancellation with respect to $I$ with parameter $\epsilon$, we can write
\begin{equation*}
\varphi^{I}= \sum_{S\subset \{1,2,\ldots, n-1\}}2^{-\epsilon\sum_{l\in B}(j_{l+1}-j_{l})}\eta^{I}_{S}
\end{equation*}
where the sum is over all subsets $S$ of $\{1, \ldots, n-1\}$ and each $\eta^{I}_{S}$ is a normalized bump function which has integral zero in each multi-variable $\x_{r}$ with $r\not\in S$. Thus we have
\begin{equation}\label{0.3asd}
K =\sum_{I\in \mathcal E_{n}} \sum_{S\subset \{1,2,\ldots, n-1\}}2^{-\epsilon\sum_{l\in S}(j_{l+1}-j_{l})}[\eta^{I}_{S}]_{I}.
\end{equation}
We will prove by induction on $k$, for $1\le k\le n$, that $K$ can be written
\begin{equation}\label{thesis}
K=\sum_{I\in \mathcal E_{n}}\sum_{S\subseteq\{1,2,\dots,n-k\}}2^{-\eps\sum_{l\in S}(j_{l+1}-j_l)}[\eta^{I}_{S}]_I + \sum_{\substack{\mathcal B\succ \mathcal A}}K^{\mathcal B},
\end{equation}
where the functions $\{\eta^{I}_{S}\}$ and $\{K^{\mathcal B}\}$ have the following properties.
\begin{enumerate}[(i)]
\item Each normalized bump function $\eta^{I}_{S}$ has integral zero with respect to each variable $\x_{r}$ with $r\not\in S$.
\item For each $\mathcal B\succ \mathcal A$, (\textit{i.e.} for each decomposition $\mathcal B$ strictly coarser than $\mathcal A$ and hence whose corresponding flag has step strictly less than $n$), the function $K^{\mathcal B}$ can be written as a finite sum
\begin{equation*}
K^{\mathcal B}=\sum_{J\in\mathcal E(\mathcal B)}[\theta^{J}_{\mathcal B}]_{J},
\end{equation*} 
where each $\theta^{J}_{\mathcal B}$ is a normalized bump function with weak cancellation with some parameter $\epsilon' > 0$ relative to the flag arising from the decomposition $\mathcal B$. (It follows from the induction hypothesis that each such function can be rewritten as a sum of dilates of normalized bump functions with strong cancellation.)
\item The bump functions $\{\eta^{I}_{S}\}$ and the $\{\theta^{J}_{\mathcal B}\}$ are uniformly normalized relative to the normalized bump functions $\{\varphi^{I}\}$ defined in  \eqref{oiu}.
\end{enumerate}
Clearly equation \eqref{0.3asd} gives the desired conclusion in \eqref{thesis} for $k=1$. Moreover, when $k=n$ the set $S$ must be empty and thus we will have written $K$ as an appropriate sum of dilates of normalized bump functions with strong cancellation.

Thus we turn to the induction step. Suppose that equation \eqref{thesis} holds for a given $k<n$. We must show that \eqref{thesis} also holds with $k$ replaced by $k+1$.  We split the first sum into two parts depending on whether or not the subset $S$ contains the element $n-k$:
\begin{align}
K&=\!\!\!\sum_{\substack{J\in \mathcal E_{n}\\S\subseteq\{1,2,\dots,n-k\}\\n-k\in S}}\!\!\!2^{-\eps\sum_{l\in S}(j_{l+1}-j_l)}[\eta^{J}_{S}]_J 
+\!\!\!\sum_{\substack{I\in \mathcal E_{n}\\S\subseteq\{1,2,\dots,n-(k+1)\}}}\!\!\!2^{-\eps\sum_{l\in S}(j_{l+1}-j_l)}[\eta^{I}_{S}]_I + \sum_{\substack{\mathcal B\succ \mathcal A}}K^{\mathcal B}\notag\\
&=
K_{1}+K_{2}+K_{3}.\label{thesis(2)}
\end{align}
Now $K_{2}+K_{3}$ are already of the form in \eqref{thesis} with $k$ replaced by $k+1$, so we only need to deal with $K_{1}$.

Thus let $n-k\in S\subset \{1, 2, \ldots, n-k\}$, and consider the corresponding term $\eta=\eta^{J}_{S}$ in $K_{1}$.  Let $Q_{n-k}$ denote the homogeneous dimension of the space $\R^{a_{n-k}}$. Then the function
$$
(\eta^{J}_{S})'(\x_{1}, \ldots, \x_{n})=2^{Q_{n-k}}\eta^{J}_{S}\big(\x_{1},\dots,2\cdot\x_{{n-k}},\dots,\x_{n}\big)-\eta^{J}_{S}\big(\x_{1}, \ldots, \x_{n-k}, \ldots, \x_{n}\big)
$$
has cancellation in the variables $\x_{r}$ for all $r\not\in S'=S\setminus\{{n-k}\}$. Note that  $S'\subset \{1, 2, \ldots, n-(k+1)\}$.  Let  $J=(j_1,\dots,j_n)$. Using a telescoping series we have
\begin{align*}
\eta^{J}_{S}(\x)&=\sum_{i=1}^{j_{n-k+1}-j_{n-k}}2^{-iQ_{n-k}}(\eta^{J}_{S})'\big(\x_{1},\dots,2^{-i}\cdot\x_{n-k},\dots,\x_{n}\big)\\
&\qquad\qquad
+2^{-(j_{n-k+1}-j_{n-k})Q_{n-k}}\eta^{J}_{S}\big(\x_{1},\dots,2^{-(j_{n-k+1}-j_{n-k})}\cdot\x_{n-k},\dots,\x_{n}\big),
\end{align*}
and hence
\begin{equation}\label{Eq0.5bn}
[\eta^{J}_{S}]_J=\sum_{i=j_{n-k}+1}^{j_{n-k+1}}[(\eta^{J}_{S})']_{(j_1,\dots,j_{{n-k}-1},i,j_{{n-k}+1},\dots,j_s)} +[\eta^{J}_{S}]_{(j_1,\dots,j_{{n-k}-1},j_{{n-k}+1},j_{{n-k}+1},\dots,j_s)}\ .
\end{equation}
We regard the last term as associated to the coarser flag of step $(n-1)$ associated to the decomposition
\begin{equation*}
(\mathcal B)\quad \R^{a_{1}}\oplus \R^{a_{n-k-1}}\oplus \big[\R^{a_{n-k}}\oplus\R^{a_{n-k+1}}\big] \oplus \R^{a_{n-k+2}}\oplus\cdots \oplus\R^{a_{n}}
\end{equation*}
where $\R^{a_{n-k}}\oplus\R^{a_{n-k+1}}$ is now considered one factor. If  we set 
\begin{align*}
\text{$J_i^{(n-k)}=(j_1,\dots,j_{{n-k}-1},i,j_{{n-k}+1},\dots,j_n)$, $J^{\widehat{n-k}}=(j_1,\dots,j_{{n-k}-1},j_{{n-k}+1},\dots,j_n)$},
\end{align*}
then $J^{(n-k)}_{i}\in \mathcal E_{n}$ for $ j_{n-k}+1\le i\le j_{n-k+1}$ and $J^{\widehat{n-k}}\in \mathcal E_{\mathcal B}$. Moreover, $\eta^{J}_{S}$ has weak cancellation with respect to the flag $(\mathcal B)$ with parameter $\epsilon$.
Formula \eqref{Eq0.5bn} then becomes
\begin{equation}\label{eta-eta'}
[\eta^{J}_{S}]_J=\sum_{i=j_{n-k}+1}^{j_{n-k+1}}[(\eta^{J}_{S})']_{J_i^{(n-k)}}+[\eta^{J}_{S}]_{J^{\widehat{n-k}}}\ .
\end{equation}
Applying the identity \eqref{eta-eta'}, we obtain 
\begin{align*}
K_{1}&=\sum_{J\in \mathcal E_{n}}\sum_{\substack{S\subseteq\{1,2,\dots,n-k\}\\n-k\in S}}2^{-\eps\sum_{l\in S}(j_{l+1}-j_l)}[\eta^{J}_{S}]_J\\
&=
\sum_{J\in \mathcal E_{n}}\sum_{\substack{S\subseteq\{1,2,\dots,n-k-1\}}}2^{-\eps\sum_{l\in S\cup\{n-k\}}(j_{l+1}-j_l)}[\eta^{J}_{S\cup\{n-k\}}]_J
\\
&=
\sum_{J\in \mathcal E_{n}}\sum_{\substack{S\subseteq\{1,2,\dots,n-k-1\}}}2^{-\eps\sum_{l\in S\cup\{n-k\}}(j_{l+1}-j_l)}\sum_{i=j_{{n-k}+1}}^{j_{n-k}-1}\big[(\eta^{J^{(n-k)}_{i}}_{S})'\big]_{J^{(n-k)}_{i}}
\\
&\qquad\qquad
+\sum_{J\in \mathcal E_{n}}\sum_{\substack{S\subseteq\{1,2,\dots,n-k-1\}}}2^{-\eps\sum_{l\in S\cup\{n-k\}}(j_{l+1}-j_l)}[\eta^{J^{\widehat{n-k}}}_{S}]_{J^{\widehat{n-k}}}
\\
&=
\Sigma_{1}+\Sigma_{2}.
\end{align*}

In $\Sigma_1$ we change the order of summation, grouping together all the terms for which $J_i^{(n-k)}$ is a given $J'=(j'_1,\dots,j'_n)\in\mathcal E_{n}$. Clearly, this condition forces $J$ to differ from $J'$ only in its component $j_{n-k}$, and we have $j'_{n-k-1}\le j_{n-k}<j'_{n-k}$. 
In order to express the factor $2^{-\eps\sum_{l\in S\cup\{n-k\}}(j_{l+1}-j_l)}$ in terms of $J'$, we split the summation over the subsets $S$ into two parts; the first consists of subsets $S$ not containing $n-k-1$ and the second consists of the subsets which do contain $n-k-1$. We then have
$$
\begin{aligned}
&\Sigma_1\\
&=
\sum_{J'\in\mathcal E_{n}}\,
\sum_{S\subseteq\{1,2,\dots,n-k-2\}}
2^{-\eps\sum_{l\in S}(j'_{l+1}-j'_l)}
\Big(\sum_{i=j_{n-k-1}'}^{j'_{n-k}-1}2^{-\eps(j'_{n-k+1}-i)}\big[(\eta^{J'}_{S})'\big]_{J'}\Big)\\
&
+\sum_{J'\in\mathcal E_{n}}\,\sum_{S\subseteq\{1,2,\dots,n-k-2\}}2^{-\eps\sum_{l\in S}(j'_{l+1}-j'_l)}2^{-\eps(j'_{n-k+1}-j'_{n-k-1})}\Big(\sum_{i=j_{n-k-1}'}^{j'_{n-k}-1}\big[(\eta^{J'}_{S\cup\{n-k-1\}})'\big]_{J'}\Big)\\
&=\Sigma_{1,1}+\Sigma_{1,2}\ .
\end{aligned}
$$
Each function appearing in $\Sigma_{1,1}$ has integral zero in each variable $\x_{r}$ with $r\not\in B\subset\{1,\ldots, n-k-2\}$, and each function appearing in $\Sigma_{1,2}$ has integral zero in each variable $\x_{r}$ with $r\not\in B\cup\{n-k-1\}\subset \{1,\ldots, n-k-1\}$.  All the $\eta'$ in the above formula are normalized relative to the initial data. Hence, the term in $\Sigma_{1,1}$ indexed by  $(J',S)$ contains a function $\tilde\eta^{J'}_{S}$ normalized relative to the initial data, and multiplied by a factor
$$
2^{-\eps\sum_{l\in S}(j'_{l+1}-j'_l)}\sum_{j_{n-k-1}'}^{j'_{n-k}-1}2^{-\eps(j'_{n-k+1}-i)}\lesssim 2^{-\eps\sum_{l\in S}(j'_{l+1}-j'_l)}\ ,
$$
whereas the corresponding term in $\Sigma_{1,2}$ is a bump function $\tilde\eta^{J'}_{S\cup\{n-k-1\}}$ normalized relative to the initial data and multiplied by
$$
\begin{aligned}
2^{-\eps\sum_{l\in S}(j'_{l+1}-j'_l)}2^{-\eps(j'_{n-k}-j'_{n-k-1})}(j'_{n-k}-j'_{n-k-1})
\lesssim 2^{-\eps'\sum_{l\in S\cup\{n-k-1\}}(j'_{l+1}-j'_l)}\ ,
\end{aligned}
$$
with $\eps'<\eps$. Hence,
$$
\Sigma_1=\sum_{J'\in\mathcal E_{n}}\sum_{S\subseteq\{1,2,\dots,n-k-1\}}2^{-\eps\sum_{l\in S}(j'_{l+1}-j'_l)}[\tilde\eta^{J'}_{S}]_{J'}\ ,
$$
and $\Sigma_1$ together with $K_{2}$ gives the first sum in \eqref{thesis} with $k$ replaced by $k+1$.

It remains to prove that $\Sigma_2$ can be absorbed in the remainder term (second sum) of \eqref{thesis}. We group together the terms at the same scale and separating the sets $S$ containing $n-k-1$ form the others. Indexing the elements $J'$ of $\mathcal E_{\mathcal B}$ as $J'=(j'_1,\dots,j'_{n-k-1},j'_{n-k},\dots,j'_{n-1})\in \mathcal E_{n-1}$, we have
$$
\begin{aligned}
\Sigma_2
&=
\sum_{J\in \mathcal E_{n}}\sum_{\substack{S\subseteq\{1,2,\dots,n-k-1\}}}2^{-\eps\sum_{l\in S\cup\{n-k\}}(j_{l+1}-j_l)}[\eta^{J^{\widehat{n-k}}}_{S}]_{J^{\widehat{n-k}}}
\\
&=\sum_{J'\in \mathcal E_{n-1}}\sum_{S\subseteq\{1,2,\dots,n-k-2\}}
\Big(\sum_{J\in\mathcal E_{n}:J^{\widehat{n-k}}=J'}2^{-\eps\sum_{l\in S\cup\{n-k\}}(j_{l+1}-j_l)}\big[\eta^{J^{\widehat{n-k}}}_{S}\big]_{J^{\widehat{n-k}}}\Big)
\\
&\qquad 
+\sum_{J'\in \mathcal E_{n-1}}\sum_{S\subseteq\{1,2,\dots,n-k-2\}}\Big(\sum_{J\in\mathcal E_{n}:J^{\widehat{n-k}}=J'}2^{-\eps\sum_{l\in S\cup\{n-k-1,n-k\}}(j_{l+1}-j_l)}\big[\eta^{J^{\widehat{n-k}}}_{S}\big]_{J^{\widehat{n-k}}}\Big)\\
&=\Sigma_{2,1}+\Sigma_{2,2}\ .
\end{aligned}
$$
As in the previous discussion, each term in parentheses is a function normalized relative to the initial data, multiplied by a factor controlled by
$$
2^{-\eps\sum_{l\in S}(j_{l+1}-j_l)}\sum_{i=j'_{n-k-1}}^{j'_{n-k}}2^{-\eps(j'_{n-k}-i)}\lesssim 2^{-\eps\sum_{l\in S}(j'_{l+1}-j'_l)}\ ,
$$
for the terms in $\Sigma_{2,1}$ and by
$$
2^{-\eps\sum_{l\in S}(j_{l+1}-j_l)}2^{-\eps(j'_{n-k}-j'_{n-k-1})}(j'_{n-k}-j'_{n-k-1})\lesssim 2^{-\eps'\sum_{l\in S\cup\{n-k-1\}}(j'_{l+1}-j'_l)}\ ,
$$
with $\eps'<\eps$, for the terms in $\Sigma_{2,2}$. It follows that 
$$
\Sigma_3=\sum_{J'\in\mathcal E_{n-1}}\sum_{S\subseteq\{1,2,\dots,n-k-1\}}2^{-\eps\sum_{l\in S}(j'_{l+1}-j'_l)}\big[\tilde\eta^{J'}_{S}\big]_{J'}\ ,
$$
and this concludes the proof.
\end{proof}

\subsection{Restricted cancellation conditions}\quad

\smallskip

Our next result shows that the cancellation conditions in the definition of a flag kernel can be relaxed. As usual, let $\mathcal F$ denote the standard flag 
$$
(0)\subset\R^{a_n} \subset\R^{a_{n-1}}\oplus\R^{a_n}\subset \cdots\subset \R^{a_2}\oplus\cdots\oplus\R^{a_n}\subset \R^N\ .
$$

\begin{theorem}\label{restricted}
Let $\K$ be a distribution in $\Sc(\R^N)$ which satisfies the conditions of Definition \ref{Def2.1} for the flag $\mathcal F$, except that in condition {\rm(\ref{Def2.1b})}, the values of the parameters   $R_1,\dots,R_s$ are restricted to $R_1\ge R_2\ge\cdots\ge R_s$.
Then $\K$ is a flag kernel.
\end{theorem}

Before giving the formal proof, let us explain what needs to be done. We are asserting that if the cancellation conditions in part {\rm(\ref{Def2.1b})} of Definition \ref{Def2.1} are satisfied when $R_1\ge R_2\ge\cdots\ge R_s$, then they are satisfied for \emph{all} values of the scaling parameters. To do this we fix a non-empty subset $I\subseteq\{1,\dots,n\}$, and a constant $\rho_i>0$ for every $i\in I$. For each $i\in I$ we choose $\psi_i\in C^\infty_0(\R^{a_i})$ equal to 1 on the $N_i$-ball of radius 1 and supported on the $N_i$-ball of radius 2, and with bounds on the norms $\{||\psi_{i}||_{(m)}\}$.
We set $\x_I=(\x_i)_{i\in I}$ and $\x'_I=(\x_i)_{i\not\in I}$, and put
$$
\Psi_\rho(\x_I)=\prod_{i\in I}\psi_i(\rho_i\cdot\x_i)\ .
$$
Let $\K^\#_{\Psi,\rho}=\K^\#_{\Psi,\rho}(\x_{I}^{'})= \langle \K, \Psi_{\rho}\rangle$ denote the distribution in the variables $\x'_I$ such that
$
\lan \K^\#_{\Psi,\rho},\ph\ran=\lan \K,\Psi_\rho\otimes\ph\ran
$
for every test function $\varphi$ in the variables $\x_{I}^{'}$. 
We must then prove the following:  

\smallskip

\noindent \textit{Suppose that $\K$ satisfies the hypotheses of Theorem \ref{restricted}. Then
\begin{enumerate}[(a)]
\smallskip
\item \label{parta}If $I= \{1, \ldots, n\}$, then $|\langle \K, \Psi_{\rho}\rangle|\leq C$ where $C$ is independent of the choice of $\{\rho_{i}\}$ and $\{\psi_{i}\}$.
\smallskip
\item \label{partb}If $I$ is a proper, non-empty subset of $\{1, \ldots, n\}$, let $i_{0}=\min\{j\,\big\vert\,j\notin I\}$. Then $\K^\#_{\Psi,\rho}$ coincides with a smooth function for $\x_{i_0}\ne 0$, and for every multi-index $\bar\al$,
\begin{equation}\label{inequalities}
\big|\de^{\bar\al}_{\x'_I}\K^\#_{\Psi,\rho}(\x'_I)|\le C_{\alpha}\prod_{i\not\in I}\Big(\sum_{\substack{l\not\in I \\ l\le i }}N_l(\x_l)\Big)^{-Q_i-[\![\bar\al_i]\!]}
\end{equation}
where $\{C_{\alpha}\}$ are independent of the choice of $\{\rho_{i}\}$ and $\{\psi_{i}\}$.
\end{enumerate}}

In proving (\ref{parta}) or (\ref{partb}), we may assume that $\K$ has compact support, and look for non-restricted cancellation estimates that only depend on the constants in Definition 4.1, and not on the size of the support. For general $\K$, the conclusion will then follow by a limiting argument, based on the following construction. We fix a $C_0^\infty$-function $\ph$ on the real line, equal to 1 on a neighborhood of the origin, and set $\Phi=\ph\otimes\cdots\otimes\ph\in C^\infty_0(\R^N)$, $\Phi_r=\Phi\circ \del_{r\inv}$. Then $\K_r=\Phi_r\K$ has compact support, satisfies the hypotheses of Theorem \ref{restricted} uniformly in $r$, and $\lim_{r\to\infty}\K_r=\K$ in the sense of distributions.

\begin{proof}

For $i\in I$, define $R_i$ as
\begin{equation}\label{Ri}
R_i=\max_{\substack{l\in I \\ l\ge i }}\rho_\ell\ .
\end{equation}
Then $R_{1}\geq R_{2}\geq \cdots$, and so by hypothesis  $\K^\#_{\Psi,R}$ satisfies the estimates in \eqref{inequalities}. Thus it suffices to prove that the difference $\K^\#_{\Psi,\rho}-\K^\#_{\Psi,R}$ does as well.

Denote by $I^0$, (respectively $I^+$), the set of $i\in I$ such that $R_i=\rho_i$, (respectively $R_i>\rho_i$). Setting $\eta_i(\x_i)=\psi_i(\rho_i\cdot\x_i)-\psi_i(R_i\cdot\x_i)$, we have

\begin{equation*}
\begin{aligned}
\Psi_\rho(\x)-\Psi_R(\x)&=\Big(\prod_{i\in I^0}\psi_i(R_i\cdot\x_i)\Big)\Big(\prod_{i\in I^+}\psi_i(\rho_i\cdot\x_i)-\prod_{i\in I^+}\psi_i(R_i\cdot\x_i)\Big)\\
&=\Big(\prod_{i\in I^0}\psi_i(R_i\cdot\x_i)\Big)\Big(\prod_{i\in I^+}\big(\psi_i(R_i\cdot\x_i)+\eta_i(\x_i)\big)-\prod_{i\in I^+}\psi_i(R_i\cdot\x_i)\Big)\\
&=\Big(\prod_{i\in I^0}\psi_i(R_i\cdot\x_i)\Big)\sum_{\emptyset\ne J\subseteq I^+}\Big(\prod_{i\in J}\eta_i(\x_i)\Big)\Big(\prod_{i\in I^+\setminus J}\psi_i(R_i\cdot\x_i)\Big)\\
&=\sum_{\emptyset\ne J\subseteq I^+}\Big(\prod_{i\in J}\eta_i(\x_i)\Big)\Big(\prod_{i\in I\setminus J}\psi_i(R_i\cdot\x_i)\Big)\ .
\end{aligned}
\end{equation*}

Fix $J\subseteq I^+$, $J\ne\emptyset$.
By definition of $I^0$ and $I^+$, for each $i\in J$, $R_i=\rho_l$ for some $l\in I^0$, $l>i$. Set 
$\bar i=\min\{l\in I^0:l>i\text{ and }R_i=\rho_l\}$ and $\bar J=\{\bar i:i\in J\}$. 
By the cancellation  of $\K$ in the variables $\x_i$ for $i\in I\setminus(J\cup\bar J)$, 
$$
\K^\#_{\Psi,\rho}(\x'_I)-\K^\#_{\Psi,R}(\x'_I)=\sum_{\emptyset\ne J\subseteq I^+}\Big\lan \K^\#_J\,,\,\Big(\prod_{i\in J}\eta_i\Big)\Big(\prod_{i\in \bar{J}}(\psi_i\circ\del_{R_i})\Big)\Big\ran_{J\cup\bar J}\ ,
$$
where each $\K^\#_J$ is a distribution in the variables $\x_i$ with $i\in {}^cI\cup J\cup\bar J$, satisfying condition {\rm(\ref{Def2.1a})} of Definition \ref{Def2.1}.

Notice that this pairing can be expressed as an integral because the right-hand side $f$ in the pairing above is supported where $\K^\#_J$ is smooth. To see this, denote by  $i_0$ the smallest element of ${}^cI\cup J\cup \bar J$. Then $\K^\#_J$ is smooth for $\x_{i_0}\ne0$.  On the other hand, if $\x=(\x_i)_{i\in{}^cI\cup J\cup \bar J} \in\supp f$, all coordinates $\x_i$ are bounded away from zero except for those in $\bar J$. But every element of $\bar J$ is strictly larger than some element of $J$, therefore $i_0\not\in\bar J$.

We estimate the $\bar\al$-derivative of each $\K^\#_J$ by
\begin{equation}\label{diffineq'}
\begin{aligned}
|\de^{\bar\al}_{\x'_I}&\K^\#_J(\x'_I,\x_J,\x_{\bar J})|\\
&\le C_{\alpha}\Big(\prod_{i\in J}N_i(\x_i)^{-Q_i}\Big)\Big(\prod_{l\in \bar J}\big(\sum_{i\in J_l}N_i(\x_i)\big)^{-Q_l}\Big)\prod_{i\not\in I}\Big(\sum_{\substack{l\not\in I \\ l\le i }}N_l(\x_l)\Big)^{-Q_i-[\![\bar\al_i]\!]}\ ,
\end{aligned}
\end{equation}
where $J_l=\{i\in J:\bar i=l\}=\{i\in J:R_i=R_l\}$. Since the $J_l$ form a partition of $J$, we have
$$
\big|\K^\#_{\Psi,\rho}(\x'_I)-\K^\#_{\Psi,R}(\x'_I)\big|\le C\Big(\prod_{i\not\in I}\Big(\sum_{\substack{l\not\in I \\ l\le i }}N_l(\x_l)\Big)^{-Q_i-[\![\bar\al_i]\!]}\Big)\sum_{\emptyset\ne J\subseteq I^+}\prod_{l\in\bar J}V_l\ .
$$
With $m_l$ denoting the cardinality of $J_l$,
\allowdisplaybreaks{
\begin{align*}
V_l&=\int\limits_{\{N_i(\x_i)>R_l\inv \,,\,\forall i\in J_l\}}\int\limits_{N_l(\x_l)<R_l\inv}\Big(\prod_{i\in J_l}N_i(\x_i)^{-Q_i}\Big)\Big(\sum_{i\in J_l}N_i(\x_i)\Big)^{-Q_l}\,d\x_l\,\prod_{i\in J_l}d\x_i\\
&=R_l^{-Q_l}\int\limits_{\{N_i(\x_i)>R_l\inv \,,\,\forall i\in J_l\}}\Big(\prod_{i\in J_l}N_i(\x_i)^{-Q_i}\Big)\Big(\sum_{i\in J_l}N_i(\x_i)\Big)^{-Q_l}\,\prod_{i\in J_l}d\x_i\\
&\le CR_l^{-Q_l}\prod_{i\in J_l}\int\limits_{N_i(\x_i)>R_l\inv}N_i(\x_i)^{-Q _i-Q_l/m_l}\,d\x_i
\le C\ .
\end{align*}
}

This concludes the proof.
\end{proof}

\subsection{Invariance of flag kernels under changes of variables}\label{ChVariables}\quad

\smallskip

We study the effect of a change of variables on the class of flag distributions.  If $\K\in \mathcal S'(\R^{N})$, then formally $\langle \K,\psi\rangle = \int_{\R^{N}}K(\x)\psi(\x)\,d\x$, where $K$ is the `kernel' associated to $\K$. Let $F:\R^{N}\to \R^{N}$ be a diffeomorphism with inverse $G$. We want to define a new distribution $\K^{\#}$ which is the composition of $\K$ with the change of variables $F$. Now formally 
\begin{equation*}
\big\langle \K\circ F, \psi\big\rangle = \int_{\R^{N}}K(F(\x))\psi(\x)\,d\x= \int_{\R^{N}}K(\y) \psi(G(\y))\det(JG)(\y)\,dy
\end{equation*} 
where $JG$ is the Jacobian matrix of $G$. Thus if $\psi^{\#}(\y) = \psi(G(\y))\det(JG)(\y)$, and if the change of variables has the property that $\psi\in \mathcal S(\R^{N})$ implies $\psi^{\#}\in \mathcal S(\R^{N})$, we can define $\K^{\#}= \K\circ F$ by setting $\big\langle\K^{\#},\psi\big\rangle = \big\langle \K,\psi^{\#}\big\rangle$ for all $\psi \in\mathcal S(\R^{N})$.

We are primarily interested in changes of variables of the form
\begin{equation*}
F(x_{1}, \ldots, x_{N}) = (x_{1}+P_{1}(\x), \ldots, x_{n}+P_{N}(\x))
\end{equation*}
where $P_{1}, \ldots, P_{N}$ are polynomials of the form $P_{k}(x_{1}, \ldots, x_{N}) = \sum c^{\alpha}_{k}\,x_{1}^{\alpha_{1}}\,\cdots\,x_{k-1}^{\alpha_{k-1}}$ with coefficients $c^{\alpha}_{k}\in \R$. Thus $P_{k}$ depends only on the variables $\{x_{1}, \ldots, x_{k-1}\}$, and this guarantees that $F$ is a diffeomorphism with inverse $G$ of the same form. In particular, $\det (JG)(\x)$ is a polynomial. Thus if $\psi\in \mathcal S(\R^{N})$, then $\det(JG)\,\psi\circ G \in \mathcal S(\R^{N})$, so $\K^{\#}$ is well-defined. 
We want to show that if $\K$ is a flag distribution adapted to the decomposition $\R^{N}= \R^{a_{1}}\oplus \cdots \oplus \R^{a_{n}}$, then $\K^{\#}$ has the same property. In order to show this, we need to make additional assumptions on the coefficients $\{c^{\alpha}_{k}\}$. 

\smallskip

\begin{definition} A change of variables $\y=F(\x)$ of the form $y_{k}= x_{k}+P_{k}(\x)$ where $P_{k}(\x) = \sum_{\alpha\in \mathcal B_{k}}c^{\alpha}_{k}x_{1}^{\alpha_{1}}\cdots x_{k-1}^{\alpha_{k-1}}$ 
is \emph{allowable} if $$\mathcal B_{k}= \Big\{(\alpha_{1}, \ldots, \alpha_{k-1})\in \mathbb N^{k-1}\,\big\vert\,\sum_{j=1}^{k-1}\alpha_{j}d_{j} = d_{k}\Big\}.$$




\end{definition}

\begin{theorem}\label{Thm3.12}
Let $\K\in \mathcal S'(\R^{N})$ be a flag kernel adapted to the standard flag $\mathcal F$ coming from the decomposition $\R^{N}= \R^{a_{1}}\oplus\cdots \oplus \R^{a_{n}}$. If $\y=F(\x)$ is an allowable change of variables, then $\K^{\#}= \K\circ F$ is a flag distribution for the same decomposition.




\end{theorem}

\begin{proof}
We can assume that  $\K = \sum_{I\in \mathcal E_{n}}[\varphi^{I}]_{I}$ where each $\varphi^{I}$ is a normalized bump function having strong cancellation. Let $I\in E_{N}$. We consider the dilate $[\varphi^{I}]_{I}$ composed with an allowable change of variables. We have
\begin{align*}
[\varphi^{I}]_{I}(F(\x))
&
= [\varphi^{I}]_{I}\big( \ldots,x_{l}+P_{l}(x_{1}, \ldots, x_{l-1}), \ldots\big)\\
&=
2^{-\sum_{k=1}^{N}d_{k}i_{k}}\varphi^{I}\big( \ldots,2^{-d_{l}i_{l}}[x_{l}+P_{l}(x_{1}, \ldots, x_{l-1})],\ldots\big)
\end{align*}
Put
\begin{align*}
\theta^{I}(\x) &
=
\varphi^{I}\big(\ldots, x_{l}+2^{-d_{l}i_{l}}P_{l}(2^{d_{1}i_{1}}x_{1}, \ldots, 2^{d_{l-1}i_{l-1}}x_{l-1}), \ldots\big)
\end{align*}
so that $[\theta^{I}]_{I}(\x) = [\varphi^{I}]_{I}(F(\x))$. Put 
\begin{equation*}
P_{l}^{I}(\x) = 2^{-d_{l}i_{l}}P_{l}(2^{I}\cdot \x) 
=
\sum_{\mathcal B_{l}}\,2^{-d_{l}i_{l}+ \sum_{k=1}^{l-1}d_{j}\alpha_{j}i_{j}}\,c_{k}^{\alpha}x_{1}^{\alpha_{1}}\cdots x_{k-1}^{\alpha_{k-1}}.
\end{equation*}
Then $\theta^{I}(\x) = \varphi^{I}\big(x_{1}+P_{1}^{I}(\x), \ldots, x_{N}+P_{N}^{I}(\x)\big)$. Since the change of variables is allowable and $I\in \mathcal E_{n}$, we have
\begin{equation*}
-d_{l}i_{l}+\sum_{k=1}^{l-1}d_{k}\alpha_{k}i_{k}=-\sum_{k=1}^{l-1}\alpha_{k}d_{k}(i_{l}-i_{k})\leq 0.
\end{equation*}
It follows that each $P_{l}^{I}$ is normalized relative to $P_{l}$, and this shows that $[\varphi^{I}]_{I}(F(\x))= [\theta^{I}]_{I}(\x)$, where $\theta^{I}\in \mathcal C^{\infty}_{0}(\R^{N})$ is normalized relative to $\varphi^{I}$. 

\medskip

Next we study the cancellation properties of $\theta^{I}$. If we can show that each $\theta^{I}$ has weak cancellation relative to the multi-index $I\in \E_{n}$, it follows from Theorem \ref{Lemma5.4zw} that 
\begin{equation*}
\K\circ F = \sum_{I}[\varphi^{I}]_{I}\circ F = \sum_{I}[\theta^{I}]_{I}
\end{equation*}
is a flag kernel, which is what we want to show. 

To do this we use Proposition \ref{Prop5.7ijn}. Let $\{1, \ldots, n\}= A\cup B$ with $A\cap B = \emptyset$. We study
\begin{align*}
\int\limits_{\bigoplus_{k\in B}\R^{a_{k}}}&\theta^{I}(\x_{A}, \x_{B})\,d\x_{B}\\
&=
\int\limits_{\bigoplus_{k\in B}\R^{a_{k}}}\varphi^{I}\big(\ldots, x_{l}+2^{-d_{l}i_{l}}P_{l}(2^{d_{1}i_{1}}x_{1}, \ldots, 2^{d_{l-1}i_{l-1}}x_{l-1}), \ldots\big)\,d\x_{B}.
\end{align*}
Since $\varphi^{I}$ has strong cancellation, we can write it as a sum of terms of the form $\partial_{j_{1}}\cdots \partial_{j_{n}}\widetilde \varphi^{I}$, where each index $j_{l}\in J_{l}$. It suffices to consider the integrals
\begin{equation*}
\int\limits_{\bigoplus_{k\in B}\R^{a_{k}}}\!\!\!\!\!
\partial_{j_{1}}\cdots \partial_{j_{n}}\widetilde \varphi^{I}
\big(\ldots, x_{l}+2^{-d_{l}i_{l}}P_{l}(2^{d_{1}i_{1}}x_{1}, \ldots, 2^{d_{j_{r}}i_{j_{r}}}x_{j_{r}}, \ldots, 2^{d_{l-1}i_{l-1}}x_{l-1}), \ldots\big)\,d\x_{B}.
\end{equation*}
Let $r\in B$. In this last integral,  replace the term $$2^{-d_{l}i_{l}}P_{l}(2^{d_{1}i_{1}}x_{1}, \ldots, 2^{d_{j_{r}}i_{j_{r}}}x_{j_{r}}, \ldots, 2^{d_{l-1}i_{l-1}}x_{l-1})$$ for $l>j_{r}$ by the term $2^{-d_{l}i_{l}}P_{l}(2^{d_{1}i_{1}}x_{1}, \ldots, 0, \ldots, 2^{d_{l-1}i_{l-1}}x_{l-1})$; \textit{i.e.} we set $x_{j_{r}}=0$ everywhere in the integrand except where it appears by itself in the $j_{r}^{th}$ entry of $\partial_{j_{1}}\cdots \partial_{j_{n}}\widetilde \varphi^{I}$. Now
\begin{equation*}
\int_{\bigoplus_{k\in B}\R^{a_{k}}}
\partial_{j_{1}}\cdots \partial_{j_{n}}\widetilde \varphi^{I}
\big( \ldots, x_{l}+2^{-d_{l}i_{l}}P_{l}(2^{d_{1}i_{1}}x_{1}, \ldots, 0, \ldots, 2^{d_{l-1}i_{l-1}}x_{l-1}), \ldots\big)\,d\x_{B}=0
\end{equation*}
since one of the variables we integrate is $x_{j_{r}}$, and we are integrating the derivative of a Schwartz function. Thus it suffices to estimate the integral of the difference:
\begin{align*}
\int_{\bigoplus_{k\in B}\R^{a_{k}}}
&\Big[\partial_{j_{1}}\cdots \partial_{j_{n}}\widetilde \varphi^{I}
\big(\ldots, x_{l}+2^{-d_{l}i_{l}}P_{l}(2^{d_{1}i_{1}}x_{1}, \ldots, 2^{d_{j_{r}}i_{j_{r}}}x_{j_{r}}, \ldots, 2^{d_{l-1}i_{l-1}}x_{l-1}), \ldots\big)\\
&
-
\partial_{j_{1}}\cdots \partial_{j_{n}}\widetilde \varphi^{I}
\big( \ldots, x_{l}+2^{-d_{l}i_{l}}P_{l}(2^{d_{1}i_{1}}x_{1}, \ldots, 0, \ldots, 2^{d_{l-1}i_{l-1}}x_{l-1}), \ldots\big)\Big]\,
d\x_{B}.
\end{align*}
A typical term in the polynomial $P_{l}$ has the form $c_{\alpha}x_{1}^{\alpha_{1}}\cdots x_{l-1}^{\alpha_{l-1}}$ where $\alpha_{1}d_{1}+ \cdots +\alpha_{l-1}d_{l-1}= d_{l}$.  Thus we get a `gain' whose  size can be estimated by a sum of terms of the form
\begin{align*}
|c_{\alpha}|2^{-d_{l}i_{l}+ \alpha_{1}d_{1}i_{1}+\cdots +\alpha_{l-1}d_{l-1}i_{l-1}}|x_{1}|^{\alpha_{1}}\cdots |x_{l-1}|^{\alpha_{l-1}}
\end{align*}
where $\alpha_{j_{r}}>0$. However,
\begin{align*}
-d_{l}i_{l}+\sum_{t=1}^{l-1}\alpha_{t}d_{t}i_{t} &= \big[-d_{1}+\sum_{t=1}^{l-1}\alpha_{t}d_{t}\big]i_{l}+\sum_{t=1}^{l-1}\alpha_{t}d_{t}(i_{t}-i_{l})\\
&=\sum_{t=1}^{l-1}\alpha_{t}d_{t}(i_{t}-i_{l})\\
& \leq -\alpha_{r}d_{r}(i_{l}-i_{r}) \leq -\epsilon(i_{r+1}-i_{r}).
\end{align*}
Thus for every $r\in B$ we have shown that $\big|\int_{\bigoplus_{k\in B}\R^{a_{k}}}\theta^{I}(\x_{A}, \x_{B})\,d\x_{B}\Big| \lesssim 2^{-\epsilon(i_{r+1}-i_{r})}$. Thus with a smaller $\epsilon$ we have
\begin{equation*}
\big|\int_{\bigoplus_{k\in B}\R^{a_{k}}}\theta^{I}(\x_{A}, \x_{B})\,d\x_{B}\Big| \lesssim \prod_{r\in B}2^{-\epsilon(i_{r+1}-i_{r})},
\end{equation*}
and it follows from Proposition \ref{Prop5.7ijn} that $\theta^{I}$ has weak cancellation. This completes the proof.
\end{proof}

\section{Convolutions on nilpotent Lie groups}

\subsection{Homogeneous nilpotent Lie groups}\label{Nilpotent}\qquad

\smallskip

We now begin the study of operators $f\to f*\K$ where $\K$ is a flag kernel, and the convolution is on a homogeneous nilpotent Lie groug $G$ with Lie algebra $\mathfrak g$. To say that a Lie group $G$ is homogeneous means that there is a one-parameter group of automorphisms $\delta_{r}:G\to G$ for $r>0$, with $\delta_{1}= \text{Id}$.
As a manifold, $G$ is an $N$-dimension real vector space, and we assume that with an appropriate choice of coordinates,  $G= \R^{N}$ and the automorphisms are given by $\delta_{r}[\x] = r\cdot \x = (r^{d_{1}}x_{1}, \ldots, r^{d_{N}}x_{N})$ with  $1\leq d_{1}\leq d_{2}\leq \cdots \leq d_{N}$. We begin by summarizing the facts about group multiplication, invariant vector fields, and group convolution that we need in this context. Additional background information, and in particular the proofs of formulas (\ref{1.4}) and (\ref{1.6}) below, can be found in the first chapter of \cite{FoSt82}. 

The product on $G=\R^{N}$ is given by a polynomial mapping; if $\x= (x_{1}, \ldots, x_{N})$ and $\y=(y_{1}, \ldots, y_{N})$, the $k^{th}$ component of the product $\x\y$ is given by
\begin{equation}\label{1.4}
\begin{split}
(\x\y)_{k}&= x_{k}+y_{k} + M_{k}(\x,\y)
= x_{k}+y_{k} + \sum_{\alpha,\beta\in \mathcal M_{k}} c^{\alpha,\beta}_{k}x_{1}^{\alpha_{1}}\cdots x_{k-1}^{\alpha_{k-1}}y_{1}^{\beta_{1}}\cdots y_{k-1}^{\beta_{k-1}}
\end{split}
\end{equation}
where $\{c^{\alpha,\beta}_{k}\}$ are real constants, and 
$$\mathcal M_{k} = \big\{(\alpha;\beta) = (\alpha_{1}, \ldots, \alpha_{k-1}; \beta_{1},\ldots, \beta_{k-1})\,\Big\vert\,\sum_{l=1}^{k-1} d_{l}(\alpha_{l}+\beta_{l}) = d_{k}\big\}.$$
Note that $M_{k}(r\cdot\x, r\cdot\y)= r^{d_{k}}M_{k}(\x,\y)$. 



Next, let $\{X_{1}, \ldots, X_{N}\}$ and $\{Y_{1}, \ldots, Y_{N}\}$ be the left- and right-invariant vector fields on $G$ such that at the origin, $X_{k}=Y_{k}=\partial_{x_{k}}$. Then
\begin{equation}\label{1.6}
\begin{split}
X_{k} &= \frac{\partial}{\partial x_{k}} +\sum_{\substack{l=k+1\\d_{l}>d_{k}}}^{N}P_{k_{l}}(\x)\frac{\partial}{\partial x_{l}}
= \frac{\partial}{\partial x_{k}} + \sum_{\substack{l=k+1\\d_{l}>d_{k}}}^{N}  \sum_{\alpha\in \mathfrak H_{d_{l}-d_{k}}}a^{\alpha}_{k_{l}}\,x_{1}^{\alpha_{1}}\cdots x_{l-1}^{\alpha_{l-1}}\,\frac{\partial}{\partial x_{l}},\\
Y_{k} &= \frac{\partial}{\partial x_{k}} +\sum_{\substack{l=k+1\\d_{l}>d_{k}}}^{N}\widetilde P_{k_{l}}(\x)\frac{\partial}{\partial x_{l}}
=\frac{\partial}{\partial x_{k}} + \sum_{\substack{l=k+1\\d_{l}>d_{k}}}^{N} \sum_{\alpha\in \mathfrak H_{d_{l}-d_{k}}}\widetilde a^{\alpha}_{k_{l}}\,x_{1}^{\alpha_{1}}\cdots x_{l-1}^{\alpha_{l-1}} \,\frac{\partial}{\partial x_{l}},
\end{split}
\end{equation}
where $\{a^{\alpha}_{k_{l}}\}$ and $\{\widetilde a^{\alpha}_{k_{l}}\}$ are real constants, and the index set $\mathfrak H_{d}$ is defined in Proposition \ref{Prop2.1mn}. It follows that $P_{k_{l}}, \widetilde P_{k_{l}}\in \mathcal H_{d_{l}-d_{k}}$. 

The bi-invariant Haar measure on $G$ is Lebesgue measure $d\y=dy_{1}\cdots dy_{N}$.  The \textit{convolution} of functions $f, g \in L^{1}(G)$ is given by
\begin{equation*}
f*g(\x) = \int_{G}f(\x\y^{-1})g(\y)\,d\y =\int_{G}f(\y)g(\y^{-1}\x)\,d\y,
\end{equation*}
and the integral converges absolutely for almost all $x\in G$. The following result can be found on page 22 of \cite{FoSt82}.

\begin{proposition}\label{Proposition3.1} Let $f,g \in \mathcal C^{1}(G) \cap L^{1}(G)$. 

\smallskip

\begin{enumerate}[{\rm(1)}]

\item \label{Proposition3.1b} If $X$ is a left-invariant vector field and $Y$ is a right invariant vector field, then $X[f*g] = f*X[g]$ and $Y[f*g] = Y[f]*g$.

\smallskip

\item \label{Proposition3.1c}If $X$ is a left-invariant vector field and $Y=\widetilde X$ is the unique right-invariant vector field agreeing with $X$ at the origin, then $X[f]*g = f*Y[g]$.

\smallskip

\item \label{Proposition3.1d} If $\delta = \delta_{0}$ denotes the Dirac delta-function at the origin, then $\varphi(x) = \varphi * \delta(x) = \delta *\varphi(x)$ for $\varphi \in \mathcal C^{\infty}_{0}(G)$. In particular, if $X$ is a left-invariant vector field and $Y$ is a right invariant vector field, $X[\varphi] = \varphi*X[\delta]$ and $Y[\varphi] = Y[\delta] * \varphi$.
\end{enumerate}
\end{proposition}




%


%






%



Using the formulas in (\ref{1.4}), we can write the convolution of integrable functions $f$ and $g$ as
 \begin{equation}\label{5.1}
\begin{split}
f*g(\x) 
&=
\int_{\R^{N}}f(\ldots, x_{m}-y_{m}-P_{m}(\x,\y), \ldots)\,g(\ldots, y_{m}, \ldots)\,dy_{1}\cdots dy_{N}
\end{split}
\end{equation}
where each $P_{m}$ is a polynomial in the $2m-2$ variables $\{x_{1}, \ldots, x_{m-1}, y_{1}, \ldots, y_{m-1}\}$ satisfying $P_{m}(2^{K}\cdot \x, 2^{K}\cdot\y) = 2^{k_{m}d_{m}}P_{m}(\x,\y)$.  In the formula (\ref{5.1}), the variables in $\x$ appear in the argument of $f$. However by a change of variables we can move some or all of them to the argument of $g$. Thus if $S$ is any subset of $\{1, \ldots, N\}$, we can write
\begin{equation}\label{5.2}
\begin{split}
&f*g(\x) = 
\int_{\R^{N}}f\big(u_{1}(\x,\y), \ldots, u_{N}(\x,\y)\big)\,g\big(v_{1}(\x,\y), \ldots, v_{N}(\x,\y)\big)\,dy_{1}\cdots dy_{N}
\end{split}
\end{equation}
where
\begin{equation}\label{5.3}
\begin{split}
u_{m}(\x,\y) &=
\begin{cases}
x_{m}-y_{m}-Q_{m}(\x,\y) &\text{if $m\in S$,}\\
y_{m} &\text{if $m\notin S$,}
\end{cases}\\
v_{m}(\x,\y) &=
\begin{cases}
y_{m} &\text{if $m\in S$,}\\
x_{m}-y_{m}-Q_{m}(\x,\y) &\text{if $m\notin S$.}
\end{cases}
\end{split}
\end{equation}
Here each $Q_{m} = Q_{m}^{S}$ is a polynomial in the variables  $\{x_{1}, \ldots, x_{m-1}, y_{1}, \ldots, y_{m-1}\}$ with the same homogeneity as $P_{m}$; that is  $Q_{m}(2^{K}\cdot \x, 2^{K}\cdot\y) = 2^{k_{m}d_{m}}Q_{m}(\x,\y)$.

\subsection{Support properties of convolutions $[\varphi]_{I}*[\psi]_{J}$}\label{SupportProperties}\quad

\smallskip

In this section we study the support properties of the convolution of dilates of normalized bump functions with compact support. Given integers $i,j\in \Z$, we set $i\vee j = \max\{i,j\}$. Given $N$-tuples $I=(i_{1}, \ldots, i_{N}),\,J=(j_{1}, \ldots, j_{N}) \in \Z^{N}$, we set 
\begin{equation}
I\vee J =(i_{1}\vee j_{1}, \ldots, i_{N}\vee j_{N}).
\end{equation}
We want to show that the convolution  $[\varphi]_{I}*[\psi]_{J}$ is the $I\vee J$-dilate of a normalized function. 

\smallskip 

\begin{lemma} \label{Lemma5.1}Let $\varphi, \psi \in \mathcal C^{\infty}_{0}(\R^{N})$ have support in the ball $B(\rho)$. Then for any $I,J\in E_{N}$ there exists $\theta \in \mathcal C^{\infty}_{0}(\R^{N})$ supported in the ball $B(C\rho)$ such that $[\varphi]_{I}*[\psi]_{J} = [\theta]_{I\vee J}$, and $||\theta||_{(m)}\leq C_{m}||\varphi||_{(m)}||\psi||_{(m)}$. The constants $C$ and $\{C_{M}\}$ can depend on the radius $\rho$, but are independent of the functions $\varphi$ and $\psi$.
\end{lemma}

\begin{proof}

Let $K=I\vee J$ and put
\begin{equation*}
\theta = \big[[\varphi]_{I}*[\psi]_{J}\big]_{-K}.
\end{equation*}
It suffices to show that $\theta$ is supported in the set $\{\x\in \R^{N}\,\big\vert\,|x_{k}|\leq C\rho,\,1\leq k \leq N\}$  and that $||\theta||_{(m)}\leq C_{m}\,||\varphi||_{(m)}||\psi||_{(m)}$ for some absolute constants $C$ and $\{C_{m}\}$.  
Making the change of variables $y_{m}\to 2^{k_{m}d_{m}}y_{m}$ and using the homogeneity of the functions $\{u_{m}\}$ and $\{v_{m}\}$, we have
\begin{align*}
&\big[[\varphi]_{I}*[\psi_{J}]\big]_{-K}(\x)\\
&= 
2^{\sum_{m} d_{m}k_{m}}([\varphi]_{I}*[\psi]_{J})(2^{d_{1}k_{1}}x_{1}, \ldots, 2^{d_{N}k_{N}}x_{N})\\
&= 
2^{\sum_{m}2d_{m}k_{m}}
\int_{\R^{N}}[\varphi]_{I}\big(\ldots, u_{m}(2^{d_{1}k_{1}}x_{1}, \ldots, 2^{d_{N}k_{N}}x_{N},\y), \ldots\big)\\
&\qquad\qquad \qquad\qquad\qquad
[\psi]_{J}\big(\ldots, v_{m}(2^{d_{1}k_{1}}x_{1}, \ldots, 2^{d_{N}k_{N}}x_{N},\y), \ldots \big)dy_{1}\cdots dy_{N}
\\
&=
2^{\sum_{m}d_{m}(2k_{m}-i_{m}-j_{m})}
\int_{\R^{N}}\varphi\big(\ldots, 2^{d_{m}(k_{m}-i_{m})}u_{m}(\x,\y), \ldots\big)
\\
&\qquad\qquad \qquad\qquad\qquad\qquad\qquad\psi\big(\ldots, 2^{d_{m}(k_{m}-j_{m})}v_{m}(\x,\y), \ldots \big)dy_{1}\cdots dy_{N}.
\end{align*}
Note that  $2^{d_{m}(k_{m}-i_{m})}\geq 1$ and $2^{d_{m}(k_{m}-j_{m})}\geq 1$. It follows that if $[\varphi_{I}*\psi_{J}]_{-K}(\x)\neq 0$, there exists $\y = (\y_{1}, \ldots, \y_{N}) \in \R^{N}$ so that for $1 \leq m \leq N$,
\begin{equation}\label{5.4}
\begin{split}
|u_{m}(\x,\y)|\leq  2^{d_{m}(k_{m}-i_{m})}|u_{m}(\x,\y)|&\leq \rho,\\
|v_{m}(\x,\y)|\leq 2^{d_{m}(k_{m}-j_{m})}|v_{m}(\x,\y)|&\leq \rho.
\end{split}
\end{equation}

We show by induction on $m$ that these inequalities imply that $|x_{m}|+|y_{m}| \leq A_{m}\rho$ for an appropriate choice of constants $A_{1} < A_{2}<  \ldots < A_{N}$. 
When $m=1$, we have $|x_{1}-y_{1}|\leq \rho$ and $|y_{1}|\leq \rho$, so 
$|x_{1}|+|y_{1}|\leq 3\rho$, and we can take $A_{1}=2$. 
Next, assume by induction that $|x_{s}|+|y_{s}|\leq A_{s}\rho$ for $1 \leq s <  m$. Since $Q_{m}(\x,\y)$ depends only on the variables $\{x_{1}, \ldots, x_{m-1},y_{1}, \ldots, y_{m-1}\}$, it follows that $|Q_{m}(\x,\y)|\leq B_{m}\rho$ where $B_{m}$ is a constant that depends on the coefficients of the polynomial $Q_{m}$, on the constants $\{A_{s}\}$ for $s<m$, and on $\rho$. We have $|y_{m}|\leq \rho$ and $|x_{m}-y_{m}-Q_{m}(\x,\y)|\leq \rho$, so $|x_{m}|+|y_{m}|\leq (B_{m}+2)\rho$.
This completes the proof of the statement about the support of $\theta$.
\smallskip

To establish the estimate $||\theta||_{(m)}\leq C_{m}\,||\varphi||_{(m)}||\psi||_{(m)}$, we again use formula (\ref{5.2}), but this time with the set $S = \{m\in \{1, \ldots, N\}\,\big\vert\, j_{m}\leq i_{m}= k_{m}\}$. Of the two factors $\big\{2^{d_{m}(k_{m}-i_{m})}, 2^{d_{m}(k_{m}-j_{m})}\big\}$, the one which equals $1$ multiplies the expression $x_{m}-y_{m}-Q_{m}(\x,\y)$, while the term $y_{m}$ is multiplied by the larger factor $2^{d_{m}(k_{m}-(i_{m}\wedge j_{m}))}$. Thus with this representation of the convolution $\varphi_{I}*\psi_{J}$, the integration takes place over the set $E=\{y\in\R^{N}\,\big\vert\,|y_{m}|\leq 2^{-d_{m}(k_{m}-(i_{m}\wedge j_{m}))}\}$. Thus we can estimate the size of $\big[\![\varphi]_{I}*[\psi]_{J}\big]_{-K}(\x)$ by
\begin{align*}
|[\varphi_{I}*\psi_{J}]_{-K}(\x)|&\leq 2^{\sum_{m}d_{m}(2k_{m}-i_{m}-j_{m})}||\varphi||_{0}||\psi||_{0}
\int_{E}d\y\leq C_{0}||\varphi||_{0}||\psi||_{0}
\end{align*}
since $2k_{m}-i_{m}-j_{m}-k_{m}+(i_{m}\wedge j_{m}) =0$. When we take derivatives of $\big[\![\varphi]_{I}*[\psi]_{J}\big]_{-K}(\x)$, the terms involving the variables $\x$ are multiplied by the factor $1$, and so  we obtain in the same way the estimate
\begin{align*}
||[\varphi_{I}*\psi_{J}]_{-K}||_{(m)}\leq C_{m}||\varphi||_{(m)}||\psi||_{(m)}.
\end{align*}
This completes the proof.
\end{proof}

\subsection{Decay and cancellation properties of convolutions $[\varphi]_{I}*[\psi]_{J}$} \label{Decay}\quad

\smallskip

We want to study the decay and cancellation properties of the convolution $[\varphi]_{I}*[\psi]_{J}$ under the assumption that $\varphi$ has cancellation in the variables $\{x_{l_{1}}, \ldots, x_{l_{a}}\}$ and $\psi$ has cancellation in the variables $\{x_{m_{1}}, \ldots, x_{m_{b}}\}$. Here \emph{decay} means that the size of $[\varphi]_{I}*[\psi]_{J}$ is small due to the difference between the $N$-tuples $I$ and $J$; \emph{cancellation} means that  the integral of $[\varphi]_{I}*[\psi]_{J}$  with respect to some variables is zero. (See Section \ref{StrongWeak} and Definition \ref{Def1.14} for the precise definition of strong and weak cancellation). Before stating our results, let us see what we should expect by considering the much simpler case in which the convolution $[\varphi]_{I}*_{e}[\psi]_{J}$ is taken with respect to the Abelian (Euclidean) vector space structure of $\R^{N}$ rather than the general homogeneous nilpotent Lie group structure $G$. 

Let $I=(i_{1}, \ldots, i_{N})$ and $J=(j_{1}, \ldots, j_{N})$, and put
\begin{align*}
A_{0}&=\big\{s\in\{1, \ldots, a\}\,\big\vert\,i_{l_{s}}\leq j_{l_{s}}\big\},\\
A_{1}&=\big\{s\in\{1, \ldots, a\}\,\big\vert\,i_{l_{s}} > j_{l_{s}}\big\} = \{1, \ldots, a\}\setminus A_{0}, \\
B_{0}&=\big\{t\in\{1, \ldots, b\}\,\big\vert\,j_{m_{t}}\leq i_{m_{t}}\big\}, \\
B_{1}&=\big\{t\in\{1, \ldots, b\}\,\big\vert\,j_{m_{t}}> i_{m_{t}}\big\} = \{m_{1}, \ldots, m_{b}\}\setminus B_{0}.
\end{align*}
Because of the hypothesis on cancellation, we can write
\begin{align*}
\varphi &= \partial_{l_{1}}\cdots \partial_{l_{a}}\widetilde\varphi,\\
\psi&= \partial_{m_{1}}\cdots \partial_{m_{b}}\widetilde\psi.
\end{align*}
For each $s\in A_{0}$ we can integrate by parts in the variable $x_{l_{s}}$ in the integral $[\varphi]_{I}*_{e}[\psi]_{J}$, moving the derivative $\partial_{l_{s}}$ from $\widetilde\varphi$ to $\widetilde\psi$. Since the width of the dilate $[\varphi]_{I}$ is narrower in this variable than the dilate $[\psi]_{J}$, this integration by parts gives a gain of $2^{-\epsilon(j_{l_{s}}-i_{l_{s}})}$, and we get such a gain for each $s\in A_{0}$. A similar argument shows that we get a gain of $2^{-\epsilon(i_{m_{t}}-j_{m_{t}})}$ for each $t\in B_{0}$. Thus the total gain from integration by parts is $\prod_{s\in A_{0}}2^{-\epsilon(j_{l_{s}}-i_{l_{s}})}\,\prod_{t\in B_{0}}2^{-\epsilon(i_{m_{t}}-j_{m_{t}})}$. In addition to this gain, we observe that in the convolution $[\varphi]_{I}*_{e}[\psi]_{J}$, the derivatives $\partial_{l_{s}}$ for $s\in A_{1}$ and $\partial_{m_{t}}$ for $t\in B_{1}$ can be pulled outside the integral. The final result is that there is a compactly supported function $\theta$, normalized relative to $\varphi$ and $\psi$, so that
\begin{equation*}
[\varphi]_{I}*_{e}[\psi]_{J}= \prod_{s\in A_{0}}2^{-\epsilon(j_{l_{s}}-i_{l_{s}})}\,\prod_{t\in B_{0}}2^{-\epsilon(i_{m_{t}}-j_{m_{t}})}\prod_{s'\in A_{1}}\prod_{t'\in B_{1}}\partial_{l_{s'}}\partial_{m_{t'}}[\theta]_{I\vee J}.
\end{equation*}
In other words, we get exponential gains from variables where there is cancellation for the function with `narrower' dilation, and the resulting convolution still has cancellation in the remaining variables. 

When dealing with convolution on a homogeneous nilpotent Lie group, we cannot move Euclidean derivatives from one factor to the other. However, we can write Euclidean derivatives in terms of left- or right-invariant vector fields which can be moved across the convolution. But this process introduces error terms involving derivatives with respect to `higher' variables, and these come with a gain involving the differences between entries of $I$ or $J$. Thus in the case of nilpotent Lie groups, we might hope that convolution results in three kinds of terms: gains of the form $2^{-\epsilon|i_{\ell}-j_{\ell}|}$ coming from integration by parts in narrow variables with cancellation, residual cancellation of the convolution in some variables which are not used in the integration by parts, and finally gains of the type $2^{-\epsilon(i_{\ell+1}-i_{\ell})}$ and $2^{-\epsilon(j_{\ell+1}-j_{\ell})}$. This is in fact the case, and is made precise in the next Lemma.

\medskip

\medskip

Suppose we are given two decompositions
\begin{align*}
(\mathcal A):\quad \R^{N}&= \R^{a_{1}}\oplus \cdots \oplus\R^{a_{n}},\\
(\mathcal B):\quad \R^{N}&= \R^{b_{1}}\oplus \cdots \oplus\R^{b_{m}}.
\end{align*}
Let $\{J^{{\mathcal A}}_{1}, \ldots, J^{{\mathcal A}}_{n}\}$ be the indices corresponding to (${\mathcal A}$) and let $\{J^{{\mathcal B}}_{1}, \ldots, J^{{\mathcal B}}_{m}\}$ be the indices corresponding to (I). Define 
\begin{align*}
\sigma&:\{1, \ldots, N\}\to\{1, \ldots,n\}\quad\text{such that \quad $\ell \in J^{{\mathcal A}}_{\sigma(\ell)}$},\\
\tau&:\{1, \ldots, N\}\to\{1, \ldots,m\}\quad\text{such that \quad $l \in J^{{\mathcal B}}_{\tau(l)}$}.
\end{align*}
In what follows, $\pi_{{\mathcal A}}$  and $\pi_{{\mathcal B}}$ denote mappings from the set $\{1, \ldots, N\}$ to itself with the property that $\pi_{{\mathcal A}}(\ell) \in J^{{\mathcal A}}_{\sigma(\ell)}$, and  $\pi_{{\mathcal B}}(l) \in J^{{\mathcal B}}_{\tau(l)}$. Also recall from equation (\ref{E2.18asd}) in Section \ref{Section4.1} that we can introduce mappings $p_{\mathcal A}:\mathcal E_{n}\to E_{N}$ and $p_{\mathcal B}:\mathcal E_{m}\to E_{N}$ so that
\begin{align*}
p_{\mathcal A}(i_{1}, \ldots, i_{n}) &= \big(\,\overset{a_{1}}{\overbrace{i_{1}, \ldots,i_{1}}}\,,\,\overset{a_{2}}{\overbrace{i_{2}, \ldots,i_{2}}}\,, \ldots , \,\overset{a_{n}}{\overbrace{i_{n}, \ldots,i_{n}}}\,\big),\\
p_{\mathcal B}(j_{1}, \ldots, j_{m}) &= \big(\,\overset{b_{1}}{\overbrace{j_{1}, \ldots,j_{1}}}\,,\,\overset{b_{2}}{\overbrace{j_{2}, \ldots,j_{2}}}\,, \ldots , \,\overset{b_{m}}{\overbrace{j_{m}, \ldots,j_{m}}}\,\big).
\end{align*}

\begin{lemma}\label{Lemma4.3qw}
Suppose that $\varphi\in  \mathcal C^{\infty}_{0}(\R^{N})$ has cancellation in the variables $x_{\ell}$ for $\ell\in A\subset \{1, \ldots, N\}$, and that $\psi\in  \mathcal C^{\infty}_{0}(\R^{N})$ has cancellation in the variables $x_{l}$ for $l\in B\subset \{1, \ldots, N\}$. Let $I=(i_{1}, \ldots, i_{n})\in \mathcal E_{n}$ and $J=(j_{1}, \ldots, j_{m})\in \mathcal E_{m}$. Set
\begin{align*}
A_{0} &= \big\{\ell\in A\,\big\vert\, i_{\sigma(\ell)}\leq j_{\tau(\ell)}\big\},\\
B_{0} &= \big\{l\in B\,\big\vert\, j_{\tau(l)}\leq i_{\sigma(l)}\big\}.
\end{align*}
Then $[\varphi]_{I}*[\psi]_{J}$ can be written as a sum of terms of the form
\begin{align*}
\prod_{\ell\in A_{1}}2^{-\epsilon(i_{\sigma(\ell)+1}-i_{\sigma(\ell)})}
&\prod_{l\in B_{1}}2^{-\epsilon(i_{\tau(l)+1}-i_{\tau(l)})}
\prod_{\ell\in A_{2}}2^{-\epsilon(j_{\tau(\ell)}-i_{\sigma(\ell)})}
\prod_{l\in B_{2}}2^{-\epsilon(i_{\sigma(l)}-j_{\tau(l)})}\\
&\big[\prod_{\ell\in A_{3}}\partial_{\pi_{\mathcal A}(\ell)}
\prod_{l\in B_{3}}\partial_{\pi_{\mathcal B}(l)}\theta\big]_{p_{\mathcal A}(I)\vee p_{\mathcal B}(J)}
\end{align*}
where $\theta$ is normalized relative to $\varphi$ and $\psi$, $A= A_{1}\cup A_{2}\cup A_{3}$ and $B=B_{1}\cup B_{2}\cup B_{3}$ are disjoint unions, and  we have $A_{2}\subset A_{0}\subset A_{1}\cup A_{2}$ and  $B_{2}\subset B_{0}\subset B_{1}\cup B_{2}$. (We will have $\sigma(\ell) \neq n$ and $\tau(l)\neq n$).
\end{lemma}

\begin{proof}
Using the cancellation hypotheses, it follows from Lemma \ref{Lemma1.12} that we can write
\begin{align*}
\varphi &= \big(\prod_{\ell\in A}\partial_{x_{\ell}}\big)[\varphi_{A}], & 
[\varphi]_{I}&=\Big(\prod_{\ell\in A}\big(2^{d_{\ell}i_{\sigma(\ell)}}\partial_{x_{\ell}}\big)\Big)[\varphi_{A}]_{I},\\
\psi &= \big(\prod_{\ell\in B}\partial_{x_{\ell}}\big)[\psi_{B}], & 
[\psi]_{J}&=\Big(\prod_{\ell\in B}\big(2^{d_{\ell}j_{\tau(\ell)}}\partial_{x_{\ell}}\big)\Big)[\psi_{B}]_{J},
\end{align*}
where $\varphi_{A}$ is normalized relative to $\varphi$, and $\psi_{B}$ is normalized relative to $\psi$. We can use Corollary \ref{Prop2.15jkl}  to write $[\varphi]_{I}$ as finite sums of terms of the form
\begin{align*}
\Big(\prod_{\ell\in \widetilde A_{1}}2^{-\epsilon(i_{\sigma(\ell)+1}-i_{\sigma(\ell)})}\Big) \prod_{\ell\in A_{1}}(2^{d_{\pi_{\mathcal A}(\ell)}i_{\sigma(\ell)}}Z_{\pi_{\mathcal A}(\ell)})[\varphi_{A_{1}}]_{I}
\end{align*}
where ${A}_{1}\subset A$ is a possibly empty subset, $\widetilde {A}_{1}= {A}\setminus {A}_{1}$, each $Z_{\pi_{\mathcal A}(\ell)}$ is either the corresponding left- or  right-invariant vector field, and $\varphi_{{A}_{1}}$ is normalized relative to $\varphi$. Moreover, according to the Remarks \ref{Remarks4.9} following Corollary \ref{Prop2.15jkl}, the operators $\{Z_{\pi_{\mathcal A}(\ell)}\}$ can be put in any desired order. Similarly 
 $[\psi]_{J}$  is a finite sum of terms of the form
\begin{align*}
\Big(\prod_{l\in \widetilde {B}_{1}}2^{-\epsilon(j_{\tau(l)+1}-j_{\tau(l)})}\Big) \prod_{l\in {B}_{1}}(2^{d_{\pi_{{\mathcal B}}(l)}j_{\tau(l)}}Z_{\pi_{{\mathcal B}}(l)})[\varphi_{{B}_{1}}]_{J}.
\end{align*}
It follows that $[\varphi]_{I}*[\psi]_{J}$ is a finite sum of terms of the form
\begin{align*}
\prod_{\ell\in \widetilde {A}_{1}}2^{-\epsilon(i_{\sigma(\ell)+1}-i_{\sigma(\ell)})}&
\prod_{l\in \widetilde {B}_{1}}2^{-\epsilon(j_{\tau(l)+1}-j_{\tau(l)})}\\
&
\Big(\prod_{\ell\in {A}_{1}}(2^{d_{\pi_{\mathcal A}(\ell)}i_{\sigma(\ell)}}Z_{\pi_{\mathcal A}(\ell)})[\varphi_{{A}_{1}}]_{I}
*
\prod_{l\in {B}_{1}}(2^{d_{\pi_{{\mathcal B}}(l)}j_{\tau(l)}}Z_{\pi_{{\mathcal B}}(l)})[\varphi_{{B}_{1}}]_{J}\Big).
\end{align*}
Now let
\begin{align*}
{A}_{2}&=\{\ell \in {A}_{1}\,\big\vert\, i_{\sigma(\ell)}\leq j_{\tau(\ell)}\} =  A_{1}\cap A_{0},&
\widetilde {A}_{2}&=\{\ell \in {A}_{1}\,\big\vert\, i_{\sigma(\ell)}> j_{\tau(\ell)}\} = {A}_{1}\setminus {A}_{2},\\
{ B}_{2}&=\{\ell \in {B}_{1}\,\big\vert\, j_{\tau(l)}\leq i_{\sigma(l)}\}= B_{1}\cap B_{0},&
\widetilde {B}_{2}^{\prime}&=\{\ell \in {B}_{1}\,\big\vert\, j_{\tau(l)}> i_{\sigma(l)}\} = {B}_{1}\setminus {B}_{2}.
\end{align*}
We choose $Z_{\ell}$ and $Z_{l}$ as follows:
\begin{align*}
Z_{\ell}&=
\begin{cases}
L_{\ell} &\text{if $\ell \in { A}_{2}$}\\
R_{\ell} &\text{if $\ell \in \widetilde { A}_{2}$}
\end{cases}, &
Z_{l}&=
\begin{cases}
L_{l} &\text{if $l \in { B}_{2}$}\\
R_{l} &\text{if $l \in \widetilde { B}_{2}$}
\end{cases}.
\end{align*}
Then since left-invariant vector fields commute with right invariant vector fields, we can use Proposition \ref{Proposition3.1}, part (\ref{Proposition3.1b}) and then Proposition \ref{Proposition3.1}, part (\ref{Proposition3.1c}) to write 
\begin{align*}
\Big(\prod_{\ell\in { A}_{1}}&(2^{d_{\pi_{\mathcal A}(\ell)}i_{\sigma(\ell)}}Z_{\pi_{\mathcal A}(\ell)})[\varphi_{{ A}_{1}}]_{I}
*
\prod_{l\in { B}_{1}}(2^{d_{\pi_{{\mathcal B}}(l)}j_{\tau(l)}}Z_{\pi_{{\mathcal B}}(l)})[\varphi_{{ B}_{1}}]_{J}\Big)\\
&=
\prod_{\ell \in \widetilde { A}_{2}}(2^{d_{\pi_{\mathcal A}(\ell)}i_{\sigma(\ell)}}R_{\pi_{\mathcal A}(\ell)})
\prod_{l \in \widetilde { B}_{2}}(2^{d_{\pi_{{\mathcal B}}(l)}j_{\tau(l)}}L_{\pi_{{\mathcal B}}(l)})\\
&\qquad\qquad\qquad
\Big(\prod_{\ell\in { A}_{2}}(2^{d_{\pi_{\mathcal A}(\ell)}i_{\sigma(\ell)}}L_{\pi_{\mathcal A}(\ell)})[\varphi_{{ A}_{1}}]_{I}
*
\prod_{l\in { B}_{2}}(2^{d_{\pi_{{\mathcal B}}(l)}j_{\tau(l)}}R_{l})[\psi_{{ B}_{1}}]_{J}\Big)\\
&=
\prod_{\ell \in \widetilde { A}_{2}}(2^{d_{\pi_{\mathcal A}(\ell)}i_{\sigma(\ell)}}R_{\ell})
\prod_{l \in \widetilde { B}_{2}}(2^{d_{\pi_{{\mathcal B}}(l)}j_{\tau(l)}}L_{\pi_{{\mathcal B}}(l)})\\
&\qquad\qquad\qquad
\Big([\varphi_{{ A}_{1}}]_{I}
*
\prod_{\ell\in { A}_{2}}(2^{d_{\pi_{\mathcal A}(\ell)}i_{\sigma(\ell)}}L_{\pi_{\mathcal A}(\ell)})
\prod_{l\in { B}_{2}}(2^{d_{\pi_{{\mathcal B}}(l)}j_{\tau(l)}}R_{\pi_{{\mathcal B}}(l)})
[\psi_{{ B}_{1}}]_{J}\Big),
\end{align*}
where $\varphi_{{ A}_{1}}$ and $\psi_{{ B}_{1}}$ are normalized\footnote{The function $\varphi_{{ A}_{1}}$ also depends $A$  and $\psi_{{ B}_{1}}$ also depends on $B$.  With minimal risk of confusion, we shall omit such notation.} relative to $\varphi$ and $\psi$. (Here $\prod_{\ell \in { A}_{1}^{\prime}}$ is actually the product of the operators in the reverse order).
Now we want to commute the operators $\prod_{\ell\in { A}_{2}}(2^{d_{\pi_{\mathcal A}(\ell)}i_{\sigma(\ell)}}L_{\ell})$ and $\prod_{l\in { B}_{2}}(2^{d_{\pi_{{\mathcal B}}(l)}j_{\tau(l)}}R_{l})$ before applying them  to  $[\varphi_{{ B}_{1}}]_{J}$. According to the third of the Remarks \ref{Remarks4.9}, the result is a sum of terms of the form
\begin{align*}
\prod_{\ell\in \widetilde { A}_{3}}2^{-\epsilon(i_{\sigma(\ell)+1}-i_{\sigma(\ell)})}
\prod_{l\in \widetilde { B}_{3}}2^{-\epsilon(j_{\tau(l)+1}-j_{\tau(l)})}
\prod_{l\in { B}_{3}}(2^{d_{\pi_{{\mathcal B}}(l)}j_{\tau(l)}}R_{l})
\prod_{\ell\in { A}_{3}}(2^{d_{\pi_{\mathcal A}(\ell)}i_{\sigma(\ell)}}R_{\ell})[\psi_{{ B}_{3}}]
\end{align*}
where 
\begin{align*}
{ A}_{3}&\subset { A}_{2}^{\prime}, &
\widetilde { A}_{3}&= { A}_{2}^{\prime}\setminus { A}_{3},\\
{ B}_{3}&\subset { B}_{2}^{\prime}, &
\widetilde { B}_{3}&= { B}_{2}^{\prime}\setminus { B}_{3}.
\end{align*}
Thus $[\varphi_{{ A}_{1}}]_{I}*\prod_{\ell\in { A}_{2}}(2^{d_{\pi_{\mathcal A}(\ell)}i_{\sigma(\ell)}}L_{\ell}) \prod_{l\in { B}_{2}}(2^{d_{\pi_{{\mathcal B}}(l)}j_{\tau(l)}}R_{l})[\psi_{{ B}_{1}}]_{J}$ is a sum of terms of the form
\begin{align*}
&\prod_{\ell\in \widetilde { A}_{3}}2^{-\epsilon(i_{\sigma(\ell)+1}-i_{\sigma(\ell)})}
\prod_{l\in \widetilde { B}_{3}}2^{-\epsilon(j_{\tau(l)+1}-j_{\tau(l)})}\\
&\qquad\qquad\qquad
\Big([\varphi_{{ A}_{1}}]_{I}*
\prod_{l \in { B}_{3}}(2^{d_{\pi_{{\mathcal B}}(l)}j_{\tau(l)}}R_{l})
\prod_{\ell\in { A}_{3}}(2^{d_{\pi_{\mathcal A}(\ell)}i_{\sigma(\ell)}}R_{\ell})[\psi_{{ B}_{3}}]\Big)\\
&=
\prod_{\ell\in \widetilde { A}_{3}}2^{-\epsilon(i_{\sigma(\ell)+1}-i_{\sigma(\ell)})}
\prod_{l\in \widetilde { B}_{1}^{\prime\prime}}2^{-\epsilon(j_{\tau(l)+1}-j_{\tau(l)})}\\
&\qquad\qquad\qquad
\Big( \prod_{l\in { B}_{3}}(2^{d_{\pi_{{\mathcal B}}(l)}j_{\tau(l)}}R_{l})[\varphi_{{ A}_{1}}]_{I}*
 \prod_{\ell\in { A}_{3}}(2^{d_{\pi_{\mathcal A}(\ell)}i_{\sigma(\ell)}}R_{\ell})[\psi_{{ B}_{3}}]\Big)\\
&=
\prod_{\ell\in \widetilde { A}_{3}}2^{-\epsilon(i_{\sigma(\ell)+1}-i_{\sigma(\ell)})}
\prod_{l\in \widetilde { B}_{3}}2^{-\epsilon(j_{\tau(l)+1}-j_{\tau(l)})}
\prod_{l\in  { B}_{3}}2^{-d_{\pi_{{\mathcal B}}(l)}(i_{\tau(l)}-j_{\tau(l)})}
\prod_{\ell\in { A}_{3}}2^{-d_{\pi_{\mathcal A}(\ell)}(j_{\sigma(\ell)}-i_{\sigma(\ell)})}\\
&\qquad\qquad\qquad
\Big( \prod_{\widetilde { B}_{3}}(2^{d_{\pi_{{\mathcal B}}(l)}i_{\tau(l)}}R_{l})[\varphi_{{ A}_{1}}]_{I}*
 \prod_{\ell\in { A}_{3}}(2^{d_{\pi_{\mathcal A}(\ell)}j_{\sigma(\ell)}}R_{\ell})[\psi_{{ B}_{3}}]\Big)\\
&=
\prod_{\ell\in \widetilde { A}_{3}}2^{-\epsilon(i_{\sigma(\ell)+1}-i_{\sigma(\ell)})}
\prod_{l\in \widetilde { B}_{3}}2^{-\epsilon(j_{\tau(l)+1}-j_{\tau(l)})}
\prod_{l\in { B}_{3}}2^{-d_{\pi_{{\mathcal B}}(l)}(i_{\tau(l)}-j_{\tau(l)})}
\prod_{\ell\in { A}_{3}}2^{-d_{\pi_{\mathcal A}(\ell)}(j_{\sigma(\ell)}-i_{\sigma(\ell)})}\\
&\qquad\qquad\qquad
\Big(\big[ \prod_{l\in { B}_{3}}R_{l}\varphi_{{ A}_{1}}\big]_{I}*
\big[ \prod_{\ell\in { A}_{3}}R_{\ell}\psi_{{ B}_{3}}\big]\Big).
\end{align*}
Now according to Lemma \ref{Lemma5.1}, we can write
\begin{equation*}
\big[\prod_{\widetilde { B}_{3}}R_{l}\varphi_{{ A}_{1}}\big]_{I}*
\big[\prod_{\ell\in { A}_{1}^{\prime\prime}}R_{\ell}\psi_{{ B}_{3}}\big]
=
[\theta_{{ A}_{1}, { B}_{3}}]_{K}
\end{equation*}
where $\theta_{{ A}_{1}, { B}_{3}}$ is normalized relative to $\varphi$ and $\psi$ and $K = I \vee J$. Thus it follows that $[\varphi]_{I}*[\psi]_{J}$ is a finite sum of terms of the form
\begin{align*}
\prod_{\ell\in \widetilde { A}_{1}\cup \widetilde { A}_{3}}2^{-\epsilon(i_{\sigma(\ell)+1}-i_{\sigma(\ell)})}&\!\!\!\!
\prod_{l\in \widetilde { B}_{1}\cup \widetilde { B}_{3}}2^{-\epsilon(j_{\tau(l)+1}-j_{\tau(l)})}
\prod_{\ell\in { A}_{3}}2^{-d_{\pi_{\mathcal A}(\ell)}(j_{\sigma(\ell)}-i_{\sigma(\ell)})}
\prod_{l\in  { B}_{3}}2^{-d_{\pi_{{\mathcal B}}(l)}(i_{\tau(l)}-j_{\tau(l)})}\\
&
\prod_{\ell \in \widetilde { A}_{2}}(2^{d_{\pi_{\mathcal A}(\ell)}i_{\sigma(\ell)}}R_{\pi_{\mathcal A}(\ell)})
\prod_{l \in \widetilde { B}_{2}}(2^{d_{\pi_{{\mathcal B}}(l)}j_{\tau(l)}}L_{\pi_{{\mathcal B}}(l)})[\theta_{{ A}_{1}, { B}_{3}}]_{I\vee J}.
\end{align*}
However, it follows from part (\ref{Prop2.15jkl1}) of Corollary \ref{Prop2.15jkl} that we can write  the product of vector fields $\prod_{\ell \in \widetilde { A}_{2}}(2^{d_{\pi_{\mathcal A}(\ell)}i_{\sigma(\ell)}}R_{\ell})
\prod_{l \in \widetilde { B}_{2}}(2^{d_{\pi_{{\mathcal B}}(l)}j_{\tau(l)}}L_{l})[\theta_{{ A}_{1}, { B}_{3}}]_{I\vee J}
$ as a sum of terms of the form
\begin{align*}
\prod_{\ell\in \widetilde { A}_{4}}2^{-\epsilon(i_{\sigma(\ell)+1}-i_{\sigma(\ell)})}
\prod_{l\in \widetilde { B}_{4}}2^{-\epsilon(j_{\tau(l)+1}-j_{\tau(l)})}
\prod_{\ell \in  { A}_{4}}(2^{d_{\pi_{\mathcal A}(\ell)}i_{\sigma(\ell)}}\partial{\pi_{\mathcal A}(\ell)})
\prod_{l \in { B}_{4}}(2^{d_{\pi_{{\mathcal B}}(l)}j_{\tau(l)}}\partial_{\pi_{{\mathcal B}}(l)})[\theta_{{ A}_{4}, { B}_{4}}]_{I\vee J}.
\end{align*}
Thus we have shown that $[\varphi]_{I}*[\psi]_{J}$ is a sum of terms of the form
\begin{align*}
\prod_{\ell\in \widetilde { A}_{1}\cup \widetilde { A}_{3}\cup A_{4}}\!\!\! 2^{-\epsilon(i_{\sigma(\ell)+1}-i_{\sigma(\ell)})}\!\!\!&
\prod_{l\in \widetilde { B}_{1}\cup \widetilde { B}_{3}\cup B_{4}}\!\!\! 2^{-\epsilon(j_{\tau(l)+1}-j_{\tau(l)})}
\prod_{\ell\in { A}_{3}} 2^{-d_{\pi_{\mathcal A}(\ell)}(j_{\sigma(\ell)}-i_{\sigma(\ell)})}
\prod_{l\in  { B}_{3}}2^{-d_{\pi_{{\mathcal B}}(l)}(i_{\tau(l)}-j_{\tau(l)})}\\
&
\prod_{\ell \in { A}_{4}}(2^{d_{\pi_{\mathcal A}(\ell)}i_{\sigma(\ell)}}R_{\pi_{\mathcal A}(\ell)})
\prod_{l \in { B}_{4}}(2^{d_{\pi_{{\mathcal B}}(l)}j_{\tau(l)}}L_{\pi_{{\mathcal B}}(l)})[\theta_{{ A}_{4}, { B}_{4}}]_{I\vee J}.
\end{align*}
This has the form asserted by the Lemma, and so completes the proof.
\end{proof}

\smallskip

\begin{corollary}\label{Cor7.4}
Let $\varphi, \psi \in \mathcal C^{\infty}_{0}(\R^{N})$  have strong cancellation relative to the same decomposition $\R^{N}= \R^{a_{1}}\oplus \cdots\oplus \R^{a_{n}}$. There exists $\epsilon>0$ so that if $I, J\in \mathcal E_{n}$, it follows that $[\varphi]_{I}*[\psi]_{J}$ is a finite sum of terms of the form 
\begin{equation*}
\prod_{\ell\in A}2^{-\epsilon(i_{\ell+1}-i_{\ell})}\,\prod_{\ell\in \widetilde A}2^{-\epsilon|i_{\ell}-j_{\ell}|}
\prod_{l \in B}2^{-\epsilon(j_{l+1}-j_{l})}\,\prod_{l \in \widetilde B}2^{-\epsilon|i_{l}-j_{l}|}
\,[\theta_{A,B}]_{I\vee J}
\end{equation*}
where:
\begin{enumerate}[{\rm(a)}]

\smallskip

\item the set $\{1, \ldots,n\}$ is the disjoint union of the sets $A$ and $\widetilde A$, and of the sets $B$ and $\widetilde B$; 

\smallskip

\item each function $\theta_{A,B}$ is normalized relative to $\varphi$ and $\psi$.
\end{enumerate}
\end{corollary}

\begin{proof} Let $J_{\ell}$ denote the set of subscripts $\ell$ such that $x_{\ell}\in \R^{a_{\ell}}$, and let $\sigma:\{1, \ldots, N\}\to \{1, \ldots, n\}$ be the mapping such that $\sigma(\ell) \in J_{\ell}$ for all $\ell$. Since $\varphi$ and $\psi$ both have strong cancellation relative to the decomposition $\R^{N}= \R^{a_{1}}\oplus \cdots\oplus \R^{a_{n}}$, if $A_{0}$ and $B_{0}$ are the sets defined in Lemma \ref{Lemma4.3qw}, it follows that $\{\sigma(\ell)\,\big\vert\,\ell \in A_{0}\cup B_{0}\}= \{1, \ldots, n\}$. This means that the sets $A_{3}$ and $B_{3}$ of that Lemma must be empty, and the result follows.
\end{proof}

\subsection{Truncated flag kernels}\label{SecTruncated}\quad

\smallskip

\begin{definition}\label{DefTruncated}
A flag distribution $\K$ adapted to the decomposition $\R^{N}= \R^{a_{1}}\oplus \cdots \oplus\R^{a_{n}}$ is a \emph{truncated kernel of width $a>0$} if the differential inequalities given in  part (\ref{Def2.1a}) of Definition \ref{Def2.1} are replaced by
\begin{equation*}
\big\vert\partial^{\balpha}_{\x}K(\x)\big\vert \leq C_{\balpha}\,\prod_{k=1}^{n} \left[a+N_{1}(\x_{1}) + \cdots N_{k}(\x_{k})\right]^{-Q_{k}-[\![\balpha_{k}]\!]}.
\end{equation*}
$\K$ is an \emph{improved truncated kernel} if it is a truncated kernel, and in addition satisfies
\begin{equation*}
\big\vert\partial^{\balpha}_{\x}K(\x)\big\vert \leq C_{\balpha}\,\frac{a}{a+N_{1}(\x_{1})}\,\prod_{k=1}^{n} \left[a+N_{1}(\x_{1}) + \cdots N_{k}(\x_{k})\right]^{-Q_{k}-[\![\balpha_{k}]\!]}.
\end{equation*}
\end{definition}

Our objective is to establish the following:

\goodbreak

\begin{proposition}\label{Prop4.5z}
Let $\K$ be a flag distribution.
\begin{enumerate}[{\rm(1)}]

\smallskip

\item If $\psi\in \mathcal C^{\infty}_{0}(\R^{N})$ has support in the unit ball, then $\K*\psi$ and $\psi*\K$ are truncated flag kernels of width $1$. 

\smallskip

\item If $\psi\in \mathcal C^{\infty}_{0}(\R^{N})$ has support in the unit ball, and if $\int_{\R^{N}}\psi(\x)\,d\x=0$, then then $\K*\psi$ and $\psi*\K$ are improved truncated flag kernels of width $1$.
\end{enumerate}
\end{proposition}

\begin{proof}
We can write  $\K= \sum_{I\in \mathcal E_{n}}[\varphi^{I}]_{I}+\sum_{j=1}^{r}\K_{r}$, where $\{\varphi^{I}\}$ are normalized unit bump functions with strong cancellation, and $\{\K_{1}, \ldots, \K_{r}\}$ are flag distributions adapted to coarser flags. If $\K_{0}=  \sum_{I\in \mathcal E_{n}}[\varphi^{I}]_{I}$,  it suffices to show that the proposition is true for $\K_{0}$. We consider $\psi*\K_{0}$. The case of $\K_{0}*\psi$ is handled similarly.
We have
\begin{align*}
\psi*\K_{0}(\x)&= \sum_{I\in \mathcal E_{n}}\psi*[\varphi^{I}]_{I}\\
&= 
\sum_{I\in \mathcal E_{n}^{-}}\psi*[\varphi^{I}]_{I}+\sum_{k=1}^{n-1}\sum_{I\in \mathcal E_{n}^{k}}\psi*[\varphi^{I}]_{I}+\sum_{I\in \mathcal E_{n}^{+}}\psi*[\varphi^{I}]_{I}\\
&=
I(\x)+\sum_{k=1}^{n-1}II_{k}(\x)+III(\x)
\end{align*}
where
\begin{align*}
E_{n}^{-}&= \{I=(i_{1}, \ldots, i_{n})\in \mathcal E_{n}\,\big\vert\, i_{n}\leq 0\},\\
E_{n}^{\,k}&= \{I=(i_{1}, \ldots, i_{n})\in \mathcal E_{n}\,\big\vert\, i_{k}\leq0<i_{k+1} \},\\
E_{n}^{+}&= \{I=(i_{1}, \ldots, i_{n})\in \mathcal E_{n}\,\big\vert\, 0< i_{1}\}.
\end{align*}
Denote the element $(0,\ldots, 0)\in \mathcal E_{n}$ by $\mathbf 0$. If $I=(i_{1},\ldots, i_{n})\in \mathcal E_{n}^{k}$, then $i_{1}\leq i_{2}\leq \cdots \leq i_{k}\leq 0$, and we put 
\begin{equation*}
\widetilde I_{k} = (0,\ldots, 0,i_{k+1}, \ldots, i_{n}).
\end{equation*} 
Then
\begin{equation*}
I\vee \mathbf 0 =
\begin{cases}
\mathbf 0 & \text{if $I\in \mathcal E_{n}^{-}$,}\\
\widetilde I_{k}& \text{if $I\in \mathcal E_{n}^{k}$,}\\
I& \text{if $I\in \mathcal E_{n}^{+}$.}
\end{cases}
\end{equation*}
Let 
\begin{equation*}
\widetilde E_{n}^{k}= \{I\in \mathcal E_{n}\,\big\vert\,i_{1}= \cdots = i_{k}= 0\}.
\end{equation*} 
Note that if $I\in \mathcal E_{n}^{k}$ then $\widetilde I_{k}\in \widetilde E_{n}^{k}$. According to Lemmas \ref{Lemma5.1} and \ref {Lemma4.3qw}, each term $\psi*[\varphi^{I}]_{I}$ has weak cancellation. Moreover, for each $I\in \mathcal E_{n}$ there exists $\theta^{I}\in \mathcal C^{\infty}_{0}(\R^{N})$, normalized with respect to $\psi$ and $\varphi^{I}$, so that 
\begin{equation}\label{Eqn8.1}
\begin{aligned}
&&&&I&\in \mathcal E_{n}^{-} &&\Longrightarrow & \psi*[\varphi^{I}]_{I} &= 2^{-\epsilon(|i_{1}|+\cdots+|i_{n}|)}[\theta^{I}]_{\bar0}&&&&\\
&&&&I&\in \mathcal E_{n}^{k} &&\Longrightarrow & \psi*[\varphi^{I}]_{I} &= 2^{-\epsilon(|i_{1}|+\cdots+|i_{k}|)}[\theta^{I}]_{\widetilde I_{k}}&&&&\\
&&&&I&\in \mathcal E_{n}^{+} &&\Longrightarrow & \psi*[\varphi^{I}]_{I} &= [\theta^{I}]_{I}.&&&&
\end{aligned}
\end{equation}

\smallskip

We have
\begin{align*}
I(\x) &= 
\sum_{I\in \mathcal E_{n}^{-}}2^{-\epsilon(|i_{1}|+\cdots+|i_{n}|)}[\theta^{I}]_{\mathbf 0}(\x)
=
\sum_{I\in \mathcal E_{n}^{-}}2^{-\epsilon(|i_{1}|+\cdots+|i_{n}|)}\theta^{I}(\x) = \widetilde \theta^{\mathbf 0}(\x)
\end{align*}
where the series  converges to a normalized unit bump function $\widetilde \theta^{\mathbf 0}$. Next, we can write
\begin{align*}
II_{k}(\x)
&=
\sum_{I\in \mathcal E_{n}^{k}}\psi*[\varphi^{I}]_{I} (\x)
= 
\sum_{J\in \widetilde E_{n}^{k}} \Big[\sum_{\substack{I\in \mathcal E_{n}^{k}\\ \widetilde I_{k}= J}}\psi*[\varphi^{I}]_{I} \Big](\x)\\
&=
\sum_{J\in \widetilde E_{n}^{k}} \Big[\sum_{\substack{I\in \mathcal E_{n}^{k}\\ \widetilde I_{k}= J}}2^{-\epsilon(|i_{1}|+\cdots+|i_{k}|)}[\theta^{I}]_{\widetilde I_{k}}\Big](\x)\\
&=
\sum_{J\in \widetilde E_{n}^{k}}
 \Big[\sum_{\substack{I\in \mathcal E_{n}^{k}\\ \widetilde I_{k}= J}}2^{-\epsilon(|i_{1}|+\cdots+|i_{k}|)}\theta^{I}\Big]_{J}(\x)= 
\sum_{J\in \widetilde E_{n}^{k}}\big[\widetilde \theta^{J}\big]_{J}(\x),
\end{align*}
where
\begin{equation*}
\sum_{\substack{I\in \mathcal E_{n}^{k}\\ \widetilde I_{k}= J}}2^{-\epsilon(|i_{1}|+\cdots+|i_{k}|)}\theta^{I}=\widetilde \theta^{J}
\end{equation*}
converges to  a normalized unit bump function. Thus each $II_{k}$ is a flag kernel with a decomposition into dilates of normalized bump functions where all of the dilation parameters are non-negative. This is also true of the term $III$. 

Thus we have shown that
\begin{equation*}
\psi*\K_{0} = \sum_{I\in \mathcal E_{n}^{+}}[\widetilde\theta^{I}]_{I}
\end{equation*}
where
\begin{equation*}
E_{n}^{+}= \{I=(i_{1}, \ldots,i_{n})\in \mathbb Z^{n}\,\big\vert\,0\leq i_{1}\leq i_{2}\leq \cdots \leq i_{n}\}.
\end{equation*}
But then it follows from the second inequality in Proposition \ref{Prop2.11}, as in the proof of Proposition \ref{Proposition5.3z} that  $\psi*\K_{0}$ satisfies the differential inequalities of a truncated kernel of width one. 

\smallskip

If we assume in addition that $\int_{\R^{N}}\psi(\x)\,d\x=0$, then we can write $\psi = \sum_{l=1}^{N}\psi_{l}$, with 
\begin{equation*}
\int_{\R}\psi_{k}(x_{1}, \ldots, x_{l-1},t,x_{l+1}, \ldots, x_{n})\,dt =0.
\end{equation*}
We can repeat the argument given above with $\psi$ replace by $\psi_{l}$. It then follows from  Lemma \ref{Lemma4.3qw} that instead of the formulas given in equation (\ref{Eqn8.1}), we get 
\begin{equation}\label{Eqn8.2}
\begin{aligned}
&&&&I&\in \mathcal E_{n}^{-}\ &&\Longrightarrow & \psi_{l}*[\varphi^{I}]_{I} &= 2^{-\epsilon(|i_{1}|+\cdots+|i_{n}|)}[\theta^{I}]_{\bar0}&&&&\\
&&&&I&\in \mathcal E_{n}^{k} &&\Longrightarrow & \psi_{l}*[\varphi^{I}]_{I} &= 2^{-\epsilon(|i_{1}|+\cdots+|i_{k}|)}[\theta^{I}]_{\widetilde I_{k}}&\quad\text{if $k<l$}&&&&\\
&&&&I&\in \mathcal E_{n}^{k} &&\Longrightarrow & \psi_{l}*[\varphi^{I}]_{I} &= 2^{-d_{l} i_{l}}\,2^{-\epsilon(|i_{1}|+\cdots+|i_{k}|)}[\theta^{I}]_{\widetilde I_{k}}&\quad\text{if $k\geq l$}&&&&\\
&&&&I&\in \mathcal E_{n}^{+} &&\Longrightarrow & \psi_{l}*[\varphi^{I}]_{I} &= 2^{-d_{l} i_{l}}\,[\theta^{I}]_{I}.&&&&
\end{aligned}
\end{equation}
Again using the second inequality in Proposition \ref{Prop2.11}, and observing that the case $l=1$ gives the worst estimate, we see that $\psi_{l}*\K_{0}$ satisfies the differential inequalities of an improved truncated kernel of width one. 
\end{proof}

\begin{remark}\label{Remark7.7}
{\rm We point out that we can relax the $\mathcal C^{\infty}$ requirement on the function $\psi$ in Propsition \ref{Prop4.5z} in the following way. An examination of the arguments in Sections  \ref{Nilpotent} and \ref{Decay} and the proof just given show that for any\label{Here} integer m there exists an integer M, so that $\K*\psi$ and $\psi*\K$ satisfy the required differential inequalities and cancellation properties for orders of differentiation not exceeding $m$, if $\psi$ is supposed to be of class $\mathcal C^{M}$.}
\end{remark}

\section{Convolution of flag kernels}

Let $\K\in \mathcal S'(\R^{N})$ be a flag distribution on the homogeneous nilpotent Lie group $G=\R^{N}$, adapted to the standard flag associated with the decomposition $\R^{N}= \R^{a_{1}}\oplus \cdots \oplus \R^{a_{n}}$.  Define a left-invariant operator $T_{\K}:\mathcal S(\R^{N})\to\mathcal C^{\infty}(\R^{N})$ by setting
\begin{equation*}
T_{\K}[\phi](\x) = \phi * \K(\x) = \langle \K,\phi^{\#}_{\x}\rangle
\end{equation*} 
where, if $\phi\in \mathcal S(\R^{N})$, we set $\phi^{\#}_{\x}(\y) = \phi(\x\y^{-1})$. If $\K_{1}$ and $\K_{2}$ are two flag kernels on $G$, we want to make sense of the composition $T_{\K_{2}}\circ T_{\K_{1}}$, and show that the resulting operator is of the form $T_{\K_{3}}$ where $\K_{3}$ is a third flag kernel on $G$. Now formally
\begin{equation*}
T_{\K_{2}}\circ T_{\K_{1}}[\phi]= (T_{\K_{1}}[\phi])*\K_{2}= (\phi*\K_{1})*\K_{2}=\phi*(\K_{1}*\K_{2}),
\end{equation*}
so the operator $T_{\K_{2}}\circ T_{\K_{1}}$ should be given by convolution with the distribution $\K_{1}*\K_{2}$. However we cannot directly define the composition $T_{\K_{2}}\circ T_{\K_{1}}[\phi]= T_{\K_{2}}\big(T_{\K_{1}}[\phi]\big)$, even if $\phi \in \mathcal C^{\infty}_{0}(\R^{N})$, since $T_{\K_{1}}[\phi]$ need not belong to $\mathcal S(\R^{N})$. 
Also, in general one cannot convolve an arbitrary pair of distributions unless one of them has compact support. 

We will define the convolution $\K_{1}*\K_{2}$ somewhat indirectly. In Section \ref{L2boundedness} we show that if $\phi\in \mathcal S(\R^{N})$, then $T_{\K}[\phi]\in L^{2}(\R^{N})$ and the mapping $T_{\K}:\mathcal S(\R^{N}) \to L^{2}(\R^{N})$ has a (unique) continuous extension to a mapping of $L^{2}(\R^{N})$ to itself. This allows us to define $T_{\K_{2}}\circ T_{\K_{1}}$ as the composition of two mappings from $L^{2}(\R^{N})$ to itself. Then in Section \ref{Reduction}, we show that this  composition is given by convolution with a distribution which is given as sum of convolutions of dilates of  bump functions. The key is then to recognize this sum as a flag kernel. The combinatorics are rather complicated, so in Section \ref{Example} we work out an explicit example. In Section \ref{General} we prove the main result, Theroem \ref{Theorem8.4}, which shows that the convolution of two flag kernels is a sum of flag kernels. Finally in Section \ref{FlagAlgebra} we work out some additional examples.

\smallskip

\subsection{Boundedness on $L^{2}$}\label{L2boundedness}\quad

\smallskip

In this section we show that convolution with a flag kernel extends to a bounded operator on $L^{2}(\R^{N})$. Later in Section \ref{LpEstimates} we will see more: such operators are bounded on $L^{p}(\R^{N})$ for $1<p<\infty$.

\begin{lemma}\label{Lemma5.1qw}
Let $\K$ be a flag kernel on $\R^{N}$. Then there is a constant $C$ so that if  $T_{\K}[\phi] = \phi*\K$ for $\phi\in \mathcal S(\R^{N})$ then $||T_{\K}[\phi]||_{L^{2}(\R^{N})}\leq C\,||\phi||_{L^{2}(\R^{N})}$. As a consequence, there is a unique extension of $T_{\K}$  to a bounded operator from  $L^{2}(\R^{N})$  to itself.
\end{lemma}

\begin{proof}
Using Corollary \ref{Cor3.2xz}, we can assume that $\K$ is a flag kernel adapted to a standard flag $\mathcal F$ of length $n$ as given in equation (\ref{Eqn2.6tyu}), and that there is a uniformly bounded family of functions $\{\varphi^{I}\}\subset \mathcal C^{\infty}_{0}(\R^{N})$, each having strong cancellation relative to the decomposition $\R^{N}= \R^{a_{1}}\oplus \cdots \oplus \R^{a_{n}}$ such that $\K = \sum_{I\in \mathcal E_{n}}[\varphi^{I}]_{I}$ with convergence in the sense of distributions.

For any $I\in \mathcal E_{n}$ let $T_{I}[f]= f*[\varphi^{I}]_{I}$.
Then $||T_{I}[f]||_{L^{2}}\leq ||\varphi^{I}||_{L^{1}}\,||f||_{L^{2}}$. If $(f,g) = \int_{\R^{N}}f(\x)\overline{g(\x)}\,d\x$ is the standard inner product in $L^{2}(\R^{N})$, it follows from Fubini's theorem that  $\big(T_{I}[f],g\big) = \big(f,\widetilde T_{I}[g]\big)$ where $T_{\widetilde I}[g] = g*[\widetilde \varphi^{I}]_{I}$ and $\widetilde \varphi^{I}(\x)= \overline{\varphi^{I}(\x^{-1})}$. Thus the Hilbert space adjoint of the operator $T_{\varphi^{I}}$ is the operator $T_{\varphi}^{*}= T_{\widetilde\varphi^{I}}$. Fubini's theorem also shows that 
$T_{I}\circ T_{J}^{*}= T_{\widetilde\varphi^{I} * \varphi^{J}}$ and $T_{J}^{*}\circ T_{I}= T_{\varphi^{I}*\widetilde\varphi^{J}}$. Thus if $I,J\in \mathcal E_{n}$, the $L^{2}$-norm of the operators $T_{{I}}\circ T_{{J}}^{*}$ and $T_{{J}}^{*}\circ T_{{I}}$ are bounded by the $L^{1}$ norms of $[\widetilde\varphi^{J}]_{J} * [\varphi^{I}]_{I}$ and $[\varphi^{I}]_{I}*[\widetilde\varphi^{J}]_{J}$. It follows from Corollary \ref{Cor7.4} that 
\begin{equation}\label{eqn5.1dw}
\begin{aligned}
||[\widetilde\varphi]_{J} * [\varphi]_{I}||_{L^{1}(\R^{N})}&+ ||[\varphi]_{I}*[\widetilde\varphi]_{J}||_{L^{1}(\R^{N})} \\
&\leq 
C\,2^{-\epsilon|i_{n}-j_{n}|}\,\prod_{\ell=1}^{n-1}\big[2^{-\epsilon|i_{\ell}-j_{\ell}|}+\min\big\{ 2^{-\epsilon(i_{\ell+1}-i_{\ell})}, 2^{-\epsilon(j_{\ell+1}-j_{\ell})}\big\}\big].
\end{aligned}\end{equation}

For any finite subset $F\subset E_{n}$, set $\K_{F}(\x)= \sum_{I\in F}[\varphi^{I}]_{I}(\x)$. Then for any $\phi\in \mathcal S(\R^{N})$, $\big\langle \K,\phi\big\rangle = \lim_{F\nearrow E_{n}}\big\langle \K_{F},\phi\big\rangle$, 
and in particular, if $\phi^{\#}_{\x}(\y) = \phi(\x\y^{-1})$, 
\begin{align*}
T_{\K}[\phi](\x)
= 
\big\langle \K,\phi^{\#}_{\x}\big\rangle= \lim_{F\nearrow E_{n}}\big\langle \K_{F},\phi^{\#}_{\x}\big\rangle
&= \lim_{F\nearrow E_{n}}\sum_{I\in F}\phi*[\varphi^{I}]_{I}(\x)
=
\lim_{F\nearrow E_{n}}\sum_{I\in F}T_{I}[\phi](\x).
\end{align*}
It follows from the almost orthogonality estimate in (\ref{eqn5.1dw}) and the Cotlar-Stein Theorem (see for example \cite{St93}, page 280) that there is a constant $C$ independent of the finite set $F$ such that 
\begin{equation*}
||\sum_{I\in F}T_{[\varphi^{I}]_{I}}[\phi]||_{L^{2}}\leq C\,||\phi||_{L^{2}}.
\end{equation*}
But then Fatou's lemma implies that $||T_{\K}[\phi]||_{L^{2}}\leq C\,||\phi||_{L^{2}}$ for all $\phi\in \mathcal S(\R^{N})$. This completes the proof.
\end{proof}

\begin{corollary} 
Supposet that $\K$ is a flag kernel, and that $\K = \sum_{I\in \mathcal E_{n}}[\varphi^{I}]_{I}$. Then for all $f\in L^{2}(\R^{N})$,
$$\lim_{K\nearrow E_{n}} ||\sum_{I\in F}T_{[\varphi]_{I}}[f]-T_{\K}[f]||_{L^{2}}=0.$$ 

\end{corollary}

\begin{proof}
Since $\mathcal S(\R^{N})$ is dense in $L^{2}(\R^{N})$, and since $T_{\K}$ is bounded on $L^{2}(\R^{N})$, it suffices to show that $\lim_{K\nearrow E_{n}} ||\sum_{I\in F}T_{[\varphi^{I}]_{I}}[\psi]-T_{\K}[\psi]||_{L^{2}}=0$ for $\psi\in \mathcal S(\R^{N})$. (Both $T_{\K}$ and $\sum_{I\in F}T_{[\varphi^{I}]_{I}}$ are bounded on $L^{2}(\R^{N})$ with norm independent of $F$). But for $\psi\in \mathcal S(\R^{N})$, the result follows from Theorem \ref{Lemma5.4zw} and the discussion following it on page \pageref{remarks}. This completes the proof.
\end{proof}



\subsection{Composition of convolution operators}\label{Reduction}\quad

\smallskip

Let $\K_{j}$, $j=1,\,2$,  be two flag kernels on $\R^{N}$, and let $T_{\K_{j}}[\phi]=\phi*\K_{j}$  be the corresponding convolution operators. According to Lemma \ref{Lemma5.1qw}, each of these operators is bounded on $L^{2}(\R^{N})$, and hence the composition $T_{\K_{2}}\circ T_{\K_{1}}$ is well-defined as a bounded operator on $L^{2}(\R^{N})$. Our main result is the following.

\begin{theorem}\label{Thm5.3dc}
Let $\mathcal F_{1}, \mathcal F_{2}$ be two standard flags on $\R^{N}$, and let $\mathcal F_{0}$ be the coarsest flag on $\R^{N}$ which is finer than both $\mathcal F_{1}$ and $\mathcal F_{2}$. For $j=1,2$, let $\K_{j}$ be a flag kernel adapted to the flag $\mathcal F_{j}$. Then $T_{\K_{2}}\circ T_{\K_{1}}$ is a flag kernel adapted to the flag $\mathcal F_{0}$.
\end{theorem}

In order to study the composition $T_{\K_{2}}\circ T_{\K_{1}}$, we want to relate it to the decompositions of $\K_{1}$ and $\K_{2}$ as sums of dilates of normalized bump functions. According to Corollary \ref{Cor3.2xz}, we can assume that  the flags $\mathcal F_{1}$ and $\mathcal F_{2}$ are given by
\begin{align*}
\F_{1}:\qquad (0) &\subseteq \R^{a_{n}}\subseteq \R^{a_{n-1}}\oplus\R^{a_{n}}\subseteq \cdots \subseteq \R^{a_{2}}\oplus\cdots \oplus\R^{a_{n}}\subseteq \R^{N},\\
\F_{2}:\qquad (0) &\subseteq \R^{b_{m}}\subseteq \R^{b_{m-1}}\oplus\R^{b_{m}}\subseteq \cdots \subseteq \R^{b_{2}}\oplus\cdots \oplus\R^{b_{m}}\subseteq \R^{N},
\end{align*}
and the flag kernels are given by $\K_{1} = \sum_{I\in \E_{n}}[\varphi^{I}]_{I}$ and $\K_{2} = \sum_{J\in \E_{m}}[\psi^{J}]_{J}$, where $\{\varphi^{I}\,\vert\, I\in \E_{n}\}$ is a uniformly bounded family of compactly supported functions with strong cancellation relative to the flag $\mathcal F_{1}$ and $\{\psi^{J}\,\vert\, J\in \E_{m}\}$ is a uniformly bounded family of compactly supported functions with strong cancellation relative to the flag $\mathcal F_{2}$.

If $\phi,\theta\in \mathcal S(\R^{N})$, then 
\begin{align*}
T_{\K_{1}}[\phi] &= \lim_{F\nearrow \E_{n}}\sum_{I\in F}\phi * [\varphi^{I}]_{I}, \quad  \text{and}\\
T_{\K_{2}}[\theta] &= \lim_{G\nearrow \E_{m}}\sum_{J\in G}\theta * [\psi^{J}]_{J},
\end{align*}
where the limits are in $L^{2}(\R^{N})$ and are taken over finite subsets $F\subset \E_{n}$ and $G\subset \E_{m}$.   For every fixed finite set $F\subset \E_{n}$, the function $\sum_{I\in F}\phi * [\varphi^{I}]_{I}\in \mathcal S(\R^{N})$. Since $T_{\K_{2}}$ is a continuous mapping from $L^{2}(\R^{N})$ to itself, it follows that
\begin{align*}
T_{\K_{2}}\big(T_{\K_{1}}[\phi]\big) &= \lim_{F\nearrow \E_{n}}T_{\K_{2}}\big(\sum_{I\in F}\phi*[\varphi^{I}]_{I}\big)
=\lim_{F\nearrow \E_{n}}\lim_{G\nearrow \E_{m}}\sum_{I\in F}\sum_{J\in G}\phi*[\varphi^{I}]_{I}*[\psi^{J}]_{J}\\
&=
\lim_{F\nearrow \E_{n}}\lim_{G\nearrow \E_{m}}\,\phi*\Big[\sum_{I\in F}\sum_{J\in G}[\varphi^{I}]_{I}*[\psi^{J}]_{J}\Big].
\end{align*}
Thus in order to prove Theorem \ref{Thm5.3dc}, we must study the finite sums $\sum_{I\in F}\sum_{J\in G}[\varphi^{I}]_{I}*[\psi^{J}]_{J}$, and show that these converge in the sense of distributions to a finite sum of flag kernels, each adapted to a flag which is equal to or coarser than $\mathcal F_{0}$. 

\smallskip

Since the general situation is rather complicated, we first present an example which may help understand the difficulties.

\subsection{An Example} \label{Example}\quad

\smallskip

Suppose that  we are working in $\R^{5}$ with the family of dilations given by
\begin{equation}
\delta\cdot \x=\delta\cdot (x_{1},\,x_{2},\,x_{3},\,x_{4},\,x_{5}) =  (\delta^{d_{1}}x_{1},\delta^{d_{2}}x_{2},\delta^{d_{3}}x_{3},\delta^{d_{4}}x_{4},\delta^{d_{5}}x_{5})
\end{equation}
with $d_{1}\leq d_{2}\leq d_{3}\leq d_{4}\leq d_{5}$. The standard flags on $\R^{5}$ correspond to partitions of $N=5$. Consider two flags $\F_{1}$ and $\F_{2}$ corresponding to the partitions $\mathfrak A=(2,3)$, where we write $\R^{5}=\R^{2}\oplus\R^{3}$, and $\mathfrak B=(1,2,2)$, where we write $\R^{5}=\R\oplus\R^{2}\oplus\R^{2}$. Thus
\begin{align*}
\text{$\mathcal F_{1}$ is the flag} \qquad (0)&\subset \{x_{1}=x_{2}=0\}\cong \R^{3}\subset \R^{5},\\
\text{$\mathcal F_{2}$ is the flag} \qquad (0)&\subset \{x_{1}=x_{2}=x_{3}=0\}\cong\R^{2}\subset \{x_{1}=0\}\cong\R^{2}\oplus\R^{2}\subset \R^{5}.
\end{align*}
We are given flag kernels $\K_{1}=\sum_{I\in \E_{2}}[\varphi^{I}]_{I}$ and $\K_{2}= \sum_{J\in \E_{3}}[\psi^{J}]_{J}$ adapted to these two flags. We then want to study the infinite sum 
\begin{equation}\label{6.1q}
\sum_{I\in \E_{2}}\sum_{J\in \E_{3}}[\varphi^{I}]_{I}*[\psi^{J}]_{J}
\end{equation} arising from the composition of the operators $T_{\K_{1}}\circ T_{\K_{2}}$. 

Suppose $I=(i_{1},i_{2})\in \E_{2}$ and $J=(j_{1},j_{2},j_{3})\in \E_{3}$, so that $i_{1}\leq i_{2}$, and $j_{1}\leq j_{2}\leq j_{3}$. If $\varphi^{I}, \psi^{J}\in \mathcal C^{\infty}_{0}(\R^{N})$, we have
\begin{align*}
[\varphi^{I}]_{I}(\x)&= 2^{-i_{1}(d_{1}+d_{2})-i_{2}(d_{3}+d_{4}+d_{5})}\,\,
\varphi^{I}(2^{-d_{1}i_{1}}x_{1},2^{-d_{2}i_{1}}x_{2},2^{-d_{3}i_{2}}x_{3},2^{-d_{4}i_{2}}x_{4},2^{-d_{5}i_{2}}x_{5})\\
[\psi^{J}]_{J}(\x)&=
2^{-j_{1}d_{1}-j_{2}(d_{2}+d_{3})-j_{3}(d_{4}+d_{5})}\,
\psi^{J}(2^{-d_{1}j_{1}}x_{1},2^{-d_{2}j_{2}}x_{2},2^{-d_{3}j_{2}}x_{3},2^{-d_{4}j_{3}}x_{4},2^{-d_{5}j_{3}}x_{5}).
\end{align*}
Note that dilation by $I=(i_{1},i_{2})$ on $\R^{2}\oplus\R^{3}$ is the same as dilation by the  $5$-tuple $\tilde I= (i_{1},i_{1},i_{2},i_{2},i_{2})$ on $\R^{5}$, and dilation by $J=(j_{1}, j_{2}, j_{3})$ on $\R\oplus\R^{2}\oplus\R^{2}$ is the same as dilation by the $5$-tuple $\tilde J=(j_{1},j_{2},j_{2}, j_{3},j_{3})$ on $\R^{5}$. Also note that we can reconstruct $I$ and $J$ from $\tilde I$ and $\tilde J$ by consolidating repeated indices. By Lemma \ref{Lemma5.1}, the convolution $[\varphi^{I}]_{I}*[\psi^{J}]_{J}$ is equal to  $[\theta^{I,J}]_{K}$ where $\theta^{I,J}\in \mathcal C^{\infty}_{0}(\R^{5})$, and where 
\begin{equation}\label{Equation6.4rt}
\begin{aligned}
K&=(k_{1},k_{2},k_{3},k_{4},k_{5})=\tilde I\vee \tilde J =(i_{1},i_{1},i_{2},i_{2},i_{2}) \vee (j_{1}, j_{2}, j_{2}, j_{3}, j_{3})\\
&= (\max\{i_{1}, j_{1}\}, \max\{i_{1}, j_{2}\}, \max\{i_{2},j_{2}\}, \max\{i_{2},j_{3}\}, \max\{i_{2},j_{3}\}).
\end{aligned}
\end{equation}

We must consider the sum in (\ref{6.1q}) of the convolutions $[\varphi^{I}]_{I}*[\psi^{J}]_{J}$, taken over all $I\in \E_{2}$ and $J\in \E_{3}$.  Each pair $(I,J)\in \E_{2}\times \E_{3}$ gives rise to a $5$-tuple $K\in \E_{5}$. 
However, not all elements of $\E_{5}$ actually arise in this sum. (For example, it is clear from (\ref{Equation6.4rt}) that we must have $k_{4}=k_{5}$, so the 5-tuple $(1, 2, 3, 4, 5)$  does not arise). Let $\E( {\mathfrak A} , {\mathfrak B} )$ denote the set of all $5$-tuples $K=(k_{1}, k_{2}, k_{3},k_{4},k_{5})$ that do arise as in (\ref{Equation6.4rt}). (The notation reflects the fact that this set of $5$-tuples is determined by the partitions $ {\mathfrak A} =(2,3)$ and $ {\mathfrak B} =(1,2,2)$ of $N=5$).  Then for each $K\in \E( {\mathfrak A} , {\mathfrak B} )$, let $\E(K)$ denote the set of pairs $(I,J)\in \E_{2}\times \E_{3}$ which give rise to the $5$-tuple $K$. Once $K\in \E( {\mathfrak A} , {\mathfrak B} )$ is fixed, each of the terms in the inner infinite sum on the right-hand side of (\ref{Equation6.4by}) is the $K$ dilate of a normalized bump function  $\theta^{I,J}$. Then we can write the sum in  (\ref{6.1q}) as
\begin{align}\label{Equation6.4by}
\sum_{I\in \E_{2}}\sum_{J\in \E_{3}}[\varphi^{I}]_{I}*[\psi^{J}]_{J}
&=
\sum_{K\in \E( {\mathfrak A} , {\mathfrak B} )}\,\sum_{I,J\in \E(K)}[\varphi^{I}]_{I}*[\psi^{J}]_{J}
=
\sum_{K\in \E( {\mathfrak A} , {\mathfrak B} )}\,\sum_{I,J\in \E(K)}[\theta^{I,J}]_{K}.
\end{align}

We  will need to show that the infinite inner sum $\sum_{I,J\in \E(K)}[\theta^{I,J}]_{K}$ actually converges and is the $K$ dilate of a normalized bump function. However, this is not enough to give the right description of the sum in (\ref{6.1q}) as a flag kernel. In the outer sum on the right-hand side of (\ref{Equation6.4by}), the $5$-tuple $K$ runs over $\E( {\mathfrak A} , {\mathfrak B} )$ and  \emph{not}
over all of $\E_{5}$. We still need to partition $\E( {\mathfrak A} , {\mathfrak B} )$ into a finite number of subsets based on which indices in $K$ are repeated. To make this clear, we further analyze $\E( {\mathfrak A} , {\mathfrak B} )$.

As one sees from (\ref{Equation6.4rt}), the  $5$-tuple $K= \tilde I\vee \tilde J$  depends not only on the tuples $I=\{i_{1},i_{2}\}$ and $J=\{j_{1},j_{2},j_{3}\}$, but also on the \textit{ordering} of the larger set consisting of $\{i_{1},i_{2},j_{1},j_{2},j_{3}\}$. We know that $i_{1}\leq i_{2}$ and $j_{1}\leq j_{2}\leq j_{3}$, but this does not determine the ordering of the larger set. Such orderings are in one-to-one correspondence with decompositions of the set $\{1,2,3,4,5\}$ into two disjoint subsets of sizes $2$ and $3$, where elements of the first set are indices from $I$, and elements of the second set are indices from $J$. Thus there are $\binom{5}{2}=\binom{5}{3}=10$ such orderings. A description of these is given in the following Table 1:

\smallskip

\centerline{\textbf{Table 1:} \quad Decomposition of $\E( {\mathfrak A} , {\mathfrak B} )$}

\smallskip

{\scriptsize
\begin{center}
\begin{tabular}{|c|c|c|c|c|}
\hline
Decomposition & Ordering & $K$ & New Decomposition & Free\\\hline &&&&\\
$ \E( {\mathfrak A} , {\mathfrak B} )_{1}=\{1,2\} \cup \{3,4,5\}$ &$i_{1}\leq i_{2}\leq j_{1}\leq j_{2}\leq j_{3}$&$\{j_{1},j_{2},j_{2},j_{3},j_{3}\}$&$ \R\oplus\R^{2}\oplus\R^{2}$ &$i_{1}, i_{2}$\\&&&&\\ \hline&&&&\\
$ \E( {\mathfrak A} , {\mathfrak B} )_{2}=\{1,3\} \cup \{2,4,5\}$ &$i_{1}\leq j_{1} < i_{2}\leq j_{2}\leq j_{3}$&$\{j_{1},j_{2},j_{2},j_{3},j_{3}\}$&$ \R\oplus\R^{2}\oplus\R^{2}$ &$i_{1}, i_{2}$\\&&&&\\\hline&&&&\\
$ \E( {\mathfrak A} , {\mathfrak B} )_{3}=\{1,4\} \cup \{2, 3,5\}$ &$i_{1}\leq j_{1}\leq j_{2}< i_{2}\leq j_{3}$&$\{j_{1},j_{2},i_{2},j_{3},j_{3}\}$&$ \R\oplus\R\oplus\R\oplus\R^{2}$ &$i_{1}$\\&&&&\\\hline&&&&\\
$ \E( {\mathfrak A} , {\mathfrak B} )_{4}=\{1,5\} \cup \{2,3,4\}$ &$i_{1}\leq j_{1}\leq j_{2}\leq j_{3} < i_{2}$&$\{j_{1},j_{2},i_{2},i_{2},i_{2}\}$&$\R\oplus\R\oplus\R^{3}$&$i_{1}, j_{3}$\\&&&&\\\hline&&&&\\
$\E( {\mathfrak A} , {\mathfrak B} )_{5}= \{2,3\} \cup \{1,4,5\}$ &$j_{1}< i_{1}\leq i_{2}\leq j_{2}\leq j_{3}$&$\{i_{1},j_{2},j_{2},j_{3},j_{3}\}$&$ \R\oplus\R^{2}\oplus\R^{2}$ &$i_{2}, j_{1}$\\&&&&\\\hline&&&&\\
$ \E( {\mathfrak A} , {\mathfrak B} )_{6}=\{2,4\} \cup \{1,3,5\}$ &$j_{1}< i_{1}\leq j_{2}< i_{2}\leq j_{3}$&$\{i_{1}, j_{2}, i_{2}, j_{3},j_{3}\}$&$ \R\oplus\R\oplus\R\oplus\R^{2}$&$j_{1}$\\&&&&\\\hline&&&&\\
$\E( {\mathfrak A} , {\mathfrak B} )_{7}= \{2,5\} \cup \{1,3,4\}$ &$j_{1}< i_{1}\leq j_{2}\leq j_{3}<i_{2}$&$\{i_{1}, j_{2}, i_{2}, i_{2},i_{2}\}$&$\R\oplus\R\oplus\R^{3}$&$j_{1}, j_{3}$\\&&&&\\\hline&&&&\\
$\E( {\mathfrak A} , {\mathfrak B} )_{8}= \{3,4\} \cup \{1,2,5\}$ &$j_{1}\leq j_{2}<i_{1}\leq i_{2}\leq j_{3}$&$\{i_{1}, i_{1}, i_{2}, j_{3}, j_{3}\}$&$\R^{2}\oplus\R\oplus\R^{2}$&$j_{1}, j_{2}$\\&&&&\\\hline&&&&\\
$\E( {\mathfrak A} , {\mathfrak B} )_{9}= \{3,5\} \cup \{1,2,4\}$&$j_{1}\leq j_{2}<i_{1}\leq j_{3}<i_{2}$&$\{i_{1}, i_{1}, i_{2},i_{2},i_{2}\}$&$\R^{2}\oplus\R^{3}$&$j_{1}, j_{2}, j_{3}$\\&&&&\\\hline&&&&\\
$\E( {\mathfrak A} , {\mathfrak B} )_{10}= \{4,5\} \cup \{1,2,3\}$ &$j_{1}\leq j_{2}\leq j_{3}<i_{1}\leq i_{2}$&$\{i_{1}, i_{1}, i_{2},i_{2},i_{2}\}$&$\R^{2}\oplus\R^{3}$&$j_{1}, j_{2}, j_{3}$\\&&&&\\\hline
\end{tabular}
\end{center}}

\smallskip

\noindent In the first column, we have given the decomposition of $\{1,2,3,4,5\}$ into two subsets, the first with two elements and the second with three. This then gives an ordering of the elements in the set $\{i_{1},i_{2}, i_{3}, i_{4}, i_{5}\}$ which is given in the  second column.  The third column gives the value of the 5-tuple $K = \tilde I\vee\tilde J$. In this tuple, certain entries can be repeated, and this corresponds to a new decomposition of $\R^{5}$.  The fourth column gives this new decomposition of $\R^{5}$ dictated by the repeated indices of $K$. Finally, in each of the decompositions, certain of the indices from $I$ or $J$ appear in the 5-tuple $K$. In the sixth column of Table 1, we list the `free'-variables which do \emph{not} appear in $K$ are listed. It is precisely these free variables which appear in the inner sum on the right-hand side of equation (\ref{Equation6.4by}).


%
Table 1 shows that if $K\in \E( {\mathfrak A} , {\mathfrak B} )$, then $K$ takes one of five forms: 
{
\begin{equation}\label{5.6pk}
\begin{aligned}
&&(k_{1},k_{1},k_{2},k_{2},k_{2}) &&&\text{(decompositions 9 and 10)}&&&&\text{leading to the flag $\R^{2}\oplus\R^{3}$,}\\
&&(k_{1},k_{2},k_{2},k_{3},k_{3}) &&&\text{(decompositions 1, 2, and 5)}&&&&\text{leading to the flag $\R\oplus\R^{2}\oplus\R^{2}$,}\\
&&(k_{1},k_{2},k_{3},k_{3},k_{3}) &&&\text{(decompositions 4 and 7)}&&&&\text{leading to the flag $\R\oplus\R\oplus\R^{3}$,}\\
&&(k_{1},k_{1},k_{2},k_{3},k_{3}) &&&\text{(decomposition 8)}&&&&\text{leading to the flag $\R^{2}\oplus\R\oplus\R^{2}$,}\\
&&(k_{1},k_{2},k_{3},k_{4},k_{4}) &&&\text{(decompositions 3 and 6)}&&&&\text{leading to the flag $\R\oplus\R\oplus\R\oplus\R^{2}$.}
\end{aligned}
\end{equation}}
The outer sum on the right-hand side of (\ref{Equation6.4by}) thus splits into five separate sums:
\begin{equation*}
\begin{split}
\sum_{I\in \E_{2}}\sum_{J\in \E_{3}}[\varphi^{I}]_{I}*[\psi^{J}]_{J}
&=
\sum_{K\in \E_{9}\cup \E_{10}}\,\sum_{I,J\in \widetilde \E(K)}[\varphi^{I}]_{I}*[\psi^{J}]_{J}+
\sum_{K\in \E_{1}\cup \E_{2}\cup \E_{5}}\,\sum_{I,J\in \widetilde \E(K)}[\varphi^{I}]_{I}*[\psi^{J}]_{J}\\
&\quad\quad
+\sum_{K\in \E_{4}\cup \E_{7}}\,\sum_{I,J\in \widetilde \E(K)}[\varphi^{I}]_{I}*[\psi^{J}]_{J}+\sum_{K\in E_{8}}\,\sum_{I,J\in \widetilde \E(K)}[\varphi^{I}]_{I}*[\psi^{J}]_{J}\\
&\quad\quad\quad\quad
+\sum_{K\in E_{3}\cup E_{6}}\,\sum_{I,J\in \widetilde \E(K)}[\varphi^{I}]_{I}*[\psi^{J}]_{J}.
\end{split}
\end{equation*}
Our object is to show that these five sums are flag kernels, each adapted to one of the five flags listed on the right-hand side of (\ref{5.6pk}). To see this, we must show that in each case, the inner infinite sum converges, and has weak cancellation. Let us see why this happens in one case.

\smallskip

\noindent \textbf{Case 1: $K\in E_{9}\cup E_{10}$}.

\smallskip

In this case, $K= \{i_{1}, i_{1}, i_{2}, i_{2}, i_{2}\}$ is fixed, and the inner sum $\sum_{I,J\in \widetilde \E(K)}[\varphi^{I}]_{I}*[\psi^{J}]_{J}$ is over the free variables $\{j_{1}, j_{2}, j_{3}\}$ which satisfy the inequalities
\begin{equation*}
j_{1}\leq j_{2}<i_{1}\leq j_{3}<i_{2} \qquad\text{or}\qquad j_{1}\leq j_{2}\leq j_{3}<i_{1}\leq i_{2}.
\end{equation*}
In order to apply Theorem \ref{Lemma5.4zw}, we need to check that the sum converges to the $K$-dilate of a normalized bump function $\theta^{I,J}$, and moreover that $\theta^{I,J}$  has weak cancellation relative to the decomposition $\R^{2}\oplus \R^{3}$.

To show that the sum over the free variables $\{j_{1}, j_{2}, j_{3}\}$ converges, we want to show that each term in the sum can be bounded by $2^{-\epsilon[(l_{1}-j_{1})+(l_{2}-j_{2})+(l_{3}-j_{3})]}$ where $l_{1}, l_{2}, l_{3} \in \{i_{2}, i_{2}\}$, and $j_{1}\leq l_{1}$, $j_{2}\leq l_{2}$, and $j_{3}\leq l_{3}$. This will follow because, by hypothesis, $\psi^{J}$ has strong cancellation relative to the decomposition $\R\oplus\R^{2}\oplus\R^{2}$. Thus $\psi^{J}$ has cancellation in $x_{1}$, in either $x_{2}$ or $x_{3}$, and in either $x_{4}$ or $x_{5}$. In the variable $x_{1}$, $j_{1}<i_{1}$, and so by Lemma \ref{Lemma4.3qw}, we get a gain of $2^{-\epsilon|i_{1}-j_{1}|}\leq 2^{-\epsilon|i_{1}-j_{2}|}$. If there is cancellation in $x_{2}$, we have $j_{2}<i_{1}$, so we get a gain of $2^{-\epsilon|i_{1}-j_{2}|}$, while if there is cancellation in $x_{3}$, we have $j_{3}\leq i_{2}$ and so we get a gain of $2^{-\epsilon|i_{2}-j_{3}|}$. Finally, if there is cancellation in $x_{4}$ or $x_{5}$, we have $j_{3}<i_{2}$, and so we get a gain of $2^{-\epsilon|i_{2}-j_{3}|}$. Taking the best of these estimates, we see that the size of $[\varphi^{I}]_{I}*[\psi^{J}]_{J}$ is dominated by a constant times $2^{-\epsilon[|i_{1}-j_{1}|+|i_{1}-j_{2}|+|i_{2}-j_{3}|]}$. Thus in this case we can take $l_{1}= i_{1}$, $l_{2}=i_{1}$, and $l_{3}=i_{2}$.

The key points in this convergence argument are the following:

\begin{enumerate}[(a)]

\item If $f_{s}$ is a free variable, it does not appear in $K$. Since the entries  $k_{s}\in K$ are the maxima of the corresponding entries of $i_{s}\in \tilde I$ and $j_{s}\in \tilde J$, the free variable must satisfy $f_{s}\leq k_{s}$. 

\item Since we will sum over the free variable $f_{s}$, but the variable $k_{s}\in K$ is fixed, it suffices to show that there is a gain $2^{-\epsilon(k_{s}-f_{s})}$. 

\item The function $\varphi^{I}$ or $\psi^{J}$ with the free variable $f_{s}$ may not necessarily have cancellation in the variable $x_{s}$. (For example, the free variable $j_{2}$ comes from the function $\psi^{(j_{1},j_{2},j_{2},j_{3},j_{3})}$, and we only know that this function has cancellation in the variable $x_{2}$ \emph{or} the variable $x_{3}$). However, if there is no cancellation in the free variable, there is a \emph{smaller} free variable where there is cancellation, and where the corresponding element of $K$ is the same. (In our example, $\psi^{(j_{2},j_{2},j_{3})}$ has cancellation in $x_{1}$, and $k_{1}=k_{2}$).

\end{enumerate}

To see that the sum of the terms $[\varphi^{I}]_{I}*[\psi^{J}]_{J}$ has weak cancellation relative to the decomposition $\R^{2}\oplus \R^{3}$, we again use Lemma \ref{Lemma4.3qw}. We only need to observe that  either $\varphi^{I}$ or $\psi^{J}$ has cancellation in one of the variables $\{x_{1}, x_{2}\}$, and also that either $\varphi^{I}$ or $\psi^{J}$ has cancellation in one of the variables $\{x_{3}, x_{4}, x_{5}\}$.  But this is clear: for example,  $\psi^{J}$ has strong cancellation relative to the decomposition $\R\oplus\R^{2}\oplus\R^{2}$, and so has cancellation in $x_{1}$, and  $\varphi^{I}$ has strong cancellation relative to the decomposition $\R^{2}\oplus\R^{3}$, and so has cancellation in one of the variables $\{x_{3}, x_{4}, x_{5}\}$. 

\subsection{The general decomposition}\label{General}\quad

\smallskip

Now let us return to the general situation. Suppose $\mathcal F_{1}$ and $\mathcal F_{2}$ are standard flags arising from two (in general different) decompositions we label as $(\mathfrak A)$ and $(\mathfrak B)$:
\begin{align*}
(\mathfrak A):\quad &\R^{N}= \R^{a_{1}}\oplus \cdots \oplus \R^{a_{n}},\\
(\mathfrak B):\quad&\R^{N}= \R^{b_{1}}\,\oplus \cdots \oplus \R^{b_{m}}.
\end{align*}  
Let $\mathcal K_{1}$ and $\K_{2}$ be flag kernels adapted to the flag $\mathcal F_{1}$ and $\mathcal F_{2}$. We only need to consider the parts of these kernels given by sums of dilates of normalized bump functions with strong cancellation. (That is, for each kernel we focus on the part called $\K_{0}$ in Theorem \ref{Lemma2.3} and disregard the other terms since they correspond to coarser flags). Thus we  can write 
\begin{align}\label{Eqn8.7wer}
\K_{1}&= \sum_{I\in \mathcal E_{n}}[\varphi^{I}]_{I},&
\K_{2}&= \sum_{J\in \mathcal E_{m}}[\psi^{J}]_{J},
\end{align}
where each $\varphi^{I}$ has strong cancellation relative to the decomposition $\mathfrak A$ and each $\psi^{J}$ has strong cancellation relative to the decomposition $\mathfrak B$. Let
\begin{align*}
\K_{1}^{F}&= \sum_{I\in F\subset \mathcal E_{n}}[\varphi^{I}]_{I}, & \K_{2}^{G}&= \sum_{J\in G\subset \mathcal E_{m}}[\psi^{J}]_{J},
\end{align*}
where $F \subset \mathcal E_{n}$ and $G\subset \mathcal E_{m}$ are finite subsets. We study the double sum
\begin{equation}\label{E5.6gh}
\K_{1}^{F}*\K_{2}^{G}=\sum_{I\in F\subset \mathcal E_{n}}\sum_{J\in G\subset \in \mathcal E_{m}}[\varphi^{I}]_{I}*[\psi^{J}]_{J}. 
\end{equation} 
Let $\F_{0}$ be the coarsest flag which is finer than both $\F_{1}$ and $F_{2}$. 

\begin{theorem} \label{Theorem8.4}
Let $\K_{1}$ and $\K_{2}$ be flag kernels given in (\ref{Eqn8.7wer}). Then
\begin{equation*}
 \lim_{\substack{F \nearrow \mathcal E_{n}\\G\nearrow \mathcal E_{m}}}\K_{1}^{F}*\K_{2}^{G}= \K_{1}*\K_{2}
\end{equation*} converges in the sense of distributions to a finite sum of flag kernels $\sum\K_{\mu}$, each of which is adapted to a flag $\F_{\mu}$ which is equal to or coarser than the flag $\F_{0}$.
\end{theorem}

\smallskip

Before outlining the proof, we review our notation. If $\x\in \R^{N}$, we can write $\x=(\x_{1}', \ldots, \x_{n}')$ with $\x_{j}'=(x_{p_{j}'}, \ldots, x_{q_{j}'})\in \R^{a_{j}}$, or $\x=(\x_{1}'', \ldots, \x_{m}'')$ with $\x_{k}''=(x_{p_{k}''}, \ldots, x_{q_{k}''})\in \R^{b_{k}}$. We let $J'_{j}= \{p_{j}', \ldots, q_{j}'\}$ and $J_{k}''=\{p_{k}'', \ldots, q_{k}''\}$ so that $\{1, \ldots, N\} = \bigcup_{j=1}^{n}J_{j}' = \bigcup_{k=1}^{m}J_{k}''$. Define 
\begin{align*}
&&&&\sigma_{\mathfrak A}&:\{1, \ldots,N\}\to \{1, \ldots, n\}&&\text{so that} &l&\in J_{\sigma_{\mathfrak A}(l)}'\,\,\text{for}\,\,1 \leq l \leq N;&&&&\\
&&&&\tau_{\mathfrak B}&:\{1, \ldots, N\}\to \{1, \ldots, m\}&&\text{so that}& l&\in J_{\sigma_{\mathfrak B}(l)}''\,\,\text{for}\,\,1 \leq l \leq N. &&&&
\end{align*} 
If $Q_{j}'$ is the homogeneous dimension of $\R^{a_{j}}$ and $Q_{k}''$ is the homogeneous dimension of $\R^{b_{k}}$, then
\begin{align*}
Q_{j}' &= d_{p_{j}'}+ \cdots + d_{q_{j}'}= \sum_{l\in J_{j}'}d_{l},\\
Q_{k}''&= d_{p_{k}''}+ \cdots + d_{q_{k}''}= \sum_{l \in  J_{k}''}d_{l}.
\end{align*}
If $I\in \E_{n}$ and $J\in \E_{m}$, the notation $[\varphi^{I}]_{I}$ and $[\psi^{J}]_{J}$ refers to the families of dilations 
\begin{align*}
[\varphi^{I}](\x) &= 2^{-[Q_{1}'i_{1}+\cdots +Q_{n}'i_{n}]}\,\,\,\varphi(2^{-i_{1}}\cdot \x_{1}', \ldots, 2^{-i_{n}}\cdot \x_{n}'),\\
[\psi^{J}](\x) &= 2^{-[Q_{1}''j_{1}+\cdots +Q_{m}''j_{m}]}\varphi(2^{-j_{1}}\cdot \x_{1}'', \ldots, 2^{-j_{m}}\cdot \x_{m}'')
\end{align*}

In order to compare multi-indices $I=(i_{1}, \ldots, i_{n})\in \mathcal E_{n}$ and $J=(j_{1}, \ldots, j_{m})\in \mathcal E_{m}$ which parameterize different families of dilations, we  identify them with multi-indices of length $N$ with repeated entries. Thus we define $p_{\mathfrak A}:\mathcal E_{n}\to E_{N}$ and $p_{\mathfrak B}:\mathcal E_{m}\to E_{N}$ so that $p_{\mathfrak A}(I)$ is the $N$-tuple with $i_{1}$ repeated $a_{1}$ times, $i_{2}$ repeated $a_{2}$ times, \textit{etc.} We define $p_{\mathfrak B}$  analogously. Thus
\begin{equation*}\label{E5.7gh}
\begin{split}
p_{\mathfrak A}(I)&=I=(I_{1}, \ldots, I_{N}) = \big(\,\overset{a_{1}}{\overbrace{i_{1}, \ldots\ldots\ldots,i_{1}}}\,,\ldots,\,\overset{a_{r}}{\overbrace{i_{r}, \ldots,i_{r}}}\,, \ldots \ldots\ldots, \,\overset{a_{n}}{\overbrace{i_{n}, \ldots,i_{n}}}\,\big),\\
p_{\mathfrak B}(J)&=J=(J_{1}, \ldots, J_{N})= \big(\,\overset{b_{1}}{\overbrace{j_{1}, \ldots,j_{1}}}\,,\ldots\ldots\ldots\ldots\,,\overset{b_{s}}{\overbrace{j_{s}, \ldots\ldots,j_{s}}}\,, \ldots , \,\overset{b_{m}}{\overbrace{j_{m}, \ldots\ldots,j_{m}}}\,\big).
\end{split}
\end{equation*}
Explicitly, $p_{\mathfrak A}(I)=(I_{1}, \ldots, I_{N})$ and $p_{\mathfrak B}(J)=(J_{1}, \ldots, J_{N})$ where $I_{l}=i_{\sigma(l)}$ and $J_{l}= j_{\tau(l)}$.  Next set $K =(K_{1}, \ldots, K_{N})= p_{\mathfrak A}(I) \vee p_{\mathfrak B}(J)=(k_{1}, \ldots, k_{N})$. This means that for $1 \leq l \leq N$
\begin{align}\label{Eqn8.9hjk}
K_{l}= i_{\sigma(l)}\vee j_{\tau(l)}= \max\big\{i_{\sigma(l)}, j_{\tau(l)}\big\}
=
\begin{cases}
j_{\tau(l)}  & \text{if $i_{\sigma(l)}\leq j_{\tau(l)}$}\\\\
i_{\sigma(l)}& \text{if $j_{\tau(l)} < i_{\sigma(l)}$}
\end{cases}.
\end{align}
Note that we can then define a map $\pi_{I,J}:\{1, \ldots, N\}\to \{1, \ldots, n, n+1, \ldots, n+m\}$ so that
\begin{align}\label{Eqn8.10hjk}
\pi_{I,J}(l)=
\begin{cases}
\ell \in \{1, \ldots n\} & \text{if \,\,$K_{l}= i_{\sigma(\ell)}$};\\\\
n+\ell \in \{n+1, \ldots, n+m\} &\text{if \,\,$K_{l}= j_{\tau(\ell)}$}.
\end{cases}
\end{align}

\bigskip

Now let us outline the proof of Theorem \ref{Theorem8.4}. For each $I\in \mathcal E_{n}$ and $J\in \mathcal E_{m}$ it follows from Lemma \ref{Lemma5.1} that there is a function $\theta^{I,J}\in \mathcal C^{\infty}_{0}(\R^{N})$, normalized relative to $\varphi^{I}$ and $\psi^{J}$, so that 
\begin{equation}
[\varphi^{I}]_{I}*[\psi^{J}]_{J}= [\theta^{I,J}]_{p_{\mathfrak A}(I)\vee p_{\mathfrak B}(J)},
\end{equation}
and hence equation (\ref{E5.6gh}) can be written
\begin{equation}
\K_{1}^{F}*\K_{2}^{G}= \sum_{(I,J)\in F\times G\subset \E_{n}\times \E_{m}}[\theta^{I,J}]_{p_{\mathfrak A}(I)\vee p_{\mathfrak B}(J)}.
\end{equation}
We analyze this sum by decomposing the set $\E_{n}\times \E_{m}$ into  disjoint subsets. Let $\mathfrak P(n,m)$ denote the set of permutations $\mu:\{1, \ldots, n,n+1, \ldots, n+m\}\to \{1, \ldots, n,n+1, \ldots, n+m\}$ which preserve the order of the first $n$ elements $\{1, \ldots, n\}$ and of the last $m$ elements $\{n+1 , \ldots, n+m\}$. (Explicitly, this means that if $\mu \in \mathfrak P(n,m)$,  then $1\leq s<t\leq n$ implies $\mu(s)<\mu(t)$ and $n+1\leq s<t\leq n+m$ implies $\mu(s)<\mu(t)$). This corresponds to the ten cases in the example studied in Section \ref{Example}. The cardinality of $\mathfrak P(n,m)$ is $\binom{n+m}{n}$. Let $I\in \E_{n}$ and $J\in \E_{m}$, and let us write 
\begin{align*}
I=(\alpha_{1}, \ldots, \alpha_{n})\qquad&\text{and}\qquad J=(\alpha_{n+1}, \ldots, \alpha_{n+m}), \qquad \text{so that}\\
\alpha_{1}\leq \alpha_{2}\leq \cdots \leq \alpha_{n}\qquad&\text{and}\qquad 
\alpha_{n+1}\leq \alpha_{n+2}\leq \cdots \leq \alpha_{n+m}.
\end{align*}
Then let $L(I,J)$ be the (weakly) increasing rearrangement of the set $I\cup J = \{\alpha_{1}, \ldots, \alpha_{n+m}\}$ so that
\begin{equation}\label{Eqn8.10ert}
\begin{aligned}
&\text{if \quad $1\leq r<s\leq n$} && \text{then} && \text{$\alpha_{r}$ comes to the left of $\alpha_{s}$,}\\
&\text{if  \quad $n+1\leq r<s\leq n+m$} && \text{then} && \text{$\alpha_{r}$ comes to the left of $\alpha_{s}$, }\\
&\text{if \quad $1\leq r\leq n$, $n+1\leq s\leq n+m$ and $\alpha_{r}<\alpha_{s}$} && \text{then} &&\text{$\alpha_{r}$ comes to the left of $\alpha_{s}$,}\\
&\text{if \quad $1\leq r\leq n$, $n+1\leq s\leq n+m$ and $\alpha_{s}\leq\alpha_{r}$} && \text{then} && \text{$\alpha_{s}$ comes to the left of $\alpha_{r}$. }
\end{aligned}
\end{equation}
This rearrangement of $\{\alpha_{1}, \ldots, \alpha_{n+m}\}$ is given by $L(I,J) = \{\alpha_{\mu(1)}, \ldots, \alpha_{\mu(n+m)}\}$ where $\mu$ is a permutation of the set of subscripts $\{1, \ldots, n,n+1, \ldots, n+m\}$, and it follows from (\ref{Eqn8.10ert}) that $\mu \in \mathfrak P(n,m)$. In this way we associate to each pair $(I,J)\in \E_{n}\times \E_{m}$ a unique $\mu= \mu(I,J)\in \mathfrak P(n,m)$. Conversely, for each $\mu \in\mathfrak P(n,m)$, let
\begin{equation}\label{Eqn8.14wer}
\E_{n,m}(\mu) = \big\{(I,J)\in \E_{n}\times\E_{m}\,\big\vert\, L(I,J) = \{\alpha_{\mu(1)}, \ldots, \alpha_{\mu(n+m)}\}\big\}.
\end{equation}
It follows that we have a disjoint decomposition $\E_{n}\times\E_{m}= \bigcup_{\mu\in \mathfrak P(m,n)}\E_{n,m}(\mu)$,
and we can write equation (\ref{E5.6gh}) as
\begin{equation}\label{Eqn8.15ert}
\K_{1}^{F}*\K_{2}^{G}= \sum_{\mu\in \mathfrak P(m,n)}\Big( \sum_{(I,J)\in \E_{n,m}(\mu)\cap (F\times G)}[\theta^{I,J}]_{p_{\mathfrak A}(I)\vee p_{\mathfrak B}(J)}\Big).
\end{equation}
Now let
\begin{equation}\label{E5.9gh}
\begin{split}
E_{N}(\mu)&= \left\{K \in E_{N}\,\big\vert\, \text{$K = p_{\mathfrak A}(I)\vee p_{\mathfrak B}(J)$ where $(I,J)\in \E_{n,m}(\mu)$}\right\}.
\end{split}
\end{equation}
Note that in general $E_{N}(\mu)$ is a proper subset of $E_{N}$, and an element $K\in E_{N}(\mu)$ can be represented in many ways as $p_{\mathfrak A}(I) \vee p_{\mathfrak B}(J)$ with $(I,J)\in \E_{n,m}(\mu)$. We can then rewrite (\ref{Eqn8.15ert}) as
\begin{equation}\label{Eqn8.17yui}
\begin{aligned}
\K_{1}^{F}*\K_{2}^{G}
&=
\sum_{\mu\in \mathfrak P(m,n)}\Big(
\sum_{K\in E_{N}(\mu)}\Big(
\sum_{\substack{(I,J)\in \E_{n,m}(\mu)\cap (F\times G)\\ p_{\mathfrak A}(I)\vee p_{\mathfrak B}(J) = K}}[\theta^{I,J}]_{K}\Big)\Big)\\
&=
\sum_{\mu\in \mathfrak P(m,n)}\Big(
\sum_{K\in E_{N}(\mu)}
\Big[\sum_{\substack{(I,J)\in \E_{n,m}(\mu)\cap (F\times G)\\ p_{\mathfrak A}(I)\vee p_{\mathfrak B}(J) = K}}\theta^{I,J}\Big]_{K}\Big)
\end{aligned}
\end{equation}
We will prove in Lemma \ref{Cor5.5ty} below that because the functions $\{\varphi^{I}\}$ and $\{\psi^{J}\}$ have strong cancellation, the innermost sum
\begin{equation}
\sum_{\substack{(I,J)\in \E_{n,m}(\mu)\\ p_{\mathfrak A}(I)\vee p_{\mathfrak B}(J) = K}}\theta^{I,J}= 
\lim_{\substack{F\nearrow \E_{n}\\G\nearrow \E_{m}}}\sum_{\substack{(I,J)\in \E_{n,m}(\mu)\cap (F\times G)\\ p_{\mathfrak A}(I)\vee p_{\mathfrak B}(J) = K}}\theta^{I,J}
\end{equation}
converges to a function $\Theta^{K}\in \mathcal C^{\infty}_{0}(\R^{N})$ which is normalized relative to the families $\{\varphi^{I}\}$ and $\{\psi^{j}\}$. From this it follows from (\ref{Eqn8.17yui}) that
\begin{equation}\label{Eqn8.19hjk}
\K_{1}*\K_{2}=\lim_{\substack{F\nearrow \E_{n}\\G\nearrow \E_{m}}} \K_{1}^{F}*\K_{2}^{G} 
= 
\sum_{\mu\in\mathfrak P(n,m)}\Big(\sum_{K\in E_{N}(\mu)}[\Theta^{K}]_{K}\Big).
\end{equation}
We will also see in Lemma \ref{Cor5.5ty} below that for each fixed $\mu$, the functions $\{\Theta^{K}\}$ for $K\in E_{N}(\mu)$ have weak cancellation relative to a decomposition of $\R^{N}$ depending on $\mu$, $\R^{N}= \R^{c_{1}}\oplus \cdots \oplus \R^{c_{r}}$,  and hence the inner sum  on the right hand side of (\ref{Eqn8.19hjk}) is a flag kernel relative to the corresponding standard flag $\mathcal F_{\mu}$. Once this is done, we will have established Theorem \ref{Theorem8.4}.

\bigskip

We now turn to the details of the proof. We begin by studying the $N$-tuple $K=(K_{1}, \ldots, K_{N}) =  p_{\mathfrak A}(I)\vee p_{\mathfrak B}(J)$ if $(I,J)\in \E_{n,m}(\mu)$.  Partition $K$ into disjoint subsets of consecutive entries where two successive elements $K_{l}$ and $K_{l+1}$ belong to the same subset if and only if  either
\begin{enumerate}[(i)]

\smallskip

\item \label{78} $K_{l}= i_{\sigma(l)}\geq j_{\tau(l)}$, $K_{l+1}= i_{\sigma(l+1)}\geq j_{\tau(l+1)}$, and $\sigma(l)= \sigma(l+1)$; \quad or

\smallskip

\item \label{79} $K_{l}= j_{\tau(l)}\geq i_{\sigma(l)}$, $K_{l+1}= j_{\tau(l+1)}\geq i_{\sigma(l+1)}$, and $\tau(l)= \tau(l+1)$.

\smallskip
\end{enumerate}
In particular, if two successive elements $K_{l}$ and $K_{l+1}$ belong to the same subset, they must be equal. Thus we write 
\begin{align*}
K &= 
(K_{1}, \ldots, K_{N})
=\big(\{K_{\alpha_{1}}, \ldots, K_{\beta_{1}}\}, \,\{K_{\alpha_{2}}, \ldots, K_{\beta_{2}}\}, \ldots, \{K_{\alpha_{r}}, \ldots, K_{\beta_{r}}\}\big)
\end{align*}
where 
\begin{align*}
K_{\alpha_{1}}&= K_{\alpha_{1}+1}= \cdots = K_{\beta_{1}},&
K_{\alpha_{2}}&= K_{\alpha_{2}+1}= \cdots = K_{\beta_{2}}, &
&\,\,\,\cdots &
K_{\alpha_{r}}&= K_{\alpha_{r}+1}= \cdots = K_{\beta_{r}}.
\end{align*} 
We can also write
\begin{align}\label{Eqn5.12vn}
K= (K_{1}, \ldots, K_{N}) = \big(\,\overset{c_{1}}{\overbrace{k_{1}, \ldots,k_{1}}}\,;\,\overset{c_{2}}{\overbrace{k_{2}, \ldots,k_{2}}}\,; \ldots ; \,\overset{c_{r}}{\overbrace{k_{r}, \ldots,k_{r}}}\,\big),
\end{align} 
so that $c_{1}+ c_{2}+ \cdots + c_{r}=N$. Note that $1 \leq r <m+n$ since either $i_{1}$ or $j_{1}$ does not appear in $K$. 

We have the following properties of this decomposition.

\begin{proposition}\label{Prop8.5lkj}
The integers $\{\alpha_{1}, \beta_{1}, \ldots, \alpha_{r}, \beta_{r}\}$ depend only on the permutation $\mu$ and are independent of the choice of $(I,J)\in \E_{n,m}(\mu)$. In fact, let $\widetilde K_{l}=\{K_{\alpha_{l}}, \ldots, K_{\beta_{l}}\}$ be one of the subsets of consecutive indices in $K= p_{\mathfrak A}(I)\vee p_{\mathfrak B}(J)\in \E_{n,m}(\mu)$. Then:
\begin{enumerate}[{\rm(1)}]

\smallskip

\item \label{Prop8.5lkj1}
The starting position $\alpha_{l}$ coincides either with the starting position of one of the subsets of $p_{\mathfrak A}(I)$ or with the position of one of the subsets of $p_{\mathfrak B}(J)$. More precisely,

\smallskip

\begin{enumerate}[{\rm(\ref{Prop8.5lkj1}a)}]

\item \label{Prop8.5lkj1a}
If $K_{\alpha_{l}}= i_{\sigma(\alpha_{l})}$, then $\sigma(\alpha_{l})>\sigma(\alpha_{l}-1)$, so the start of $\widetilde K_{l}$ coincides with the start of the index $i_{\sigma(\alpha_{l})}$ in $p_{\mathfrak A}(I)$. 

\smallskip

\item \label{Prop8.5lkj1b}
If $K_{\alpha_{l}}= j_{\tau(\alpha_{l})}$, then $\tau(\alpha_{l})>\tau(\alpha_{l}-1)$, so the start of $\widetilde K_{l}$ coincides with the start of the index $j_{\tau(\alpha_{l})}$ in $p_{\mathfrak B}(J)$. 

\smallskip

\end{enumerate}

\item \label{Prop8.5lkj2}
If the ending position of $\widetilde K_{l}$ does not coincide with the end of the corresponding segment of $p_{\mathfrak A}(I)$ or $p_{\mathfrak B}(J)$, then the entries of the segment which do not appear in $\widetilde K_{l}$ are bounded above by the corresponding entries of $K$. More precisely

\begin{enumerate}[{\rm(\ref{Prop8.5lkj2}a)}]

\smallskip

\item \label{Prop8.5lkj2a}
Suppose that $K_{\alpha_{l}}= i_{\sigma(\alpha_{l})}$, so that $K_{t}=i_{\sigma(\alpha_{l})}$ for $\alpha_{l}\leq t \leq \beta_{l}$. Suppose that $I_{t} = I_{\alpha_{l}}$ for $\alpha_{l}\leq t \leq \gamma$ and that $\gamma > \beta_{l}$. Then $I_{t}\leq J_{t}=K_{t}$ for $\beta_{l}+1\leq t \leq \gamma$.

\smallskip

\item \label{Prop8.5lkj2b}
Suppose that $K_{\alpha_{l}}= j_{\tau(\alpha_{l})}$, so that $K_{t}=j_{\tau(\alpha_{l})}$ for $\alpha_{l}\leq t \leq \beta_{l}$. Suppose that $J_{t} = J_{\alpha_{l}}$ for $\alpha_{l}\leq t \leq \gamma$ and that $\gamma > \beta_{l}$. Then $J_{t}\leq I_{t}=K_{t}$ for $\beta_{l}+1\leq t \leq \gamma$.
\end{enumerate}
\end{enumerate}
\end{proposition}

\begin{proof}
We begin by establishing part (\ref{Prop8.5lkj1a}). If it were not true, then since $j_{\tau(\alpha_{l}-1)}\leq j_{\tau(\alpha_{l})}$, we would have 
\begin{equation*}
j_{\tau(\alpha_{l}-1)}\leq j_{\tau(\alpha_{l})}\leq i_{\sigma(\alpha_{l})}= i_{\sigma(\alpha_{l}-1)},
\end{equation*} 
and by condition (\ref{78}) in the definition of $\widetilde K_{l}$ it would follow that $K_{\alpha_{l}}$ and $K_{\alpha_{l}-1}$ belong to the same subset. Part (\ref{Prop8.5lkj1b}) follows in the same way. To establish part (\ref{Prop8.5lkj2a}), observe that since $J_{\beta_{l}+1}\leq J_{t}$, it suffices to show this for $t=\beta_{l}+1$. But if $J_{\beta_{l}+1}<I_{\beta_{l}+1}$, it would follow from (\ref{78}) that $K_{\beta_{l}+1}$ belongs to the same subset as $K_{\alpha_{l}}$. Part (\ref{Prop8.5lkj2b}) follows in the same way.

Now parts (\ref{Prop8.5lkj1}) and (\ref{Prop8.5lkj2}) show that the positions where we decompose $K$ depend on the \emph{ordering} of the entries of $p_{\mathfrak A}(I)$ and $p_{\mathfrak B}(J)$, and not on the entries themselves. This shows that the decomposition depends only on $\mu$, which completes the proof.
\end{proof}

\begin{definition}
Let $\mu\in \mathfrak P(n,m)$. 

\begin{enumerate}[{\rm(1)}]

\smallskip

\item It follows from Proposition \ref{Prop8.5lkj} that the permutation $\mu$ determines $r$ and the integers $\{c_{1}, \ldots, c_{r}\}$. Thus $\mu$ determines  the decomposition $\R^{N}= \R^{c_{1}}\oplus \cdots \oplus \R^{c_{r}}$. We let $\F_{\mu}$ denote the corresponding standard flag $(0) \subseteq \R^{c_{r}} \subseteq \R^{c_{r-1}}\oplus \R^{c_{r}}\subseteq \cdots \subseteq \R^{c_{2}}\oplus \cdots \oplus \R^{c_{n}}\subseteq \R^{c_{1}}\oplus \cdots \oplus \R^{c_{r}}= \R^{N}$. For $\x\in \R^{N}$, we write $\x= (\widetilde\x_{1}, \ldots, \widetilde \x_{r})$, and we let $\{\widetilde J_{1}, \ldots, \widetilde J_{r}\}$ denote the corresponding sets of subscripts so that $x_{l}$ is a coordinate in $\R^{c_{k}}$ if and only if $l\in \widetilde J_{k}$.

\smallskip

\item Let $I=(i_{1}, \ldots, i_{n}), J=(j_{1}, \ldots, j_{m}) \in \mathcal E_{n,m}(\mu)$. An index $i_{l}$  is \emph{free} if $I_{t}\leq J_{t}$ for all $t$ such that $\sigma(t) = l$. An index $j_{l}$  is \emph{free} if $J_{t}<I_{t}$ for all $t$ such that $\sigma(t) = l$. In particular, a free index does \emph{not} appear in the set $ K=p_{\mathfrak A}(I)\vee p_{\mathfrak B}(J)\in E_{N}(\mu)$. 
\smallskip
\end{enumerate}
Note that whether or not an index is free depends only on the choice of $\mu$, and not on the choice of $(I,J)\in \E_{n,m}(\mu)$. The number of free elements is equal to $m+n-r$, and $1\leq m+n-r <m+r$. 
\end{definition}

\smallskip

\begin{proposition}\label{Lemma5.5gs} 
Fix $\mu\in \mathfrak P(n,m)$, and let  $(I,J)\in \E_{n,m}(\mu)$ with $I=(i_{1}, \ldots, i_{n})$ and  $J=(j_{1}, \ldots, j_{m})$. Let $\varphi^{I}, \psi^{J}\in \mathcal C^{\infty}_{0}(\R^{N})$, and suppose that $\varphi^{I}$ has strong cancellation relative to the decomposition $\R^{N}= \R^{a_{1}}\oplus \cdots \oplus \R^{a_{n}}$, and that $\psi^{J}$ has strong cancellation relative to the decomposition $\R^{N}= \R^{b_{1}}\oplus \cdots \oplus \R^{b_{m}}$.

\begin{enumerate}[{\rm(1)}]

\smallskip

\item \label{Lemma5.5gs1}
Suppose that an index $i_{l}$ is free, and that $i_{l}=I_{r}=I_{r+1}=\cdots =I_{s}$ is the corresponding group of indices in $p_{\mathfrak A}(I)$, so that $I_{r-1}=i_{l-1}$ and $I_{s+1}=i_{l+1}$.  Then the function $\varphi^{I}$ has cancellation in the variables  $\{x_{I_{r}}, \ldots, x_{I_{s}}\}$, and $I_{t}\leq J_{t}$ for $r\leq t\leq s$.\footnote{It follows from Lemma \ref{Lemma1.12} that $\varphi^{I}$ can be written as a sum of functions each of which has cancellation in one of the variables $\{x_{I_{r}}, \ldots, x_{I_{s}}\}$.}

\smallskip

\item \label{Lemma5.5gs2}
Suppose that an index $j_{l}$ is free, and that $\{j_{l}=J_{r}=J_{r+1}=\cdots =J_{s}\}$ is the corresponding group of indices in $p_{\mathfrak B}(J)$, so that $J_{r-1}=j_{l-1}$ and $J_{s+1}=j_{l+1}$.  Then the function $\psi^{J}$ has cancellation with respect to one of the variables $\{x_{J_{r}}, \ldots, x_{J_{s}}\}$, and $J_{t}\leq I_{t}$ for $r\leq t \leq s$.\footnote{It follows from Lemma \ref{Lemma1.12} that $\psi^{J}$ can be written as a sum of functions each of which has cancellation in one of the variables $\{x_{J_{r}}, \ldots, x_{J_{s}}\}$.}

\smallskip

\item \label{Lemma5.5gs3}
Let $\R^{N}= \R^{c_{1}}\oplus \cdots \oplus \R^{c_{r}}$ be the decomposition corresponding to $\mu$, and let $\{\widetilde J_{1}, \ldots, \widetilde J_{r}\}$ be the corresponding sets of subscripts. Let $1 \leq l \leq r$.

\begin{enumerate}[{\rm(a)}]

\smallskip

\item \label{Lemma5.5gs3a}
Suppose that $K_{\alpha_{l}}= \cdots = K_{\beta_{l}}= i_{\ell}$ so that $\sigma(\alpha_{l}) = \cdots = \sigma(\beta_{l}) = \ell$. Then either $\varphi^{I}$ has cancellation in a variables $x_{t}$ with $t\in \widetilde J_{l}$, or  $\varphi^{I}$ has cancellation in a coordinate $x_{t}$ with $t>\beta_{l}$, in which case $K_{t}=J_{t}\geq I_{t}=i_{\ell}$.

\smallskip

\item \label{Lemma5.5gs3b}
Suppose that $K_{\alpha_{l}}= \cdots = K_{\beta_{l}}= j_{\ell}$ so that $\tau(\alpha_{l}) = \cdots = \tau(\beta_{l}) = \ell$. Then either $\psi^{J}$ has cancellation in a variables $x_{t}$ with $t\in \widetilde J_{l}$, or  $\psi^{J}$ has cancellation in a coordinate $x_{t}$ with $t>\beta_{l}$, in which case $K_{t}=I_{t}\geq J_{t}=j_{\ell}$.

\end{enumerate}
\end{enumerate}
\end{proposition}

\begin{proof}
To prove assertion (1), note that since $i_{l}$ is free, it does not appear in $K$. Hence for any $r\leq t \leq s$, $i_{l}\neq \max\{I_{t}, J_{t}\}= \max\{i_{l}, J_{t}\}$, and so $J_{\ell}\geq i_{\ell}= I_{t}$.  Since $\varphi^{I}$ is assumed to have strong cancellation and $\{x_{I_{r}}, \ldots, x_{I_{s}}\}$ are precisely the variables corresponding to the index $i_{r}$, this establishes (1). The proof of assertion (2) proceeds in the same way.

To prove assertion (\ref{Lemma5.5gs3a}), note that by Proposition \ref{Prop8.5lkj}, part (\ref{Prop8.5lkj1a}), the only way in which it is possible for $\varphi^{I}$ not to have cancellation in a variable $x_{t}$ with $t\in \widetilde J_{l}$ is if $I_{t} = I_{\alpha_{l}}$ for $\alpha_{l}\leq t \leq \gamma$ and that $\gamma > \beta_{l}$. But then the conclusion follows from Proposition \ref{Prop8.5lkj}, part (\ref{Prop8.5lkj2a}). The proof of assertion (\ref{Lemma5.5gs3b}) follows in the same way.
\end{proof}

\goodbreak

\begin{lemma}\label{Cor5.5ty} 
Fix $\mu\in \mathfrak P(n,m)$.
\begin{enumerate}[{\rm(1)}]

\smallskip

\item \label{Cor5.5ty1}
Let $K\in E_{N}(\mu)$. Then there exists $\Theta^{K}\in \mathcal C^{\infty}_{0}(\R^{N})$, normalized relative to the families $\{\varphi^{I}\}$ and $\{\psi^{J}\}$ so that the sum
\begin{equation*}
\sum_{\substack{(I,J)\in \E_{n,m}(\mu)\\p_{\mathfrak A}(I)\vee p_{\mathfrak B}(J)=K}}[\varphi^{I}]_{I}*[\psi^{J}]_{J}
\end{equation*}
converges (uniformly) to $[\Theta^{K}]_{K}$. 

\smallskip

\item \label{Cor5.5ty2}
The function $\Theta^{K}$ has weak cancellation relative to the decomposition of $\R^{N}$ corresponding to $\mu$.
\end{enumerate}
\end{lemma}

\begin{proof}
Let $(I,J) \in \E_{n,m}(\mu)$. 
Suppose that $\varphi^{I}$ has cancellation in the variables $\{x_{r_{1}}, \ldots, x_{r_{n}}\}$ with $r_{l}\in J_{l}'$, and that $\psi^{J}$ has cancellation in the variables $\{x_{s_{1}}, \ldots, x_{s_{m}}\}$ with $s_{l}\in J_{l}''$. Let
\begin{align*}
&&A_{0}&=\big\{l\in \{1, \ldots,n\}\,\big\vert\, i_{r_{l}}<j_{r_{l}}\big\},
&
B_{0}&=\big\{l\in \{1, \ldots,m\}\,\big\vert\, j_{s_{l}}<i_{s_{l}}\big\}.&&
\end{align*}
Then it follows from Proposition \ref{Lemma5.5gs} that if $i_{l}$ is a free index, $l\in A_{0}$ and if $j_{l}$ is a free index, then $l\in B_{0}$. On the other hand, according to Lemma \ref{Lemma4.3qw}, we can write each $[\varphi^{I}]_{I}*[\psi^{J}]_{J}$ as a finite sum of terms of the form
\begin{align}\label{Eqn8.21okm}
\prod_{\substack{s\in A'}}
2^{-\epsilon |j_{l_{s}}-i_{l_{s}}|}
\prod_{\substack{t\in B'}}
2^{-\epsilon|i_{m_{t}}-j_{m_{t}}|}
\prod_{\substack{s\in A''}}
2^{-\epsilon|i_{l_{s}+1}-i_{l_{s}}|}
\prod_{\substack{t\in B''}}
2^{-\epsilon|j_{m_{t}+1}-j_{m_{t}}|}
\prod_{\substack{s\in A'''}}\partial_{l_{s}}\prod_{t\in B'''}\partial_{m_{t}}[\widetilde \theta]
\end{align}
where 

\begin{enumerate}

\smallskip

\item $A'$, $A''$, $A'''$ are disjoint subsets of $\{1, \ldots, n\}$ with $A'\cup A'' \cup A'''=\{1, \ldots, n\}$; 

\smallskip

\item $B'$, $B''$, $B'''$ are disjoint subsets of $\{1, \ldots, m\}$ with $B'\cup B'' \cup B'''=\{1, \ldots, m\}$;

\smallskip

\item 
$
A'\subset A_{0}\subset A'\cup A''$ and $B'\subset B_{0}\subset B'\cup B''$.

\smallskip

\item each function $\widetilde \theta$ depends on $\{A', A'', A''', B', B'', B'''\}$ and is normalized relative to $\varphi$ and $\psi$.
\end{enumerate}
Since $K$ is fixed, the sum for $p_{\mathfrak A}(I)\vee p_{\mathfrak B}(J)=K$ is precisely the sum over the set of free indices in $\{i_{1}, \ldots, i_{n}, j_{1}, \ldots, j_{m}\}$, and these are contained in the indices in $A'\cup A'' \cup B'\cup B''$. The exponential decay in the powers of $2$ in equation (\ref{Eqn8.21okm}) show that the sum over all the free indices converges, and what remains satisfies the requirements for weak cancellation in Definition \ref{Def1.14a2}. This completes the proof.
\end{proof}

\subsection{Further Examples}\label{FlagAlgebra}\quad

\smallskip

It may help to consider two additional examples.

\smallskip

\noindent\textbf{Example 2:} \quad Suppose that $N=5$, and that the two partitions of $\{1,2,3,4,5\}$ are $A = \{2,3\}$ and $B=\{2,3\}$. Thus $A$ and $B$ come from the same decomposition $\R^{5}= \R^{2}\oplus\R^{3}$.  There are  $\binom {4}{2} = 6$ different decompositions of $\{1,2,3,4\}$ into two disjoint subsets of cardinality $2$ and $2$. The six decompositions, the $5$-tuples $\tilde I$, $\tilde J$, and $\tilde I \vee \tilde J$, and the resulting new decomposition $C$ of $\{1,2,3,4,5\}$ are listed in Table 2 below. This example is typical of the convolution of two kernels coming from the \emph{same} flag (in this case coming from the decomposition $\{2,3\}$). Note that all decompositions lead to the same decomposition $\{2,3\}$. Thus the convolution will be a flag of the same type.

\centerline{\textbf{Table 2}}

\smallskip

{\SMALL

\begin{center}
\begin{tabular}{|c|c|c|c|c|c|}
\hline
{Decomposition} & Ordering & $K$ &New Decomposition &$C$&\textit{Free variables}\\\hline&&&&&\\
$ \{1,2\} \cup \{3,4\}$ &$i_{1}\leq i_{2}\leq j_{1}\leq j_{2}$&$\{j_{1},j_{1}, j_{2},j_{2},j_{2}\}$&$ \R^{2}\oplus\R^{3}$&$\{2,3\}$&$i_{1},i_{2}$\\&&&&&\\ \hline&&&&&\\
$ \{1,3\} \cup \{2,4\}$ &$i_{1}\leq j_{1}<i_{2}\leq j_{2}$&$\{j_{1},j_{1}, j_{2},j_{2},j_{2}\}$&$ \R^{2}\oplus\R^{3}$&$\{2,3\}$&$i_{1}, i_{2}$\\&&&&&\\\hline&&&&&\\
$ \{1,4\} \cup \{2, 3\}$ &$i_{1}\leq j_{1}\leq j_{2}< i_{2}$&$\{j_{1},j_{1}, i_{2},i_{2},i_{2}\}$&$ \R^{2}\oplus\R^{3}$&$\{2,3\}$&$i_{1},j_{2}$\\&&&&&\\\hline&&&&&\\
$ \{2,3\} \cup \{1,4\}$ &$j_{1}<i_{1}\leq i_{2}\leq j_{2}$&$\{i_{1},i_{1}, j_{2},j_{2},j_{2}\}$&$ \R^{2}\oplus\R^{3}$&$\{2,3\}$&$i_{2}, j_{1}$\\&&&&&\\\hline&&&&&\\
$ \{2,4\} \cup \{1,3\}$ &$j_{1}<i_{1}\leq j_{2}< i_{2}$&$\{i_{1},i_{1}, i_{2},i_{2},i_{2}\}$&$ \R^{2}\oplus\R^{3}$&$\{2,3\}$&$j_{1}, j_{2}$\\&&&&&\\\hline&&&&&\\
$ \{3,4\} \cup \{1,2\}$ &$j_{1}\leq j_{2}< i_{1}\leq i_{2}$&$\{i_{1},i_{1}, i_{2},i_{2},i_{2}\}$&$ \R^{2}\oplus\R^{3}$&$\{2,3\}$&$j_{1}, j_{2}$\\&&&&&\\\hline
\end{tabular}
\end{center}}

\vspace{.3in}
\bigskip

\noindent\textbf{Example 3:} \quad Suppose that $N=5$, and that the two partitions of $\{1,2,3,4,5\}$ are $A = \{5\}$ and $B=\{1,2,3,4,5\}$. Thus $m=1$ and $n=5$. There are then $\binom {6}{5} = 6$ different decompositions of $\{1,2,3,4,5,6\}$ into two disjoint subsets of cardinality $1$ and $5$. The six decompositions, the $5$-tuples $\mathcal I$, $\mathcal J$, and $\mathcal I \vee \mathcal J$, and the resulting new decomposition $C$ of $\{1,2,3,4,5\}$ are listed in Table 3 below.

\medskip

\centerline{\textbf{Table 3}}

\smallskip
\vbox{
{\tiny\begin{center}
\begin{tabular}{|c|c|c|c|c|c|}
\hline
\textit{Decomposition} & Ordering & $K$ &New Decomposition&$C$&\textit{Free variables}\\\hline&&&&&\\
$ \{1\} \cup \{2,3,4,5,6\}$ &$i_{1}\leq j_{1}\leq j_{2}\leq j_{3}\leq j_{4}\leq j_{5}$&$\{j_{1},j_{2},j_{3},j_{4},j_{5}\}$&$\R\oplus\R\oplus\R\oplus\R\oplus\R$&$\{1,1,1,1,1\}$&$i_{1}$\\&&&&&\\ \hline&&&&&\\
$ \{2\} \cup \{1,3,4,5,6\}$ &$j_{1}< i_{1}\leq j_{2}\leq j_{3}\leq j_{4}\leq j_{5}$&$\{i_{1},j_{2},j_{3},j_{4},j_{5}\}$&$\R\oplus\R\oplus\R\oplus\R\oplus\R$&$\{1,1,1,1,1\}$&$j_{1}$\\&&&&&\\\hline&&&&&\\
$ \{3\} \cup \{1,2,4,5,6\}$ &$j_{1}\leq j_{2}< i_{1}\leq j_{3}\leq j_{4}\leq j_{5}$&$\{i_{1},i_{1},j_{3},j_{4},j_{5}\}$&$\R^{2}\oplus\R\oplus\R\oplus\R$&$\{2,1,1,1\}$&$j_{1}, j_{2}$\\&&&&&\\\hline&&&&&\\
$ \{4\} \cup \{1,2,3,5,6\}$ &$j_{1}\leq j_{2}\leq j_{3}< i_{1}\leq j_{4}\leq j_{5}$&$\{i_{1},i_{1},i_{1},j_{4},j_{5}\}$&$\R^{3}\oplus\R\oplus\R$&$\{3,1,1\}$&$j_{1}, j_{2},j_{3}$\\&&&&&\\\hline&&&&&\\
$ \{5\} \cup \{1,2,3,4,6\}$ &$j_{1}\leq j_{2}\leq j_{3}\leq j_{4}< i_{1}\leq j_{5}$&$\{i_{1},i_{1},i_{1},i_{1},j_{5}\}$&$\R^{4}\oplus\R$&$\{4,1\}$&$j_{1}, j_{2}, j_{3}, j_{5}$\\&&&&&\\\hline&&&&&\\
$ \{6\} \cup \{1,2,3,4,5\}$ &$j_{1}\leq j_{2}\leq j_{3}\leq j_{4}\leq j_{5}< i_{1}$&$\{i_{1},i_{1},i_{1},i_{1},i_{1}\}$&$\R^{5}$&$\{5\}$&$j_{1}, j_{2}, j_{3}, j_{4}, j_{5}$\\&&&&&\\\hline
\end{tabular}
\end{center}}
}

\smallskip

\noindent {This example is typical of the convolution of a Calder\'on-Zygmund kernel with a kernel that is as fine as possible. In this case, the flag $A$ is coarser than flag $B$.}

\smallskip

\section{$L^{p}$-estimates for flag convolutions}\label{LpEstimates}

In this section we establish the boundedness in $L^{p}(G)$ for $1 < p < \infty$ of the operator $f\to \K*f$ given by convolution on $G$ with a flag kernel. To simplify the notation, we limit ourselves to the
special situation where the exponents of the dilations  $d_1 , d_2 , \ldots d_N$ equation (\ref{E2})  are 
positive integers.  The results proved below will go over to the
more general context with essentially no change in the proofs. We will also find it convenient to consider a
continuous parameter $s_k$ for the dilation of the $x_k$
variable, in place of the dyadic version $2^{i_k}$ appearing in
the previous sections.  Again, the various results above stated
for the dyadic dilations have simple modifications valid for
their continuous analogues.

\subsection{Maximal Functions}\label{SecMaxFns}\quad

\smallskip

As usual, $G$ is a homogeneous nilpotent Lie group that we identify with $\R^{N}$ as in Section \ref{Nilpotent}. We also let
\begin{equation*}
\begin{aligned}
G_{k}&= \left\{\x=(\x_{1}, \ldots, \x_{n})\in \R^{N}\,\big\vert\,\x_{1}= \cdots = \x_{k-1}=0\right\}\\
&= 
\left\{(\mathbf 0, \cdots , \mathbf 0, \x_{k}, \ldots, \x_{n}) \in \R^{N}\,\big\vert\,\x_{j}\in \R^{a_{j}},\,k \leq j \leq n\right\}.
\end{aligned}
\end{equation*}
We can identify $G_{k}$ with $\R^{a_{k}}\oplus \cdots \oplus \R^{a_{n}}$, and it follows from the formula (\ref{1.4}) for group multiplication that $G_k$ is a subgroup of $G$. We let $m(E)$ denote the Lebesgue measure of a set $E\subseteq G = G_{1}$, and $m_{k}(E)$ denote the Lebesgue measure on $G_{k}$ of a subset $E\subseteq G_{k}$. For $\s= (s_{k}, \ldots, s_{n})$, let 
\begin{equation*}
R_{\s} = R^{(k)}_{\s}= \{ (\x_{k}, \ldots, \x_{n}) \in G_k : | \x_k | \leq s^{k}_k , \ldots ,
| \x_n | \leq s^{n}_n \}.
\end{equation*} 
We say that the size of the rectangle $R_s$ is  \emph{acceptable} if $s_k \leq s_{k
+ 1} \leq \cdots \leq s_n$.

\begin{definition}\label{Def6.1}
The maximal function $M$, defined on $G=G_{1}$, is given by
\begin{equation*}
M ( f ) ( x ) = {\sf \sup} \, \frac{1}{m ( R_{\s} )} \:
\displaystyle{\int_{R_{\s}}} \, | f ( x \cdot y^{-1} ) | dy
\end{equation*}
where the supremum is taken over all acceptable rectangles $R_{\s}=R^{(1)}_{\s}\subseteq G=G_{1}$.
\end{definition}

\begin{theorem}\label{Theorem6.2}\quad
\begin{enumerate}[{\rm(a)}]
\smallskip
\item $M$ {\it is a bounded map of} 
$L^p ( G )$ {\it to itself, for}
$1 < p < \infty$.
\smallskip
\item For $1 <p<\infty$ there are constants $A_{p}$ so that if $\{ f_j \}$ are scalar-valued functions on $G$ then
\begin{equation*} 
\big\| \big(
{\sum\limits_{j}} \, M ( f_j )^2 \big)^{1/2}
\big\|_{L^p ( G )} \, \leq \,
{A}_p \big\| \big( 
\displaystyle{\sum\limits_{j}} | f_j|^2
\big)^{1/2} \big\|_{L^p ( G )}.
\end{equation*}
\end{enumerate}
\end{theorem}

To prove this theorem, we consider the standard maximal function $M_k$ on
the subgroup $G_k$ defined by
\[ M_k ( f ) ( x ) = \sup_{\rho>0} \, \frac{1}{m ( B{(\rho)})} \,
\displaystyle{\int\limits_{B{(\rho)}}} \, | f ( xy^{-1} )| dy \]
where $B{( \rho )} = B^{k}{( \rho)}$ is the automorphic
one-parameter ball given by 
\begin{equation*}
B^{(k)} ( \rho) = \big\{ (\x_{k}, \ldots, \x_{n})\in G_{k}\,\big\vert\,| \x_k | \leq \rho^k, 
| \x_{k+1} | \leq \rho^{k+1}, \ldots , |\x_n | \leq \rho^n \big\}.
\end{equation*}  
Let $\widetilde{M}_k$ be the maximal function in $G$ obtained by lifting
$M_k$ in $G_k$ to $G$.  (Facts about lifting are reviewed in the Appendix, Section \ref{Lifting}). The key lemma is

\smallskip
\begin{lemma}\label{Lemma6.3} There is a constant $C$ so that
\[
M \leq C\,\widetilde{M}_n \circ \widetilde{M}_{n - 1} \circ\ldots \circ \widetilde{M}_1 \,
.
\]
\end{lemma}

\begin{proof}
Let $\s=(s_{k}, s_{k+1}, \ldots, s_{n})$ and $\bar\s=(s_{k+1}, \ldots, s_{n})$. Let 
\begin{align*}
\chi_{R^{(k)}_{\s}}&= \text{the characteristic function of the rectangle $R^{(k)}_{\s}$ in the subgroup $G_k$},\\
\chi_{R_{\bar{\s}}}^{(k + 1)} &= \text{the characteristic function of  $R^{(k +1)}_{\bar{\s}}$ in the subgroup $G_{k+1}$},\\
\chi_{B^{(k)}(\s_k)}&=\text{the characteristic function of the ball
$B^{(k)} ( \s_k)$ in $G_k$}.
\end{align*}
Let $\eta_{R^{(k)}_\s}$,
$\eta_{R^{(k+1)}_{\bar{\s}}}$, 
$\eta_{B^{(k)}_{s_k}}$  be 
the normalized versions of these functions, so that, for example,   
$
\eta_{R^{(k)}_{\s}} ( x )  = m_{k} ( R^{(k)}_\s )^{-1} \,
\chi_{R^{(k)}_{\s}}
$
with a similar definition for 
$\eta_{B^{(k)}_{\s_k}}$.
The first observation to make is that if $s_k \leq s_{k+1} \ldots \leq s_n$, there is an estimate
$
\eta_{R^{(k)}_\s} 
\lesssim \eta_{B^{(k)}_{s_k}} \, \ast \, 
\eta_{R^{(k + 1)}_{\bar{\s}}},
$
in the sense there are constants
$c, C$, so that
\begin{equation}\label{Eqn6.1}
\eta_{R^{(k)}_{c \s}} \leq 
C \eta_{B^{(k)}_{s_k}} \, \ast \,
( \delta_{x_k} \otimes \eta_{R^{(k + 1)}_{\bar{\s}}} )
\end{equation}
where the convolution is now on the group $G_k$ and $\delta_{x_k}$
denote the delta function of the $x_k$ variables.\footnote{Inequality (\ref{Eqn6.1})
is essentially contained in Subsection \ref {SupportProperties}.} In fact, 
\begin{equation*}\chi_{B ( s_k)} \ast ( \delta_{x_k} \otimes
\chi_{R^{k + 1)} ( \bar{\s})} ) = 
\int_{G_{k}}
\chi_{B (s_k)} ( x \cdot y^{-1} ) 
\, \chi_{R^{k+1}_{\bar{\s}}} ( y ) dy . 
\end{equation*}
We introduce a new coordinate system in $G_k$, so that if $\x \in
G_k$, then 
\begin{equation*}
\x = ( \x_k , \x_{k+1} , \ldots \x_n ) = (\x_k , 0 , \ldots 0 )
\cdot ( 0 , \x^\prime_{k+1} , \ldots \x^\prime _n )= (\x_k) \cdot
\x^\prime ,
\end{equation*}
 with $\x^\prime \in G_{k+1}$.  In this new coordinate system
the integral can be written as 
\begin{equation}\label{Eqn6.2}
\displaystyle{\int_{G_{k + 1}}} 
\chi_{B ( s_k)} 
( \x_k \cdot \y^\prime ) \, 
\chi_{R^{k+1}_{\bar{\s}}} 
( \y^{\prime -1} \cdot \x^\prime ) d\y^\prime.
\end{equation}
Now if $\x \in R^{(k)}_{c\s}$, and $c > 0$ is small, then $| \x_k | \leq 
c^k s^k_k$ and $\x^\prime \in R^{(k + 1)}_
{c^\prime \bar{\s}}$, with
$c^\prime$ small with $c$.  (This is because 
$\x^\prime_j = q_j ( \x )$ 
with $q_j$ homogeneous polynomials of degree
$j$, $k + 1 \leq j \leq n$.) So if 
$\y^\prime \in B^{k + 1)}_{c s_{k+1}}$
then 
$\y^\prime{^{-1}} \cdot 
\x^\prime \in R^{k+1}_{\bar{\s}}$.  Thus for $\x \in R_{c\s}^{(k)}$, 
the integrand above is 1, whenever 
$\y^\prime \in B^{k+1}_{c s_{k+1}}$.  
The result is that the last integral exceeds 
$m ( B^{k+1}_{c s_{k+1}} )$ for $\x \in R^{(k)}_{cs}$.  Dividing through by
the normalizing factors and observing that 
\[
\frac{1}{m_{k} ( B^{(k)}_{s_k})} \, \cdot \, 
\frac{1}{m_{k+1} ( R^{k+1}_{\bar{\s}})}
\, \cdot \, m_{k+1} (B^{k+1}_{c s_{k+1}})  
= c^{\prime\prime} \, \frac{1}{m_{k} ( R^{(k)}_\s )}
\]
proves the claim (\ref{Eqn6.1}).  Proceeding this way by downward induction,
starting with the trivial case $k = n$, gives
\begin{equation}\label{Eqn6.3}
\eta_{R^{(1)}_{cs}} 
\leq C \eta_{B^{(1)}_{s_1}} \ast
(\delta_{x_1} 
\otimes \eta_{B^{(2)}_{s_2}} ) 
\cdots \ast ( 
\delta_{x_1 \cdots x_{n-1}} 
\otimes \eta_{B^{(n)}_{s_n}} )
\end{equation}
whenever $s_1 \leq s_2 \ldots \leq s_n$.  The inequality (\ref{Eqn6.3}) then
implies Lemma \ref{Lemma6.3}.
\end{proof}

We now turn to the proof of Theorem \ref{Theorem6.2}. For each $k$, the maximal functions $M_k$ satisfying the usual
weak-type and $L^p$ estimates on $L^p ( G_k)$ (because the balls
$B^{(k)} ( s_k )$ satisfy the required properties for the Vitali covering
argument).  Moreover, the vector-valued version
\begin{equation}\label{Eqn6.4}
\big\| \big( \displaystyle{\sum_{j}} 
( M_k ( f_j ))^2)^{1/2} \big\|_{L^p(G_k)} \leq \mathcal{A}_p 
\big\| ( \sum | f_j |^{2} )^{1/2} \big\|_{L^p ( G_k)}
\end{equation}
also holds.  This can be shown by following the main steps in the case
of $\mathbb{R}^n$ (see e.g. \cite{St93}, Chapter 2).  In fact, one proves first
a weak-type inequality for the vector-valued case, using a
Calder\'{o}n-Zygmund decomposition, which establishes (\ref{Eqn6.4}) for $1 < p
\leq 2$.  An additional argument is needed for $p > 2$, and is based on
the fact that 
$$
\int_{G_k} M_k ( f )^2 (x)
\omega(x) dx \leq A \int_{G_k} | f ( x ) |^2 (
M_k \omega ) ( x ) dx
$$ 
for all positive functions $\omega$.
Next a lifting argument (see the Appendix) 
allows one to lift (\ref{Eqn6.4})
on $G_k$ to $G$ to get 
\begin{equation*}\tag{\ref{Eqn6.4}$^{\prime}$}
\big\| \big( \displaystyle{\sum\limits_{j}} |
\tilde{M}_k ( f_j )|^2 \big)^{1/2} \big\|_{L^p ( G )} \leq A_p \big\|
\big( \displaystyle{\sum\limits_{j}} | f_j |^2 \big)^{1/2} \big\|_{L^p (
G )}
\end{equation*}
As a result we obtain a similar inequality for $\tilde{M}_n \circ
\tilde{M}_{n-1} \circ\ldots \circ \tilde{M}_1$, and an application of Lemma
\ref{Lemma6.3} then proves Theorem \ref{Theorem6.2}. 

\smallskip

Our actual application of the estimate in equation (\ref{Eqn6.4}$^{\prime}$) is contained in the following.

\begin{corollary}\label{Cor6.4}
Suppose $F_t ( x )$ is a 
measurable function of $(t , x ) \in ( \mathbb{R}^+)^n \times
\mathbb{R}^N$.  Then
\[
\big\| \big(
\displaystyle{\int\limits_{( \mathbb{R}^+ )^n}} 
( {M} ( F_t) ) ( x ) )^2 dt \big)^{1/2} 
\big\|_{L^p ( \mathbb{R}^M)} \, \leq \,
A_p \big\| \big( \int\limits_{( \mathbb{R}^+ )^n} | F_t ( x ) |^2 dt
\big)^{1/2} \big\|_{L^p ( \mathbb{R}^N)}
\]
\end{corollary}


\begin{proof}

\noindent
Assume first that $F_t ( x )$ is jointly
continuous and has compact support.  For each $\epsilon > 0$, apply the
conclusion (\ref{Eqn6.4}$^{\prime}$) to the case where $\{ f_j ( x ) \}$ are an
enumeration of the $\epsilon^{n/2} F_{\epsilon i_1 , \epsilon i_2 ,
\ldots \epsilon i_n} ( x )$, for $(i_1 , i_2 , \ldots i_n )$ ranging
over $( \mathbb{Z}^+ )^n$, and then let $\epsilon \rightarrow 0$,
obtaining the desired result in this case.  For the general $F_t$,
assuming that $\| \big({\int_{(
\mathbb{R}^+)^n}} 
\, | F_t ( x ) |^2 dt \big)^{1/2} \|_{L^p ( G)}$ is
finite, find a sequence $F^{(n)}_t ( x )$ of continuous functions of
compact support, with $F^{(n)}_t ( x ) \rightarrow F_t ( x )$ almost
everywhere, so that 
\[
\big\| \big( \displaystyle{\int_{( \mathbb{R}^+)^n}} | F^{(n)}_t
|^2 dt \big)^{1/2} 
\big\|_{L^p} \rightarrow \big\| \big( \displaystyle{\int_{(
\mathbb{R}^+)^n}} | F_t |^2 dt )^{1/2} \big\|_{L^p}
\]
and apply the previous case, via Fatou's lemma.
\end{proof}

It will also be useful to observe that effectively the estimate (\ref{Eqn6.3}) can
be reversed in the following way.

\begin{lemma}\label{Lemma6.5}
We have
\begin{equation}\label{Eqn6.5}
\eta_{B^{(1)}_{s_1}} \, {\ast} \,( \delta_{x_1} \otimes
\eta_{B^{(2)}_{s_2}} ) \,\ast\,\cdots \,\ast \,( \delta_{x_1 , \ldots x_{n-1}}
\otimes \eta_{B^{(n)}_{(s_n)}} ) 
\leq C \eta_{R^{(1)}_{cs^\ast}}
\end{equation}
for an appropriate $C > 0$.  Here $s^\ast = ( s^\ast_1, \ldots
s^\ast _n )$, with $s^\ast_k = \max \{ s_j , \, 1 \leq j \leq k \}$. Note that we do not require that $s_1 \leq s_2 \cdots \leq s_n$.
\end{lemma}


\smallskip

The proof is based on the observation that \,\,$\eta_{B^{(k)}_{s_k}} \ast ( \delta_{x_k} \, \otimes \,\eta_{R^{k+1}_{\bar{\s}}}) \leq c \,\eta_{{R^{k}}_{c{ \!s{\tilde{ ~ }}}}}
$\,\,
where 
\begin{equation*}
\widetilde s_{j}=
\begin{cases}
\max\{s_{k},s_{j}\}&\text{if $j>k$,}\\
s_{k}&\text{if $j=k$.}
\end{cases}
\end{equation*}
In fact if $\x \notin R^{k}_{C \widetilde{\s}}$ 
(for some large $C$)  then either
$|\x_k | \geq  C^\prime \,{s_{k}^{k}}$ 
or 
$\x^\prime \notin R^{(k + 1)}_{c \widetilde{\bar{s}}} \,.$
Looking back at the integral in equation (\ref{Eqn6.2}) we see that the integral vanishes,
because $\chi_{B ( s_k)} ( \x_k \cdot \y^\prime ) = 0$ in the first
case, or $\chi_{R^{(k + 1)}_{\bar{s}_{\tilde{~}}}} 
( \y^{\prime - 1} \cdot \x^\prime ) = 0$ 
in the second case.  Moreover as above, this integral
is majorized by 
$C\, m ( B^{(k+1)}_{c s^{~}_{k + 1}} )$.
Altogether then, we have $\eta_{B^{(k)}_{s_k}}\,{\ast} \,( \delta_{x_k} \otimes 
\eta_{R^{(k + 1)}_{\bar{s}}} ) \leq c \, 
\eta_{R_{\widetilde s}^{(k)}}$, and an induction proves Lemma \ref{Lemma6.5}.

\smallskip

Now let $A_s$ denote the function appearing on the left-side of (\ref{Eqn6.5}).
Then as a consequence we have
\begin{equation}\label{Eqn6.6}
| ( f \ast A_s ) ( x ) | \leq C \, M ( f ) ( x ),
\end{equation} 
for all $\s = ( s_1 , \ldots , s_n )$, not necessarily in increasing order. Similarly if 
$A^\ast_s = ( \delta_{x_1 , \ldots x_{n-1}} \otimes
\eta_{B^{(n)}_{s_n}} ) \ast 
\cdots \ast ( \eta_{B^{(1)}_{s_1}})$  
we also have $A^\ast_s ( x ) \leq
c \, \eta_{R^{(1)}_{s^\ast}}$.  This follows from (5) if we observe that
$A^\ast_s ( x ) = A_s ( x^{-1})$.  As a result, in analogy to
(\ref{Eqn6.6}),  we have for all $\s$
\begin{equation*}\tag{\ref{Eqn6.6}$^{\prime}$}
| ( f \ast A^\ast_s ) ( x ) | \leq c \, M ( f ) ( x
)
\end{equation*}
Indeed, one has $R^{-1}_{s^\ast} = R_{s^\ast}$ if, in defining $R_s$, a
coordinate system is used where the inverse of $\x=(\x_{1}, \ldots, \x_{n})$ is given by $\x^{-1} = ( - \x_1 , \ldots , - \x_n)$.  Alternatively, if we
use cannonical coordinates of the second kind, as above, then one has
$R_{c_1 s^\ast} \subset R^{-1}_{s^\ast} \subset R_{c_2 s^\ast}$, for two
appropriate constants $c_1$ and $c_2$; this also leads to (\ref{Eqn6.6}$^{\prime}$).

\subsection{Comparisons}\label{Sec6.2}\quad

\smallskip

The basic comparison function is 
\[
\Gamma_t ( x ) = t_1 \cdot t_2 \cdots \, \cdot t_n \cdot 
\displaystyle{\mathop{\Pi}\limits_{k = 1}^{n}} ( t_1 + t_2 \cdots + t_k 
+ N_1 ( x ) \cdots + N_k ( x ))^{- Q_k - 1}
\]
for $\t = ( t_1 , \ldots , t_n )$, $t_j > 0$.  Recall that $N_k ( x ) = |
x_k |^{1/k}$ and $Q_k = ka_k$ with $a_k$ the dimension of the $x_k$
space.


\begin{theorem}\label{Theorem6.6}
\begin{equation}\label{Eqn6.7}
\sup_{t}\,\big\vert (f*\Gamma_{t})(\x)\big\vert \leq C\,M(f)(\x),
\end{equation}
where the supreme is taken over all $t$, with $t_j > 0$.
\end{theorem}


\begin{proof}

Note that it suffices to restrict attention to 
$\t$'s that are of acceptable size.  Indeed, let $s_j = t_1 + t_2 \cdots
+ t_j$, $1 \leq j \leq k$.  Then $s_j \leq s_{j+1}$ but $k ( t_1 \cdots +
t_k ) \geq s_1 + s_2 \cdots + s_k$.  
Hence $( t_1 \cdots + t_k + N_1
+ \cdots + N_k )^{-Q_k - 1} 
\leq c_k ( s_1 + s_2 \cdots + s_k + N_1
+ \cdots + N_k )^{-Q_k - 1}$, with $c_k = k^{Q_k + 1}$.  Therefore
$\Gamma_\t ( \x ) \leq c \Gamma_\s ( \x )$, which shows that it suffices to
consider $\t$'s that are increasing.

We next fix $\t  = ( t_1 , \ldots t_n )$ and decompose the space $G=\R^{N}$ into
a preliminary dyadic partition as follows.  For each $J = (j_1 , \ldots
j_n) \in \mathbb{Z}^{n}_{+}$ we let 
\begin{equation*}
R_J = \{ \x\in \R^{N}\,\big\vert\, \text{$2^{j_k -1} < \,
\frac{N_1 ( x ) + \cdots N_k ( x )}{t_1 + t_2 \cdots + t_k} \, \leq \,
2^{j_k}$, for $k = 1 , 2 , \ldots n $} \},
\end{equation*} 
with the understanding
that if $j_k = 0$ the inequality should be taken to be $$\frac{N_1(x)
\cdots + N_k ( x )}{t_1 + t_2 \cdots t_k} \leq 1.$$
Notice that $\bigcup_{J \in\mathbb{Z}^{n}_{+}} R_J = G$ gives a partitioning of the space $G$.
However, in general each $R_J$ is not comparable to an acceptable
rectangle.  We remedy this as follows.  For a suitable constant $c$,
define $s^J = (s^J_1 , \cdots s^J_n)$ by 
\begin{equation}\label{Eqn6.8}
s^J_1 = c t_1 2^{j_1}, \quad s^J_2 = c^2 t_2 2^{j_2}, \quad \ldots \quad 
s^J_n = c^n t_n 2^{j_n}. 
\end{equation}
Now if $R_J$ is non-empty, then since $t_1 \leq t_2 \cdots \leq t_n$, we
have 
\begin{equation*}
N_1 ( x ) \approx t_1 2^{J_1}, \quad N_1 ( x ) + N_2 ( x ) \approx t_2
2^{j_2}, \quad \cdots \quad N_1 ( x) + N_2 ( x ) \cdots + N_n ( x ) \approx t_n
2^{j_n}.
\end{equation*}  
As a result, for sufficiently large $c$, it follows that
$s^J_1 \leq s^J_2 \cdots \leq s^J_n$.

Now define $R^\ast_J = \{ x: N_k ( x ) \leq s^J_k, k = 1, \ldots n
\}$ for those $J$ where $R_J$ is not empty.  Then clearly $R_J \subset
R^\ast_J$  and each $R^\ast_J$ is a rectangle of acceptable size (in fact,
essentially the smallest rectangle of acceptable size containing $R_J$).
However for $f \geq 0$, 
\[
\displaystyle{\int\limits_{G}} f ( \x \cdot \y^{-1} ) \Gamma_t ( \y ) d\y \,
= \,
\displaystyle{\sum\limits_{J \in \mathbb{Z}^n_+}} \; \; 
\displaystyle{\int\limits_{R_J}} f ( \x \y^{-1}) \Gamma_t ( \y ) d\y .
\]
Recall that by (\ref{Eqn6.8}), $t_k \approx s^J_k 2^{-j_k}$, and on $R_J$ we have 
$N_1 + \cdots N_k \approx ( t_1 + t_2 \cdots + t_k ) 
2^{j_k} \approx s^J_k$.
Thus,  on $R_J$ we have $\Gamma_t ( y ) \lesssim
\displaystyle{\mathop{\Pi}\limits_{k = 1}^{n}} 
( s^J_k )^{-Q_k} \cdot ( 2^{-Q_k j_k})$. 
So
\begin{equation*}
{\int\limits_{G}} 
f ( x y^{-1} ) \Gamma_t ( y ) dy \lesssim
\displaystyle{\sum\limits_{J}} 
\displaystyle{\mathop{\Pi}\limits_{k = 1}^{n}} 2^{-Q_k j_k} 
( s^J_k )^{-Q_k} \cdot 
\displaystyle{\int_{R^\ast_J}} f ( x y^{-1}) dy.
\end{equation*} 
But each $R^\ast_J$
is a rectangle of acceptable size and
$\displaystyle{\mathop{\Pi}\limits_{k = 1}^{n}} ( s^J_k )^{Q_k} = c \, m
( R^\ast_J)$.  Thus by the definition of $M$, the last sums is
majorized by $c  \displaystyle{\sum\limits_{J}} \,
\displaystyle\mathop{\Pi}\limits_{k = 1}^n 2^{-Q_k j_k} M ( f ) = c^\prime M ( f )$, 
and (\ref{Eqn6.7}) is proved. 
\end{proof}

\smallskip

\subsection{Truncated kernels}\label{Sec6.3}\quad

\smallskip

Recall that we defined truncated kernels and improved truncated kernels in Definition \ref{DefTruncated}  (Section \ref{SecTruncated}). Suppose that $\psi\in \mathcal C^\infty_{0}(\R^{N})$ with support in the unit ball.  For $b > 0$ write $$\psi_{b}(\x) = b^{-Q_1
-Q_2 \cdots Q_n} \psi ( b^{-1}\x_{1} , b^{-2}\x_2, \ldots b^{-n}\x_n),
$$
the automorphically dilated $\psi$.  We also say that $\psi_b$ has {\it
width} $b$.

\begin{theorem}\label{Theorem6.7}
Suppose $\K$ is a truncated flag kernel
of width $a$, and $\psi_b$ is as above of width $b$. Then
\begin{enumerate}[{\rm (1)}]
\item $\K \ast \psi_b$ 
{\it and} $\psi_b \ast \K$ 
{\it are truncated kernels of width} 
$a + b$.  
\item {\it If in addition 
$\int_{G}
\psi(\x)\, d\x = 0$, then $K \ast \psi_b$ and $\psi_b \ast K$ 
are improved truncated kernels of width $a+b$.  Moreover, then $K \ast
\psi_b$ and 
$\psi_b \ast K$ are actually improved truncated kernels of width $a
+ b$, multiplied by the further factor $\frac{b}{a+b}$.} 
\end{enumerate} 
\end{theorem}


 Note that the statements of the hypotheses and
conclusions have an automorphic-dilation invariance, so in proving Theorem \ref{Theorem6.7}, it suffices to
consider two cases: $b = 1$, $a \leq 1$; and $a = 1$, $b \leq 1$.  In
the first case we use Proposition \ref{Prop4.5z}, since any truncated kernel is
actually an un-truncated kernel.  Thus we get that $K \ast \psi_1$ and
$\psi_1 \ast K$ have width $1$, which is essentially the same as having
width $1 + a$, since $a \leq 1$.

The second case could be easy if every truncated kernel of width 1 were
of the form $K \ast \psi_1$.  
Its proof is a little more involved and requires the following
lemma.

\begin{lemma}\label{Lemma6.8}
Given any $M$, there exist $\eta_0$ and
$\eta_1$ both of class $C^{(M)}$, supported in the unit ball, and a 
(non-communicative) polynomial $P ( X_1 , \ldots X_N) = P ( X )$ in
the right-invariant vector fields of $G$, so that
\begin{equation}\label{Eqn6.9}
P ( X ) \eta_0 = \delta_0 + \eta_1.
\end{equation}
with $\delta_0$ the Dirac delta at the origin.
\end{lemma}

\begin{proof}
Consider the elliptic operator of order $2r$,
$P ( X ) = \big( {\sum\limits_{j = 1}^N} 
X^2_j \big)^r$, with $r$ a positive integer.  Then by the standard theory of
pseudo-differential operators there is a locally integrable function $F$
which is $C^\infty$ away from the origin, so that $P ( X ) F = \delta_0
+ \eta^\prime$, with $\eta^\prime$ a $C^\infty$ function.  Moreover, $F$
satisfies the estimate $| \left( \frac{\partial}{\partial x}
\right)^\alpha F ( \x ) | \leq A_\alpha$, whenever $| \x | \leq 1$ and
$2 r > N + | \alpha |$. (These estimates also follow from \cite{MR1044793}, Theorem 1.) Thus we only need to take $2 r > M + N$ and set 
$\eta_0 = \mu \cdot F$,
where $\mu$ is a $C^\infty$ function supported in the unit ball, and
$\mu ( x ) = 1$ in the ball of radius $1/2$.  Then since 
$\eta_0$ is
supported in the unit ball, so is 
$\eta_1 = P ( X ) \eta_0 - \delta_0$;
and since $F$ is $C^\infty$ away from the origin it follows that
$\eta_1$ is in fact $C^\infty$ everywhere. This completes the proof of Lemma \ref{Lemma6.8}.
\end{proof}

We now return to the proof of Theorem \ref{Theorem6.7}. We consider $K \ast \psi_b$ when
$K$ has width $1$, and $b \leq 1$.  Now by the lemma $K \ast \psi_b = K
\ast \delta_0 \ast \psi_b = K \ast P ( X ) \ast \eta_0 \ast \psi_b + K
\ast \eta_1 \ast \psi_b$, since $P ( X )$ is a right-invariant
differential operator.  Now since $K$ has width $1$, $K \ast P ( X )$ is
also a truncated kernel of width one, and in particular an un-truncated
kernel. However, $\eta_0 \ast \psi_b$ has width $1 + b$, which is
essentially one.  Also, it is of class $C^{(M)}$ (uniformly in $b$),
since $\eta_0$ is of class $C^{(M)}$.  Thus  $K \ast P (
X ) \ast \eta_0 \ast \psi_b$ satisfy the differential inequalities for
a truncated kernel of width one for all orders $\leq m$.  However, the
term $K \ast \eta_1 \ast \psi_b$ clearly does the same, for all orders.
Notice we can make $m$ as large as we wish by making $M$ sufficiently
large. (See Remark \ref{Remark7.7} on page \pageref{Here}.) A similar argument works for $\psi_b \ast K$ and thus part (1)
of Theorem \ref{Theorem6.7} is proved.  

Part (2) is
proved in the same way, using conclusion (2) of Proposition \ref{Prop4.5z}.  The
further improvement given by the factor $b(a+b)^{-1}$ comes about as
follows.  As before, we may take $a = 1$, and $b \leq 1$.  Since $ \int
\psi(\x| d\x = 0$, both $\eta_0 \ast \psi_b$ and $\eta_1 \ast \psi_b$ give an
improvement of $b$. In fact, since $\int \psi(\x) d\x = 0$, it follows from Lemma \ref{Lemma1.12} and Proposition \ref{Prop1.8} that we can write
\begin{equation}\label{Eqn6.10}
\psi_b = \displaystyle{\sum} b^k X_k ( \psi^{(k)}_b )
\end{equation}
for
suitable $C^\infty$ functions $\psi^{(k)}$ supported in the unit ball,
with $\{ X_k \}$ ranging over right-invariant vector fields of degree
$k$ , $k \geq 1$.
Thus $\eta_0 \ast \psi_b = \displaystyle{\sum} b^k ( \eta_0 \ast X_k )
\ast \psi^k_b$ and this gives a gain $b$, $b \leq 1$.  Similarly for the
term $\eta_1 \ast \psi_b$.


\subsection{Key estimates: kernels}\label{Sec6.4}\quad

\smallskip

Suppose $\varphi^{(k)}\in \mathcal C^\infty_{0}$ is supported on the unit
ball of the group $G_k$, with
$${\int_{{G_k}}}
\varphi^{(k)} ( \x ) d \x = 0.$$  We set $\varphi^{(k)}_t ( \x ) =
t^{-Q_k - \cdots Q_n } \varphi ( \delta_{t^{-1}} ( \x ) )$, with 
$Q_k = k a_k$ and let 
$\tilde{\varphi}^{(k)}$ to be the corresponding
distributions lifted to the full group $G$;\, \textit{i.e.}\,
$\tilde{\varphi}^{(k)}_t = \delta_{x_1 , x_2 \cdots x_{k - 1}} \otimes
\varphi^{(k)}_t$.  We let 
$\Phi_\t = \tilde{\varphi}_{t_1}^{(1)} 
\ast \tilde{\varphi}^{(2)}_{t_2} 
\cdots \ast \tilde{\varphi}^{(n )}_{t_n}$
for $\t = ( t_1 , \ldots , t_n )$, and write 
$$\Phi^\ast_\t = \tilde{\varphi}^{(n)}_{t_n} 
\ast \tilde{\varphi}^{(n - 1)}_{t_{n - 1}}
\cdots \ast \tilde{\varphi^{(1)}}_{t_1}.$$
Recall the comparison function $\Gamma_\t$ discussed in Section \ref{Sec6.2}.  Note that
here we will allow the functions $\varphi$ and $\Phi$ to take their
values in finite-dimensional vector spaces.

\begin{theorem}\label{Theorem6.9} Suppose $\K$ is a flag kernel.  Then
\begin{enumerate}[{\rm(1)}]
\item{\it $| K \ast \Phi_\t ( \x ) |$ and $| \Phi^\ast_\t \ast K ( \x ) |$
are both majorized by $c \,\Gamma_\t ( \x )$ for all $\t$}.

\smallskip

\item If $X^R_k$ is any right-invariant vector field of degree
$k$, then $$| X^R_k ( K \ast \Phi_t ) | \leq c ( t_1 + \cdots + t_k
)^{-k} \Gamma_t ( x ).$$ 

\smallskip

\item If $X^L_k$ is any left-invariant vector field of degree
$k$, then $$| X^L_k ( K \ast \Phi_t ) | \leq c ( t_1 + \cdots + t_k
)^{-k} \Gamma_t ( x ).$$

\end{enumerate}

\end{theorem}
For the proof we need to do our calculations in a particular coordinate
system, already used in the proof Lemma \ref{Lemma6.3}.  Here we represent a point $\x = ( \x_1 , \ldots, \x_n) \in G$ via
adapted canonical coordinates of the second kind; \textit{i.e.}  we take $\x
= \exp ( \x^\prime_1 \cdot X_1 ) \, \exp ( \x^\prime_2 \cdot X_2) \cdots
\exp ( \x^\prime_n \cdot X_n )$, where $\x^\prime = (x'_{k,1}, \ldots, x'_{k,a_{k}})$, $\{ X_{k,1}, \ldots, X_{k,a_{k}} \}$ is a
basis of the sub-space of vector fields of degree $k$, and $\x^\prime_k
\cdot X_k = {\sum_{j}} x^{\prime}_{k,j}
X_{k,j}$.
The passage from the initial $(\x_1 , \ldots \x_n)$ coordinates to
the $( \x^\prime_1, \ldots \x^\prime_n )$ coordinates is of the form
treated in Section \ref{ChVariables}, so that the basic comparison function $\Gamma_t$
is essentially unchanged when passing from $x$ to $x^\prime$.  We
therefore freely use instead the new coordinate system $\x^\prime$, and
now relabel $\x^\prime$ by $\x$.
The advantage of this coordinate system is that, firstly $\y = ( \y_1 ,
\ldots \y_k , \y_{k + 1}, \ldots, \y_n ) \in G_k$ if and only if $\y_1 = 0 ,
\ldots , \y_{k-1} = 0$, but more importantly, if $\x \in G$, and $\y \in
G_k$, then $\x \cdot \y = ( \x_1 , \cdots \x_{k - 1} , \bar{\x}_k , \ldots
\bar{\x}_n )$, with $\bar{\x}_\ell$ ( for $\ell \geq k )$, depending only
on $\x_k , \ldots\ x_n$, and $\y_k , \ldots \y_n$, and not on $\x_1 , \ldots
\x_{k-1}$.

\smallskip

Now set $K^{(k)} {(x)} = K \ast \tilde{\varphi}^{(1)}_{t_1} 
\ast \cdots \ast \tilde{\varphi}^{(k)}_{t_k}$.  Consider first $K^{(1)}(x) = K \ast
\varphi^{(1)}_{t_1}$.  According to Proposition \ref{Prop4.5z}, part (2), since
${\int_{G}} \varphi^{(1)} d\x = 0$, then for each
$\x_1$, the kernel $K^{(1)} ( \x_1 , \x_2 , \ldots, \x_n)$ as a function of
$(\x_2 , \ldots , \x_n)$ on $G_2$, is a truncated kernel of width $t_1 +
N_1 ( x_1)$, multiplied by the ``constant'' factor $t_{1}\,\big[t_1 + N_1 (\x_1)\big]^{-Q_{1}-1}$. Also we have a similar conclusion for $\partial_{x_{1}}^{\beta_1} K^{(1)} ( \x )$, except
now the improving factor is $t_{1}\,\big[t_1 + N_1 (\x_1)\big]^{-Q_{1}-1-\beta_{1}}$.

Consider next the inductive hypothesis:  for a given $k$,
\begin{enumerate}[(a)]

\smallskip

\item
For each $\x_1, \ldots \x_k$,  $K^{(k)} {(\x_{1}, \ldots, \x_{n})}$, thought of as a function of $(\x_{k+1}
, \ldots, \x_n)$ on $G_{k+1}$ is a truncated kernel of width $t_1 + \cdots
+ t_k + N_1 ( \x_1) \cdots + N_k ( \x_k )$, multiplied by the improving
factor ${\prod\limits_{j = 1}^k}  t_{j}\,\big[t_1
+ \cdots t_j + N_1 + \cdots N_j\big]^{-Q_{j}-1}$.

\smallskip

\item For
each $r \leq k$, a similar statement holds for $\partial_{x_{r}}^{\beta_r}
K^{(k)}(\x)$, except that now the $r^{\sf th}$ part of
the improving factor is 
$t_{r}\big[t_1 + \cdots t_r + N_1 \cdots + N_r\big]^{-Q_{r}-1-r\beta_{r}}$.

\smallskip
\end{enumerate}
Notice that if the inductive hypothesis holds for $k$, that is for
$K^{(k)}$, then since $K^{(k + 1)} ( \x_1 \cdots \x_n ) = K^{(k)} \ast
\varphi^{(k + 1)}_{t_{k+1}}$, where the convolution is
taken on the group $G_{k+1}$, therefore does not involve the variables
$\x_1, \ldots \x_k$ because of the nature of our coordinate system.As a
result, we get the conclusion for $k + 1$, that is for $K^{(k + 1)}$.
To see this, we merely apply Proposition \ref{Prop4.5z}, part (2), for the case of the
group $G_{k + 1}$.  

More precisely, we are convolving a truncated kernel of
width $a = t_1 + \cdots + t_k + N_1 + \cdots + N_k$ (on $G_{k + 1}$) with a function $\varphi$ (which equals $\varphi^{(k + 1)}_{t_{k + 1}}$)
of width $b = t_{k +
1}$.  The result is a truncated kernel of width $a + b = t_1 \cdots +
t_{k+1} + N_1 \cdots + N_k$ on $G_{k + 1}$, together with a further
factor $\frac{b}{a+b}$.  That is, together for $K^{(k+1)}$ we have, as a
function $x_{k + 2} , \ldots x_n$, a truncated kernel of width $t_1
\cdots + t_{k+1} + N_1 + \cdots N_{k+1}$, times a factor of the form
$b(a + b)^{-1}$.  So the full improvement is 
\begin{equation*}
\frac{b}{a+b} \cdot
\frac{a+b}{(a + b + N_{k + 1} )^{Q_{k+1}+1}} = \frac{t_{k+1}}{(t_1 \cdots
+ t_{k+1} + N_1 \cdots + N_{k + 1} )^{Q_{k+1}+1}},
\end{equation*} 
as was needed.  The
same kind of improvement holds for the estimates of 
$\partial_{x_{l}}^{\beta_\ell} K^{(k + 1)} ( \x
)$, for $\ell \leq k + 1$. 

Thus, the inductive hypothesis (now the conclusion of the induction)
holds for $k = n$. As a result, it is clear that $| K \ast \Phi_t ( x )
| \leq c \Gamma_t ( x )$ and $| \partial_{x_{k}} ( K
\ast \Phi_t ( x ) ) | \leq c \big[t_1 + \cdots + t_k + N_1 \cdots +
N_k\big]^{-k} \Gamma_t ( \x )$, for every $k$, $1 \leq k \leq n$.  Since $X^R_k
= \partial_{x_{k}}+ {\sum_{\ell > k}}
h^k_\ell \partial_{x_{ell}}$ where $h^k_\ell$ is a
homogeneous polynomial of degree $\ell - k$, it follows, in particular,
that $X^R_k ( K \ast \Phi_t) ( x ) \leq c \,\big[t_1 \cdots + t_k \big]^{-k}
\Gamma_t ( \x )$.

The results for $\Phi^\ast_t \ast K$ and $X^L_k ( \Phi^\ast_t \ast K )$
follow in the same way, but require a canonical coordinate  
system in the reverse order.  Alternatively we can deduce it from the
previous case by using the inversion $\x \rightarrow \x^{-1}$.

\smallskip

\subsection{Key estimates: operators}\label{Sec6.5}\quad

\smallskip

We define $P_\t ( f ) = f \ast \Phi_\t$, and $P^\ast_\t = f \ast
\Phi^\ast_t$, with $\t = ( t_1 , t_2, \cdots t_n )$ and $t_j > 0$, $1
\leq j \leq n$, with $\Phi_\t$ and $\Phi^\ast_\t$ defined at the beginning of
the previous section.  We suppose $\K$ is a flag kernel and $Tf = f \ast
K$ when $f$ is a Schwartz function.  We recall the maximal operator $M$
and let $\mathcal{M} = M \circ M$, i.e. $\mathcal{M} ( f ) = M ( M ( f
))$.

\begin{theorem}\label{Theorem6.10}\quad
\begin{enumerate}[{\rm(a)}]

\item {\it $| P_t T ( f ) ( x ) | \leq c \, M ( f ) ( x )$, all
$t$}.

\smallskip

\item[{\rm(a$^\prime$)}] {\it Similarly $| ( T P^\ast_t) ( f ) ( x ) | \leq c M ( f
) ( x )$, all $t$}.

\smallskip

\item {\it $| P_t TP^\ast_s ( f ) ( x ) | \leq \gamma ( s , t )
\mathcal{M} ( f ) ( x )$}, where for some $\delta > 0$, $$\gamma ( s , t ) \leq c \left(
\displaystyle{\mathop{\Pi}\limits_{k = 1}^{n}} \min \left(
\frac{s_k}{t_k} , \frac{t_k}{s_k} \right) \right)^\delta.$$
\end{enumerate}
\end{theorem}

\noindent
Note: the conclusion will be seen to hold for $\delta = \frac{1}{n^2}$.

\begin{proof}
The function $P_t T f$ is given by $( f \ast K
) \ast \Phi_t = f \ast ( K \ast \Phi_t )$.  Hence conclusion (a) is a direct consequence of Theorem \ref{Theorem6.6} and Theorem \ref{Theorem6.9},
part (1). The same is true  for conclusion (a$^\prime$).

Turning to conclusion (b), we first fix $k$, and consider the situation
when $t_k/s_k = \rho \geq 1$.  With this $\rho$ given, we next divide
our consideration in two cases.

\begin{namelist}{xxxxxx}
\item[\sf Case I:]  With $\sigma$ a positive constant, to be specified
below, $s_{j-1}\slash{s_j} > \rho^\sigma$ for at least one $j$, with $2 \leq j
\leq k$.  
\item[\sf Case II:] $s_{j-1}\slash{s_j} \leq \rho^\sigma$ for all $j$, with $2
\leq j \leq k$.  
\end{namelist}

To handle Case I we need the following observation that
will give us the needed gain. Recall the notation 
$\eta_{B^{(k)}_\rho}$
$={m_{k}} ( B^{(k)}_\rho )^{-1} \chi_{B^{(k)}_{\rho}}$ used
above in Section \ref{SecMaxFns}.

\begin{lemma}\label{Lemma6.11}
If $s_{j-1} \slash {s_j} \geq
\rho^\sigma$, $\rho \geq 1$, then
\begin{equation*}
\left| 
\tilde{\varphi}^{(j)}_{s_j} \ast 
\varphi^{(j-1)}_{s_{j-1}} 
\right|
\leq c \, 
\rho^{- \sigma j} 
\left( \tilde{\eta}^{(j)}_{B^{(j)}_{s_j}} \right) 
\ast \left( \eta^{(j-1)}_{B^{(j - 1)}_{s_{j-1}}} \right).
\end{equation*}
\end{lemma}

\begin{proof}[Proof of Lemma \ref{Lemma6.11}] In analogy with equation (\ref{Eqn6.10}) in Section \ref{Sec6.3}, we have that on $G_j$
\begin{equation}\label{Eqn6.11}
\varphi^{(j)}_{s_j} = \displaystyle{\sum\limits_{r \geq j}}
\, (s_j)^r \, X^L_r ( \psi^{(r)}_{s_j} ) 
\end{equation}
where $X^L_r$ are left-invariant vector fields 
on $G_j$, with
$\psi^{(r)}$ $C^\infty$ functions supported in the unit ball.
Now 
\begin{equation*}
\left(
\delta_{x_{j-1}} \otimes X^L_r ( \psi^{(r)}_{s_j}) \right) \ast
\varphi^{(j - 1)}_{s_{j-1}} \;
= \; \left(
\delta_{x_{j-1}} \otimes \psi^{(r)}_{s_j} \right) \ast
X^R_r \left(
\varphi^{(j - 1)}_{s_{j-1}} \right) , 
\end{equation*}
and $X^R_r \left( \varphi^{(j-1)}_{s_{j-1}} \right)$  
is of the form $s^{-r}_{j-1}
\psi^{\prime ( r )}_{s_{j-1}}. 
\)  Combining these gives a sum $\displaystyle{\sum\limits_{r \geq j}} (
s_j \slash{s_{j-1}})^r \, \tilde{\psi}^{(r)}_{s_j} \ast \psi^{\prime ( r
)}_{s_{j-1}}$ and since $s_j \slash {s_{j-1}} \leq \rho^{- \sigma}$, the
lemma is proved.
\end{proof}

Consider the operator $P^\ast_s$  given by
$P^\ast_s ( f ) = f \ast \Phi^\ast_s$ with $\Phi^\ast_s =
\tilde{\varphi}^{(k)}_{s_n} \ast \tilde{\varphi}^{(n-1)}_{s_{k-1}}
\cdots \ast \tilde{\varphi}^{(1)}_{s_1}$.  Because of Lemma \ref{Lemma6.11} and since $j\geq 2$, we have $| \Phi^\ast_s | \leq c \rho^{-2 \sigma}
\tilde{\eta}_{B^{(k)}_{s_n}}$ $\ast \cdots \ast
\tilde{\eta}_{B^{(1)}_{s_1}} = c \, \rho^{ - \sigma j} = c \rho^{- 2
\sigma} A^\ast_s$, in the terminology in Section \ref{SecMaxFns}.  So by equation
(\ref{Eqn6.6}$^{\prime}$) there, it follows that 
$
| P^\ast_ s ( f ) ( \x ) | \leq c \rho^{- 2 \sigma} M ( f ) ( \x )
$
Combining this with the first conclusion already proved that $| P_t T 
( F ) ( x ) | \leq c M ( F ) ( x )$, with $F = P^\ast_t ( f )$ yields
\begin{equation}\label{Eqn6.12}
| P_t T P^\ast_s ( f ) ( x ) | \leq c \rho^{-2 \sigma} \,
\mathcal{M} ( f )
\end{equation}
since $\mathcal{M} =   M  \cdot M $.  

\smallskip

We now turn to Case II.
Here $s_{j-1} \slash{s_j} \leq \rho^\sigma$, for all $j$, $2 \leq j
\leq k$.  Thus $s_j \leq s_k \rho^{\sigma ( k - j)}$ for $1 \leq j \leq
k$.  If we set $s_\ast = \sup\limits_{1 \leq j \leq k} s_j$, then $s_*
\leq s_k \rho^{\sigma ( k - 1)}$.  Next we recall the following fact: 

\smallskip
\noindent
If $X^R_\ell$ is a right-invariant\label{LeftInv}
vector field of degree $\ell \geq j$ on $G_j$, then one can write
\[
X^R_\ell ( \psi_a) = \displaystyle{\sum_{r \geq \ell}} a^{r-\ell} \,
X^L_r ( \psi^{(r)}_a )
\]
where $X^L_r$ are left-invariant vector fields of degree $r$. This
follows by writing $X^R_\ell = \displaystyle{\sum} h_{\ell, r} ( x )
X^L_r$ with $h_{\ell , r }$ a homogeneous polynomial of degree $r -
\ell$ and arguing as in Proposition 2.3. 

\medskip

 With this in hand, consider 
$\tilde{\varphi}_{s_k}^{k} \ast
\tilde{\varphi}^{k-1}_{s_{k-1}} \cdots \ast$ 
$\tilde{\varphi}^{(1)}_{s_1}$.
First use that 
${\int_{G_k}} \varphi^{(k)} dx = 0$,
which by the analogue of assertion (\ref{Eqn6.10}) 
(for left-invariant vector fields on the group $G_k$) 
gives an expression involving the action of
left-invariant vector fields (of degrees $\geq k$).  Next pass from the
left-invariant vector-fields acting on $\psi^{(r)}_{s_j}$ to the
corresponding right-invariant vector fields acting on
$\tilde{\varphi}^{k - 1}_{s_{k-1}}$, via the rule 
\begin{equation}\label{Eqn6.13}
( X^L_r \psi ) \ast
\varphi = \psi \ast X^R_r \varphi. 
\end{equation}
At this point utilize the remark on page \pageref{LeftInv}  to pass to left-invariant vector
fields, and then use the rule (\ref{Eqn6.13}) above to pass to $\tilde{\varphi}^{k -
2}_{s_{k - 2}}$, etc.

Putting this all together leads quickly to the following conclusion: the convolution
$$\tilde{\varphi}^{(k)}_{s_k} \ast \tilde{\varphi}^{(k-1)}_{s_{k-1}}
\cdots \ast \tilde{\varphi}^{(1)}_{s_1}$$ is a finite sum of expression
of the form 
\begin{equation}\label{Eqn6.14}
( s_\ast)^r \tilde{\psi}^{(k)}_{s_k} \ast \tilde{\psi}^{(k -
1)}_{s_{k-1}} \cdots \ast X^L_r ( \tilde{\psi}^{(1)}_{s_1} ) 
\end{equation}
for 
$r \geq k$. Here we use the fact that $s_\ast \geq s_j$, $j \leq k$.
Now consider
\begin{equation}\label{Eqn6.15}
\tilde{\varphi}^{(k)}_{s_k} \ast \tilde{\varphi}^{k -
1}_{s_{k-1}} \cdots \ast \tilde{\varphi}_{s_1} \ast K \ast \Phi_t.
\end{equation}

By applying the rule (\ref{Eqn6.3}) we can pass the left-invariant vector field $X^L_r$ as a
right-invariant vector field acting on $K \ast \Phi_t$.  We keep in mind
that $r \geq k$, and use Theorem \ref{Theorem6.9}, part (2).  Therefore
$\Phi^\ast_s \ast K \ast \Phi_t = \tilde{\varphi}^{(n)}_{s_n} \ast
\tilde{\varphi}^{(k - 1)}_{s_{k - 1}} \cdots \ast
\tilde{\varphi}^{(1)}_{s_1} \ast K \ast \Phi_t$ is majorized by a
constant multiple of $s^r_\ast \cdot t^{-r}_{k} \, A^\ast_s \ast
\Gamma_t$ where 
$A^\ast_s = \tilde{\eta}_{B^{(n)}_{s_n}} \ast
\tilde{\eta}_{B^{n - 1}_{s_{n-1}}} \cdots 
\ast \eta_{B^{(1)}_{s_1}}$.

In view of Theorem \ref{Theorem6.6} and inequality (\ref{Eqn6.6}$^\prime$) in Section \ref{SecMaxFns}, we get
that 
\begin{equation}\label{Eqn6.16}
| P_t T P^\ast_s ( f ) ( \x ) | \leq c \sup\limits_{r \geq k}
( s^r_\ast t^{-r}_k ) \, M ( M ( f ))(\x).
\end{equation}
We now pick $\sigma = 1/2k$.  
Since $s_\ast \leq s_k \rho^{\sigma (k-1)}$,
and $s_k \slash t_k = \rho^{-1}$, $\rho \geq 1$, we get 

\[ | P_t T P^\ast_s ( f ) ( x ) | \leq c \, \rho^{-1/2} \, 
M ( M ( f ) ) ( x ) \, .\]
Combining this with the previous case given by (\ref{Eqn6.12}) yields 
\[
| P_t T P^\ast_s ( f ) ( x ) | \leq c 
\left(
\frac{s_k}{t_k} \right)^{1/n} \, 
\mathcal{M} ( f ) ( x ), \ \mbox {if}  \ t_k
\slash s_k \geq 1.
\]
By a parallel argument the analogous result holds of $s_k \slash t_k
\geq 1$.  Hence 
\[
| P_t T P^\ast_s ( f ) ( x ) | \leq c \, \min \Big( \frac{s_{k}}{t_{k}},\,\frac{t_{k}}{s_{k}}\Big)^{\frac{1}{n}} \mathcal{M} ( f ) ( x ) \, .
\]
Since this holds for all $k$, $1 \leq k \leq n$, we can take the
geometric mean of these inequalities.  The result is conclusion (2) of
Theorem \ref{Theorem6.10}, with $\delta = 1 \slash n^2$.
\end{proof}

\smallskip

\subsection{Square functions and $L^p$-boundedness}\quad

\smallskip

We will construct the square functions for $G$ as products of the
(one-parameter) square functions of the sub-groups $G_k$, $1 \leq k \leq
n$. Each $G_{k}$ is a homogeneous group with family of dilations $\delta_{r}$, and so there exists a finite-dimensional inner-product space $V_k$ and a pair
$\varphi^{(k)}$, $\psi^{(k)}$ of $V_k$-valued functions, with
$\varphi^{(k)}\in \mathcal C^\infty_{0}(G_{k})$ supported in the unit ball,  and
$\psi^{(k)}\in \mathcal S(G_{k})$ a Schwartz function, so that
$\int_{G_k} \varphi^{(k)} (\x)\,d\x = 
\int_{G_k} \psi^{(k)}(\x)\,d\x = 0$, and
\begin{equation}\label{Eqn6.13w}
\int_{0}^{\infty} \psi^{(k)}_a ( \x
\y^{-1} ) \cdot \varphi^{(k)}_a ( \y ) \frac{da}{a} = \delta_0.
\end{equation}
Here $\varphi^{(k)}_{a} ( \x ) = a^{- Q_k - Q_{k+1} \ldots - Q_n} \varphi^{(k)}
( \delta_{a^{-1}} ( \x ))$, with a similar definition for $\psi^{(k)}_{a}
( \x )$. Also $\cdot $ denotes the inner product in $V_k$. See \cite{FoSt82},
Theorem 1.61.

We define operators $P^{(k)}_a$ and $Q^{(k)}_a$, acting on
functions on $G_k$, by  setting $P^{(k)}_a ( f ) = f \, \ast \, \varphi^{(k)}_a$ and 
$Q^{(k)}_a ( f ) = f \, \ast \, \psi^{(k)}_a$.
Note that (\ref{Eqn6.13w}) shows that 
\begin{equation}\label{Eqn6.14w}
\int_{0}^{\infty} P^{(k)}_{a} \cdot
Q^{(k)}_a  \frac{da}{a} = \text{Id}.
\end{equation}
Next, define the square functions $S_k$ and $S^\#_k$ by setting
\begin{align*}
S_k ( f ) ( \x ) \, &= \, \Big( \displaystyle{\int\limits_{0}^{\infty}} |
P^{(k)}_a ( f ) ( \x ) |^2 \frac{da}{a} \Big)^{1/2}, \\
 S^\#_k ( f ) ( \x ) \, &= \, \Big(
\displaystyle{\int\limits_{0}^{\infty}} | Q^{(k)}_{a} ( f ) ( \x ) |^2
\frac{da}{a} \Big)^{1/2} \, .
\end{align*}
The usual theory of singular integrals in \cite{St93} and \cite{FoSt82} 
together with (\ref{Eqn6.14w}) then gives the
inequalities 
\begin{equation}\label{Eqn6.15w}
\parallel f \parallel_{L^p} \approx \parallel S_k ( f )
\parallel_{L^p} \approx \parallel S^{\#}_{k} ( f ) \parallel_{L^p}
\end{equation}
for $1 <p<\infty$ on $G_{k}$.
The result is valid not only for
scalar-valued $f$, but also for $f$ that take their values in a Hilbert
space.

Having recalled the known results for $G_k$ we transfer them to the
whole group $G$ by writing
\begin{align*}
\tilde{P}^{(k)}_{a} ( f ) & =  f \, \ast\, 
( \delta_{x_1 \cdots x_{k-1}} \otimes
\varphi^{(k)}_{a} ) &
\tilde{Q}^{(k)}_{a} ( f )& =  f \, \ast\, 
( \delta_{x_1 \cdots x_{k-1}} \otimes
\psi^{(k)}_{a} ) \\ 
\tilde{S}_k ( f ) & =  \Big(
\displaystyle{\int\limits_{0}^{\infty}} | \tilde{P}^{(k)}_{a} ( f ) |^2
\frac{da}{a} \Big)^{1/2} &\tilde{S}^{\#}_{k} ( f ) & =  \Big( \displaystyle{\int\limits_{0}^{\infty}}
| \tilde{Q}^{(k)}_a ( f ) |^2 \frac{da}{a} \Big)^{1/2} \, .
\end{align*}
We get as a consequence 
\begin{align*}
\tag{\ref{Eqn6.14w}$^{\prime}$}\int\limits_{0}^{\infty}
\tilde{P}^{(k)}_{a} \cdot 
\tilde{Q}^{(k)}_{a} \frac{da}{a} & = \text{Id}
\end{align*}
on $G$, and \begin{align*}
\tag{\ref{Eqn6.15w}$^{\prime}$}
\parallel f \parallel_{L^p ( G )} &\approx \parallel
\tilde{S}_k ( f ) \parallel_{L^p ( G )} \approx \parallel \tilde{S}^{\#}
( f ) \parallel_{L^p ( G )}.
\end{align*}

With the above one-parameter theory arising from each $G_k$ we come to
the square functions on $G$ that are relevant for us.  For each $\t =
(t_1 , \ldots t_n ) \in ( \mathbb{R}^+)^n$ we set $P_\t =
\tilde{P}^{(n)}_{t_n} \cdot \tilde{P}^{(n - 1)}_{t_{n-1}} \cdot
\tilde{P}^{(1)}_{t_1}$.  That is, $P_\t ( f ) = f \ast \Phi_\t$, where
$\Phi_\t = \tilde{\varphi}^{(1)}_{t_1} \ast \tilde{\varphi}^{(2)}_{t_2}
\cdots \ast \tilde{\varphi}_{t_n}^{(n)}$ and
$\tilde{\varphi}^{(k)}_{t_k} = \delta_{x_1 \cdots x_{k - 1}} \otimes
\varphi^{(k)}_{t_k}$.  Note also that $\Phi_\t$ is a $V$-valued function,
where $V = V_1 \otimes V_2 \cdots \otimes V_n$.  Similarly, we define 
$P^{\ast}_\t = \tilde{P}^{(1)}_{t_1} \cdot
\tilde{P}^{(2)}_{t_2} \cdots \, \cdot \, 
\tilde{P}^{(n)}_{t_n} $,  $Q_t = \tilde{Q}^{(n)}_{t_n} \cdots
\, \cdot \tilde{Q}^{(1)}_{t_n}$, and 
$Q^\ast_\t = \tilde{Q}^{(1)}_{t_1} \cdots \, \cdot \, \tilde{Q}^{(n)}_{t_n}$.
Also $Q_t ( f ) = f \ast \bar{\psi}_t$, with $\bar{\psi}_\t =
\tilde{\psi}^{(1)}_{t_1} \ast \cdots \, \cdot \,
\tilde{\psi}^{(n)}_{t_n}$ and $\bar{\psi}_{\t}$ is also $V$-valued.
Finally, we set
\begin{align*}
S ( f ) ( x ) \, &= \, \Big(
\displaystyle{\int\limits_{( \mathbb{R}^+ )^n}} | P_t ( f ) |^2 \,
\frac{dt}{[t]} \Big)^{1/2} , &
\mathfrak{S} ( f ) ( x ) &= \Big(
\displaystyle{\int\limits_{( \mathbb{R}^+ )^n}} \, | \mathcal{M} Q_t f
|^2 \, 
\frac{dt}{[t]} \Big)^{1/2} \, .
\end{align*}
Here we use the abbreviation that $[t] = t_1 \cdot t_2 \cdots \, \cdot
t_n$.

\begin{lemma}
We have
\begin{enumerate}[{\rm(a)}]

\smallskip

\item $\displaystyle{\int\limits_{\left( \mathbb{R}^+ \right)^n}}
P^\ast_t Q_t \, \frac{dt}{[t]} \, = \, Id$;

\smallskip

\item $\parallel f \parallel_{L^p} \leq A_p \parallel S ( f )
\parallel_{L^p}$, $1 < p < \infty$;

\smallskip

\item $\parallel \mathfrak{S} ( f ) \parallel_{L^p} \leq A_p \parallel
f \parallel_{L^p}$, $1 < p < \infty$.
\end{enumerate}
\end{lemma}


To prove (a), we take first the identity (\ref{Eqn6.14w}) when $k = 1$, and $a =
t_1$.  Next we multiply on the left of both sides by $P^{(2)}_{t_2}$ and
on the right of both sides by $\tilde{Q}^{(2)}_{t_2}$ and integrate in
$t_2$, using (\ref{Eqn6.14w}$^\prime$) for $k = 2$.  Continuing this way yields (a).
The inequality (b) follows from repeated comparisons of the
corresponding inequalities (\ref{Eqn6.15w}$^\prime$) for $\tilde{S}_k$.  Also (c)
follows by applying (\ref{Eqn6.15w}$^\prime$) for $\tilde{S}_k^{\#}$ and 
a two-fold application of the
vector-valued maximal function in Corollary \ref{Cor6.4}, for $F_t ( x ) = Q_t (
f ) ( x ) \cdot [t]^{- 1/2}$.  

The final lemma needed is as follows.

\begin{lemma}\label{Lemma6.13w}
Suppose $\K$ is a flag kernel end $T ( f) = f \ast \K$. Then
\begin{equation}\label{Eqn6.16w}
S ( T ( f )) ( x ) \leq c \, \mathfrak{S} ( f ) ( x ).
\end{equation}
\end{lemma}


\begin{proof}  Now 
$P_t T ( f ) ( x ) = {\int_{( \mathbb{R}^+ )^n}} \,
P_t \cdot T P^\ast_s Q_s ( f ) ( x ) \, \frac{ds}{[s]}$
by part (a) of the previous lemma.  Hence Theorem \ref{Theorem6.10}, part (b) shows that
\[
| P_t T ( f ) ( x ) | \leq 
\displaystyle{\int\limits_{( \mathbb{R}^+)^n}}
\, \gamma ( s, t) \, \mathcal{M} ( Q_s ( f ) ) ( x ) \, \frac{ds}{[s]}
\, ,
\] where $\gamma (s, t ) = c 
\Big(
{\prod_{k = 1}^{n}} \, \min 
\Big(
\frac{s_k}{t_k} , 
\frac{t_k}{s_k}  \Big)\Big)^\delta$.  Thus $$| P_t T ( f ) ( x
) |^2 \leq {\int_{( \mathbb{R}^+ )^n}} \, \gamma
(s, t ) ( \mathcal{M} Q_s ( f ) ( x ))^2 \, \frac{ds}{[s]} \, \cdot \,
I ( t )$$
with $I ( t ) = {\int\limits_{( \mathbb{R}^+ )^n}} \,
\gamma (s, t ) \, \frac{ds}{[s]}$, by Schwarz's inequality.  But
$\sup\limits_{t} I ( t ) = A < \infty$, since 
$${\int\limits_{0}^{\infty}} \, \min \left( \frac{s_k}{t_k}
\, , \, \frac{t_k}{s_k} \right)^\delta \, \frac{ds_k}{s_k} \, = \,
t^{- \delta}_{k} \, \displaystyle{\int\limits_{0}^{t_k}} \, s_k^{-1 +
\delta} \, ds_k \, + \, t^\delta_k
\displaystyle{\int\limits_{t_k}^{\infty}} \, s^{-1-\delta}_{k} \, ds_k =
\frac{2}{\delta}.
$$
A further integration in $t$ (noting that also
${\int_{( \mathbb{R}^+)^n}} \, \gamma (s, t ) \,
\frac{dt}{[t]} \leq A$) then gives the desired result.
\end{proof}

\begin{theorem} \label{Theorem6.14}
With $T f = f \ast K$ as above, we
have $\parallel T f \parallel_{L^p ( G )} \leq A_p \parallel f
\parallel_{L^p ( G )}$, $1 < p < \infty$.
\end{theorem}  

This now follows directly
from (4), once we apply Lemma 14.1 part (b), for $Tf$ in place of $f$,
and then part (c) of that lemma.

\section{Appendix I: Lifting}\label{Lifting}

Suppose  that $T ( f) = f \ast \K$ is a convolution operator on $G_k$ with $f\in \mathcal S(G_{k})$ and 
 $\K \in \mathcal{S}^\prime ( G_k)$, a tempered distribution.  Then $Tf$ can be written  \,
$T ( f ) ( \x ) = \int_{G_k} \, \K ( \y ) \, f (
\x \y^{-1} ) \, d\y = \langle \K , F_\x \rangle$, \,
where $F_\x$ is the element of $\mathcal{S} (G_k)$ given by $F_\x ( \y ) = f (
\x \y^{-1} )$ for $\y \in G_k$.
We can lift $T$ to a convolution operator
on $G$, denoted by $\tilde{T}$, given by \,
$
\tilde{T} ( f ) ( \x )  =  \int_{G_k}  K ( \y
)  f ( \x\y^{-1}) \, d\y = \big\langle \K, F_{\x}\big\rangle
$, 
\,where  $f \in \mathcal{S} ( G )$, and $F_{\x}(\y) = f(\x\y^{-1})$ for $\y\in G_{k}$.

We describe this lifting in terms of the coordinate system used in
the proof of Theorem \ref{Theorem6.9}.  We can write each $\x \in G$ as a
product $\x = \x^\prime \cdot \bar{\x}$, with $\x^\prime = ( \x_1 , \ldots
\x_{k-1} , 0 , \ldots 0 )$ and $\bar{\x} \in G_k$, where $\bar{\x} = (0 ,
\ldots 0 , \x_k , \ldots \x_n)$.
With this coordinate system, we define $\tilde{K} \in \mathcal{S}^\prime ( G )$ as
$\delta_{\x^\prime} \otimes K$, where $\x = \x^\prime \cdot \bar{\x}$, and we set 
$
\tilde{T} ( f ) ( \x ) \, = \, ( f \ast \tilde{K} ) ( \x )$ for $f \in \mathcal{S} (G) $ and  $\x \in G$.
Then $\tilde{T} ( f ) ( \x ) \, = \, \tilde{T}(f) ( \x^\prime \cdot
\bar{\x})$.  However observe that $\tilde{T}(f) ( \x^\prime \cdot \bar{\x} ) \,
= \, T (f^{\x^\prime} ) ( \bar{\x})$, where $f^{\x^\prime}$ is the element
of $\mathcal{S} ( G_k )$ given by $f^{\x^\prime} ( \y ) = f ( \x^\prime
\cdot \y )$, $\y \in G_k$.
Therefore
\begin{equation}\label{Eqn7.1}
\tilde{T} ( f ) ( \x^\prime \cdot \bar{\x} ) \, = \, T
(f^{\x^\prime} ) ( \bar{\x})
\end{equation}


Next, suppose that $T$ satisfies the bound
\begin{equation}
\parallel T ( f ) \parallel_{L^p ( G_k)} \, \leq \, 
A \parallel f \parallel_{L^p ( G_k)}
\end{equation}
for each $f \in \mathcal{S} ( G )$.  Then applying this to $f =
f^{x^\prime}$ via (\ref{Eqn7.1}), (and assuming $p < \infty$), gives 
\[
\displaystyle{\int\limits_{G_k}} \, | \tilde{T} ( f ) (x^\prime \bar{x}
) |^p \, d\bar{x} \, \leq \, A^p \, \displaystyle{\int\limits_{G_k}} \,
| f^{x^\prime} ( \bar{x} )|^p \, d \bar{x} \, ,
\] 
for each $x^\prime$,
and an integration in $x^\prime$ yields
\begin{equation}
\parallel T ( f ) \parallel_{L^p ( G )} \, \leq \, A
\parallel f \parallel_{L^p ( G )}.
\end{equation}

Suppose next that $K$ depends on a parameter $t$, $K = K_t$ and set $T_t
( f ) = f \ast K_t$.  There the same argument shows that 
\[
\parallel \sup\limits_{t} \, | \tilde{T}_t ( f ) | \parallel_{L^p ( G)}
\, \leq \, A \, \parallel f \parallel_{L^p ( G )} \, , \ \mbox{for all}
\ f \in \mathcal{S} ( G )
\]
whenever
\[
\parallel \sup\limits_{t} | T_t ( f ) | \parallel_{L^p ( G_k)} \, \leq \,
A \parallel f \parallel_{L^p ( G_k)} \, , \ \mbox{for all} \ f \in
\mathcal{S} ( G_k) \, .
\]

This proves that the lifted maximal inequality $\parallel \tilde{M}_k (
f ) \parallel_{L^p ( G)} \, \leq \, A_p \parallel f \parallel_{L^p ( G
)}$ follows from the corresponding inequality on $G_k$, by considering
first the case when $f$ is non-negative, (and 
$K_t = \eta_{B_t}^{(k)}$),
and then by replacing $f$ by $| f |$.

In the same way the vector-valued maximal inequality (\ref{Eqn6.4}) on $G_k$ can be
lifted to the corresponding inequality (\ref{Eqn6.4}$^\prime$) on $G$.
In fact, it suffices to prove (\ref{Eqn6.4}$^\prime$) when there are only $m$
non-zero $f_j$'s, $1 \leq j \leq m$, with bounds independent of $m$.
With this understanding, set 
\begin{align*}
f &= ( f_1 , \ldots f_m ),&T_t ( f ) &= ( T^1_{t_1}
( f_1) , \ldots T^m_{t_m} ( f_m) ),&| f | &= \Big( \displaystyle{\sum\limits_{j= 1}^{m}} \, |
f_j |^2 \Big)^{1/2},\\
t &= (t_1 \ldots t_m),&T^j_{t_j} ( f ) &= f \ast
K^j_{t_j},&| T_t ( f ) | &= \Big(
\displaystyle{\sum\limits_{j = 1}^m} \, | T^j_{t_j} ( f_j ) |^2
\Big)^{1/2}.
\end{align*}
Note that $\Big( \displaystyle{\sum\limits_{j = 1}^{m}}
\, \sup\limits_{t_j} \, | T^j_{t_j} ( f ) |^2 \Big)^{1/2} = \:
\sup\limits_{t} \, |T_t| ( f ) |$.
Then as before the inequality 
\[
\parallel \sup\limits_{t} | T_{t} ( f ) | \parallel_{L^p ( G_n)} \, \leq
\, A \parallel | f | \parallel_{L^p ( G_k)} \, , \, f \in \mathcal{S} (
G_n)
\]
implies the corresponding inequality for $T_t$ lifted to $G$, that is
for $\tilde{T}_t$, and this then yields the desired result.

We should remark that the lifting procedure used here can be viewed in
terms of the more abstract ``transference'' method presented in \cite{CoifmanWeiss}.

\section{Appendix II: An estimate for a geometric sum}

Recall that $\mathcal E_{n}= \{I=(i_{1}, \ldots, i_{n})\in \Z^{n}\,\big\vert\,i_{1}\leq i_{2}\leq \cdots \leq i_{n}\}$.  More generally, if $0\leq B_{1}\leq B_{2}\leq \cdots \leq B_{n}$, let
\begin{equation}\label{E2.9}
\Lambda(B) = \left\{I=(i_{1}, \ldots, i_{n})\in \mathcal E_{n}\,\big\vert\,B_{j}\leq 2^{i_{j}}\quad\text{for}\quad 1 \leq j \leq n\right\}.
\end{equation}
We establish the following estimate for geometric sums which we shall use several times. 

\smallskip

\begin{proposition}\label{Prop2.11}
Let $\alpha_{j}>0$ for $1 \leq j \leq n$, and let $M >\sum_{k=1}^{n}\alpha_{k}$. There is a constant $C$ depending on $n$, on $M$, and on the numbers $\{\alpha_{j}\}$ so that for any $A_{1}, \ldots , A_{n}\in (0,\infty)$ and any $0\leq B_{1}\leq B_{2}\leq \cdots \leq B_{n}$, 
\begin{align}\label{E2.10}
\sum_{I\in \mathcal E_{n}}\frac{\prod_{k=1}^{n}(2^{-i_{k}})^{\alpha_{k}}}{\big(1+\sum_{k=1}^{n}2^{-i_{k}}A_{k}\big)^{M}} &\leq C\,\prod_{j=1}^{n}(A_{1}+A_{2}+\cdots + A_{j})^{-\alpha_{j}},\\\notag\\\label{E2.11}
\sum_{I\in \Lambda(B)}\frac{\prod_{k=1}^{n}(2^{-i_{k}})^{\alpha_{k}}}{\big(1+\sum_{k=1}^{n}2^{-i_{k}}A_{k}\big)^{M}}&\leq C\,\prod_{j=1}^{n}(A_{1}+A_{2}+\cdots + A_{j}+B_{j})^{-\alpha_{j}}.
\end{align}
\end{proposition}

\begin{proof}
Note that if we take $B_{1}= \cdots = B_{n}=0$, then the inequality (\ref{E2.11}) gives the inequality (\ref{E2.10}). If for $x_{1}, \ldots, x_{n}\geq 0$ we put
\begin{equation*}
\varphi(x_{1}, \ldots, x_{n})=\Big(\prod_{k=1}^{n}x_{k}^{\alpha_{k}}\Big)\Big(1+\sum_{k=1}^{n}x_{k}A_{k}\Big)^{-M},
\end{equation*}
then if $1\leq s_{j}\leq 2$ for $1 \leq j \leq n$, we have
\begin{align*}
\varphi(s_{1}x_{1}, \ldots, s_{n}x_{n}) 
& 
\leq 2^{|\alpha|}\prod_{k=1}^{n}x_{k}^{\alpha_{k}}\Big(1+\sum_{k=1}^{n}x_{k}A_{k}\Big)^{-M}
\leq 2^{|\alpha|}\varphi(x_{1}, \ldots, x_{n})
\end{align*}
For each $I=(i_{1}, \ldots, i_{n})\in \Z^{n}$, let $Q_{I}= \{\x\in\R^{n}\,\big\vert\, \frac{1}{2}2^{-i_{k}}\leq x_{k}<2^{-i_{k}}\}$. It follows that there is a constant $C= C(n,\alpha)$ depending only on $n$ and $\alpha$ such that 
\begin{equation*}
C^{-1}\int_{Q_{I}}\varphi(\x)\,\frac{dx_{1}\cdots dx_{n}}{x_{1}\cdots x_{n}}\leq 
\frac{\prod_{k=1}^{n}\big(2^{-i_{k}}\big)^{\alpha_{k}}}{ \big(1+\sum_{k=1}^{n}2^{-i_{k}}A_{k}\big)^{M}}
\leq C\,\int_{Q_{I}}\varphi(\x)\,\frac{dx_{1}\cdots dx_{n}}{x_{1}\cdots x_{n}}.
\end{equation*}
Thus 
\begin{equation*}
\sum_{I\in \Lambda(B)}\frac{\prod_{k=1}^{n}(2^{-i_{k}})^{\alpha_{k}}}{\big(1+\sum_{k=1}^{n}2^{-i_{k}}A_{k}\big)^{M}}\leq C\int\limits_{\bigcup_{I\in \Lambda(B)}Q_{I}}\varphi(\x)\,\frac{dx_{1}\cdots dx_{n}}{x_{1}\cdots x_{n}}.
\end{equation*}
On the other hand, it is easy to check that if $I\in \Lambda(B)$ and $\x\in Q_{I}$, then $0<x_{k+1}\leq 2 x_{k}$ for $1\leq k \leq n-1$ and $x_{k}\leq B_{k}^{-1}$ for $1 \leq k \leq n$. Thus if we put
\begin{equation*}
\Omega(B) =\Bigg\{\x\in \R^{n}_{+}\,:\, 
\begin{cases}
\frac{1}{2}x_{k+1}\leq x_{k}\leq B_{k}^{-1}&\text{for $1\leq k \leq n-1$}\\\\
0\leq x_{n}\leq B_{n}^{-1}
\end{cases}
\Bigg\},
\end{equation*}
then $\bigcup_{I\in \Lambda(B)}Q_{I}\subseteq \Omega(B)$. Thus to prove the Proposition, it suffices to show that
\begin{equation*}
\begin{aligned}
\int\limits_{\Omega(B)}\frac{\prod_{k=1}^{n}x_{k}^{\alpha_{k}}}{\big(1+\sum_{k=1}^{n}x_{k}A_{k}\big)^{M}}&\frac{dx_{1}\cdots dx_{n}}{x_{1}\cdots x_{n}}
\leq C\,\prod_{j=1}^{n}(A_{1}+A_{2}+\cdots + A_{j}+B_{j})^{-\alpha_{j}}.
\end{aligned}
\end{equation*}
However,
\begin{equation*}
\begin{aligned}
\int\limits_{\Omega(B)}\prod_{k=1}^{n}x_{k}^{\alpha_{k}}\Big(1+\sum_{k=1}^{n}x_{k}A_{k}\Big)^{-M}&\frac{dx_{1}\cdots dx_{n}}{x_{1}\cdots x_{n}}\\
&=
\int\limits_{0}^{\infty}t^{M}e^{-t}\Big[\int\limits_{\Omega(B)}\prod_{k=1}^{n}x_{k}^{\alpha_{k}}e^{-x_{k}A_{k}t}\frac{dx_{1}\cdots dx_{n}}{x_{1}\cdots x_{n}}\Big]\,\frac{dt}{t}.
\end{aligned}
\end{equation*}
We will show  that we can estimate the inner integral on the right hand side by
\begin{equation}\label{E2.12}
\begin{aligned}
\int_{\Omega(B)}\prod_{k=1}^{n}&x_{k}^{\alpha_{k}}e^{-x_{k}A_{k}t}\frac{dx_{1}\cdots dx_{n}}{x_{1}\cdots x_{n}}\\&\leq C(n,\alpha)\prod_{j=1}^{n}(1+t^{-\alpha_{j}})(A_{1}+A_{2}+\cdots + A_{j}+B_{j})^{-\alpha_{j}},
\end{aligned}
\end{equation}
and since $M>\sum_{k=1}^{n}\alpha_{k}$, this will complete the proof.

To establish (\ref{E2.12}), we first establish an estimate for $\int_{\frac{x}{2}}^{B^{-1}}s^{\alpha}e^{-sAt}\,\frac{ds}{s}$. On the one hand, we have $$\int_{\frac{x}{2}}^{B^{-1}}s^{\alpha}e^{-sAt}\,\frac{ds}{s}\leq e^{-\frac{1}{2}xAt}\int_{0}^{B^{-1}}s^{\alpha-1}\,ds = \alpha^{-1}e^{-\frac{1}{2}xAt} B^{-\alpha}.$$ On the other hand we  have $$\int_{\frac{x}{2}}^{B^{-1}}s^{\alpha}e^{-sAt}\,\frac{ds}{s}
\leq (At)^{-\alpha}\int_{\frac{1}{2}xAt}^{\infty}s^{\alpha-1}e^{-s}ds
\leq C_{\alpha} (At)^{-\alpha}e^{-\frac{1}{4}xAt}.$$ Putting the two together, we have the estimate
\begin{equation}\label{E2.13}
\int_{\frac{x}{2}}^{B^{-1}}s^{\alpha}e^{-sAt}\,\frac{ds}{s} \leq C_{\alpha}e^{-\frac{1}{4}xAt}(At+B)^{-\alpha}\leq C_{\alpha}\,e^{-\frac{1}{4}xAt}(A+B)^{-\alpha}(1+t^{-\alpha}).
\end{equation}

We now establish (\ref{E2.12}) by induction on $n$. When $n=1$, we use (\ref{E2.13}) with $x=0$ to get
\begin{equation*}
\begin{aligned}
\int_{0}^{B_{1}^{-1}}x_{1}^{\alpha_{1}}e^{-\alpha_{1}A_{1}t}\,\frac{dx_{1}}{x_{1}}\leq C_{\alpha}(A_{1}+B_{1})^{-\alpha_{1}}(1+t^{-\alpha_{1}}).
\end{aligned}
\end{equation*}
For the induction step, we have
\begin{equation*}
\begin{aligned}
\int_{\Omega(B)}\prod_{k=1}^{n}&x_{k}^{\alpha_{k}}e^{-x_{k}A_{k}t}\frac{dx_{1}\cdots dx_{n}}{x_{1}\cdots x_{n}}\\
&=
\int_{\Omega'(B)}\prod_{k=2}^{n}x_{k}^{\alpha_{k}}e^{-x_{k}A_{k}t}\Big[
\int_{\frac{x_{2}}{2}}^{B_{1}^{-1}}x_{1}^{\alpha_{1}}e^{-\alpha_{1}A_{1}t}\,\frac{dx_{1}}{x_{1}}\Big]\frac{dx_{2}\cdots dx_{n}}{x_{2}\cdots x_{n}}\\
&\leq
C_{\alpha}(A_{1}+B_{1})^{-\alpha_{1}}(1+t^{-\alpha_{1}})\int_{\Omega'(B)}\prod_{k=2}^{n}x_{k}^{\alpha_{k}}e^{-x_{k}A_{k}t}e^{-\frac{1}{4}x_{2}A_{1}t}\frac{dx_{2}\cdots dx_{n}}{x_{2}\cdots x_{n}}.
\end{aligned}
\end{equation*}
Here \,$\displaystyle \Omega'(B) =\Bigg\{\x\in \R^{n}_{+}\,:\, 
\begin{cases}
\frac{1}{2}x_{k+1}\leq x_{k}\leq B_{k}^{-1}&\text{for $2\leq k \leq n-1$}\\
0\leq x_{n}\leq B_{n}^{-1}
\end{cases}
\Bigg\}$, and we have used the estimate in (\ref{E2.13}). The last integral on the right-hand side is thus of the same form as the original integral, except that $n$ has been replaced by $n-1$, and $A_{2}$ has been replaced by $A_{2}+\frac{1}{4}A_{1}$. We can thus use our inductive hypothesis on this integral, and we obtain the desired estimate. This completes the proof.
\end{proof}

}

\end{document}